%% file: article.tex
\documentclass[11pt,leqno]{article}

\usepackage{macros}

\usepackage{mathrsfs}

\usepackage{makeidx}

\usepackage[]{hyperref}
\hypersetup{
    colorlinks=true,       
    linkcolor=black,          
    citecolor=black,        
}
\usepackage[disable]{todonotes}

\usepackage{amsmath}
\usepackage{amsfonts,amssymb}
\usepackage{enumerate}
\usepackage{amsthm}
\usepackage{a4wide}
\usepackage{epsfig}
\usepackage{tikz}
\usepackage{authblk}

\theoremstyle{plain}
\newtheorem*{main}{Main result}
\newtheorem*{prop*}{Proposition}
\newtheorem{theo}{Theorem}[section]
\newtheorem{prop}[theo]{Proposition}
\newtheorem{lemm}[theo]{Lemma}
\newtheorem{coro}[theo]{Corollary}

\newtheorem{defi}[theo]{Definition}
\theoremstyle{definition}
\newtheorem{rema}[theo]{Remark}
\newtheorem*{rema*}{Remark}
\newtheorem*{remas}{Remarks}

\newtheorem*{nota*}{Notation}
\newtheorem{exam}[theo]{Example}
\DeclareMathOperator{\cnx}{div}
\DeclareMathOperator{\cn}{div}

\DeclareMathOperator{\diff}{d}

\DeclareMathOperator{\Op}{Op}

\numberwithin{equation}{section}

\title{\huge{Global solutions and asymptotic behavior \\ 
for two dimensional gravity water waves}}
\author{Thomas Alazard\\D\'epartement de Math\'ematiques et Applications\\
 \'Ecole normale sup\'erieure et CNRS UMR 8553\\ 
 45, rue d'Ulm\\
 F-75230 Paris, France \and Jean-Marc Delort\\ Universit\'e Paris 13\\
Sorbonne Paris Cit\'e, LAGA, CNRS (UMR 7539)\\
 99, avenue J.-B. Cl\'ement\\
F-93430 Villetaneuse}
\date{}

\begin{document}


\maketitle


\parskip=10pt 
\lineskip=4pt

\def\indexname{Index of notations}


\begin{abstract}
 This paper is devoted to the proof of a global existence result for the water waves equation with smooth, small, and
 decaying at infinity Cauchy data. We obtain moreover an asymptotic description in physical coordinates of the solution,
 which shows that modified scattering holds. 

The proof is based on a bootstrap argument involving $L^2$ and $L^\infty$ estimates. The $L^2$ bounds are proved in the
companion paper~\cite{AlDel} of this article. They rely on a normal forms paradifferential method allowing one to obtain energy estimates on the Eulerian
formulation of the water waves equation. We give here the proof of the  uniform bounds, interpreting the equation in a
semi-classical way, and combining Klainerman vector fields with the description of the solution in terms of semi-classical
lagrangian distributions. This, together with the $L^2$ estimates of~\cite{AlDel}, allows us to deduce our main global existence result.
\end{abstract}


\input{introduction}

\pagebreak
\input{chapitre1-v2}

\input{chapitre3}

\appendix

\input{appendix2}

\addcontentsline{toc}{chapter}{Bibliography}



\end{document}

%% file: introduction.tex
\section*{Introduction}
\renewcommand{\theequation}{\arabic{equation}}
\renewcommand{\thesubsection}{\arabic{subsection}}

\subsection{Main result}

Consider an homogeneous and incompressible fluid in a gravity field, 
occupying a time-dependent domain with a free surface. 
We assume that the motion is the same in every vertical section and consider the two-dimensional 
motion in one such section. At time~$t$, the fluid domain, denoted by 
$\Omega(t)$, is therefore a two-dimensional domain. We assume that its boundary 
is a free surface 
described by the equation~$y=\eta(t,x)$, so that
$$
\Omega(t)=\left\{\, (x,y)\in \xR\times \xR\,;\, y<\eta(t,x)\,\right\}.
$$
The velocity field is assumed to satisfy the incompressible Euler equations. 
Moreover, the fluid motion is assumed to have been generated from rest 
by conservative forces and is 
therefore irrotational in character. It follows that the 
velocity field~$v\colon \Omega \rightarrow \xR^{2}$ 
is given by~$v=\nabla_{x,y} \phi$ 
for some velocity potential~$\phi\colon \Omega\rightarrow \xR$ satisfying 
\begin{equation}\label{intro:1}
\Delta_{x,y}\phi=0,\quad 
\partial_{t} \phi +\mez \la \nabla_{x,y}\phi\ra^2 + P +g y = 0,
\end{equation}
where~$g$ is the modulus of the acceleration of gravity ($g>0$) and where~$P$ is the pressure term. 
Hereafter, the units of length and time are chosen so that~$g=1$. 

The problem is then given by two boundary conditions on the free surface:
\begin{equation}\label{intro:2}
\left\{
\begin{aligned}
&\partial_{t} \eta = \sqrt{1+(\px \eta)^2}\, \partial_n \phi 
&&\text{on }\partial\Omega, \\
& P=0
&&\text{on }\partial\Omega,
\end{aligned}
\right.
\end{equation}
where~$\partial_n$ is the outward normal derivative of~$\Omega$, so that 
$\sqrt{1+(\px \eta)^2}\, \partial_n \phi =\partial_y\phi-(\px\eta)\px\phi$. The former condition 
expresses that the velocity of the free surface coincides with the one of the fluid particles. 
The latter condition is a balance of forces across the free surface.  

Following Zakharov~\cite{Zakharov1968} and Craig and Sulem~\cite{CrSu}, we 
work with the trace of~$\phi$ at the free boundary
$$
\psi(t,x)\defn\phi(t,x,\eta(t,x)).
$$
To form a system of two evolution equations for~$\eta$ and 
$\psi$, one introduces the Dirichlet-Neumann operator~$G(\eta)$ 
that relates~$\psi$ to the normal derivative 
$\partial_n\phi$ of the potential by 
$$
(G(\eta) \psi)  (t,x)=\sqrt{1+(\partial_x\eta)^2}\,
\partial _n \phi\arrowvert_{y=\eta(t,x)}.
$$
(This definition is made precise in the first section of the companion paper~\cite{AlDel}. See proposition~\ref{ref:118} below).
Then~$(\eta,\psi)$ solves (see~\cite{CrSu}) 
the so-called Craig--Sulem--Zakharov system
\begin{equation}\label{intro:3}
\left\{
\begin{aligned}
&\partial_t \eta=G(\eta)\psi,\\
&\partial_t \psi + \eta+ \frac{1}{2}(\partial_x \psi)^2
-\frac{1}{2(1+(\partial_x\eta)^2)}\bigl(G(\eta)\psi+(\partial_x  \eta )(\partial_x \psi)\bigr)^2= 0.
\end{aligned}
\right.
\end{equation}
In \cite{Bertinoro}, it is proved that 
if $(\eta,\psi)$ is a classical solution of \e{intro:3}, such that $(\eta,\psi)$ belongs 
to $C^0([0,T];H^s(\xR))$ for some $T>0$ and $s>3/2$, then one can define a velocity potential $\phi$ and 
a pressure $P$ satisfying \e{intro:1} and \e{intro:2}. Thus it is sufficient to solve the 
Craig--Sulem--Zakharov formulation of the water waves equations. 

Our main result is stated in full generality in the first section of this paper. A weaker statement is the following:

\begin{main}
For small enough initial data of size $\eps\ll1$,  sufficiently decaying at infinity, the Cauchy problem 
for \e{intro:3} is globally in time well-posed. Moreover, $u=\Dxmez \psi+i\eta$ admits the following 
asymptotic expansion as $t$ goes to $+\infty$: 
There is a continuous function $\underline{\alpha}\colon \xR
\rightarrow \xC$, depending of $\eps$ but 
bounded uniformly in $\eps$, such that
$$
u(t,x)=\frac{\eps}{\sqrt{t}}\underline{\alpha}\Bigl(\frac{x}{t}\Bigr)
\exp\Bigl(\frac{it}{4|x/t|}+\frac{i\eps^2}{64}\frac{\la \underline{\alpha}(x/t)\ra^2}{\la x/t\ra^5}\log (t)\Bigr)
+\eps t^{-\mez-\kappa}\rho(t,x)
$$
where $\kappa$ is some positive number and $\rho$ 
is a function uniformly bounded for $t\ge 1$, $\eps\in ]0,\eps_0]$.
\end{main}

As an example of small enough initial data sufficiently decaying at infinity, consider
\be\label{intro:4}
\eta\arrowvert_{t=1}=\eps \eta_0,\quad 
\psi\arrowvert_{t=1}=\eps \psi_{0},
\ee
with $\eta_0, \psi_0$ in $C^\infty_0(\R)$. 
Then there exists a unique solution 
$(\eta,\psi)$ in $C^\infty([1,+\infty[;H^\infty(\xR))$ of \e{intro:3}. In fact, in Theorem~\ref{ref:122} we allow 
$\psi$ to be merely in some homogeneous Sobolev space.

The strategy of the proof will be 
explained in the following sections of this introduction. 
We discuss at the end of this paragraph some related previous works. 

For the equations obtained by neglecting the nonlinear terms, the computation of the asymptotic behavior of the solutions 
was performed by Cauchy \cite{Cauchy} who computed the phase of oscillations. The reader is 
referred to \cite{Darrigol} and \cite{Craik} 
for many historical comments on Cauchy's memoir.

Many results have been obtained in the study of 
the Cauchy problem for the water waves equations, 
starting from the pioneering work  
of Nalimov~\cite{Nalimov} who proved that the Cauchy problem is well-posed 
locally in time, in the framework of Sobolev spaces, 
under an additional smallness assumption 
on the data. 
We also refer the reader to Shinbrot~\cite{Shinbrot}, Yoshihara~\cite{Yosihara} and 
Craig~\cite{Craig1985}. 
Without smallness assumptions on the data, 
the well-posedness of the Cauchy problem     
was first proved by Wu for the case without surface 
tension (see~\cite{WuInvent,WuJAMS}) and by Beyer-G\"unther in~\cite{BG} 
in the case with surface tension. 
Several extensions of their results have been obtained and 
we refer the reader to C{\'o}rdoba, C{\'o}rdoba and Gancedo~\cite{CCG}, Coutand-Shkoller~\cite{CS}, 
 Lannes~\cite{LannesJAMS,LannesKelvin,LannesLivre}, Linblad~\cite{LindbladAnnals}, Masmoudi-Rousset~\cite{MR} and Shatah-Zeng \cite{SZ,SZ2} for recent results on the Cauchy problem for the gravity water waves equation. 
 
Our proof of global existence is based on the analysis of the 
Eulerian formulation of the water waves equations by means of microlocal 
analysis. In particular, the energy estimates discussed in~\cite{AlDel} are influenced by the papers by Lannes~\cite{LannesJAMS} 
and Iooss-Plotnikov~\cite{IP} and  follow the paradifferential analysis 
introduced in \cite{AM} and further developed in \cite{ABZ1,ABZ3}. 

\smallbreak
 
It is worth recalling that the only known 
coercive quantity for \e{intro:3} is the 
hamiltonian, which reads (see~\cite{Zakharov1968,CrSu})
\be\label{i10}
\mathcal{H}=\mez \int \eta^2 \,dx+\mez \int \psi G(\eta)\psi \, dx.
\ee
We refer to the paper by Benjamin and Olver \cite{BO} for considerations on the conservation laws of the water waves equations. 
One can compare the hamiltonian with the critical threshold given by the scaling invariance of the equations. 
Recall (see~\cite{BO,CMSW}) that 
if $(\eta,\psi)$ solves \e{intro:3}, then the functions $(\eta_\lambda,\psi_\lambda)$ defined by
\be\label{i11}
\eta_\lambda(t,x)=\lambda^{-2}\eta\left(\lambda t, \lambda^2 x\right),\quad \psi_\lambda(t,x)=\lambda^{-3}\psi\left(\lambda t,\lambda^2 x\right)
\qquad (\lambda>0)
\ee
are also solutions of \e{intro:3}. In particular, one notices that the critical space for the scaling corresponds to $\eta_0$ in $\dot{H}^{3/2}(\xR)$. 
Since the hamiltonian~\e{i10} only controls the $L^2(\xR)$-norm of $\eta$, one 
sees that the hamiltonian is highly supercritical for the water waves equation and hence one cannot  use it 
directly to prove global well-posedness of the Cauchy problem.

\smallbreak

Given $\eps \ge 0$, consider the 
solutions to the water waves system \e{intro:3} with initial data satisfying \e{intro:4}. 
In her breakthrough 
result \cite{Wu09}, Wu 
proved that the maximal time of existence  $T_\eps$ is larger or equal to $e^{c/\eps}$ for $d=1$. 
Then Germain--Masmoudi--Shatah~\cite{GMS} and Wu~\cite{Wu10} have shown 
that the Cauchy problem for three-dimensional waves 
is globally in time well-posed for $\eps$ small enough 
(with linear scattering in Germain-Masmoudi-Shatah and no assumption about the 
decay to $0$ at spatial infinity of $\Dxmez \psi$ in Wu). Germain--Masmoudi--Shatah recently proved global existence 
for pure capillary waves in dimension $d=2$ in \cite{GMS2}. 

There is at least one other case where the global existence of solutions is now understood, namely 
for the equations with viscosity (see~\cite{Beale}, \cite{GT} and the references therein). 
Then global well-posedness is obtained by using the dissipation of energy. 
Without viscosity, the analysis of global well-posedness is based on dispersive estimates. Our 
approach follows a variant of the vector fields method introduced by Klainerman in \cite{Kl,Kl2} 
to study the wave and Klein-Gordon equations  
(see the book by H\"ormander in \cite{Hormander} or the Bourbaki seminar by Lannes~\cite{LannesBourbaki} 
for an introduction to this method). More precisely, as it is discussed later in this introduction, we shall follow the approach introduced in \cite{Delort} 
for the analysis of the Klein-Gordon equation in space dimension one, to cope with the fact that solutions of the equation do
not scatter. Results for one dimensional Schr\"odinger equations, that display the same non scattering behavior, have been
proved by Hayashi and Naumkin~\cite{HayashiNaumkin}, and global existence for a simplified model of the water waves equation
studied by Ionescu and Pusateri in~\cite{IonescuPusateri0}.

Let us discuss two other questions related to our analysis : the possible 
emergence of singularities in finite time and the existence of solitary waves. 

An important question is to determine 
whether the lifespan could be finite. 
Castro, C\'ordoba, Fefferman, Gancedo 
and G\'omez-Serrano conjecture (see~\cite{CCFGGS-PNAS}) that 
blow-up in finite time is possible for some initial data. 
It is conjectured in \cite{CCFGGS-PNAS} that 
there exists at least one water-wave solution such that, 
at time $0$, the fluid interface is a graph, 
at a later time $t_1>0$ the fluid interface is not a graph, and, 
at a later time $t_2>t_1$, the fluid self-intersects. Notice that, according to this 
conjecture, one does not expect global well-posedness for arbitrarily large initial data. 
One can quote several results supporting this conjecture 
(see \cite{CCFGGS-arxiv,CCFGLF,CS2}). 
In \cite{CCFGGS-arxiv} (see also~\cite{CS2}), the authors prove the following result: 
there exists an initial data such the free surface is a self-intersecting curve, and such that solving 
backward in time the Cauchy problem, one obtains for small enough negative times a non self-intersecting 
curve of $\xR^2$. 
On the other hand, it was conjectured 
that there is no blow-up in finite time 
for small enough, sufficiently decaying initial data (see the survey paper by Craig and Wayne \cite{CW}). 

Our main result precludes the existence of solitary waves sufficiently small and sufficiently decaying at infinity. 
In this direction, notice that Sun \cite{Sun} 
has shown that in infinitely deep water, no two-dimensional solitary water waves exist. 
For further comments and references on solitary waves, we refer the 
reader to~\cite{Craig02}
as well as to \cite{BGSW,Hur1,MRT} for recent results.

We refer the reader to \cite{ABZ1,ABZ2,CHS} for the study of other 
dispersive properties of the water waves equations 
(Strichartz estimates and smoothing effect of surface tension). 

Finally, let us mention that Ionescu and Pusateri~\cite{IonescuPusateri} independently obtained  a  global existence result very similar to the one
we get here. The main difference is that they assume less decay on the initial data, and get asymptotics not for the solution
in physical space, with control of the remainders in $L^\infty$, but for its space Fourier transform, with remainders in
$L^2$. These asymptotics, as well as ours, show that solutions do not scatter. 
To get asymptotics with remainders estimated in $L^\infty$, we shall 
commute iterated vector field $Z=t\partial_t+2x\partial_x$ to the water waves equations. 
This introduces several new difficulties and requires that the initial data be sufficiently decaying at 
infinity.   

\subsection{General strategy of proof}\label{Sintro:GSP}

Let us describe our general strategy, the difficulties one has to cope with, and the ideas used to overcome them. The general
framework we use is the one of Klainerman vector fields. Consider as a (much) simplified model an equation of the form
\begin{equation}
  \label{eq:1}\begin{split}
  (D_t-P(D_x))u &= N(u)\\
u\vert_{t=1} &= \eps u_0,\end{split}
\end{equation}
where $D_t = \frac{1}{i}\frac{\partial}{\partial t}$, $P(\xi)$ is a real valued symbol (for the linearized water waves
equation, $P(\xi)$ would be $\lvert\xi\rvert^{1/2}$), and $N(u)$ is a nonlinearity vanishing  at least at order two at
zero. Recall that a Klainerman vector field for $D_t - P(D_x)$ is a space-time vector field $Z$ such that $[Z,D_t-P(D_x)]$ is
zero (or a multiple of $D_t - P(D_x)$). 
For the water waves system, $Z$ will be $t\partial_t+2x\partial_x$ or $D_x$. In that way,
$(D_t - P(D_x))Z^k u = Z^kN(u)$ for any $k$, and since $P(\xi)$ is real valued, an easy energy inequality shows that
\begin{equation}
  \label{eq:2}
  \norm{Z^k u(t,\cdot)}_{L^2} \leq \norm{Z^k u(1,\cdot)}_{L^2} + \int_1^t \norm{Z^k N(u)(\tau,\cdot)}_{L^2}\,d\tau,
\end{equation}
for any $t\geq 1$. Assume first that $N(u)$ is cubic, so that
\begin{equation}
  \label{eq:3}
  \norm{Z^k N(u)}_{L^2} \leq C\norm{u}_{L^\infty}^2 \norm{Z^ku}_{L^2} + C\sum_{\substack{k_1+k_2+k_3\leq k\\ k_1,
      k_2\leq k_3\leq k-1}} \norm{Z^{k_1}u}_{L^\infty}\norm{Z^{k_2}u}_{L^\infty}\norm{Z^{k_3}u}_{L^2}.
\end{equation}
Assuming an a priori $L^\infty$ bound, one can deduce from \e{eq:2} an $L^2$ estimate. More precisely, introduce the
following property, where $s$ is a large even integer:
\[\begin{split}
\textrm{ For } t \textrm{ in some interval } [1,T[,\ \norm{u(t,\cdot)}_{L^\infty} = O(\eps/\sqrt{t}) \\
\textrm{ and for }
k = 0,\dots,s/2,\ \norm{Z^k u(t,\cdot)}_{L^\infty} = O\bigl(\eps t^{-\frac{1}{2} + \tilde\delta'_k}\bigr),
\end{split}\tag{A}\]
where $\tilde\delta'_k$ are small positive numbers. Plugging these a priori bounds in \e{eq:2}, \e{eq:3}, we get
\begin{equation}\label{eq:4}
\ba
\norm{Z^k u(t,\cdot)}_{L^2} \leq \norm{Z^k u(1,\cdot)}_{L^2} &+ C\eps^2\int_1^t  \norm{Z^k
   u(\tau,\cdot)}_{L^2}\,\frac{d\tau}{\tau}\\
 &+ C\eps^2\int_1^t \norm{Z^{k-1}
   u(\tau,\cdot)}_{L^2}\tau^{2\tilde{\delta}'_{k/2}-1}\,d\tau.
\ea
\end{equation}
Gronwall inequality implies then that
\[\norm{Z^k u(t,\cdot)}_{L^2} = O(\eps t^{\delta_k}), \ k\leq s,\tag{B}\]
for some small $\delta_k>0$ ($\delta_k>C\eps^2$ and $\delta_k>2\tilde\delta'_{k/2}$).

The proof of global existence is done classically using a bootstrap argument allowing one to to show that if $(A)$ and $(B)$
are assumed to hold on some interval, they actually hold on the same interval with smaller constants in the estimates. 

We
have outlined above the way of obtaining $(B)$, assuming $(A)$ for a solution of the model equation \e{eq:1}. In this 
subsection of the introduction, we shall explain, in a non technical way, the new difficulties that have to be solved to prove
$(B)$ for the water waves equation. Actually, the proof of a long time energy inequality for system \e{intro:3} faces two
serious obstacles, that we describe now.

\medskip

\textbf{$\bullet$ Apparent loss of derivatives in energy inequalities}

\medskip

This difficulty already arises for local existence results, and was solved initially by Nalimov~\cite{Nalimov} and
Wu~\cite{WuInvent,WuJAMS}. For long time existence problems, Wu~\cite{Wu10} uses arguments combining the Eulerian and
Lagrangian formulations of the system. The approach followed in our companion paper~\cite{AlDel} is purely Eulerian. We explain the idea on the model obtained from
\e{intro:3} paralinearizing the equations and keeping only the quadratic terms. If we denote $U =
\bigl[\begin{smallmatrix}\eta\\ \abs{D_x}^{1/2}\psi\end{smallmatrix}\bigr]$, such a model may be written as
\[\partial_t U = T_AU\]
where $T_A$ is the paradifferential operator with symbol $A$, 
and where $A(U,x,\xi)$ is a
matrix of symbols $A(U,x,\xi) = A_0(U,x,\xi) + A_1(U,x,\xi)$, with
\[
A_0(U,x,\xi) = \begin{bmatrix}
  -i(\partial_x\psi)\xi & \abs{\xi}^{1/2}\\ -\abs{\xi}^{1/2} &  -i(\partial_x\psi)\xi
\end{bmatrix}, \ 
A_1(U,x,\xi) = (\abs{D_x}\psi)\la \xi\ra
\begin{bmatrix}
  -1& 0\\0 &1
\end{bmatrix}.\]
Because of the $A_1$ contribution, which is self-adjoint, the eigenvalues of $A(U,x,\xi)$ are not purely imaginary. For large
$\abs{\xi}$, there is one eigenvalue with positive real part, which shows that one cannot expect for the solution of
$\partial_tU = T_AU$ energy inequalities without derivative losses. A way to circumvent this difficulty is well known, and
consists in using the ``good unknown'' of Alinhac~\cite{Alipara}. For our quadratic model, this just means introducing as
a new unknown $\tilde{U} = \bigl[
\begin{smallmatrix}
  \eta \\ \abs{D_x}^{1/2}\omega
\end{smallmatrix}\bigr]$, where $\omega = \psi-T_{\abs{D_x}\psi}\eta$ 
is the (quadratic approximation of the) good
unknown. In that way, ignoring again remainders and terms which are at least cubic, one gets for $\tilde{U}$ an evolution
equation $\partial_t\tilde{U} = T_{A_0}\tilde{U}$. Since $A_0$ is anti-self-adjoint, one gets $L^2$ or Sobolev energy
inequalities for $\tilde{U}$. In particular, if for some $s$, 
$\norm{\abs{D_x}^{1/2}\omega}_{H^s} + \norm{\eta}_{H^s}$ is under
control, and if one has also an auxiliary bound for 
$\norm{\Dx\psi}_{L^\infty}$, one gets an estimate for
$\norm{\abs{D_x}^{1/2}\psi}_{H^{s-1/2}} + \norm{\eta}_{H^s}$.

\medskip

\textbf{$\bullet$  Quadratic terms in the nonlinearity}

\medskip

In the model equation \e{eq:1} discussed above, we considered a cubic nonlinearity: this played an essential role to make
appear in the first integral in the right hand side of \e{eq:4} the almost integrable factor $1/\tau$. For a quadratic
nonlinearity, we would have had instead a $1/\sqrt{\tau}$-factor, which would have given in $(B)$, through Gronwall, a
$O(e^{\eps\sqrt{t}})$-bound, instead of $O(\eps t^{\delta_k})$. The way to overcome such a difficulty is well known
since the work of Shatah~\cite{Shatah} devoted to the non-linear Klein-Gordon equation: it is to use a normal forms method to
eliminate the quadratic part of the nonlinearity, up to terms that do not contribute to the Sobolev energy inequality. 

In
practice, one looks for a local diffeomorphism at 0 in $H^s$, for $s$ large enough, so that the Sobolev energy inequality
written for the equation obtained by conjugation by this diffeomorphim be of the form \e{eq:4}. Nonlinear changes of
unknowns, reducing the water waves system to a cubic equation, have been known for quite a time (see Craig \cite{Craig96} or 
 Iooss and Plotnikov \cite[Lemma $1$]{IP-SW1}). However, these transformations were losing derivatives, as a consequence of
 the quasi-linear character of the problem. Nevertheless, one can construct a bona fide change of unknown, without
 derivatives losses, if one notices that it is not necessary to eliminate the whole quadratic part of the nonlinearity, but
 only the part of it that would bring non zero contributions in a Sobolev energy inequality. This is what we do in our companion
 paper~\cite{AlDel}. Let us also mention  that the analysis of normal forms for the water waves system 
is motivated by physical considerations, such as the derivations of various equations 
in asymptotic regimes (see~\cite{CrSuSu,CSS,SW,TW}). 

\medskip

Our proof of $L^2$-estimates of type $(B)$, assuming that a priori inequalities of type $(A)$ hold, is performed in~\cite{AlDel}
 using the ideas that we just outlined. Of course, the models we have discussed so far do not make justice
to the full complexity of the water waves system. In particular, the good unknown $\omega$ is given by a more involved
formula than the one indicated above, and one also needs to define precisely the Dirichlet-Neumann operator. The latter is
done in~\cite{AlDel}. We recall in   section~\ref{chap:1} below the main properties of the operator $G(\eta)$ when $\eta$ belongs to a space of the
form $C^\gamma(\R)\cap L^2(\R)$ with $\gamma>2$, and is small enough. Once $G(\eta)\psi$ has been defined, one can introduce
functions of $(\eta,\psi)$, $B = (\partial_y\phi)\vert_{y=\eta}$, 
$V = (\partial_x\phi)\vert_{y=\eta}$, where $\phi$ is the 
harmonic potential solving \e{intro:1}. Explicit expressions of these quantities are given by
\[B = \frac{G(\eta)\psi + (\partial_x\eta)(\partial_x\psi)}{1+ (\partial_x\eta)^2}, \ V
= \partial_x\psi - B\partial_x\eta.\]
The good unknown for the water waves equation is given by $\omega = \psi-T_B\eta$. 
Following the analysis in \cite{ABZ3,ABZ1,AM}, we prove in~\cite{AlDel} an expression for $G(\eta)\psi$
in terms of $\omega$:
\[G(\eta)\psi = \abs{D_x}\omega - \partial_x(T_V\eta) + F(\eta)\psi,\]
where $F(\eta)\psi$ is a quadratic \emph{smoothing term}, that belongs to $H^{s+\gamma-4}$ if $\eta$ is in $C^\gamma\cap H^s$
and $\abs{D_x}^{1/2}\psi$ belongs to $C^{\gamma-1/2}\cap H^{s-1/2}$. This gives a quite explicit expression for the main
contributions to $G(\eta)\psi$. Moreover, we prove as well tame estimates, that complement similar results  due to
Craig, Schanz and Sulem (see \cite{CSS}, \cite[Chapter~$11$]{SuSu} 
and \cite{ASL,IP}), and establish bounds
for the approximation of $G(\eta)\psi$ (resp.\ $F(\eta)\psi$) 
by its Taylor expansion at order two $G_{\leq 2}(\eta)\psi$
(resp.\ $F_{\leq 2}(\eta)\psi$). 

\subsection{Klainerman-Sobolev inequalities}

As previously mentioned, the proof of global existence relies on a bootstrap argument on properties $(A)$ and $(B)$. We have
indicated in the preceding section how $(B)$ may be deduced from $(A)$. On the other hand, one has to prove that conversely,
$(A)$ and $(B)$ imply that~$(A)$ holds with smaller 
constants in the inequalities. The first step is to show that if the
$L^2$-estimate~$(B)$ holds for $k\leq s$, then bounds of the form
\[\norm{Z^k u(t,\cdot)}_{L^\infty} = O(\eps t^{-\frac{1}{2}+\delta'_k}),\ k\leq s-100\tag{A'}\]
are true, for small positive $\delta'_k$. This is not $(A)$, since the $\delta'_k$ may be larger that the $\tilde{\delta}'_k$
of $(A)$, and since this does not give a \emph{uniform} bound for $\lA u(t,\cdot)\rA_{L^\infty}$. But this first information
will allow us to deduce, in the last step of the proof, estimates of the form $(A)$ from $(A')$ and the equation.

Let us make a change of variables $x\to x/t$ in the water waves system. If $u(t,x)$ is given by 
$u(t,x)= (\abs{D_x}^{1/2}\psi+ i\eta)(t,x)$, we
define $v$ by $u(t,x) = \frac{1}{\sqrt{t}}v(t,x/t)$. We set $h = 1/t$ and eventually consider $v$ as a family of functions of $x$
depending on the semi-classical parameter~$h$. Moreover, for $a(x,\xi)$ a function satisfying convenient symbol estimates, and
$(v_h)_h$ a family of functions on $\R$, we define
\[\Oph(a)v_h = a(x,hD)v_h = \frac{1}{2\pi}\int e^{ix\xi} a(x,h\xi)\hat{v}_h(\xi)\,d\xi.\]
Then the water waves system is equivalent to the equation
\begin{equation}
  \label{eq:5}
  \bigl(D_t - \Oph\bigl(x\xi+\abs{\xi}^{1/2}\bigr)\bigr)v 
  = \sqrt{h}Q_0(V) + h\left[C_0(V) -\frac{i}{2}v\right] + h^{1+\kappa}R(V),
\end{equation}
where we used the following notations

$\bullet$ $Q_0$ (resp.\ $C_0$) is a nonlocal quadratic (resp.\ cubic) form of $V = (v,\bar{v})$ that may be written as a linear
combination of expressions $\Oph(b_0)[\prod_{j=1}^\ell\Oph(b_j)v_\pm]$, $\ell=2$ 
(resp.\ $\ell = 3$), where $b_\ell(\xi)$ are
homogeneous functions of degree $d_\ell\geq 0$ with $\sum_0^2 d_\ell = 3/2$ 
(resp.\ $\sum_0^3 d_\ell = 5/2$) and $v_+ = v, v_-
= \bar{v}$.

$\bullet$ $R(V)$ is a remainder, made of the contributions vanishing at least at order four at $V=0$.

To simplify the exposition in this introduction, we shall assume that $v$ satisfies $\varphi(hD)v = v$ for some
$C^\infty_0(\R-\{0\})$-function $\varphi$, equal to one on a large enough compact subset of of $\R-\{0\}$. Such a property is
not satisfied by solutions of \e{eq:5}, but one can essentially reduce to such a situation performing a dyadic decomposition
$v = \sum_{j\in\xZ}\varphi(2^{-j}hD)v$.

The Klainerman vector field associated to the linearization of the water waves equation may be written, in the new
coordinates that we are using, as $Z = t\partial_t+x\partial_x$. Remembering $h=1/t$ and 
expressing $\partial_t$ from $Z$ in equation \e{eq:5}, we
get
\begin{equation}
  \label{eq:6}
  \Oph\bigl(2x\xi + \abs{\xi}^{1/2}\bigr)v = -\sqrt{h}Q_0(V) + h\left[\frac{i}{2}v - iZv - C_0(V)\right] - h^{1+\kappa}R(V).
\end{equation}
Since we factored out the expected decay in $1/\sqrt{t}$, 
our goal is to deduce from assumptions $(A)$ and $(B)$ estimates of the form $\norm{Z^k v}_{L^\infty} = O(\eps
h^{-\delta'_k})$ for $k\leq s-100$.

\begin{prop*}
  Assume that for $t$ in some interval $[T_0,T[$ (i.e. for $h$ in some interval $]h',h_0]$), one has estimates $(A)$ and
  $(B)$:
  \begin{equation}
    \label{eq:7}
    \norm{Z^kv}_{L^\infty} = O\bigl(\eps h^{-\tilde\delta'_k}\bigr),\ k\leq s/2,
    \ \norm{Z^kv}_{L^2} = O\bigl(\eps h^{-\delta_k}\bigr),\ k\leq s.
  \end{equation}
Denote $\Lambda = \{(x,d\omega(x)); x \in  \R^*\}$ where $\omega(x) = 1/(4\abs{x})$.  Then, if $\gamma_\Lambda$ is
smooth,  supported close to $\Lambda$ and equal to one on a neighborhood of  $\Lambda$, and if $\gamma_\Lambda^c = 1
- \gamma_\Lambda$, we have for $k\leq s-100$
\begin{equation}
  \label{eq:8}
  \norm{Z^k\Oph(\gamma_\Lambda^c)v}_{L^2} = O\big(\eps h^{\frac{1}{2}-\delta'_k}\big),
\end{equation}
\begin{equation}
  \label{eq:9}
  \norm{(hD_x-d\omega)Z^k\Oph(\gamma_\Lambda)v}_{L^2} = O\big(\eps h^{1-\delta'_k}\big),
\end{equation}
\begin{equation}
  \label{eq:10}
  \norm{Z^k v}_{L^\infty} = O\big(\eps h^{-\delta'_k}\big).
\end{equation}
\end{prop*}

\textsl{Idea of proof.} One applies $k$ vector fields $Z$ to \e{eq:6} and uses their commutation properties to the linearized
equation. In that way, taking into account the assumptions, one gets
\begin{equation}
  \label{eq:11}
  \Oph\big(2x\xi+\abs{\xi}^{1/2}\big)Z^k v = O_{L^2}\big(\eps h^{\frac{1}{2}-\delta'_k}\big)
\end{equation}
for some small $\delta'_k>0$. One remarks that $2x\xi + \abs{\xi}^{1/2}$ vanishes exactly on $\Lambda$. Consequently, this symbol is elliptic on the
support of $\gamma_\Lambda^c$, and this allows one to get \e{eq:8} by ellipticity. 

To prove the second inequality, one uses
the fact that, 
\begin{equation}
  \label{eq:12}
  \Oph\big(2x\xi+\abs{\xi}^{1/2}\big)Z^k \Oph(\gamma_\Lambda)v 
  = -\sqrt{h}\Oph(\gamma_\Lambda)Z^kQ_0(V) + O\big(\eps h^{1-\delta'_k}\big).
\end{equation}
We may decompose $v= v_\Lambda + v_{\Lambda^c}$ where $v_\Lambda = \Oph(\gamma_\Lambda)v$ and $v_{\Lambda^c} =
\Oph(\gamma_\Lambda^c)v$. We may write $Z^kQ_0(V)-Z^k Q_0(v_\La,\bar{v}_{\La})=
B(v_{\La},Z^k v_{\La^c})+\cdots$ where $B$ is the polar form of 
$Q_0$. By~\e{eq:8}, $\norm{Z^k v_{\Lambda^c}}_{L^2} = O(\eps h^{\frac{1}{2}-\delta'_k})$, and by
assumption $\norm{v_\La}_{L^\infty} = O(\eps)$. It follows that 
$\blA B(v_{\La},Z^k v_{\La^c})\brA_{L^2}=O\bigl(\eps h^{\mez-\delta_k'}\bigr)$. The other contributions to 
$Z^kQ_0(V)-Z^k Q_0(v_\La,\bar{v}_{\La})$ may be estimated in a similar way, up to extra 
contributions, that we do not write explicitly in this outline, and that may be absorbed in the left hand side of \e{eq:10} at the end of the reasoning. 
The right hand side of \e{eq:12} may thus be
written
\begin{equation}\label{eq:13}
  -\sqrt{h}\Oph(\gamma_\Lambda)Z^kQ_0(V_\Lambda) + O_{L^2}\big(\eps h^{1-\delta'_k}\big),
\end{equation}
where $V_\Lambda = (v_\Lambda,\bar{v}_\Lambda)$. One notices then that since $v_\Lambda$ (resp.\ $\bar{v}_\Lambda$) is
microlocally supported close to $\Lambda$ (resp. $-\Lambda$), $Q_0(V_\Lambda)$ is microlocally supported close to the union of $2\Lambda$, 
$0\Lambda$ and $-2\Lambda$, so far away from the support of the cut-off $\gamma_\Lambda$ (where 
$\ell\Lambda=\{(x,\ell d\omega(x));x\in\xR^*\}$). 

Consequently, the first term in \e{eq:13} vanishes, 
and we get 
\[\Oph\big(2x\xi+\abs{\xi}^{1/2}\big)Z^k \Oph(\gamma_\Lambda)v= O_{L^2}\big(\eps h^{1-\delta'_k}\big).\] 
Since $2x\xi+\abs{\xi}^{1/2}$ and $\xi-d\omega(x)$ have the same zero set, namely 
$\Lambda$, one deduces  \e{eq:9} from this estimate using symbolic calculus. 

Finally, to obtain \e{eq:10}, we write
\[\norm{Z^{k} v_\Lambda}_{L^\infty} = \norm{e^{-i\omega/h}Z^{k} v_\Lambda}_{L^\infty} 
\leq C \norm{e^{-i\omega/h}Z^{k} v_\Lambda}^{1/2}_{L^2} 
\norm{D_x(e^{-i\omega/h}Z^{k} v_\Lambda)}^{1/2}_{L^2}.\] 
The last factor is $h^{-1/2}\norm{(hD_x -d\omega)Z^{k} v_\Lambda}^{1/2}_{L^2}$, which is $O(\sqrt{\eps}h^{-\delta'_k/2})$
by \e{eq:9}. Moreover, \e{eq:8} and Sobolev inequality imply that $\norm{Z^{k} v_{\Lambda^c}}_{L^\infty} =
O\bigl(\eps h^{-\delta'_k}\bigr)$, since we have assumed that $v$ is spectrally localized for $\abs{\xi} \sim 1/h$. This
gives \e{eq:10}.

\subsection{Optimal $L^\infty$ bounds}

As seen in the preceding section, one can deduce from the $L^2$-estimates $(B)$ some $L^\infty$-estimates \e{eq:10}, which
are \emph{not} the optimal estimates of the form $(A)$ that we need (because the exponents $\delta'_k$ are larger than
$\tilde{\delta}'_k$, and because $\delta'_0$ is positive, while we need a uniform estimate when no $Z$ field acts on $v$). In
order to get $(A)$, we deduce from the PDE \e{eq:5} an ODE satisfied by $v$.

\begin{prop*}
  Under the conclusions of the preceding proposition, we may write
  \begin{equation}
    \label{eq:14}
   v = v_\Lambda + \sqrt{h}(v_{2\Lambda} +  v_{-2\Lambda}) + h(v_{3\Lambda} + v_{-\Lambda} +v_{-3\Lambda}) +h^{1+\kappa}g,
  \end{equation}
where $\kappa>0$, $g$ satisfies bounds of the form $\norm{Z^k g}_{L^\infty} = O(\eps h^{-\delta'_k})$, and
$v_{\ell\Lambda}$ is microlocally supported close to 
$\ell\Lambda$ and is a semi-classical lagrangian distribution
along $\ell\Lambda$, 
as well as $Z^kw_{\ell\Lambda}$ for $k\leq s/2$, in the following sense
\begin{equation}
  \label{eq:15}
  \norm{Z^k v_{\ell\Lambda}}_{L^\infty} = O(\eps h^{-\delta'_k}),
\end{equation}
\begin{equation}
  \label{eq:16}
  \norm{\Oph(e_\ell(x,\xi)) Z^k v_{\ell\Lambda}}_{L^\infty} = O(\eps h^{1-\delta'_k}),\ \ell \in \{1, -2,2\},
\end{equation}
\begin{equation}
  \label{eq:17}
  \norm{\Oph(e_\ell(x,\xi)) Z^k v_{\ell\Lambda}}_{L^\infty} = O(\eps h^{\frac{1}{2}-\delta'_k}),\ \ell \in \{-3, -1,3\},
\end{equation}
if $e_\ell$ vanishes on $\ell\Lambda$.
\end{prop*}
\begin{rema*}
  Consider a function $w = \alpha(x)\exp(i\omega(x)/h)$. If $\alpha$ is smooth and bounded as well as its derivatives, we see
  that $(hD_x - d\omega(x))w = O_{L^\infty}(h)$ i.e.\ $w$ satisfies the second of the above conditions with $\ell=1$, where $e_1(x,\xi) = \xi
  - d\omega(x)$ is an equation of $\Lambda$. The conclusion of the proposition thus means 
  that $v_{\ell\Lambda}$ enjoys a
  weak form of such an oscillatory behavior.
\end{rema*}

The proposition is proved using equation \e{eq:6}. For instance, the bound \e{eq:16} for $v_\Lambda = \Oph(\gamma_\Lambda)v$
is proved in the same way as \e{eq:9}, with $L^2$-norms replaced by $L^\infty$ ones, using \e{eq:10} to estimate the right
hand side. In the same way, one defines $v_{\pm2\Lambda}$ as the cut-off of $v$ close to $\pm2\Lambda$. As in the proof of
\e{eq:8}, one shows an $O_{L^\infty}(h^{\frac{1}{2}-\delta'_k})$ bound for $Z^k v_{\Lambda^c}$, which implies that the main
contribution to $Q_0(v,\bar{v})$ is  $Q_0(v_\Lambda,\bar{v}_{\Lambda})$. Localizing \e{eq:6} close to $\pm2\Lambda$, one gets an
elliptic equation that allows to determine  $v_{\pm2\Lambda}$ as a quadratic function of
$v_\Lambda,\bar{v}_{\Lambda}$. Iterating the argument, one gets the expansion of the proposition. 
One does not get in the $\sqrt{h}$-terms of the expansion a 
contribution associated to $0\Lambda$ because 
$Q_0(V)$ may be factored out by a Fourier multiplier 
vanishing on the zero section. Consequently, non 
oscillating terms form part of the $O(h^{1+\kappa})$ remainder.

\medskip

Let us use the result of the preceding proposition to obtain an ODE satisfied by $v$:
\begin{prop*}
  The function $v$ satisfies an ODE of the form
  \begin{equation}
    \label{eq:18}
\ba
    D_tv &= \frac{1}{2}(1-\chi(h^{-\beta}x))\abs{d\omega}^{1/2}v  
    - i\sqrt{h}(1-\chi(h^{-\beta}x))\Big[\Phi_2(x)v^2 +
    \Phi_{-2}(x)\bar{v}^2\Big]\\
&\quad + h (1-\chi(h^{-\beta}x))\Big[\Phi_3(x)v^3 +  \Phi_{1}(x)\abs{v}^2v + \Phi_{-1}(x)\abs{v}^2\bar{v} + \Phi_{-3}(x)\bar{v}^3\Big] \\
&\quad+
O(\eps h^{1+\kappa}),
\ea
  \end{equation}
where $\kappa>0, \beta>0$ are small, $\Phi_\ell$ are real valued functions of $x$ defined on $\R^*$ and $\chi$ is in
$C^\infty_0(\R)$, equal to one close to zero.
\end{prop*}
To prove the proposition, one plugs expansion \e{eq:14} in equation \e{eq:5}. The key point is to use \e{eq:16}, \e{eq:17} to
express all (pseudo-)differential 
terms from multiplication operators and remainders. For instance, if $b(\xi)$ is some symbol, one may write
$b(\xi) = b\vert_{\ell\Lambda} + e_\ell$ where $e_\ell$ vanishes on $\ell\Lambda = \{\xi=\ell d\omega\}$. Consequently
\[\Oph(b)v_{\ell\Lambda} = b(\ell d\omega)v_{\ell\Lambda} + \Oph(e_\ell)v_{\ell\Lambda},\]
and by \e{eq:16}, when $\ell = -2, 1, 2$, one gets $\norm{\Oph(e_\ell)v_{\ell\Lambda}}_{L^\infty} = O(\eps
h^{1-\delta'_0})$. Since $Q_0(v_\Lambda,\bar{v}_\Lambda)$ is made of expressions of type
\[S = \Oph(b_0)[(\Oph(b_1)v_\Lambda)(\Oph(b_2)v_\Lambda)]\]
(and similar ones replacing $v_\Lambda$ by $\bar{v}_{\Lambda}$), one gets, using that $v_\Lambda^2$ is lagrangian along
$2\Lambda$,
\[S = b_0(2d\omega)b_1(d\omega)b_2(d\omega)v_\Lambda^2 + O_{L^\infty}(h^{1-\delta'_0}).\]
One applies a similar procedure to the other pseudo-differential terms of equation \e{eq:5}, namely $\Oph(x\xi+\abs{\xi}^{1/2})v$ and
$C_0(V)$, where $v$ is expressed using \e{eq:14} in which the $v_{\ell\Lambda}$ are written as explicit quadratic or cubic forms
in $(v_\Lambda,\bar{v}_{\Lambda})$. This permits to write all those terms 
as polynomial expressions in
$(v_\Lambda,\bar{v}_{\Lambda})$ with $x$-depending coefficients, 
up to a remainder vanishing like $h^{1+\kappa}$ when $h$ goes to zero. Expressing back
  $v_\Lambda$ from $v$, one gets the ODE \e{eq:18}.

\medskip

As soon as the preceding proposition has been established, 
the proof of optimal $L^\infty$-estimates for $v$ is
straightforward. Applying a Poincar\'e normal forms method to \e{eq:18}, one is reduced to an equivalent ODE of the form
\[D_t f = \frac{1}{2}(1-\chi(h^{-\beta}x))\abs{d\omega}^{1/2}
\bigg[1+\frac{\abs{d\omega}^2}{t}\abs{f}^2\bigg]f +
O(\eps t^{-1-\kappa}).\]
This implies that $\partial_t\abs{f}^2$ is integrable in time, whence a uniform bound for $f$ and explicit asymptotics when
$t$ goes to infinity. Expressing $v$ in terms of $f$, and writing $u(t,x) = \frac{1}{\sqrt{t}}v(t,x/t)$, one obtains the
uniform $O(t^{-1/2})$ bound for $u$ given in $(A)$ as well as the asymptotics of the statement of the main theorem. Estimates
for $Z^ku$ are proved in the same way.

\renewcommand{\theequation}{\thesection.\arabic{equation}}
\renewcommand{\thesubsection}{\arabic{section}.\arabic{subsection}}

%% file: chapitre1-v2.tex
\section{Statement of the main result}\label{chap:1} 

We have already written in the introduction the water waves equations under the form of the Craig-Sulem-Zakharov system
\e{intro:3}. We shall give here the precise definition of the Dirichlet-Neuman operator that is used in that system, and
state some of its properties that are used in the rest of this paper, as well as in the companion paper~\cite{AlDel}. Theses
properties, that are essentially well known,  are proved in that reference. Once the Dirichlet-Neuman operator has been
properly defined, we give the precise statement of our global existence result. Next, we explain the strategy of proof, which
relies on a bootstrap argument on some a priori $L^2$ and $L^\infty$ estimates. The $L^2$ bounds are proved in the companion
paper~\cite{AlDel}. The $L^\infty$ ones, that represent the main novelty of our method, are established in
sections~\ref{S:3.1} to \ref{S:35} of the present paper.

\subsection{Dirichlet-Neumann operator}\label{S:11}

Let $\eta\colon \xR\rightarrow\xR$ be a smooth enough function and consider the 
open set
$$
\Omega\defn \{\,(x,y)\in\xR\times\xR \,;\, y<\eta (x)\,\}.
$$
It $\psi\colon\xR\rightarrow\xR$ is another function, and if we call 
$\phi\colon\Omega\rightarrow \xR$ the unique solution of 
$\Delta\phi=0$ in~$\Omega$ satisfying 
$\phi\arrowvert_{y=\eta(x)}=\psi$ and a convenient vanishing condition at $y\rightarrow -\infty$, one defines 
the Dirichlet-Neumann operator\index{Dirichlet-Neumann operator!$G(\eta)$} $G(\eta)$ by 
\begin{equation*}
G(\eta)\psi   =
\sqrt{1+(\partial_x\eta)^2}\,
\partial _n \phi\arrowvert_{y=\eta},
\end{equation*}
where~$\partial_n$ is the outward normal derivative on~$\partial\Omega$, so that
$$
G(\eta)\psi=(\partial_y \phi)(x,\eta(x))- (\partial_x \eta)(\partial_x \phi)(x,\eta(x)).
$$
In this subsection, we recall the estimates obtained in \cite{AlDel} for $G(\eta)$.

One may reduce the problem to the negative half-space through the change of coordinates
$(x,y)\mapsto (x,z=y-\eta(x))$, which sends $\Omega$ on 
$\{ (x,z)\in \xR^2\,;\, z<0\}$. 
Then $\phi(x,y)$ solves $\Delta\phi=0$ if and only if $\varphi(x,z)=\phi(x,z+\eta(x))$ is a solution of 
$P\varphi=0$ in $z<0$, where
\be\label{111}
P=(1+\eta'^2)\partial_z ^2+\px^2-2\eta'\px\partial_z-\eta''\partial_z 
\ee
(we denote by $\eta'$ the derivative $\px\eta$). The boundary condition 
becomes $\varphi(x,0)=\psi(x)$ and $G(\eta)$ is given by
$$
G(\eta)\psi=\bigl[ (1+\eta'^2)\partial_z\varphi-\eta'\partial_x \varphi\bigr]\big\arrowvert_{z=0}.
$$
It is convenient and natural to try to solve the boundary value problem 
$$
P\varphi=0, \quad\varphi\arrowvert_{z=0}=\psi
$$
when $\psi$ lies in homogeneous Sobolev spaces. 
Let us introduce them and fix some notation.

We denote by \index{Function spaces!$\Sr'_\infty(\xR)$, $\Sr'_1(\xR)$} 
$\Sr'_\infty(\xR)$ (resp.\ $\Sr'_1(\xR)$) the quotient space 
$\Sr'(\xR)/\xC[X]$ (resp.\ $\Sr'(\xR)/\xC$). If $\Sr_\infty(\xR)$ (resp.\ $\Sr_1(\xR)$) 
is the subspace of $\Sr(\xR)$ made of the functions orthogonal to any polynomial 
(resp.\ to the constants), $\Sr'_\infty(\xR)$ (resp.\ $\Sr'_1(\xR)$) is the dual of 
$\Sr_\infty(\xR)$ (resp.\ $\Sr_1(\xR)$). Since 
the Fourier transform realizes an isomorphism 
from $\Sr_\infty(\xR)$ (resp.\ $\Sr_1(\xR)$) to
$$
\widehat{\Sr}_\infty(\xR)=\{ u\in \Sr(\xR)\,;\, u^{(k)}(0)=0 \text{ for any }k\text{ in }\xN\}
$$ 
(resp. $\widehat{\Sr}_1(\xR)=\{ u\in \Sr(\xR)\,;\, u(0)=0\}$), we get by duality that the Fourier transform defines an isomorphism 
from $\Sr'_\infty(\xR)$ to $(\widehat{\Sr}_\infty(\xR))'$, which is the quotient of $\Sr'(\xR)$ by the subspace of distributions supported in $\{0\}$ 
(resp.\ from $\Sr'_1(\xR)$ to $(\widehat{\Sr}_1(\xR))'=\Sr'(\xR)/{\rm Vect}\, (\delta_0)$). 

Let $\phi\colon\xR\rightarrow \xR$ be a function defining a Littlewood-Paley decomposition
and set for $j\in \xZ$, $\Delta_j=\phi(2^{-j}D)$. Then for any $u$ in $\Sr'_\infty(\xR)$, the series $\sum_{j\in\xZ} \Delta_j u$ 
converges to $u$ in $\Sr'_\infty(\xR)$ (for the weak-$*$ topology associated to the natural topology on $\Sr_\infty(\xR)$). 
Let us recall (an extension of) the usual definition of 
homogeneous Sobolev or H\"older spaces. 

\begin{defi}\label{ref:111}
Let $s',s$ be real numbers. One denotes by \index{Function spaces!$\h{s',s}$, homogeneous Sobolev spaces}
$\h{s',s}(\xR)$ (resp.\ \index{Function spaces!$\C{s',s}$, homogeneous Zygmund spaces}$\C{s',s}(\xR)$) 
the space of elements $u$ in $\Sr'_\infty(\xR)$ 
such that there is a sequence $(c_j)_{j\in\xZ}$ in $\ell^2(\xZ)$ (resp.\ a constant $C>0$) with for any $j$ 
in $\xZ$,
$$
\lA \Delta_j u\rA_{L^2}\le c_j 2^{-js' -j_+ s}
$$
(resp.\
$$
\lA \Delta_j u\rA_{L^\infty}\le C 2^{-js'-j_+ s})
$$
where $j_+=\max(j,0)$. We set $\h{s'}$ (resp.\ $\C{s'}$) when $s=0$.
\end{defi}
The series $\sum_{j=0}^{+\infty}\Delta_j u$ always converges in $\Sr'(\xR)$ under 
the preceding assumptions, but 
the same is not true for $\sum_{j=-\infty}^{-1}\Delta_j u$. If $u$ is in $\h{s',s}(\xR)$ 
with $s'<1/2$ (resp.\ in $\C{s',s}(\xR)$ with $s'<0$), 
then $\sum_{j=-\infty}^{-1}\Delta_j u$ converges normally in $L^\infty$, so in $\Sr'(\xR)$, and 
$u\rightarrow \sum_{-\infty}^{+\infty}\Delta_j u$ 
gives the unique dilation and translation invariant realization of $\h{s',s}$ (resp.\ $\C{s',s}(\xR)$) 
as a subspace of $\Sr'(\xR)$. 
One the other hand, if $s'\in [1/2,3/2\por$ (resp. $s'\in [0,1\por$), the space $\h{s'}(\xR)$ (resp. $\C{s'}(\xR)$) admits no translation commuting realization as a 
subspace of $\Sr'(\xR)$, but the map $u\rightarrow \sum_{-\infty}^{+\infty}\Delta_j u$ defines a dilation 
and translation commuting realization of these spaces as subspaces of $\Sr'_1(\xR)$. We refer to Bourdaud~\cite{Bourdaud} for
these properties. 

Recall also that if $s$ is in $\R$ (resp. $\gamma$ is in $\R-\N$), the usual Sobolev space $H^s(\R)$ (resp. the space
$C^\gamma(\R)$) is defined as the space of elements $u$ of $\Sr'(\xR)$ satisfying, for any $j$ in $\N$, $\norm{\Delta_j u}_{L^2} \leq c_j2^{-js}$
(resp. $\norm{\Delta_j u}_{L^\infty} \leq C2^{-js}$) for some $\ell^2(\N)$-sequence $(c_j)_j$ (resp. some constant $C$), and
$\chi(D)u\in L^2$ (resp. $\chi(D)u\in L^\infty$) for some $C^\infty_0(\R)$-function $\chi$ equal to one on a large enough
neighborhood of zero. Moreover, if $\gamma$ is in $\N
$, we denote by $C^\gamma(\R)$ the space of $\gamma$ times continuously differentiable functions, which are \emph{bounded}
as well as their derivatives (endowed with the natural norms).

The main result about the Dirichlet-Neumann operator that we shell use in that paper is the following proposition, which is
proved in the companion paper~\cite{AlDel} (see Corollary 1.1.8.): \todo{Ref Compagnon}

\begin{prop}\label{ref:118} Let $\gamma$ be a real number, $\gamma>2, \gamma\not\in\frac{1}{2}\N$. There is some $\delta>0$ such that, for any  $\eta$  in $L^2\cap \eC{\gamma}(\xR)$ 
satisfying $\etapetitgamma+\lA \eta'\rA_{\eC{-1}}^{1/2}\lA \eta'\rA_{H^{-1}}^{1/2}<\delta$, one may define for $\psi$ in $\Hmez$ the Dirichlet-Neumann operator $G(\eta)$ as 
a bounded operator from $\dot{H}^{1/2}(\xR)$ 
to $\dot{H}^{-1/2}(\xR)$ that satisfies an estimate
\be\label{1137}
\lA G(\eta)\psi\rA_{\dot{H}^{-1/2}}\le 
\Cetagamma\Dxmezpsi.
\ee
In particular, if we define $G_{1/2}(\eta)=\Dx^{-\mez}G(\eta)$, we 
obtain a bounded operator from $\dot{H}^{1/2}(\xR)$ to $L^2(\xR)$ satisfying 
\be\label{1138}
\lA G_{1/2}(\eta)\psi\rA_{L^{2}}\le 
\Cetagamma \Dxmezpsi.
\ee
Moreover, $G(\eta)$ satisfies when $\psi$ is in $\C{\mez,\gamma-\mez}(\xR)$
\be\label{1139}
\lA G(\eta)\psi\rA_{\eC{\gamma-1}}\le 
\Cetagamma \blA \Dxmez \psi\brA_{\eC{\gamma-\mez}}.
\ee
where $C(\cdot)$ is a non decreasing continuous function of 
its argument.

If we assume moreover that for some $0<\theta'<\theta<\mez$, 
$\blA \eta'\brA_{H^{-1}}^{1-2\theta'}\blA \eta'\brA_{\eC{-1}}^{2\theta'}$ is bounded, then 
$\Dx^{-\mez+\theta}G(\eta)$ satisfies
\be\label{1140}
\blA \Dx^{-\mez+\theta}G(\eta)\psi\brA_{\eC{\gamma-\mez-\theta}}
\le C\bigl( \blA \eta'\brA_{\eC{\gamma-1}}\bigr) \blA \Dxmez\psi\brA_{\eC{\gamma-\mez}}.
\ee
\end{prop}

\subsection{Global existence result}

The goal of this paper is to prove global existence of small 
solutions with decaying Cauchy data of the Craig-Sulem-Zakharov system. We thus look for a couple of real valued functions $(\eta,\psi)$ 
defined on $\xR\times \xR$ satisfying for $t\ge 1$ the system 
\begin{equation}\label{121}
\left\{
\begin{aligned}
&\partial_t \eta=G(\eta)\psi,\\
&\partial_t \psi + \eta+ \frac{1}{2}(\partial_x \psi)^2
-\frac{1}{2(1+(\partial_x\eta)^2)}\bigl(G(\eta)\psi+\partial_x  \eta \partial_x \psi\bigr)^2= 0,
\end{aligned}
\right.
\end{equation}
with Cauchy data small enough in 
a convenient space. 

The operator $G(\eta)$ in \eqref{121} being defined as in the preceding subsection, we set, for $\eta,\psi$ smooth enough 
and small enough functions
\be\label{122}
\B(\eta)\psi=\frac{G(\eta)\psi+\partial_x  \eta \partial_x \psi}
{1+(\partial_x\eta)^2}\cdot
\ee

Before stating our global existence  result, let us recall a known local existence theorem 
(see~\cite{WuInvent,LannesLivre,ABZ3}).  

\begin{prop}\label{ref:121}
Let  $\gamma$ be in $]7/2,+\infty[\setminus \mez\xN$, $s\in \xN$ with $s>2\gamma-1/2$. 
There are $\delta_0>0$, $T>1$ such that for any couple $(\eta_0,\psi_0)$ 
in $H^{s}(\xR)\times \h{\mez,\gamma}(\xR)$ satisfying
\be\label{123}
\psi_0-T_{\B(\eta_0)\psi_0}\eta_0 \in \h{\mez,s}(\xR), \quad 
\lA \eta_0\rA_{\eC{\gamma}}+\blA \Dxmez \psi_0\brA_{\eC{\gamma-\mez}}<\delta_0,
\ee
equation \eqref{121} with Cauchy data $\eta\arrowvert_{t=1}=\eta_0$, $\psi\arrowvert_{t=1}
=\psi_0$ has a unique solution $(\eta,\psi)$ which is continuous on $[1,T]$ with values in 
\be\label{124}
\left\{\, (\eta,\psi)\in H^s(\xR)\times \h{\mez,\gamma}(\xR)\,;\, \psi-
T_{B(\eta)\psi}\eta\in \h{\mez,s}(\xR)\,\right\}.
\ee
Moreover, if the data are $O(\eps)$ on the indicated spaces, then $T\ge c/\eps$.
\end{prop}
\begin{remas} $\bullet$ The assumption $\psi_0\in \h{\mez,\gamma}$ implies that $\psi_0$ is in 
$\C{\mez,\gamma-\mez}$ so that Proposition~\ref{ref:118} shows that 
$G(\eta_0)\psi_0$ whence $\B(\eta_0)\psi_0$ is in 
$\eC{\gamma-1}\subset L^\infty$. Consequently, by the 
first half of \eqref{123}, $\Dxmez\psi$ is in $H^{s-\mez}\subset \eC{\gamma-\mez}$ as 
our assumption on $s$ implies that $s>\gamma+1/2$. This gives sense to the second assumption~
\eqref{123}. 

$\bullet$ As already mentioned in the introduction, the difficulty in the analysis of equation~\eqref{121} is that writing 
energy inequalities on the function $(\eta,\Dxmez\psi)$ makes appear an apparent loss 
of half a derivative. A way to circumvent that difficulty is to bound the 
energy not of $(\eta,\Dxmez\psi)$, but of $(\eta,\Dxmez\omega)$, where $\omega$ is 
the ``good unknown'' of Alinhac, defined by $\omega=\psi-T_{\B(\eta)\psi}\eta$ (see subsection~\ref{Sintro:GSP} of the introduction). This explains why the regularity assumption~\eqref{123} on the Cauchy data concerns 
$\psi_0-T_{\B(\eta_0)\psi_0}\eta_0$ and not $\psi_0$ itself. Notice that this function is in 
$\h{\mez,s}$ while $\psi_0$ itself, written 
from $\psi_0=\omega_0+T_{\B(\eta_0)\psi_0}\eta_0$ is only in $\h{\mez,s-\mez}$, because 
of the $H^s$-regularity of $\eta_0$.

$\bullet$ By \eqref{1139} if $\psi$ is in $\C{\mez,\gamma-\mez}$ and $\eta$ is in $\eC{\gamma}$, 
$G(\eta)\psi$ is in $\eC{\gamma-1}$, so $\B(\eta)\psi$ is also in $\eC{\gamma-1}$ with 
$\lA \B(\eta)\psi\rA_{\eC{\gamma-1}}\le \Cetagamma \blA \Dxmez\psi\brA_{\eC{\gamma-\mez}}$. 
In particular, as a paraproduct with an $L^\infty$-function acts on any H\"older space,
$$
\blA \Dxmez T_{\B(\eta)\psi}\eta\brA_{\eC{\gamma-\mez}}
\le \Cetagamma \lA \eta\rA_{\eC{\gamma}}\blA \Dxmez\psi\brA_{\eC{\gamma-\mez}}.
$$
This shows that for $\lA \eta\rA_{\eC{\gamma}}$ small enough, $\psi\rightarrow \psi-T_{\B(\eta)\psi}\eta$ is an isomorphism
from $\C{\mez,\gamma-\mez}$ to itself. In particular, if we are given a small enough $\omega$ in $\h{\mez,s}\subset \C{\mez,\gamma-\mez}$, we may find a unique $\psi$ in 
$\C{\mez,\gamma-\mez}$ such that $\omega=\psi-T_{\B(\eta)\psi}\eta$. In other words, 
when interested only in $\eC{\gamma-\mez}$-estimates for $\Dxmez\omega$, we may as well 
establish them on $\Dxmez\psi$ instead, as soon as 
$\lA \eta\rA_{\eC{\gamma}}$ stays small enough. 

\end{remas}

Let us state now our main result. 

We fix real numbers 
$s,s_1,s_0$ satisfying, for some large enough numbers $a$ and $\gamma$ 
with $\gamma\not\in \mez\xN$ and $a\gg \gamma$, 
the following conditions
\be\label{125}
s,s_0,s_1\in\xN,\quad 
s-a\ge s_1\ge s_0\ge \frac{s}{2}+\gamma.
\ee

\begin{theo}\label{ref:122}
There is $\eps_0>0$ such that for any $\eps\in ]0,\eps_0]$, any couple 
of functions $(\eta_0,\psi_0)$ satisfying for any integer $p\le s_1$
\be\label{126}
\begin{aligned}
&(x\px)^p\eta_0\in H^{s-p}(\xR),\quad (x\px)^p\psi_0\in \h{\mez,s-p-\mez}(\xR),\\
&(x\px)^p\bigl( \psi_0-T_{\B(\eta_0)\psi_0}\eta_0\bigr)\in \h{\mez,s-p}(\xR),
\end{aligned}
\ee
and such that the norm of the above functions in the indicated spaces is 
smaller than $1$, equation \eqref{121} with the Cauchy data 
$\eta\arrowvert_{t=1}=\eps\eta_0$, $\psi\arrowvert_{t=1}=\eps\psi_0$ 
has a unique solution 
$(\eta,\psi)$ which is defined and continuous on $[1,+\infty[$ with values 
in the set~\eqref{124}.

Moreover, $u=\Dxmez \psi+i\eta$ admits the following 
asymptotic expansion as $t$ goes to $+\infty$: 

There is a continuous function $\underline{\alpha}\colon \xR
\rightarrow \xC$, depending of $\eps$ but 
bounded uniformly in $\eps$, such that
\be\label{127}
u(t,x)=\frac{\eps}{\sqrt{t}}\underline{\alpha}\Bigl(\frac{x}{t}\Bigr)
\exp\Bigl(\frac{it}{4|x/t|}+\frac{i\eps^2}{64}\frac{\la \underline{\alpha}(x/t)\ra^2}{\la x/t\ra^5}\log (t)\Bigr)
+\eps t^{-\mez-\kappa}\rho(t,x)
\ee
where $\kappa$ is some positive number and $\rho$ 
is a function uniformly bounded for $t\ge 1$, $\eps\in ]0,\eps_0]$.
\end{theo}
\begin{rema*}
If the integers $s,s_1,s_0$ are large enough, we shall see in section~\ref{S:35} that $\underline{\alpha}(x/t)$ vanishes when $x/t$ goes to 
zero at an order that increases with these integers. Because of that, we see that the singularity of the phase at $x/t=0$ is quite irrelevant: 
for $|x/t|$ small enough, the first term in the expansion is not larger 
than the remainder.
\end{rema*}

\subsection{Strategy of the proof}\label{S:13}

The proof of the main theorem relies on the simultaneous propagation through a bootstrap of $L^\infty$ and
$L^2$-estimates. We state here these two results. The first one is proved in the companion paper~\cite{AlDel}. The proof of
the second one is the bulk of the present paper. We show below   how these two results together  imply Theorem~\ref{ref:122}. 

The main point will be to prove $L^2$ and $L^\infty$-estimates for the action of the vector field 
\be\label{131}
Z=t\partial_t+2x\px
\ee
on the unknown in equation \eqref{121}. We introduce the following notation: 

We assume given $\gamma,s,a,s_0,s_1$ satisfying \eqref{125}. For $(\eta,\psi)$ a local smooth enough solution of \eqref{121}, we set $\omega=\psi-T_{\B(\eta)\psi}\eta$ and for any integer $k\le s_1$,\index{Norms!$M_s^{(k)}$}
\be\label{132}
M_s^{(k)}(t)=\sum_{p=0}^k \bigl( 
\blA Z^p \eta(t,\cdot)\brA_{H^{s-p}}
+\blA \Dxmez Z^p \omega(t,\cdot)\brA_{H^{s-p}}\bigr).
\ee
In the same way, for $\rho$ a positive number (that will be larger 
than $s_0$), we set for $k\le s_0$, \index{Norms!$N_\rho^{(s_0)}$}
\be\label{133}
N_\rho^{(k)}(t)=\sum_{p=0}^k \bigl( 
\blA Z^p \eta(t,\cdot)\brA_{\eC{\rho-p}}
+\blA \Dxmez Z^p \psi(t,\cdot)\brA_{\eC{\rho-p}}\bigr).
\ee
By local existence theory, for any given $T_0>1$, 
there is $\eps_0'>0$ such that 
if $\eps<\eps_0'$, equation \eqref{121} has a solution for 
$t\in [1,T_0]$. Moreover, 
assumptions \eqref{126} remain valid at $t=T_0$ (see Proposition~A.4.2. \todo{Ref Compagon} in the companion paper). Consequently, it is enough to prove 
Theorem~\ref{ref:122} with Cauchy data at $t=T_0$. 

The $L^2$ estimates that we need are given by the following theorem, that is proved in the companion paper~\cite{AlDel} (see
Theorem~1.2.2.\ of that paper).

\begin{theo}\label{ref:131}
There is a constant $B_2>0$ such that $M_s^{(s_1)}(T_0)<\uq B_2 \eps$, and for any constants $B_\infty>0$, 
$B_\infty'>0$ there is $\eps_0$ such that the following holds: 
Let $T>T_0$ be a number such that equation \eqref{121} 
with Cauchy data satisfying \eqref{126} has a solution satisfying the regularity properties of Proposition~\ref{ref:121} 
on $[T_0,T[\times \xR$ and such that 

$i)$ For any $t\in [T_0,T[$, and any $\eps\in ]0,\eps_0]$, 
\be\label{134}
\blA \Dxmez \psi(t,\cdot)\brA_{\eC{\gamma-\mez}}
+\lA \eta(t,\cdot)\rA_{\eC{\gamma}}\le B_\infty\eps t^{-\mez}.
\ee 

$ii)$ For any $t\in [T_0,T[$, any $\eps\in ]0,\eps_0]$
\be\label{135}
N_\rho^{(s_0)}(t)\le B_\infty \eps t^{-\mez+B_\infty' \eps^2}.
\ee
Then, there is an increasing sequence $(\delta_k)_{0\le k\le s_1}$, depending only on $B_\infty'$ and $\eps$ with $\delta_{s_1}<1/32$ such that for any $t$ in $[T_0,T[$, any 
$\eps$ in $]0,\eps_0]$, any $k\le s_1$,
\be\label{136}
M_s^{(k)}(t)\le \mez B_2\eps t^{\delta_k}.
\ee
\end{theo}
\begin{rema*}
We do not get for the $L^2$-quantities 
$M_s^{(k)}(t)$ a uniform estimate 
when $t\rightarrow +\infty$. Actually, the form of the principal term 
in the expansion \eqref{127} shows that the action of a $Z$-vector field on it generates a $\log(t)$-loss, so that one 
cannot expect \eqref{136} to hold true with $\delta_k=0$. 
For similar reasons, one could not expect that 
$N_\rho^{(s_0)}(t)$ in \eqref{135} be $O(t^{-1/2})$ when 
$t\rightarrow +\infty$. Such an estimate can be true only 
if no $Z$-derivative acts on the solution, as in \eqref{134}.
\end{rema*}

Let us write down next the $L^\infty$-estimates.

\begin{theo}\label{ref:132}
Let $T>T_0$ be a number such that the 
equation \e{121} with Cauchy data satisfying \e{126} has a solution on $[T_0,T[\times \xR$ satisfying the regularity properties of Proposition~$\ref{ref:121}$. Assume 
that, for some constant $B_2>0$, for any $t\in [T_0,T[$, any $\eps$ in $]0,1]$, 
any $k\le s_1$, 
\be\label{137}
\ba
&M_{s}^{(k)}(t)\le B_2 \eps t^{\delta_k},\\
&N_{\rho}^{(s_0)}(t)\le \sqrt{\eps}<1
\ea
\ee
Then there are constants $B_\infty,B_\infty'>0$ depending only on $B_2$ and some 
$\eps_0'\in ]0,1]$, independent of 
$B_2$, 
such that, for any $t$ in $[T_0,T[$, any $\eps$ 
in $]0,\eps_0']$,
\be\label{138}
\ba
&N_\rho^{(s_0)}(t)\le \mez B_\infty \eps t^{-\mez+\eps^2 B_\infty'},\\
&\blA \Dxmez \psi(t,\cdot)\brA_{\eC{\gamma-\mez}}+
\lA \eta(t,\cdot)\rA_{\eC{\gamma}}\le \mez B_\infty \eps t^{-\mez}.
\ea
\ee
\end{theo}

We deduce form the above results the global existence statements in Theorem~\ref{122}. 

\begin{proof}[Proof of Theorem~\ref{ref:122}] 
We take for $B_2$ the constant given by Theorem~\ref{ref:131}. Then Theorem~\ref{ref:132} 
provides constants $B_\infty>0$, $B_\infty'>0$, and given these $B_\infty,B_\infty'$, Theorem~\ref{ref:131} brings a small positive number $\eps_0$. 
We denote by $T_*$ the supremum of those $T>T_0$ such that a solution exists 
over the interval $[T_0,T[$, satisfies over this interval the regularity conditions 
of Proposition~\ref{ref:121} and the estimates 
\be\label{139}
\ba
&M_{s}^{(k)}(t)\le B_2 \eps t^{\delta_k}\quad \text{for } k\le s_1,\\
&N_\rho^{(s_0)}(t)\le B_\infty \eps t^{-\mez+\eps^2 B_\infty'},\\
&\blA \Dxmez \psi(t,\cdot)\brA_{\eC{\gamma-\mez}}+
\lA \eta(t,\cdot)\rA_{\eC{\gamma}}\le B_\infty \eps t^{-\mez}.
\ea
\ee
We have $T_*>T_0$:  By the choice of $B_2$ in the statement 
of Theorem~\ref{ref:131}, the first estimate \e{139} holds at $t=T_0$ with 
$B_2$ replaced by $B_2/2$. If $B_\infty$ is chosen from start large enough, we may 
as well assume that at $t=T_0$, the second and third inequalities in \e{139} 
hold with $B_\infty$ replaced by $B_\infty/2$. Consequently, 
the local existence results of Appendix~A.4 \todo{Ref Compagnon} in the companion paper~\cite{AlDel} show that a solution exists  
on some interval $[T_0,T_0+\delta[$, and will satisfy \e{139} 
on that interval if $\delta$ is small enough. 

If $T_*<+\infty$, and if we take $\eps_0$ small enough so that 
$B_\infty \sqrt{\eps_0}<1$, we see that \e{139} implies that assumptions 
\e{134}, \e{135}, and \e{137} of Theorem~\ref{ref:131} and 
\ref{ref:132} are satisfied. Consequently, \e{136} and \e{138} 
hold on the interval $[T_0,T_*[$ i.e.\ \e{139} is true on this interval with 
$B_2$ (resp.\ $B_\infty$) replaced by $B_2/2$ (resp.\ $B_\infty/2$). This 
contradicts the maximality of $T_*$. So $T_*=+\infty$ 
and the solution is global. We postpone the proof of \e{127} to the 
end of Section~\ref{S:35}.
\end{proof}

The rest of this paper will be devoted to the proof of Theorem~\ref{ref:132}.  Theorem~\ref{ref:131} is proved in the
companion paper~\cite{AlDel}.

%% file: chapitre3.tex
\section{Classes of Lagrangian distributions}\label{S:3.1}
We denote by $h$ a semi-classical parameter belonging 
to $]0,1]$. If $(x,\xi)\mapsto m(x,\xi)$ is an order function from 
$T^*\xR$ to $\xC$, as defined in Appendix~\ref{S:A.2}, and if $a$ is a symbol in the class $S(m)$ of Definition~\ref{ref:A.2.1}, 
we set, for $(u_h)_h$ any family of elements of $\mathcal{S}'(\xR)$
\be\label{311}\index{Semi-classical notations!$\Oph(a)$}
\Oph(a)u=\frac{1}{2\pi}\int e^{ix\xi}a(x,h\xi,h)\widehat{u}(\xi)\, d\xi.
\ee
It turns out that we shall need extensions of this definition to more general classes of symbols. 
On the one hand, we notice that if $a$ is a continuous function such that $\la a(x,\xi,h)\ra\le m(x,\xi)$ and 
if $u$ is in $L^2(\xR)$, \eqref{311} is still meaningful.

We shall also use a formula of type \eqref{311} when the symbol $a$ is defined only on a subset of $T^*\xR$. 
Denote by $\pi_1:(x,\xi)\mapsto x$ and $\pi_2\colon (x,\xi)\mapsto \xi$ the two projections. 
For $F$ a closed subset of $T^*\xR$, and $r>0$, we set
$$
F_r=\left\{ \,(x,\xi)\in T^*\xR\,;\, d\left( (x,\xi),F\right)<r\,\right\}
$$
where $d$ is the euclidian distance. 

\begin{defi}\label{ref:3.1.1}
Let $m$ be an order function on $T^*\xR$, $F$ a 
closed non empty 
subset of $T^*\xR$ such that $\pi_2(F)$ is compact. 
We denote by $S(m,F)$ \index{Symbols!$S(m,F)$} the space 
of functions $(x,\xi,h)\mapsto a(x,\xi,h)$, defined on $F_{r_0}\times ]0,1]$ for some 
$r_0>0$, and satisfying for any $\alpha,\beta$ in $\xN$, any $(x,\xi)$ in $F_{r_0}$, any 
$h$ in $]0,1]$,
$$
\la \partial_x^\alpha \partial_\xi^\beta a(x,\xi,h)\ra\le C_{\alpha,\beta}m(x,\xi).
$$
\end{defi}
We define next the notion of a family of functions microlocally supported close to a subset $F$ as above. 

\begin{defi}\label{ref:3.1.2}
$i)$ Let $p$ be in $[1,+\infty]$. We denote by $E^p_\emptyset$ the space 
of families of $L^p$ functions $(v_h)_h$ indexed by $h\in ]0,1]$, 
defined on $\xR$ with values in $\xC$ such that for any $N$ in $\xN$, 
there is $C_N>0$ with $\lA v_h\rA_{L^p}\le C_Nh^N$ for any 
$h$ in $]0,1]$. 

$ii)$ Let $F$ be a closed non empty subset of $T^*\xR$ such that $\pi_2(F)$ is compact. 
We denote by $E^p_F$ the space of families of functions $(v_h)_h$ of $L^p(\xR)$ satisfying

$\bullet$ There are $N_0$ in $\xN$, $C_0>0$ and for any $h$ in $]0,1]$, 
$\lA v_h\rA_{L^p}\le C_0 h^{-N_0}$.

$\bullet$ For any $r>r'>0$, there is an element $\phi$ of $S(1)$ supported 
in $F_r$, equal to one on $\overline{F_{r'}}$ such that 
$(\Oph(\phi)v_h -v_h)_h$ belongs to $E^p_\emptyset$. 
We say that $(v_h)_h$ is microlocally suported close to $F$.
\end{defi}

\begin{remas}
$\bullet$ Notice that definition~\ref{ref:3.1.2} is non empty only if there exists at least one function $\phi$ in $S(1)$ supported in
$F_r$, equal to one on $\overline{F_{r'}}$. This holds if $F$ is not ``too wild'' when $\abs{x}$ goes to infinity, for
instance if $F$ is compact, or if $F = \pi_2^{-1}(K)$ for some compact subset $K$ of $\R$. In the sequel, we shall always
implicitly assume that such a property holds for the closed subsets in which are microlocally supported the different classes
of distributions we shall define. 

$\bullet$ It follows from Theorem~\ref{ref:A.2.2} of Appendix~\ref{S:A.2} that the last condition in Definition~\ref{ref:3.1.2} will hold for any element $\phi$ of 
$S(1)$, supported in $F_r$, equal to one on $\overline{F_{r'}}$.
\end{remas}

We may define the action of operators associated to symbols belonging to the class 
$S(1,F)$ on functions microlocally supported 
close to $F$, modulo elements of $E^p_\emptyset$. Let us notice first that if $a$ 
is in $S(1)$ and is supported in a domain $\{(x,\xi,h)\,;\,\la\xi\ra\le C\}$ for some 
$C>0$, then \eqref{311} defines an operator bounded on $L^p(\xR)$ for any $p$, 
uniformly in $h$. It follows then from 
the theorem of symbolic calculus~\ref{ref:A.2.2} of 
Appendix~\ref{S:A.2} that, 
if $a$ is in $S(1,F)$, if $\tilde{\phi}$ is in $S(1)$ 
supported in $F_r\cap(\overline{F_{r'}})^c$, 
for some $0<r'<r\ll 1$, then $(\Oph(a\tilde{\phi})v_h)_h$ is in 
$E^p_\emptyset$ for any $(v_h)_h$ in $E^p_F$. We may thus state:

\begin{defi}\label{ref:3.1.3}
Let $F$ be a closed set as in Definition~\ref{ref:3.1.2}, $a$ be an element of $S(1,F)$. 
For $(v_h)_h$ in $E^{p}_F$, we define 
\be\label{312}
\Oph(a)v_h=\Oph(a\phi)v_h.
\ee
where the right-hand side is defined by \eqref{311} and where $\phi$ is in $S(1)$, 
supported in $F_r\times ]0,1]$ for small enough $r>0$ and equal to one on 
$\overline{F_{r'}}\times ]0,1]$, for some $r'\in ]0,r[$. The definition is 
independent of the choice of $\phi$ modulo $E^p_\emptyset$, so that 
$\Oph(a)$ is well defined from 
$E^p_F/E^p_\emptyset$ to itself.
\end{defi}

Let $K$ be a compact subset of $T^*\xR$, $K_1=\pi_1(K)$ and let $\omega$ be a real valued function defined on an 
open neighborhood $U$ of $K_1$. Denote by $\chi_\omega$ the canonical transformation 
\begin{alignat*}{2}
\chi_\omega\colon& T^*U &&\rightarrow T^* U\\
&(x,\xi)&&\mapsto (x,\xi-\diff \omega(x)).
\end{alignat*}
Set $K'=\chi_\omega(K)$. 

\begin{lemm}\label{ref:3.1.4}
$i)$ Let $a$ be in $S(1,K')$. 
There is a symbol $b$ in $S(1,K)$ such that for any $(v_h)_h$ in $E^p_K$, we may write
\be\label{313}
e^{i\omega(x)/h}\Oph(a)\left( e^{-i\omega(x)/h}v_h\right)
=\Oph(a\circ \chi_\omega(x,\xi))v_h+h\Oph(b)v_h
\ee
modulo $E^p_\emptyset$.

$ii)$ If $(v_h)_h$ is in $E^p_K$, then 
$\left( e^{-i\omega(x)/h}v_h\right)_h$ is in $E^p_{K'}$.
\end{lemm}
\begin{proof}
We shall prove both assertions at the same time. 
Remark first that since $(\theta v_h-v_h)_h$ is in $E^p_\emptyset$ 
if $\theta$ is in $C^\infty_0(U)$ equal to one on a neighborhood of $K_1$, 
we may always assume that $v_h$ is compactly supported in $U$. By symbolic calculus, 
and the assumption $a\in S(1,K')$, we may also assume that 
$a$ is compactly supported 
and that the first projection of the support is contained in $U$. Consequently, we may replace in \eqref{313} $\omega$ by a $C^\infty_0(\xR)$ function, 
equal to the given phase in a neighborhood of $K_1$.

We compute
\be\label{314}
e^{i\omega(x)/h}\Oph(a)\left(e^{-i\omega(x)/h}v_h\right) 
=\frac{1}{2\pi}\int e^{ix\xi} c(x,h\xi,h)\widehat{v_h}(\xi)\, d\xi
\ee
with
\be\label{315}
\begin{aligned}
c(x,\xi,h)&=\frac{1}{2\pi h}\int e^{-i[ y\eta-(\omega(x)-\omega(x-y))]/h}a(x,\xi-\eta,h)\, dy\, d\eta\\
&=\frac{1}{2\pi h}\int e^{-iy\eta/h}a(x,\xi-\eta-\theta(x,y),h)\, dy\, d\eta
\end{aligned}
\ee
where ${\displaystyle{\theta(x,y)=\frac{\omega(x)-\omega(x-y)}{y}}}$. Let 
$\kappa$ be a smooth function supported in a small 
neighborhood of zero in $\xR$ and equal to one close to zero. 
We insert under the last oscillatory integral in \eqref{315} a factor 
$\kappa(y)\kappa(\eta)$. The error introduced in that way is a symbol in $h^\infty S(\langle \xi\rangle^{-\infty})$. 
The action of the associated operator on $v_h$ gives an element of $E^p_\emptyset$. We have reduced ourselves to
\be\label{316}
\frac{1}{2\pi h}\int e^{-iy\eta/h}a(x,\xi-\eta-\theta(x,y),h)\kappa(y)\kappa(\eta)\, dy\, d\eta.
\ee
The argument of $a$ belongs to $K'_r$ if $(x,\xi)$ is in $K_{r'}$ with $r'\ll r$ 
and the support of $\kappa$ has been taken small enough. 
Moreover, it is given by $(x,\xi-\diff \omega(x)+O(y)+O(\eta))$, 
so that an integration by parts shows that \eqref{316} may be written 
$a\circ \chi_\omega+h b$ for some symbol $b$ in $S(1,K)$. This gives $i)$. 

To check $ii)$ we apply \eqref{314} with $a=\phi'$ an element of $S(1)$ 
supported in $K'_{r_0'}\times ]0,1]$, equal to one on $K'_{r_1'}\times ]0,1]$ 
for some $0<r'_1<r'_0$. We assume that 
$\lA \Oph(\phi)v_h-v_h\rA_{L^p}=O(h^\infty)$ for some 
$\phi$ in $S(1)$ supported in $K_{r_0}\times ]0,1]$, 
$\phi\equiv 1$ on $K_{r_1}\times ]0,1]$ with 
$r_1<r_0\ll r_1'$. Then if $(x,\xi)$ is in $K_{r_0}$ and 
$\supp \kappa$ has been taken small enough in \eqref{316}, we see that this integral is equal to
$$
\frac{1}{2\pi h}\int e^{-iy\eta/h}\kappa(y)\kappa(\eta)\, dy\, d\eta,
$$
which is equal to one modulo $O(h^\infty)$. 
We conclude from \eqref{314} where we replaced 
$v_h$ by $\Oph(\phi)v_h$ modulo $O(h^\infty)$ 
and from symbolic calculus 
that
$$
\blA e^{i\omega/h} \Oph(\phi')\bigl( e^{-i\omega/h} v_h\bigr)-v_h\brA_{L^p}=O(h^\infty),
$$
which is the wanted conclusion. 
\end{proof}

\paragraph*{Lagrangian distributions}

We consider $\Lambda$ a Lagrangian submanifold of $T^*(\xR\setminus \{0\})$ that is, 
since we are in a one-dimensional setting, a smooth curve of $T^*(\xR\setminus \{0\})$. We shall 
assume that
$$
\Lambda=\left\{\, (x,\diff \omega(x))\,;\, w\in \xR^*\,\right\}
$$
for $\omega$ a smooth function from $\xR^*$ to $\xR$. We want to define 
semi-classical lagrangian distributions on $\Lambda$ i.e.\ distributions 
generalizing families of 
oscillating functions $(\theta(x)e^{i\omega(x)/h})_h$. Since in our applications $\omega$ 
will be homogeneous of degree $-1$, so will have a singularity at zero, we shall define in a first step 
these distributions above a compact subset of $\xR\setminus\{0\}$. In a 
second step, the lagrangian distributions along $\Lambda$ will be defined as sums of 
conveniently rescaled Lagrangian distributions on a compact set. 

We fix $\sigma$, $\beta$ two small positive numbers and consider two Planck constants $h$ and $\hbar$ 
satisfying the inequalities
\be\label{317}
0<C_0^{-1}h^{1+\beta}\le \hbar\le C_0 h^\sigma\le 1
\ee
for some constant $C_0>0$. Notice that these inequalities imply that $O(\hbar^\infty)$ remainders will be also 
$O(h^\infty)$ remainders.

\begin{defi}\label{ref:3.1.5}
Let $F$ be a closed nonempty subset of $T^*\xR$ such that $\pi_2(F)$ is compact. 
Let $\nu,\mu$ be in $\xR$, $\gamma\in \xR_+$, $p\in [1,+\infty]$. 

$i)$ One denotes by 
$h^\nu \mB_p^{\mu,\gamma}[F]$\index{Lagrangian distributions!$h^\nu \mB_p^{\mu,\gamma}[F]$}
 the space of elements 
$(v_\hbar)_\hbar$ of $E^p_F/E^p_\emptyset$, indexed by $\hbar$ and 
depending on $h$, such that there is $C>0$ and for any $h,\hbar$ in $]0,1]$ 
satisfying \eqref{317} 
\be\label{318}
\lA v_\hbar\rA_{L^p}\le Ch^\nu \left(\frac{h}{\hbar}\right)^{\mu+\frac{1}{p}}\left(1+\frac{h}{\hbar}\right)^{-2\gamma}.
\ee
We denote
$$
h^\nu \mB_{p}^{\mu,\gamma}=\bigcup_{F}h^\nu \mB_p^{\mu,\gamma}[F],
$$
where the union is taken over all closed non empty subsets $F$ of $T^*\xR$ such that 
$\pi_2(F)$ is compact. 

$ii)$ Let $K$ be a compact subset of $T^*(\xR\setminus\{0\})$ such that $K\cap \Lambda\neq \emptyset$. 
Denote by $e$ an equation of $\Lambda$ defined on a neighborhood of $K$. One denotes by 
$\hLI$\index{Lagrangian distributions!$\hLI$} (resp.\ $\hLJ$) \index{Lagrangian distributions!$\hLJ$} 
the subspace of $\hBK$ made of those families of functions 
$(v_\hbar)_\hbar$ such that there is $C>0$ and for any $h,\hbar$ in $]0,1]$ 
satisfying \eqref{317}, one has the inequality
\be\label{319}
\lA \Ophbar(e)v_\hbar\rA_{L^p}\le C h^\nu 
\left(\frac{h}{\hbar}\right)^{\mu+\frac{1}{p}}\left(1+\frac{h}{\hbar}\right)^{-2\gamma}\left[ h^\mez+\hbar\right]
\ee
respectively, the inequality
\be\label{3110}
\lA \Ophbar(e)v_\hbar\rA_{L^p}\le C h^\nu 
\left(\frac{h}{\hbar}\right)^{\mu+\frac{1}{p}}\left(1+\frac{h}{\hbar}\right)^{-2\gamma}\hbar.
\ee
\end{defi}

Notice that by definition $\hLJ$ is included in $h^\nu L^p I_\Lambda^{\mu,\nu}[K]$. 
If $\tilde{e}$ is another equation of $\Lambda$ close to $K$, we may write $\tilde{e}=ae$ 
for some symbol $a\in S(1,K)$ on a neighborhood of $K$. By Theorem~\ref{ref:A.2.2}, 
$\Ophbar(\tilde e)=\Ophbar(a)\Ophbar(e)+\hbar \Ophbar(b)$ 
for another symbol $b$ in $S(1,K)$. Consequently 
\eqref{318} and \eqref{319} imply that the same estimate holds with $e$ replaced by $\tilde{e}$, 
so that the space $\hLI$ depends only on $\Lambda$. The same holds for $\hLJ$. 
In particular, because of our definition of $\Lambda$, we may take $e(x,\xi)=\xi-\diff\omega(x)$. 

\begin{exam}
Let $\theta$ be in $C^\infty_0(\xR\setminus\{0\})$ and 
set $v_\hbar(x)=\theta(x)e^{i\omega(x)/\hbar}$. Then $\vhbar$ is in $L^\infty J_\Lambda^{0,0}[K]$ 
for any compact subset of $T^*(\Rs)$ meeting $\Lambda$ such that 
$\supp \theta\subset \pi_1(K)$. Actually $\vhbar$ is microlocally supported close to $K$ 
and Lemma~\ref{ref:3.1.4} shows that 
$\Ophbar (\xi-\diff\omega(x))v_\hbar$ satisfies estimate \eqref{3110} with $p=\infty$, $\nu=\mu=\gamma=0$. 
Notice that in this example, one could apply $\Ophbar(\xi-\diff\omega(x))$ several times to 
$v_\hbar$, and gain at each step one factor $\hbar$ in the $L^\infty$ estimates. It turns 
out that the lagrangian distributions we shall have to cope with will not satisfy such a strong statement, but only 
estimates of type \eqref{319} or \eqref{3110}.
\end{exam}

\begin{prop}\label{ref:3.1.6}
Let $p\in [1,+\infty]$, $\mu,\gamma$ in $\xR$, $K$ a compact 
subset of $T^*(\Rs)$ with $\Lambda\cap K\neq \emptyset$. 

$i)$ Let $a$ be in $S(1,K)$ and $\vhbar$ an element of $\LI$. Then 
$\left((\Ophbar (a)-a(x,\diff\omega(x))v_\hbar\right)_\hbar$ is in $(h^{1/2}+\hbar)\mB_p^{\mu,\gamma}[K]$. 

Assume we are given a vector field $Z=\alpha(\hbar,x)D_\hbar+\beta(\hbar,x)D_x$ 
satisfying the following conditions: 
$\lA Z\hbar\rA_{L^\infty}=O(\hbar)$ and 
if $e$ is a symbol in $S(1,K)$ (resp.\ that vanishes on $\Lambda$), 
then $[Z,\Ophbar(e)]=\Ophbar(\tilde{e})$ for some other symbol $\tilde{e}$ in $S(1,K)$ 
(resp.\ that vanishes on $\Lambda$). Assume also that 
for some integer $k$, 
$Z^{k'}v_\hbar$ is in $\LI$ for $0\le k'\le k$. 
Then $Z^k\Ophbar(a)v_\hbar$ is in $\LI$ and 
$Z^k[(\Ophbar(a)-a(x,\diff\omega))v_\hbar]$ is in 
$(h^{1/2}+\hbar)\mB_p^{\mu,\gamma}[K]$.

$ii)$ Denote by $\Lambda_0$ the zero section of $T^*\xR$. 
Let $\vhbar$ be in $\LI$. Then $(e^{-i\omega/\hbar}v_\hbar)_\hbar$ is $L^pI_{\Lambda_0}^{\mu,\gamma}[K_0]$ 
where $K_0=\chi_\omega(K)$. Conversely, if $\vhbar$ is in $L^pI_{\Lambda_0}^{\mu,\gamma}[K_0]$, 
$(e^{i\omega/\hbar}v_\hbar)_\hbar$ is in $\LI$.

$iii)$ Let $\Lambda_1,\Lambda_2$ be two Lagrangian submanifolds of $T^*(\Rs)$ satisfying 
the same assumptions as $\Lambda$, let $K_1,K_2$ be compact subsets of $T^*(\Rs)$ with 
$\Lambda_1\cap K_1\neq \emptyset$, $\Lambda_2\cap K_2\neq \emptyset$. 
Set
$$
\Lambda_1+\Lambda_2=\left\{\, (x,\xi_1+\xi_2)\,;\, (x,\xi_1)\in \Lambda_1,(x,\xi_2)\in \Lambda_2\,\right\}
$$
and define in the same way $K_1+K_2$. Let $p_1,p_2$ be in $[1,\infty]$ 
with $\frac{1}{p_1}+\frac{1}{p_2}=\frac{1}{p}$, $\mu_1,\mu_2,\gamma_1,\gamma_2$ in $\xR$ with 
$\mu_1+\mu_2=\mu$, $\gamma_1+\gamma_2=\gamma$. Let $(v^\ell_\hbar)_\hbar$ be in 
$L^{p_\ell}I_{\Lambda_\ell}^{\mu_\ell,\gamma_\ell}[K_\ell]$ for $\ell=1,2$. 
Then $(v^1_\hbar \cdot v^2_\hbar)_{\hbar}$ is in $L^p I_{\Lambda_1+\Lambda_2}^{\mu,\gamma}[K_1+K_2]$. 

A similar statement holds for the classes $\LJ$ and $\mB_p^{\mu,\gamma}[K]$.
\end{prop}
\begin{proof} 
$i)$ We may always modify $\omega$ outside a neighborhood of $\pi_1(K)$ so that 
it is compactly supported, and this will modify the quantities at hand only by an element of 
$E^p_\emptyset$. We may find a symbol $b$ in $S(1,K)$ so that
$$
a(x,\xi)-a(x,\diff\omega(x))=b(x,\xi)(\xi-\diff\omega(x))
$$
in $K_r$ for some small $r$. By the symbolic calculus of appendix~\ref{S:A.2},
$$
\Ophbar(a)v_\hbar-a(x,\diff\omega)v_\hbar=\Ophbar(b)\Ophbar(\xi-\diff\omega(x))v_\hbar
+\hbar \Ophbar (c)v_\hbar
$$
for a new symbol $e$ in $S(1,K)$. The conclusion follows from estimate \eqref{319}. 

If we make act a vector field $Z$ as in the statement on the last equality and use the 
commutation assumptions, we obtain the last statement of $i)$.

$ii)$ We have seen in Lemma~\ref{ref:3.1.4} that $(e^{-i\omega/\hbar}v_\hbar)_\hbar$ is in 
$E^p_{K_0}/E^p_\emptyset$. Since $\Ophbar(\xi)(e^{-i\omega/\hbar}v_\hbar)=
e^{-i\omega/\hbar}\Ophbar(\xi-\diff\omega(x))v_\hbar$, we deduce from \eqref{319} 
the statement. 

$iii)$ Denote by $\omega_1,\omega_2$ two smooth functions, that may be assumed to be 
compactly supported close to $\pi_1(K_1)$, $\pi_1(K_2)$ respectively, such that 
$\Lambda_\ell=\{(x,\diff\omega_\ell(x))\}$ close to $K_\ell$, $\ell=1,2$. 
Then $\omega=\omega_1+\omega_2$ parametrizes $\Lambda=\Lambda_1+\Lambda_2$ close to 
$K_1+K_2$. We define 
$w^\ell_\hbar=e^{-i\omega^\ell/\hbar}v_\hbar^\ell$. By $ii)$, 
$(w^\ell_\hbar)_\hbar$ is in $L^{p_\ell}I_{\Lambda_0}^{\mu_\ell,\gamma_\ell}[K_{\ell,0}]$, 
where $K_{\ell,0}=\chi_{\omega_\ell}(K_\ell)$. Writing 
a product from the convolution of the Fourier transforms of the factors, we 
see that $(w^1_\hbar w^2_\hbar)_\hbar$ is in 
$E^p_{K_{1,0}+K_{2,0}}/E^p_\emptyset$. 
Let us check that 
$w^1_\hbar w^2_\hbar$ satisfies estimate \eqref{319} when $e$ is an equation 
of $\Lambda_0$ i.e.\ $e(x,\xi)=\xi$ so that 
$\Ophbar(e)=\hbar D_x$. We write
$$
\lA \hbar D_x (w^1_\hbar w^2_\hbar)\rA_{L^p}\le 
\lA \hbar D_x w^1_\hbar \rA_{L^{p_1}}\lA w^2_\hbar\rA_{L^{p_2}}
+\lA w^1_\hbar\rA_{L^{p_1}} \lA \hbar D_x w^2_\hbar\rA_{L^{p_2}}
$$
and use \eqref{318}, \eqref{319} for each factor 
to get that $(w^1_\hbar w^2_\hbar)_\hbar$ is in 
$L^{p}I_{\Lambda_0}^{\mu,\gamma}[K_{1,0}+K_{2,0}]$. We just have to apply again $ii)$ 
to $v_\hbar=e^{i\omega/\hbar} (w^1_\hbar w^2_\hbar)$ to get the conclusion. The proof is similar for classes $\LJ$ and $\mB_p^{\mu,\gamma}[K]$.
\end{proof}

We have defined, up to now, classes of Lagrangian distributions 
microlocally supported close to a compact set of the phase space. 
We introduce next classes of Lagrangian distributions that do not obey 
such a localization property. 

From now on, we consider phase functions $\omega\colon\Rs\rightarrow \xR$ 
which are smooth, non zero, and positively homogeneous of degree $-1$. 
We set
\be\label{3111}
\Lambda=\left\{(x,\diff\omega(x))\,;\,x\in\Rs\right\}\subset T^*(\Rs)
\ee
so that $\Lambda$ is invariant under the action of $\xR_+^*$ on $T^*(\Rs)$ given by 
$\lambda\cdot(x,\xi)=(\lambda x,\lambda^{-2}\xi)$. For $h\in ]0,1]$, $C$ a positive 
constant, we introduce the notations\index{Semi-classical notations!$J(h,C)$}
\index{Semi-classical notations!$h_j$}
\index{Semi-classical notations!$j_0(h,C)$, $j_1(j,C)$}
\be\label{3112}
\begin{aligned}
&J(h,C)=\left\{ j\in\xZ\,;\, C^{-1}h^{2(1-\sigma)}\le 2^j \le Ch^{-2\beta}\right\},\\
&h_j=h 2^{-j/2} \text{ if }j\in J(h,C),\\[1ex]
&j_0(h,C)=\min (J(h,C))-1,\quad j_1(h,C)=\max(J(h,C))+1.
\end{aligned}
\ee
We note that \eqref{317} is satisfied by $\hbar=h_j$ if 
$j\in J(h,C)$ (for a constant $C=C_0^2$). 
For $j\in \xZ$, $v$ a distribution on $\xR$, we set
\index{Semi-classical notations!$\Tj$}
$$
\Theta_j^*v=v\bigl( 2^{j/2}\cdot\bigr),
$$
so that in particular, if $p\in [1,\infty]$, 
$\blA \Theta_j^*\brA_{\mathcal{L}(L^p,L^p)}=2^{-j/(2p)}$. 
If $a$ belongs to the class of symbols $S(m)$ and 
if $a_j(x,\xi)=a(2^{-j/2}x,2^j\xi)$ we notice that for $j\in J(h,C)$
\be\label{3113}
\Theta_{-j}^*\Oph(a)\Theta_j^*=\Op_{h_j}(a_j).
\ee
We fix a function $\varphi$ in $C^\infty_0(\xR^d)$ such that 
${\displaystyle{\sum_{j\in \xZ}\varphi(2^{-j}\xi)\equiv 1}}$. We define
\index{Semi-classical notations!$\Djh$}
$$
\varphi_0(\xi)=\sum_{j=-\infty}^{-1}\varphi(2^{-j}\xi),\quad 
\Djh=\Oph(\varphi(2^{-j}\xi))=\varphi(2^{-j}h D).
$$
\begin{defi}\label{ref:3.1.7}
Let $\nu\in\xR$, $p\in [1,\infty]$, $b\in \xR$. 
One denotes by \index{Lagrangian distributions!$h^\nu \mR_p^b$}$h^\nu \mR_p^b$ the space of families of 
$L^p$-functions $\vhbar$ such that there is $C>0$ and
\be\label{3114}
\begin{aligned}
&\blA \Djh v_h\brA_{L^p}\le C_0 h^\nu 2^{-j+b} \quad\text{for }j\ge j_0(h,C)\\
&\blA \Oph(\varphi_0(2^{-j_0(h,C)}\xi))v_h\brA_{L^p}\le C h^\nu,
\end{aligned}
\ee
where $j_+=\max(j,0)$.
\end{defi}
Clearly the definition is independent of the choice of $\varphi_0$. 

\begin{defi}\label{ref:3.1.8}
Let $\Lambda$ be a lagrangian submanifold of form \eqref{3111}, 
$K$ a compact subset of $T^*(\Rs)$ meeting $\Lambda$. 
Let $\nu,\mu$ be in $\xR$, $\gamma\in \xR_+$, $F$ 
a closed non empty subset of $T^*\xR$ 
such that $\pi_2(F)$ is compact in $\xR$. 
One denotes by \index{Lagrangian distributions!$\htLI$}
$\htLI$ (resp.\ \index{Lagrangian distributions!$\htLJ$}$\htLJ$, resp.\ \index{Lagrangian distributions!$\htBF$}$\htBF$) 
the space of families of functions $(v_h)_{h\in ]0,1]}$ such that

$\bullet$ For any $j\in J(h,C)$, there is a family $(v^j_{h_j})_{h_j}$, 
indexed by
$$
h_j \in \Big]0,\min  \Big(C_0^{\frac{1}{1-\sigma}}2^{\frac{j\sigma}{2(1-\sigma)}},C_0^{\frac{1}{\beta}}
2^{-j\frac{1+\beta}{2\beta}}\Big)\Big]
$$
which is an element 
of $\hLI$ (resp.\ $\hLJ$, resp.\ $\hBF$) with the constants in \eqref{318}, \eqref{319}, \eqref{3110} 
uniform in $j\in J(h,C)$. 

$\bullet$ For any $h\in ]0,1]$, $v_h=\sum_{j\in J(h,C)}\Theta_j^*v^j_{h_j}$.

One defines $\htB=\bigcup\htBF$ where the union is taken over 
the sets $F$ which are closed with $\pi_2(F)$ compact in $\xR$.
\end{defi}

\begin{rema}
$\bullet$ The interval of variation imposed to 
$h_j$ in the preceding definition is the one deduced from \eqref{317} with $\hbar = h_j$. 

$\bullet$ The building blocks $(v^j_{h_j})_{h_j}$ in the above definition are defined modulo 
$O(h_j^\infty)$ so modulo $O(h^\infty)$ since $h_j\le h^\sigma$. Since the cardinal of $J(h,C)$ is 
$O(|\log h|)$, we see that the classes introduced in the above definition are well defined modulo 
$O(h^\infty)$. 

$\bullet$ It follows from the above two definitions that 
$h^\nu \tilde{\mB}_p^{0,b} \subset h^\nu \mR_p^b$. Moreover, by \e{3114} and the fact that the cardinal of 
$\xZ_-\cap \{j\ge j_0(h,C)\}$ is $O(|\log h|)$, we see that if $u,v$ are in $\mR_\infty^b$ with $b>0$, then $uv$ is in 
$h^{-0}\mR_\infty^b\defn \cap_{\theta>0}h^{-\theta}\mR_{\infty}^b$.
\end{rema}

Let us prove a statement similar to $i)$ of 
Proposition~\ref{ref:3.1.6} for elements of the classes of distributions we just defined. 

\begin{prop}\label{ref:3.1.9}
We assume that the function $\omega$ defining $\Lambda$ satisfies either $\omega(x)\neq 0$ 
for all $x\in\xR^*$ or $\omega\equiv 0$. In the first (resp.\ second) case we denote 
by $K$ a compact subset of $T^*(\Rs)\setminus 0$ (resp.\ of $T^*(\Rs)$) such that 
$K\cap \Lambda\neq \emptyset$. Let $\mu\in\xR$, $\gamma\in \xR_+$, $p\in [1,\infty]$, 
$k\in \xN$ be given. Consider a function $(x,\xi)\mapsto a(x,\xi)$ smooth on 
$\xR^*\times \xR^*$ (resp.\ $\xR^*\times\xR$) satisfying for some real numbers $\ell,\ell',d,d'$ 
(resp.\ $\ell,\ell'$, $d\ge0$, $d'\ge 0$) and all $\alpha,\beta$ in $\xN$
\be\label{3115}
\la \px^\alpha\partial_\xi^\beta a(x,\xi)\ra \le C_{\alpha\beta}|x|^{\ell-\alpha}\langle x\rangle^{\ell'}
|\xi|^{d-\beta}\langle \xi\rangle^{d'}
\ee
when $(x,\xi)\in \xR^*\times \xR^*$ (resp.\ \eqref{3115} when $\beta\le d$ and
\be\label{3115a}
\px^\alpha\partial_\xi^\beta a(x,\xi)\equiv 0\quad\text{for }\beta>d,
\ee
when $(x,\xi)\in \xR^*\times\xR$).

Denote $Z=-hD_h+xD_x$ and let $(v_h)_h$ satisfy for any $k'\le k$, 
$(Z^{k'}v_h)_h\in \tLI^{\mu,\nu}[K]$ (resp.\ $\ti{\mB}_p^{\mu,\nu}[K]$). 
Then $(Z^k(\Oph(a)v_h))_h$ 
belongs to $\tLI^{\tilde{\mu},\tilde{\gamma}}[K]$ (resp.\ $\ti{\mB}_p^{\ti{\mu},\ti{\nu}}[K]$) where 
$\tilde{\mu}=\mu+2d-\ell -\ell'$, $\tilde{\gamma}=\gamma-\frac{\ell'}{2}-d'$. 

Moreover, under the assumption $(Z^{k'} v_h)_h\in L^p \ti{I}_\La^{\mu,\gamma}[K]$, 
if $\chi$ is in $C^\infty_0(\xR)$ is equal to one close to zero, 
and has small enough support, $Z^k((1-\chi)(xh^{-\beta})a(x,\diff\omega)v_h)$ is also 
in $\tLI^{\tilde{\mu},\tilde{\gamma}}[K]$ and
\be\label{3116}
Z^k \left[ \Oph(a)v_h-(1-\chi)(xh^{-\beta})a(x,\diff\omega)v_h\right]
\ee
belongs to $h^{1/2} \tilde{\mB}_p^{\tilde{\mu},\tilde{\gamma}}[K]+h\tilde{\mB}_p^{\tilde{\mu}-1,\tilde{\gamma}}[K]$. 

If we assume that $(Z^{k'}v_{h})_{h}$ is in $\tLJ^{\mu,\gamma}[K]$ for $k'\le k$, 
we obtain instead that 
$(Z^k\Oph(a)v_h)_h$ and 
$(Z^k(1-\chi)(xh^{-\beta})a(x,\diff\omega)v_h)_h$ belong to 
$\tLJ^{\tilde{\mu},\tilde{\gamma}}[K]$ and that \eqref{3116} is in $h\tilde{\mB}_p^{\tilde{\mu}-1,\tilde{\gamma}}[K]$. 

When $\omega\equiv 0$, if $Z^{k'}v_h$ is in $\tilde{\mB}_p^{\mu,\gamma}[K]$, we 
obtain that $(Z^{k}\Oph(a)v_h)_h$ is in $\tilde{\mB}_p^{\tilde{\mu},\tilde{\gamma}}[K]$.

The same results hold if we quantize $a$ by $\Oph(\overline{a})^*$ instead of $\Oph(a)$ 
i.e.\ under the same assumptions as above 
$(Z^k\Oph(\overline{a})^*(v_h))$ belongs 
to $\tLI^{\tilde \mu,\tilde \gamma}[K]$ and 
\be\label{3116a}
Z^k \left[ \Oph(\overline{a})^*v_h-(1-\chi)(xh^{-\beta})a(x,\diff\omega)v_h\right] 
\ee
belongs to the same spaces as indicated 
above 
after \eqref{3116}. In the same way, when $\omega\equiv 0$, 
and when $(Z^{k'}v_h)_{h}$ is 
in $\tilde{\mB}_p^{\mu,\gamma}[K]$, for $k'\le k$, 
$(Z^k(\Oph(\overline{a})^*v_h))_h$ is in $\tilde{\mB}_p^{\mu,\gamma}[K]$.
\end{prop}
\begin{proof}
According to Definition~\ref{ref:3.1.8}, we represent 
$v_h=\sum_{j\in J(h,C)}\Theta_j^*v^j_{h_j}$ 
where $(v^j_{h_j})_{h_j}$ is a bounded sequence of elements of 
$\LI$, as well as $(Z^{k'}v_{h_j})_{h_j}$ for $k'\le k$. 
By \eqref{3113} and \eqref{3115},
\be\label{3117}
\Oph(a)v_h=\sum_{j\in J(h,C)}\Theta_j^* w^j_{h_j}
\ee
with $w^j_{h_j}=\Op_{h_j}(a_j)v^j_{h_j}$ and 
$$
a_j(x,\xi)=a\left(2^{-j/2}x,2^j\xi\right)=
2^{j\left(d-\frac{\ell+\ell'}{2}\right)+j_+ \left(d'+\frac{\ell'}{2}\right)}
b_j(x,\xi).
$$
When $(x,\xi)$ stays in a compact subset of $T^*(\Rs)\setminus 0$, 
\eqref{3115} shows that $\px^\alpha\partial_\xi^\beta b_j=O(1)$ 
uniformly in $j$. In the same way when $(x,\xi)$ stays in a compact subset of 
$T^*(\Rs)$, \eqref{3115} for $\beta\le d$ and \eqref{3115a} show also that 
$\px^\alpha\partial_\xi^\beta b_j=O(1)$ uniformly in $j$. Since 
$2^j=(h/h_j)^2$, it follows from Theorem~\ref{ref:A.2.2} 
of the appendix that $(w^j_{h_j})_{h_j}$ is a bounded sequence indexed by $j\in J(h,C)$ 
of elements of $L^pI_{\Lambda}^{\tilde{\mu},\tilde{\gamma}}[K]$. Moreover, 
the vector field $Z$ satisfies for any symbol $e$
$$
[Z,\Op_{h_j}(e)]=\Op_{h_j}\left((xD_x-2\xi D_\xi)e\right).
$$
Since either $\Lambda=\{(x,\diff\omega(x))\}$ with 
$\omega$ homogeneous of degree $-1$ or 
$\Lambda=\{(x,0)\}$, we see that $(xD_x-2\xi D_\xi)e$ vanishes on $\Lambda$ 
if $e$ does. Consequently, the assumption of the last statement in $i)$ of 
Proposition~\ref{ref:3.1.6} is satisfied and we conclude 
that $(Zw^j_{h_j})_{h_j}$ is a bounded sequence of elements of $L^pI_\Lambda^{\tilde{\mu},\tilde{d}}[K]$.

To prove \eqref{3116}, we use that again by $i)$ of Proposition~\ref{ref:3.1.6},
$$
w^j_{h_j}=a_j(x,\diff\omega)v^j_{h_j}+\left(h^{1/2}+h_j\right)r^j_{h_j}
$$
where $(r^j_{h_j})_{h_j}$ is a bounded sequence indexed by $j\in J(h,C)$ of 
elements of $\mB_p^{\tilde{\mu},\tilde{\gamma}}[K]$, that stay in that space 
if one applies $Z^{k'}$ ($k'\le k$) on them. Let $\chi$ be as in the statement of the proposition, 
with small enough support. Then, if $x$ is close to 
$\pi_1(K)$ and $j$ is in $J(h,C)$, $(1-\chi)(2^{-j/2}xh^{-\beta})\equiv 1$. 
Consequently, since $(v^j_{h_j})_{h_j}$, as well as 
$(Z^{k'}v^j_{h_j})_{h_j}$ is microlocalized close to $K$, we may write 
$v^j_{h_j}=v^j_{h_j}(1-\chi)(2^{-j/2}x h^{-\beta})$ modulo a remainder which is 
$O(h^{\infty}_j)=O(h^\infty)$ in $L^p$, as well as its $Z^{k'}$-derivatives, $0\le k'\le k$. 
Integrating such a remainder in the $r^j_{h_j}$ contributions, we may write, using that 
$\diff\omega$ is homogeneous of degree $-2$,
$$
w^j_{h_j}=(1-\chi)(2^{-j/2}xh^{-\beta})a\left(2^{-j/2}x,\diff\omega(2^{-j/2}x)\right)v^j_{h_j}
+h^{1/2}r^j_{h_j}+h\tilde{r}^j_{h_j}
$$
where $(r^j_{h_j})_{h_j}$ is as above and $\tilde{r}^j_{h_j}=2^{-j/2}r^j_{h_j}$ is such that 
$(Z^{k'}\tilde{r}^j_{h_j})_{h_j}$ is in $\mB_p^{\tilde{\mu}-1,\tilde{\gamma}}[K]$ for 
$k'\le k$. This gives \eqref{3116} if we plug this expansion in \eqref{3117}.

To check that $\left(Z^k ((1-\chi)(xh^{-\beta})a(x,\diff\omega)v_h)\right)_{h}$ is also 
in $\tLI^{\tilde \mu,\tilde \gamma}[K]$, we write the function on which acts $Z^k$ as 
$$
\sum_{j\in J(h,C)}\Theta_j^* \left[ a(2^{-j/2}x,\diff\omega(2^{-j/2}x))(1-\chi)(x2^{-j/2}h^{-\beta})
v^j_{h_j}\right]
$$
and remark that, as above, the assumptions of microlocal localization of $(v^j_{h_j})_{h_j}$ 
allow one to remove the cut-off $(1-\chi)$ up to $O(h^\infty)$ remainders. 
Since $\diff\omega$ is homogeneous of degree $-2$,
$$
a(2^{-j/2}x,\diff\omega(2^{-j/2}x))
=O\left(2^{j\left(d-\frac{\ell+\ell'}{2}\right)+j_+ \left(d'+\frac{\ell'}{2}\right)}\right)
$$
when $x$ stays in a compact subset of $\xR^*$, so that the above sum defines an element of 
$\tLI^{\tilde \mu,\tilde \gamma}[K]$. 

The statement of the proposition concerning the case when 
$(Z^{k'}v_h)_h$ is in $\tLJ[K]$ is proved similarly, as well as the one 
about $\tilde{\mB}_p^{\mu,\gamma}[K]$.

Finally, the statements concerning $\Oph(\overline{a})^*$ instead of 
$\Oph(a)$ are proved in the same way: one may write \eqref{3117} with 
$w^j_{h_j}$ given by $\Oph(\overline{a_j})^*v^j_{h_j}$. By Theorem~\ref{ref:A.2.2} in 
the appendix, we know that there is a symbol 
$b_j$ in $S(1,K)$ uniformly in $j$, such that $\Op_{h_j}(\overline{a_j})^*=\Op_{h_j}(b_j)$. 
Moreover, $b_j(x,\xi)=a_j(x,\xi)+h_j c_j(x,\xi)$ for some other symbol 
$c_j$ in $S(1,K)$ uniformly in $j$. The statements concerning $\Op_h(\overline{a})^*v_h$ 
thus follows from those we just proved for $\Oph(a)v_h$.
\end{proof}

Let us study products. 

\begin{prop}\label{ref:3110}
Let $p_1,p_2,p$ be in $[1,+\infty]$ with $\frac{1}{p_1}+\frac{1}{p_2}=\frac{1}{p}$, $\gamma_1,\gamma_2$ in $\xR$, 
$\Lambda_1,\Lambda_2$ be two Lagrangian submanifolds of $\cotangent$ of the form \eqref{3111}, defined in terms 
of phase functions $\omega_1,\omega_2$ homogeneous 
of degree $-1$, $\omega_1\not\equiv 0$, $\omega_2\not\equiv 0$. Let $K_1$, $K_2$ be two compact subsets of $\cotangent$ with $K_\ell\cap \Lambdal\neq \emptyset$, $\ell=1,2$. Let 
$(\vl)_h$ be an element of $\tilde{\mB}_{\pl}^{\mul,\gammal}[K_\ell]$ (resp.\ $L^{\pl}\tilde{I}_{\Lambdal}^{\mul,\gammal}[K_\ell]$, 
resp.\ $L^{\pl}\tilde{J}_{\Lambdal}^{\mul,\gammal}[K_\ell]$) $\ell=1,2$. 

There is a compact subset $K$ of $\cotangent$ with $K\cap (\Lambda_1+\Lambda_2)\neq \emptyset $ 
such that $(v_h^1\cdot v_h^2)_h$ belongs to $\tilde{\mB}_{p}^{\mu,\gamma}[K]$ 
(resp.\ $L^{p}\tilde{I}_{\Lambda_1+\Lambda_2}^{\mu,\gamma}[K]$, 
resp.\ $L^{p}\tilde{J}_{\Lambda_1+\Lambda_2}^{\mu,\gamma}[K]$) with $\mu=\mu_1+\mu_2$, $\gamma=\gamma_1+\gamma_2$. Moreover, for any neighborhood $\Omega$ of $\Lambda_1+\Lambda_2$, any compact subset $L$ 
of $\xR\setminus\{0\}$, there are neighborhoods $\Omega_\ell$ of $\Lambda_\ell$, $\ell=1,2$, 
such that if $K_\ell\subset \Omega_\ell \cap \pi^{-1}_1(L)$, $\ell=1,2$, then $K\subset \Omega$.
\end{prop}
\begin{proof}
By Definition~\ref{ref:3.1.8}, we may write for $\ell=1,2$, 
$$
v_h^\ell=\sum_{\jl\in J(h,C)}\Theta_{\jl}^*v^{\ell,\jl}_{h_{\jl}}
$$
where $(v^{\ell,\jl}_{h_{\jl}})_{h_{\jl}}$ is a bounded sequence of $L^{\pl}I_{\Lambdal}^{\mul,\gammal}[K_\ell]$. We write
\be\label{3117a} 
v_h^1\cdot v_h^2=\sum_{j_1\in J(h,C)}\Theta_{j_1}^*w^{j_1}_{h_{j_1}}
\ee
with
\be\label{3118}
w^{j_1}_{h_{j_1}}=v^{1,j_1}_{h_{j_1}}\sum_{j_2\in J(h,C)}\Theta_{j_2-j_1}^*v^{2,j_2}_{h_{j_2}}.
\ee
Because of the microlocal localization properties of $(v^{\ell,\jl}_{h_{\jl}})_{h_{\jl}}$ we may, up to an 
$O(h^\infty)$ remainder in $L^{\pl}$, replace 
$(v^{\ell,\jl}_{h_{\jl}})_{h_{\jl}}$ by $(\theta v^{\ell,\jl}_{h_{\jl}})_{h_{\jl}}$ where $\theta$ is in $C^\infty_0(\xR)$ and 
is equal to one on a large enough compact subset of $\xR^*$. This shows that in \eqref{3118}, we may limit the summation 
to those $j_2$ such that $|j_2-j_1|\le C_0$ for some large enough $C_0$, up to remainders which are $O(h^\infty)$ in 
$L^p$. Define 
$$
\tilde{v}^{1,j_1}_{h_{j_1}}=
\sum_{\substack{j_2\in J(h,C)\\ |j_1-j_2|\le C_0}}
\Theta_{j_2-j_1}^*v^{2,j_2}_{h_{j_2}}.
$$
Then $(\tilde{v}^{1,j_1}_{h_{j_1}})_{h_{j_1}}$ is a bounded sequence of 
$\mB_{p_2}^{\mu_2,\gamma_2}[\tilde{K}_2]$ 
(resp.\ $L^{p_2}I_{\Lambda_2}^{\mu_2,\gamma_2}[\tilde{K}_2]$, 
resp.\ $L^{p_2}J_{\Lambda_2}^{\mu_2,\gamma_2}[\tilde{K}_2]$) 
for some large enough compact subset $\tilde{K}_2$ of $\cotangent$, as follows 
from \eqref{3113} and the homogeneity properties of $\Lambda_2$. We just need to apply $iii)$ 
of Proposition~\ref{ref:3.1.6} to conclude that $(w^{j_1}_{h_{j_1}})_{h_{j_1}}$ is a bounded sequence of elements 
of $\mB_{p}^{\mu,\gamma}[K_1+\tilde{K}_2]$ 
(resp.\ $L^{p}I_{\Lambda_1+\Lambda_2}^{\mu,\gamma}[K_1+\tilde{K}_2]$, 
resp.\ $L^{p}J_{\Lambda_1+\Lambda_2}^{\mu,\gamma}[K_1+\tilde{K}_2]$). 

The last statement of the proposition follows 
from the fact that $K=K_1+\tilde{K}_2$, and that $\tilde{K}_2$ mat be taken 
in an arbitrary neighborhood of $\Lambda_2$ if $K_2$ is contained in an even smaller neighborhood of that submanifold. 
\end{proof}

\begin{prop}\label{ref:3.1.12ii} 
Let $F_1,F_2$ be closed subsets 
of $T^*\xR$ such that $\pi_2(F_\ell)$ is compact, 
$\ell=1,2$. Let $(v^\ell_h)$ be 
in $\tilde{\mB}_\infty^{\mu_\ell,\gamma}[F_\ell]$ 
with $\mu_\ell\ge 0$, $\gamma>\mu_\ell$. Then $v^1_h\cdot v^2_h$ is in $h^{-\theta}\mB_\infty^{\mu,\gamma}[F]$ 
with $\mu=\mu_1+\mu_2$ for any $\theta>0$ and some closed 
subset $F$ of $T^*\xR$ whose second projection is compact. 
\end{prop}
\begin{proof}
We write \eqref{3117a}
$$ 
v_h^1\cdot v_h^2=\sum_{j_1\in J(h,C)}\Theta_{j_1}^*w^{1}_{h_{j_1}}
+\sum_{j_2\in J(h,C)}\Theta_{j_2}^*w^{2}_{h_{j_2}}
$$
with
$$
w^{1}_{h_{j_1}}=v^{1,j_1}_{h_{j_1}}\sum_{\substack{j_2\in J(h,C) \\ j_2\le j_1}}
\Theta_{j_2-j_1}^*v^{2,j_2}_{h_{j_2}}
$$
and a symmetric expression for $w_{h_{j_2}}^2$. 
Then $w^1_{h_{j_1}}$ is microlocally 
supported in some closed 
subset of $T^*\xR$, whose $\xi$ projection is compact, 
as
$$
\Op_{h_{j_1}}(\tilde{\varphi})
\Big[ \Theta_{j_2-j_1}^*v^{2,j_2}_{h_{j_2}}\Big]
=\Theta_{j_2-j_1}^*v^{2,j_2}_{h_{j_2}}
$$
if $\tilde{\varphi}$ is supported for $|\xi|\le C$, equal to one 
on $|\xi|\le C/2$, for some $C>0$. The $L^\infty$-norm of $w^1_{h_{j_1}}$ is bounded 
from above by
$$
2^{j_1\mu_1/2- j_{1+}\gamma}\sum_{\substack{j_2\in J(h,C) \\ j_2\le j_1}} 2^{j_2 \mu_2/2}2^{-j_{2 +}\gamma} 
\le C|\log h| 2^{j_1\mu/2-j_{1+}\gamma},
$$
since $\mu_\ell\ge 0$, $\gamma>\mu_\ell/2$.
\end{proof}

\section{The semi-classical water waves equation}\label{S:32}

Let us recall an equivalent form of the water waves equation that is obtained in the companion paper~\cite{AlDel} (see
Corollary 4.3.13.\ in that paper). If 
$(\eta,\psi)$ is a solution of the water waves equation, if 
$\Zr$ denotes the collection 
of vector fields $\Zr=(Z,\px)$ and if 
we assume that \index{Unknowns in Chapter~\ref{chap:3}!$u=\Dxmez \psi+i\eta$}$u=\Dxmez \psi+i\eta$ satisfies for $k$ smaller than some integer $s_0$, for 
$\alpha>0$ large enough and $d\in \xN$, 
$$
\sup_{[T_0,T]}\blA \Zr^k u(t,\cdot)\brA_{H^{d+\alpha}}<+\infty,\quad 
\sup_{[T_0,T]}\blA \Zr^k u(t,\cdot)\brA_{\eC{d+\alpha}}<+\infty
$$
on an interval $[T_0,T]$, we may write, using the notation $\mU=(u,\overline{u})$,
\be\label{321}
D_t u=\Dxmez u +\mQ+\mC+\widetilde{\mR}_0(\mU),
\ee
where $\mQ$ denotes the quadratic part of the nonlinearity
\be\label{322}
\begin{aligned}
\mQ&=-\frac{i}{8}\Dxmez \Bigl[ \bigl( \OD \Dx^{-\mez} (u+\bar{u})\bigr)^2
+\bigl(\Dxmez (u+\bar{u})\bigr)^2\Bigr]\\
&\quad +\frac{i}{4} \Dx \bigl( (u-\bar{u})\Dxmez (u+\bar{u})\bigr)-\frac{i}{4}\OD \bigl( (u-\bar{u})\OD \Dx^{-\mez}
(u+\bar{u})\bigr),
\end{aligned}
\ee
$\mC$ stands for the cubic contribution 
\be\label{323}
\begin{aligned}
\mC&=\frac{1}{8}\Dxmez \Bigl[ \big(\Dxmez (u+\bar{u})\big)\Dx \Big( (u-\bar{u})\Dxmez (u+\bar{u})\Big)\Big]\\
&\quad -\frac{1}{8}\Dxmez \Big[ \big( \Dxmez (u+\bar{u})\Big((u-\bar{u})\Dx^{\tdm}(u+\bar{u})\Big)\Big]\\
&\quad -\frac{1}{8}\Dx \Big[ (u-\bar{u})\Dx \Big((u-\bar{u})\Dxmez (u+\bar{u})\Big)\Big]\\
&\quad +\frac{1}{16}\Dx \Bigl[ (u-\bar{u})^2\Dx^{\tdm}(u+\bar{u})\Bigr] +\frac{1}{16}\Dx^2 \Bigl[ (u-\bar{u})^2\Dxmez (u+\bar{u})\Bigr]
\end{aligned}
\ee
and where $\widetilde{\mR}_0(\mU)$ is a remainder, vanishing at least at order $4$ at $\mU=0$, 
which satisfies for $k\le s_0$ the following estimates
\be\label{324}
\blA \Zr^k \Dx^{-\mez}\widetilde{R}_0(\mU)\brA_{H^d}\le C_k[u]\sum_{\substack{k_1+\cdots+k_4\le k\\ k_1,k_2,k_3\le k_4}}
\prod_{j=1}^{3}\blA \Zr^{k_j}u\brA_{\eC{d+\alpha}}\blA \Zr^{k_4}u\brA_{H^{d+\alpha}}
\ee
with a constant $C_k[u]$ depending only on $\blA \Zr^{(k-1)_+}u\brA_{\eC{d+\alpha}}$ and, 
if $\theta>0$ is small enough, 
\be\label{325}
\blA \Zr^k \Dx^{-\mez+\theta}\widetilde{R}_0(\mU)\brA_{\eC{d}}\le C_k[u]\sum_{k_1+\cdots+k_4\le k}
\prod_{j=1}^{4}\blA \Zr^{k_j}u\brA_{\eC{d+\alpha}}
\ee
where $C_k[u]$ depends only on $\blA \Zr^{(k-1)_+}u\brA_{\eC{d+\alpha}}$ and on a bound on 
$\blA u\brA_{L^2}^{1-2\theta'}\blA u\brA_{L^\infty}^{2\theta'}$ for some $\theta'\in ]0,\theta[$.

We make the change of variables $t=t'$, $x=t'x'$ and set $h=t'^{-1}$, 
$u(t,x)=h^{1/2}v(t',x')$, \index{Unknowns in Chapter~\ref{chap:3}!$v$, after change of variables}so that
$$
D_t u=h^\mez \Bigl[ (D_{t'}-x'h D_{x'})v+\frac{i}{2} hv\Bigr].
$$
The vector field $Z=t\partial_t +2x\px$ becomes $Z=t'\partial_{t'}+x'\partial_{x'}$. 
We deduce from \eqref{321} the following equation for $v$, in which 
we write $(t,x)$ instead of $(t',x')$, since we shall not go back to the 
old coordinates 
\be\label{326}
\bigl(D_{t}-\Oph(x\xi+|\xi|^\mez) \bigr)v
=\sqrt{h}\QV+h\Bigl[-\frac{i}{2}v+\CV\Bigr]
+h^{\frac{11}{8}}\RV
\ee
where $V=(v,\bar{v})$\index{Unknowns in Chapter~\ref{chap:3}!$V=(v,\bar{v})$},
\be\label{327}
\begin{aligned}
\QV&=-\frac{i}{8}\Oph(|\xi|^{\mez}) \Bigl[ \bigl( \Oph(\xi|\xi|^\mez) (v+\bar{v})\bigr)^2
+\bigl(\Oph(|\xi|^\mez) (v+\bar{v})\bigr)^2\Bigr]\\
&\quad +\frac{i}{4} \Oph(|\xi|)\bigl( (v-\bar{v})\Oph(|\xi|^\mez)(v+\bar{v})\bigr)\\
&\quad-\frac{i}{4}\Oph(\xi)\bigl( (v-\bar{v})\Oph(\xi|\xi|^{-\mez})(v+\bar{v})\bigr),
\end{aligned}
\ee
$\CV$ stands for the cubic contribution 
\be\label{328}
\begin{aligned}
\CV&=\frac{1}{8}\Oph(|\xi|^{\mez})  \Bigl[ \Big(\Oph\big(\la\xi\ra^\mez\big)(v+\bar{v})\Big) \Oph(\la \xi\ra)\Big((v-\bar{v}) \Oph\big(\la\xi\ra^\mez\big)(v+\bar{v})\Big)
\Big]\\
&\quad -\frac{1}{8}\Oph\big(\la\xi\ra^\mez\big)\Big[ \Big( \Oph\big(\la\xi\ra^\mez\big)(v+\bar{v})\Big)\Big((v-\bar{v})\Oph\big(\la\xi\ra^\tdm\big)(v+\bar{v})\Big)\Big]\\
&\quad -\frac{1}{8}\Oph(\la\xi\ra)\Big[ (v-\bar{v})\Oph(\la\xi\ra)\Big((v-\bar{v})\Oph\big(\la\xi\ra^\mez\big)(v+\bar{v})\Big)\Big]\\
&\quad +\frac{1}{16}\Oph(|\xi|)  \Bigl[ (v-\bar{v})^2\Oph(|\xi|^{\tdm}) (v+\bar{v})\Bigr]\\
&\quad+\frac{1}{16}\Oph\big(|\xi|^{2}\big)\Bigl[ (v-\bar{v})^2\Oph(|\xi|^{\mez})  (v+\bar{v})\Bigr]
\end{aligned}
\ee
and where the remainder satisfies for $p=2$ or $\infty$ and a small positive number $\theta$, for any $d\in \xN$, $k\in\xN$ such that 
$\blA \japon^{\alpha+d}\Zr^k v(t,\cdot)\brA_{L^2}$ and $\blA \japon^{\alpha+d}\Zr^k v(t,\cdot)\brA_{L^\infty}$ are finite, 
the estimate
\begin{multline}\label{329}
\blA \japon^{d}\Zr^k \la hD_x\ra^{-\mez+\theta} \RV \brA_{L^p}\\
\le C_k[v] h^{\frac{1}{16}}\sum_{\substack{k_1+k_2+k_3\le k \\ k_1,k_2\le k_3}}
\prod_{j=1}^2\blA \japon^{\alpha+d}\Zr^{k_j}V\brA_{L^\infty}
\blA \japon^{\alpha+d}\Zr^{k_3}V\brA_{L^p}
\end{multline}
where $C_k[v]$ depends on
$$
h^{\frac{1}{16}} \blA \japon^{\alpha+d}\Zr^{(k-1)_+}v\brA_{L^\infty}
$$
and on a uniform bound for $\lA v\rA_{L^2}^{1-2\theta'}h^{\theta'}\lA v\rA_{L^\infty}^{2\theta'}$.

Actually \e{326}, \e{327}, \e{328} follow from \e{321}, \e{322}, \e{323}. The remainder, estimated by \e{324} and 
\e{325}, being at least quartic, would bring in factor a power $h^{3/2}$ in \e{326}. We retained only the power $h^{11/8}$, to keep the extra 
$h^{1/16}$ factor in the \rhs of \e{329}, and to keep also a $h^{1/16}$-factor in front of one of the $\blA \Zr^{k_j}v\brA_{\eC{d+\alpha}}$ in the 
\rhs of \e{324}, \e{325}. In that way, we obtain in \e{329} an estimate in terms of cubic expressions, modulo the indicated multiplicative constant. 
Notice also that the uniform bound assumed for $\lA v\rA_{L^2}^{1-2\theta'}h^{\theta'}\lA v\rA_{L^\infty}^{2\theta'}$ will be satisfied, when 
$\theta'>0$ will have been fixed. Actually, we shall obtain a uniform control of $\lA v\rA_{L^\infty}$, and a bound of $\lA v(t,\cdot)\rA_{L^2}$ in 
$O(t^\delta)$ for some $\delta>0$ as small as we want. Taking this $\delta$ smaller than $\theta'$ will provide the wanted
uniformity. Finally, notice also that the fact that order zero pseudo-differential operators are not $L^\infty$-bounded is
harmless in deriving \eqref{329} with $p=\infty$ from \eqref{325}, as we may always replace $\alpha$ by some larger value.

Our main task in the following subsections will be to deduce from equation \eqref{326} the oscillatory 
behavior of $v$ when $h$ goes to zero. We shall do that expressing $v$ from Lagrangian 
distributions as those defined in the preceding section. This structure will be uncovered writing from \eqref{3224} 
an equation for $v$ involving only $D_t$ derivatives. Actually, since $D_t=-ihZ-\Oph(x\xi)$, we may write 
\be\label{3211}
\Oph\bigl(2x\xi+|\xi|^\mez\bigr)v
=-\sqrt{h}\QV+h\Bigl[\frac{i}{2}v-iZV-\CV\Bigr]
-h^{\frac{11}{8}}\RV.
\ee
From nom on, we consider $v(t,\cdot)$ as a family of functions of $x$ indexed by $h=t^{-1}\in ]0,1]$. 
We do not write explicitly the parameter $h$ i.e.\ we write $v$ instead of $(v_h)_h$. 
Let us introduce the Lagrangian submanifold given by the zero set of the symbol in the left hand side of \eqref{3211} 
outside $\xi=0$ i.e.\ set
\be\label{3212}
\begin{aligned}
\Lambda&=\Bigl\{ (x,\xi)\in \cotangent \,;\, 
2x\xi+|\xi|^\mez =0\, ~\xi\neq 0\,\Bigr\}\\ \index{Lagrangian distributions!$\Lambda$, lagrangian}
&=\bigl\{ (x,\diff \omega(x))\,;\, x\in \xR^*\bigr\}
\end{aligned}
\ee
where\index{Lagrangian distributions!$\omega$, phase}
$$
\omega(x)=\frac{1}{4|x|}.
$$
In Section~\ref{S:35}, we shall need the exact expressions of 
$\QV$, $\CV$ given by \eqref{327}, \eqref{328}. 
Before that, we shall use only some less precise informations on the structure of these terms that we 
describe now. 

From now on, we denote by $\mZ$ the collection of vector fields $\mZ=(Z,h\px)$ and, 
if $v$ is a distribution on $\xR$, we define for any natural integer $k$ the vector 
valued function
$$
\mZ^k v=(Z^{k_1}(h\px)^{k_2}v)_{k_1+k_2\le k}.
$$

\begin{lemm}\label{ref:3.2.1}
Let $p$ be in $[1,+\infty]$. 

$i)$ Denote by $B_0$ the symmetric bilinear form associated to the quadratic form $Q_0$. 
Let $k$ be in $\xN^*$, and for every couple $(k_1,k_2)\in \xN\times \xN$ with $k_1+k_2=k$, 
take $p_{k_1}$, $p_{k_2}$ in $[1,+\infty]$ such that $\frac{1}{p_{k_1}}+\frac{1}{p_{k_2}}=\frac{1}{p}$. 
Then for any distributions $V_1=(v_1,\bar{v_1})$, $V_2=(v_2,\bar{v_2})$, 
any $j_0,j_1,j_2$ in $\xZ$,
\be\label{3213}
\begin{aligned}
\blA \Delta_{j_0}^h \mZ^k B_0\bigl( \Delta_{j_1}^h V_1, \Delta_{j_2}^h V_2\bigr) 
\brA_{L^p}
&\le C 2^{j_0+\mez \min (j_1,j_2)} \indicator{\max(j_1,j_2)\ge j_0-C}\\
& \quad \times \sum_{k_1+k_2\le k}\blA \mZ^{k_1}\Delta_{j_1}^h V_1\brA_{L^{p_{k_1}}}
\blA \mZ^{k_2}\Delta_{j_2}^h V_2\brA_{L^{p_{k_2}}}
\end{aligned}
\ee
for some positive constant $C$. In the same way
\be\label{3214}
\begin{aligned}
&\blA \Oph\bigl(\varphi_0\bigl(h^{-2(1-\sigma)} \xi\bigr)\bigr)
\mZ^k B_0\bigl( \Delta_{j_1}^h V_1, \Delta_{j_2}^h V_2\bigr) 
\brA_{L^p}\\
&\qquad \le C h^{2(1-\sigma)}2^{\mez \min (j_1,j_2)}\\
&\qquad \quad \times \sum_{k_1+k_2=k}\blA \mZ^{k_1}\Delta_{j_1}^h V_1\brA_{L^{p_{k_1}}}
\blA \mZ^{k_2}\Delta_{j_2}^h V_2\brA_{L^{p_{k_2}}}.
\end{aligned}
\ee

$ii)$ Let $T_0$ be the trilinear symmetric form associated to $C_0$. Then for any 
$k\in \xN$, for some constant $C$,
\be\label{3215}
\begin{aligned}
&\blA \Delta_{j_0}^h \mZ^k T_0\bigl( \Delta_{j_1}^h V_1, \Delta_{j_2}^h V_2,\Delta_{j_3}^h V_3\bigr) 
\brA_{L^p}\\
&\qquad \le C 2^{j_0/2+2\max (j_1,j_2,j_3)} \indicator{\max(j_1,j_2,j_3)\ge j_0-C}\\
&\qquad \quad \times \sum_{k_1+k_2+k_3\le k}\blA \mZ^{k_1}\Delta_{j_1}^h V_1\brA_{L^{p_{k_1}}}
\blA \mZ^{k_2}\Delta_{j_2}^h V_2\brA_{L^{p_{k_2}}}\blA \mZ^{k_3}\Delta_{j_3}^h V_3\brA_{L^{p_{k_3}}}
\end{aligned}
\ee
where $\frac{1}{p_{k_1}}+\frac{1}{p_{k_2}}+\frac{1}{p_{k_3}}=\frac{1}{p}$. 

In particular, for any $d$ in $\xR_+$, any $p$ in $[1,+\infty]$, any $\alpha>2$
\be\label{3216}
\begin{aligned}
&\blA \Delta_{j}^h \mZ^k T_0(V_1,V_2,V_3) \brA_{L^p}\\
&\qquad \le C 2^{j/2-j_{+}d} \sum_{\substack{k_1+k_2+k_3\le k \\ k_1,k_2\le k_3}}
\prod_{\ell=1}^2 \blA \mZ^{k_\ell}\langle h \OD\rangle^{\alpha+d} V_\ell\brA_{L^{\infty}}
\blA \mZ^{k_3}\langle h\OD\rangle^{\alpha+d} V_3\brA_{L^{p}}.
\end{aligned}
\ee

If, in the left hand side, one replaces $\Delta_j^h$ 
by $\Oph\bigl(\varphi_0\bigl(h^{-2(1-\sigma)} \xi\bigr)\bigr)$, the same 
estimates hold with the factor $2^{j/2-j_{+}d} $ in the right hand side replaced by $h^{1-\sigma}$.

$iii)$ The remainder $\RV$ satisfies for any $d\in \xR_+$, any $j$ in $\xZ$, 
with $2^j\ge c h^{2(1-\sigma)}$ estimates
\be\label{3217a}
\begin{aligned}
&\blA \Delta_{j}^h \mZ^k \RV \brA_{L^2}\\
&\qquad \le C 2^{j/2-j_{+}d} \sum_{\substack{k_1+k_2+k_3\le k \\ k_1,k_2\le k_3}}
\prod_{\ell=1}^2 \blA \mZ^{k_\ell}\langle h \OD\rangle^{\alpha+d} V\brA_{L^{\infty}}
\blA \mZ^{k_3}\langle h\OD\rangle^{\alpha+d} V\brA_{L^{2}}
\end{aligned}
\ee
and
\be\label{3217b}
\blA \Delta_{j}^h \mZ^k \RV \brA_{L^\infty}
\le C 2^{j/2-j_{+}d} \sum_{k_1+k_2+k_3\le k }
\prod_{\ell=1}^3 \blA \mZ^{k_\ell}\langle h \OD\rangle^{\alpha+d} V\brA_{L^{\infty}}
\ee
where $C$ depends only on 
$h^{1/16}\blA \langle h\OD\rangle^{\alpha+d} \Zr^{(k-1)_+}V\brA_{L^\infty}$ for some large enough $\alpha>0$. 

If, in the left hand side of \eqref{3217a}, \eqref{3217b}, $\Delta_j^h$ is replaced by 
by $\Oph\bigl(\varphi_0\bigl(h^{-2(1-\sigma)} \xi\bigr)\bigr)$, similar estimates hold 
with $2^{j/2-j_{+}d} $ replaced by $h^{1-\sigma}$.
\end{lemm}
\begin{proof}
$i)$ Consider the contribution to $B_0(V_1,V_2)$ of the first term in the right hand side of \eqref{327}. 
Its Fourier transform may be written as the symmetrization of a multiple of 
$$
h^2 \int \frac{|\xi h|^\mez}{(h |\xi-\eta|)^\mez (h|\eta|^\mez)} \bigl( (\xi-\eta)\eta
+|\xi-\eta|\la \eta\ra\bigr) \widehat{f_1}(\xi-\eta)\widehat{f_2}(\eta)\, d\eta.
$$
where $f_1=v_1+\bar{v_1}$, $f_2=v_2+\bar{v_2}$. On the support of the integrand, 
$(\xi-\eta)\eta \ge0$ so that $|\xi|=|\xi-\eta|+|\eta|$. Consequently, the contribution 
of this term to $ \Delta_{j_0}^h B_0\bigl( \Delta_{j_1}^h V_1, \Delta_{j_2}^h V_2\bigr)$ 
will be non zero only when $j_0\ge \max(j_1,j_2)-C$ for some $C>0$. 
In the same way, the contribution to $B_0(V_1,V_2)$ of the sum 
of the last two terms in \eqref{327} may be written, after Fourier transform an up to symmetries, 
as a multiple of 
$$
h^{\tdm} \int \frac{\la \xi\ra\la\eta\ra-\xi\eta}{\la\eta\ra^\mez}
\widehat{f_1}(\xi-\eta)\widehat{f_2}(\eta)\, d\eta.
$$
On the support of the integrand $\xi\eta\le 0$, whence $|\xi-\eta|\ge \max (|\xi|,|\eta|)$ so that 
the contribution to 
$$
\Delta_{j_0}^h B_0\bigl( \Delta_{j_1}^h V_1, \Delta_{j_2}^h V_2\bigr)
$$ 
will be non zero only if $j_2\le j_1+C$ for some $C>0$. Using these inequalities and taking into account 
the distribution of the derivatives on the different factors, we conclude
\be\label{3217c}
\blA \Delta_{j_0}^h B_0\bigl( \Delta_{j_1}^h V_1, \Delta_{j_2}^h V_2\bigr)\brA_{L^2}
\le C 2^{j_0+\mez \min(j_1,j_2)} \blA \Delta_{j_1}^h V_1\brA_{L^{p_{1}}}
\blA \Delta_{j_2}^h V_2\brA_{L^{p_{2}}}
\ee
if $\frac{1}{p_1}+\frac{1}{p_2}=\frac{1}{2}$. Moreover, by spectral localization, 
we have always $\max(j_1,j_2)\ge j_0-C$ for some $C>0$. If one makes act 
$\mZ^k$ on $B_0(V_1,V_2)$, the above properties of spectral localization 
are not affected since, if $a(\xi)$ is smooth outside zero, 
$[Z,\Oph(a)]=-2\Oph(\xi a'(\xi))$. Distributing the $Z$-derivatives on the different factors, 
one gets \eqref{3213}. The proof of \eqref{3214} is similar. 

$ii)$ We notice first that in all contributions in \eqref{328}, 
$\Oph(\la\xi\ra^\mez)$ is always in factor. This allows to make appear the $2^{j_0/2}$-factor in \eqref{3215}. 
Since the sum of the powers of $\la \xi\ra$ appearing in each term of \eqref{328} is equal to $5/2$, we 
get as well the factor $2^{2\max (j_1,j_2,j_3)}$ in \eqref{3215}. The cut-of for $\max(j_1,j_2,j_3)\ge j_0-C$ 
follows from the spectral localization of each factor. Finally, making act $\mZ^k$ on 
$T_0$ and commuting each vector field with $\Oph(|\xi|^\mez)$, 
$\Oph(\xi\la\xi\ra^{-\mez}), \ldots$ we obtain \eqref{3215}. 

To deduce \eqref{3216} from \eqref{3215}, we decompose in the left hand side of \eqref{3216}, 
$$
V_\ell=\Oph(\varphi_0(\xi))V_\ell+\sum_{j_\ell\ge 0}\Delta_{j_\ell}^hV_\ell.
$$
Because of the spectral localization, we get for $j_\ell\ge 0$,
\begin{align*}
&\blA Z^{k_\ell}\Delta_{j_\ell}^h V_\ell\brA_{L^p}\le C 2^{-j_\ell (\alpha+d)}
\sum_{k_\ell'\le k_\ell}\blA Z^{k_\ell'} \langle h\OD\rangle^{\alpha+d} V_\ell\brA_{L^p},\\
&\blA Z^{k_\ell}\Oph(\varphi_0) V_\ell\brA_{L^p}\le C 
\sum_{k_\ell'\le k_\ell}\blA Z^{k_\ell'} \langle h\OD\rangle^{\alpha+d} V_\ell\brA_{L^p}.
\end{align*}
We plug these estimates in \eqref{3215} with $p_{k_1}=p_{k_2}=\infty$, $p_{k_3}=p$ 
and in the similar inequality where some $\Delta_{j_\ell}^h V_\ell$ is replaced by 
$\Oph(\varphi_0) V_\ell$. We obtain a bound given by the product of 
the sum in the right hand side of \eqref{3216} multiplied by
$$
C 2^{j/2}\sum_{\substack{\max(j_1,j_2,j_3)\ge j-C \\ j_\ell\ge 0}} 2^{2\max (j_1,j_2,j_3)
-(j_1+j_2+j_3)(\alpha+d)}.
$$
Since $\alpha >2$, this is bounded by $C 2^{j/2-j_+ d}$ as wanted. 
The analogous statement, when $\Delta_j^h$ is replaced by 
$\Oph\bigl(\varphi_0\bigl( h^{-2(1-\sigma)}\xi\bigr)\bigr)$ in the left hand side 
of \eqref{3216} is obtained in the same way. 

$iii)$ Inequalities \eqref{3217a}, \e{3217b} follows from \eqref{329} with $p=2$ or 
$p=\infty$, using that the loss $2^{-j\theta}\le c h^{-2\theta(1-\sigma)}$ is absorbed 
by the extra $h^{1/16}$ factor in the \rhs of \e{329}, if $\theta$ has been taken small enough.
\end{proof}

Let us introduce the following decomposition of a solution $v$ of \eqref{326}. 
Fix $\sigma,\beta$ some small positive numbers, $\varphi_0$ in $C^\infty_0(\xR)$ the function 
equal to one close to zero introduced before Definition~\ref{ref:3.1.6}. We decompose the solution 
$v$ of \eqref{326} as 
\be\label{3218}
\begin{aligned}
&v=v_L+w+v_H,\\ \index{Unknowns in Chapter~\ref{chap:3}!$v_L$, $v_H$} 
\index{Unknowns in Chapter~\ref{chap:3}!$w$} 
&v_L=\Oph\bigl(\varphi_0\bigl(h^{-2(1-\sigma)}\xi\bigr)v,\\
&v_H=\Oph\bigl((1-\varphi_0)(h^{-\beta}\xi)\bigr)v.
\end{aligned}
\ee
We notice that if $C$ is a large enough constant, $w=v-v_L-v_H$ may be written 
$\sum_{j\in J(h,C)}\Delta_j^h w$.

\begin{prop}\label{ref:3.2.2}
Assume that for some $k\in\xN$, some $a>b+\tdm+\frac{1}{\beta}+\alpha$, some 
positive constants $\delta_k,\delta_k',A_k,A_k'$ a solution $v$ of \eqref{326} satisfies for any $h$ 
in an interval $]h',1]$, with $h'\in ]0,1]$ given, the a priori $L^2$-bounds
\be\label{3219}
\begin{aligned}
&\blA \oplow \mZ^k v\brA_{L^2}\le \eps A_k h^{-\delta_k},\\
&\blA \Djh \mZ^k v\brA_{L^2}\le \eps A_k h^{-\delta_k} 2^{-j_+ a} \quad 
\text{for }2^j\ge C^{-1}h^{2(1-\sigma)}
\end{aligned}
\ee
and the a priori $L^\infty$-bounds
\be\label{3220}
\begin{aligned}
&\blA \oplow \mZ^k v\brA_{L^\infty}\le \eps A_k' h^{-\delta_k'},\\
&\blA \Djh \mZ^k v\brA_{L^\infty}\le \eps A_k' h^{-\delta_k'} 2^{-j_+ b} \quad 
\text{for }2^j\ge C^{-1}h^{2(1-\sigma)}.
\end{aligned}
\ee
Then, if $\delta_k$, are small enough, one gets that
\be\label{3221}
h^{-\frac{3}{8}}\underline{v}=h^{-\frac{3}{8}} (v_L+v_H) \quad\text{belongs to an $\eps$-neighborhood 
of $0$ in }\mR_\infty^b,
\ee
with the notation introduced in Definition~\ref{ref:3.1.7}. Moreover, $w=v-\underline{v}$ satisfies
\be\label{3222}
\bigl( D_t-\opx\bigr)w=\sqrt{h}\QW+h\Bigl[-\frac{i}{2}w+\CW\Bigr]+h^{\frac{5}{4}}R_0(V)
\ee
where $\mZ^k R_0(V)$ belongs to $\mR_\infty^b$, and is an $\eps$-neighborhood of zero in that space. 
\end{prop}
Notice, for further reference, that as we did for \eqref{3211}, we deduce from \eqref{3222}
\be\label{3222a}
\opdeux w=-\sqrt{h}\QW
+h\Bigl[\frac{i}{2}w-iZw-\CW\Bigr]-h^{\frac{5}{4}}R_0(V).
\ee

The proposition will be proved using the following lemma.
\begin{lemm}\label{ref:3.2.3}
$i)$ Assume that estimates \eqref{3219}, \eqref{3220} hold. Then if 
$a>b+\tdm+\frac{1}{\beta}$, $a>b+\alpha+1+\frac{1}{2\beta}$, 
$b>\alpha>2$, 
\be\label{3223}
\begin{aligned}
&\blA \oplow \mZ^k \QV\brA_{L^\infty}\le c_k \eps^2 h^{\tq},\\
&\blA \Djh\ophigh  \mZ^k \QV\brA_{L^\infty}\le c_k \eps^2 h^{\tq} 2^{-j_+b} \quad 
\text{for any }j,
\end{aligned}
\ee
and
\be\label{3224}
\begin{aligned}
&\blA \oplow \mZ^k \CV\brA_{L^\infty}\le c_k \eps^3 h^{\uq},\\
&\blA \Djh\ophigh  \mZ^k \CV\brA_{L^\infty}\le c_k \eps^3 h^{\uq} 2^{-j_+b} \quad 
\text{for any }j,
\end{aligned}
\ee
if $\delta_k,\delta_k'$ in \eqref{3220} are small enough and $c_k$ is a convenient constant. 

$ii)$ Assume \eqref{3219} and the same inequalities on $a, b$ as above. 
Then if $\delta_k,\delta_k'$ are small enough 
\be\label{3225}
\begin{aligned}
&\blA \oplow \mZ^k v\brA_{L^\infty}\le c_k \eps h^{\frac{7}{16}-\sigma},\\
&\blA \Djh\ophigh  \mZ^k v\brA_{L^\infty}\le c_k \eps h^{\frac{7}{8}} 2^{-j_+b} \quad \text{for any }j,\\
&\blA 2^{j\ell}\Djh\oplow \mZ^k v\brA_{L^\infty}\le c_k \eps h^{\frac{3}{8}-\sigma+2\ell(1-\sigma)},\quad\ell\ge 0.
\end{aligned}
\ee
Moreover, if we assume \e{3219} and \e{3220},
\begin{multline}\label{3226}
\blA \oplow \mZ^k \bigl[ \QV-\QW\bigr]\brA_{L^\infty}\\
+\sup_{j\ge j_0(h,C)}2^{j_+b} \blA \Djh \Zr^k 
\bigl[ \QV-\QW\bigr]\brA_{L^\infty}
\le c_k \eps^2 h^{\tq}
\end{multline}
and 
\begin{multline}\label{3227}
\blA \oplow \mZ^k \bigl[ \CV-\CW\bigr]\brA_{L^\infty}\\
+\sup_{j\ge j_0(h,C)}2^{j_+b} \blA \Djh \Zr^k \bigl[ \CV
-\CW\bigr]\brA_{L^\infty}
\le c_k \eps^3 h^{\uq}.
\end{multline}
\end{lemm}
\begin{proof}
$i)$ To obtain the first formula in \eqref{3223} we use \eqref{3214} 
with $p_{k_1}=p_{k_2}=p=\infty$. 
Using assumption \eqref{3220} we get a bound of the left hand side by 
$$
C\eps^2 {A_k'}^2 h^{-2\delta_k'+2(1-\sigma)}\sum_{j_1,j_2\in \xZ} 
2^{\mez\min(j_1,j_2)-j_{1+}b-j_{2+}b}
$$
which gives the conclusion since $\sigma\in ]0,1/2[$ and we take 
$\delta_k'$ small enough. To get the second inequality \eqref{3223} we use 
\eqref{3213} with $p_{k_1}=p_{k_2}=p=\infty$, and we estimate the $L^\infty$ norms  
in the right hand side using Sobolev injection and \eqref{3219}. We obtain 
$$
C\eps^2 A_k^2 h^{-2\delta_k} 2^j\sum_{\max(j_1,j_2)\ge j-C} 
2^{\mez\min(j_1,j_2)-j_{1+}a-j_{2+}a+\frac{j_1}{2}+\frac{j_2}{2}}h^{-1}.
$$
If one uses that by assumption $2^j\ge ch^{-2\beta}$, and the fact that $a>b+\tdm+\frac{1}{\beta}$, 
one gets the wanted estimate (for $\delta_k$ small enough). 

To obtain \eqref{3224}, one substitutes inside \eqref{3216} 
with $p=\infty$, $d=0$,
$$
V_\ell=\oplow V_\ell+\sum_{j_\ell\in\xZ}\Delta_{j_\ell}^h \opnolow V_\ell
$$
one uses \e{3220} to estimate two of the three factors of the 
right hand side, \e{3219} and Sobolev injection 
to bound the third one, and one makes similar computations 
as above, exploiting that for 
the left hand side of the second estimate \eqref{3224} not 
to vanish, it is necessary that one of the $j_\ell$ be larger than $j-C$, and the assumptions 
on $a$. 

$ii)$ Inequalities \eqref{3225} follow from \eqref{3219} and Sobolev injections, using the assumptions 
on $a$ and the fact that $\delta_k<1/16$. 

We estimate the contribution to \eqref{3226} corresponding to $j\in J(h,C)$. 
We write $\QV-\QW$ from $B_0(V-W,V)$ and $B_0(V-W,W)$. By \eqref{3213}, $\blA \Djh \Zr^kB_0(V-W,W)\brA_{L^\infty}$ is smaller than 
$$
C 2^j \underset{\max(j_1,j_2)\ge j-C}{\sum\sum}2^{\mez \min(j_1,j_2)}\blA \mZ^k\Delta_{j_1}^h (V-W)\brA_{L^\infty}
\blA \mZ^k\Delta_{j_2}^h V\brA_{L^\infty}.
$$
By the definition of $w=v-v_L-v_H$, $\Delta_{j_1}^h(V-W)$ is non zero only 
if $2^{j_1}\les h^{2(1-\sigma)}$ or $2^{j_1}\gtrsim h^{-2\beta}$. In the first case, we bound 
$\blA \Zr^k \Delta_{j_1}^h (V-W)\brA_{L^\infty} 2^{j_1/4}$ using the third inequality \e{3225} with $\ell=1/4$. We get a bound in 
$O\big(\eps h^{\frac{7}{8}-\frac{3\sigma}{2}}\big)$. In the second case 
$\blA \Zr^k \Delta_{j_1}^h (V-W)\brA_{L^\infty}$ is 
$O\big(\eps h^{\frac{7}{8}}2^{-j_+b}\big)$ by the second estimate \e{3225}. Using assumption \e{3220} to estimate $\blA \Zr^k \Delta_{j_2}^h V\brA_{L^\infty}$ , we get a bound
$$
C\eps^2 h^{\frac{7}{8}-\delta_k'-\frac{3\sigma}{2}}2^j 
\underset{\max(j_1,j_2)\ge j-C}{\sum\sum}2^{\jd \min (j_1,j_2)-\uq j_{1-}}
2^{-(j_{1+}+j_{2+})b}
\le C\eps^2 2^{-j_+(b-1)}h^{\frac{7}{8}-\delta_k'-\tdm \sigma}.
$$
Since $2^j\le Ch^{-2\beta}$ with $\beta\ll 1$ and $\sigma\ll 1$, 
$\delta_k'\ll 1$, we obtain the wanted conclusion.

One studies in the same way the contributions of indices $j$ in $J(h,C)$ to 
\eqref{3227}, expressing $\CV-\CW$ from $T_0(V-W,V,V)$ and from similar expressions and using \eqref{3216}. 

To estimate the first term in the left hand side of \eqref{3226}, \eqref{3227}, or the contribution of 
$j\ge j_1(h,C)$ to the latter, we just need to apply \eqref{3223}, \eqref{3224}, and to notice that 
these inequalities remain true with $V$ replaced by~$W$. 
\end{proof}

\begin{proof}[Proof of Proposition~\ref{ref:3.2.2}]
Notice first that \eqref{3221} follows 
from \eqref{3225} if $\sigma\ll 1$. 
Denote $\Sigma_h=\Id-\oplow-\ophigh$ so that 
$w=\Sigma_h v$ by definition and $\underline{v}=(\Id-\Sigma_h)v$. 
We notice that 
$\bigl[ D_t-\opx ,\Sigma_h\bigr]=h\tilde{\Sigma}_h$, 
where $\tilde{\Sigma}_h$ may be written as a linear combination of quantities 
$\Oph\bigl(\tilde{\varphi_0}\bigl( h^{-2(1-\sigma)}\xi\bigr)\bigr)$, 
$\Oph\bigl(\tilde{\varphi_0}\bigl( h^{2\beta}\xi\bigr)\bigr)$ for new functions 
$\tilde{\varphi_0}$ in $C^\infty_0(\xR^*)$. We deduce from \eqref{326} 
\begin{align*}
\bigl(D_t-\opx \bigr)\underline{v}&=\sqrt{h}(\Id-\Sigma_h)\QV\\
&\quad +h(\Id-\Sigma_h)\Bigl(-\frac{i}{2}+\CV\Bigr)\\
&\quad +h^{\frac{11}{8}}(\Id-\Sigma_h)\RV\\
&\quad -h\tilde{\Sigma}_h v.
\end{align*}
By estimates \eqref{3223}, \eqref{3224} and \eqref{3225} the first, second and last terms of the 
right hand side may be written $h^{5/4}R(V)$ with $R(V)$ in $\mR_\infty^b$.

To estimate the remainder term $R_0^h(V)$, 
we estimate its $L^\infty$ norm using \e{3217b} with $d=b+\mez$. 
The factors 
in the right hand side of \e{3217b} are estimated in the following way:
\begin{align*}
\blA \langle h\OD\rangle^{\alpha+b+\mez}V\brA_{L^\infty}
&\le C \blA \oplow V \brA_{L^\infty}\\
&\quad +C \sum_{j\in J(h,C)}\blA 2^{j_{+}(\alpha+b+\mez)}\Djh V\brA_{L^\infty}\\
&\quad + C \sum_{j\ge j_1(h,C)}2^{j(\alpha+b+\mez)}\left(\frac{2^j}{h}\right)^\mez \blA \Djh V\brA_{L^2},
\end{align*}
where we used the Sobolev injection for the last term. Using assumptions \eqref{3220} and \eqref{3219} for the right hand side, 
together with the fact that $2^j\le C h^{-2\beta}$ on the the first sum, $2^j>ch^{-2\beta}$ on the last one, 
and $a>\alpha+b+1+\frac{1}{\beta}$, we bound this quantity by say 
$Ch^{-\frac{1}{24}}$ (if $\delta_k,\delta_k',\beta$ are small enough). It follows that
$$
\bigl( D_t-\opx\bigr)w=\sqrt{h}\QV+h\Bigl[-\frac{i}{2}v+\CV\Bigr]+h^{\frac{5}{4}}R(V)
$$
with $R(V)$ in $\mR_\infty^b$. Using \eqref{3225}, \eqref{3226}, \eqref{3227} we may 
replace the right-hand side of this equation by the right hand side of \eqref{3222} up to a modification of $R(V)$. 
If we make act the $\mZ$-family of vector fields on \eqref{3222}, and use the commutation relations
\be\label{3228}
\begin{aligned}
&\bigl[ t\partial_t +x\px ,D_t-\opx\bigr] =-\bigl(D_t-\opx\bigr),\\
&\bigl[ h\px , D_t -\opx\bigr]=0,
\end{aligned}
\ee
we obtain in the same way the estimates involving $\mZ^k$ derivatives.
\end{proof}

\section{\texorpdfstring{Weak $L^\infty$ estimates}{Weak estimates}}\label{S:33}

The goal of this subsection is to show that if $v$ is a solution 
to equation \eqref{326}, and if we are given an $L^2$-control 
of $Z^{k+1}v$ of type $\blA Z^{k+1}v\brA_{L^2}=O\bigl(h^{-\delta_{k+1}}\bigr)$ for some 
small 
$\delta_{k+1}>0$, we can deduce from it and the equation an $L^\infty$-bound of the form 
$\blA Z^{k}v\brA_{L^\infty}=O\bigl(h^{-\delta_{k}'}\bigr)$ for some small $\delta_k'\gg \delta_{k+1}$. 
Actually, we shall get as well bounds 
for $\japon^b v$ instead of $v$ for some given $b>0$. 

These bounds are not good enough,  
but they will be the starting point of the more elaborated reasoning that will be pursued 
in Sections \ref{S:34} and 
\ref{S:35}. Before stating the main result, we fix some notation. 

Assume given integers $s\gg N_1\gg N_0\gg 1$ and an increasing sequence of positive numbers $(\delta_k)_{0\le k\le s/2+N_1+1}$. We consider another 
increasing sequence $(\delta_k')_{0\le k\le \sd +N_1}$ 
satisfying the inequalities
\be\label{330}
\begin{aligned}
&\delta_k'>\sum_{j=0}^{\ell'}\delta_{k_j}'
+\sum_{j=\ell'+1}^\ell \delta_{k_\ell+1} 
\quad &&\text{if }\left\{
\begin{aligned}
&0\le \ell'\le \ell\le 4,~\Sigma_{j=0}^{\ell}k_j\le k, \\
& k_j< k \text{ when } 0\le j\le \ell',\end{aligned}\right.\\
&\delta_k'>\delta_{k+1}+2\delta_0+4\delta_0'\quad&&\text{if }k\ge 1,
\end{aligned}
\ee
for $k=0,\ldots, \frac{s}{2}+N_1$. Clearly such a sequence $(\delta_k')_k$ may always 
be constructed by induction, and if $\delta_s$ is small enough, we may assume moreover that
\be\label{333}
\delta_k<\frac{1}{32}\quad k=0,\ldots,\frac{s}{2}+N_1+1,~\delta_k'<\frac{\sigma}{8}<\frac{1}{32}~
k=0,\ldots,\frac{s}{2}+N_1.
\ee
We assume that the positive number $\beta$ introduced in \eqref{3112} is small enough so that 
$2\beta(\alpha+\mez)<\frac{1}{8}$, where $\alpha>2$ is the fixed large enough number introduced in \eqref{324} and \eqref{3216}, and that $\beta<\sigma/2$. 
We fix positive numbers $a>b>b'>b''$ such that
\be\label{334}
\begin{aligned}
& a>b+\tdm+\frac{1}{\beta}+\alpha,\quad b>\mez,\\
& (b-b')\beta>2,\quad (b'-b'')\beta>2.
\end{aligned}
\ee
In that way, the assumptions of Proposition~\ref{ref:3.2.2} will be fulfilled. 

For $k$ a nonnegative integer, we define 
\be\label{335}
\Er_k(v)=\sum_{k'=0}^k \max \Bigl( \blA \oplow \mZ^{k'}v\brA_{L^\infty},\sup_{j\ge j_0(h,C)}
2^{j_+b}\blA \Djh \mZ^{k'}v\brA_{L^\infty}\Bigr)\index{Norms!$\Er_k(v)$}
\ee
and
\be\label{336}
\Fr_k(v)=\sum_{k'=0}^k \max \Bigl( \blA \oplow \mZ^{k'}v\brA_{L^2},\sup_{j\ge j_0(h,C)}
2^{j_+a}\blA \Djh \mZ^{k'}v\brA_{L^2}\Bigr).\index{Norms!$\Fr_k(v)$}
\ee

Let $k\in\xN^*$. We denote by $\Tr_k^\infty$ \index{Function spaces!$\Tr_k^\infty$}the set of functions $v\mapsto P_k(v)$ satisfying 
for any $v=(v_h)_h$ with $\E{k-1}(v)\le h^{-1/4}$ a bound of type
\be\label{336a}
\ba
\la P_k(v)\ra\le C\Big[ &\E{k}(v)+\sum_{k_1+k_2\le k}\E{k_1}(v)\E{k_2}(v)\\
&+\sum_{k_1+k_2+k_3\le k}\E{k_1}(v)\E{k_2}(v)\E{k_3}(v)\Big].
\ea
\ee
In the same way, we define $\Tr_k$ \index{Function spaces!$\Tr_k$}as the set of functions 
$v\mapsto P_k(v)$ admitting 
for any $v=(v_h)_h$ with $\E{k-1}(v)\le h^{-1/4}$ a bound of type
\be\label{336b}
\ba
\la P_k(v)\ra\le C\sum_{1\le \ell\le 4}\sum_{k_1+\cdots+k_\ell\le k}
\sum_{\ell'=0}^\ell \prod_{j=1}^{\ell'} \E{k_{j}}(v)\prod_{j=\ell'+1}^\ell \F{k_{j+1}}(v).
\ea
\ee
The main result of this section is the following one. 
\begin{prop}\label{ref:3.3.1}
Let $k$ be a positive integer. Assume that we are given constants $\tilde{A}_0$, $A_0$, $A_1$,\ldots, $A_{k+1}$ 
and a solution $v$ of \eqref{326} such that for $h$ in some interval $]h',1]$ 
\be\label{337}
\E{0}(v)\le \tilde{A}_0 h^{-\delta_0'},\quad 
\F{k'}(v)\le \eps A_{k'}h^{-\delta_{k'}},\quad 0\le k'\le k+1.
\ee
Then, there are $h_0>0$, $A'_{k'}>0$, $k'=1,\ldots,k$, depending only on 
$\eps\tilde{A}_0,A_1,\ldots, A_{k+1}$ such that for any $h$ in $]h',1]$
\be\label{339}
\E{k'}(v)\le \eps A_{k'}' h^{-\delta'_{k'}},\quad k'=1,\ldots,k.
\ee
\end{prop}
\begin{rema}
$\bullet$ The result of the preceding proposition may be thought of as a ``Klainerman-Sobolev'' 
estimate, that allows one to get $L^\infty$-decay from $L^2$-bounds (There is no decay involved 
in \eqref{339} since the negative power of time $t^{-1/2}=\sqrt{h}$ has been factored out when we defined $v$ from $u$).
\end{rema}

The proof of the proposition will be made in three steps. 

First, we treat the case of small or large frequencies, for which we deduce \eqref{339} 
from the $L^2$-estimate in \eqref{337} and Sobolev injection. 

Next, we are reduced to intermediate 
frequencies i.e.\ to $\Djh v$ with $j$ belonging to $J(h,C)$. We write the equation 
for $\Djh v$ coming from \eqref{3211}. The operator 
of symbol $2x\xi+\la \xi\ra^\mez$ is elliptic 
outside the Lagrangian $\Lambda$ 
defined in \eqref{3212}. Since the right 
hand side of \eqref{3211} is $O(h^{1/2-0})$, one 
will get for the $L^\infty$-norm of $\Djh \mZ^k v$ 
cut-off outside a neighborhood of $\Lambda$ 
some $O(h^{1/2-0})$ estimates, that are better than 
what we want. 

In the last step, we decompose in the quadratic part $\QV$ of the right hand side of \eqref{3211}, 
$v$ as the sum of the contribution microlocalized outside $\Lambda$, which by the preceding 
step will give an $O(h^{1-0})$ contribution to \eqref{3211}, and a contribution 
microlocalized close to $\Lambda$. The quadratic interactions between the latter will be 
microlocally supported close to $2\cdot \Lambda$, $0\cdot \Lambda$, $-2\cdot \Lambda$ 
where
$$
\lambda\cdot \Lambda =\bigl\{ ( x,\lambda\xi)\,;\, (x,\xi)\in \Lambda\bigr\}.
$$
Consequently, if we microlocalize \eqref{3211}Ê
close to $\Lambda$, which does not meet 
$\pm 2\cdot \Lambda$, $0\cdot \Lambda$, the $\sqrt{h}$-terms of the right hand side disappear, and we 
get an $O(h^{1-0})$ estimate for the $L^2$-norm of the left hand side. 
This allows to deduce the wanted $L^\infty$-estimate from a Sobolev embedding, after reduction of $\Lambda$ 
to the zero section, through a canonical transformation. 

\underline{First step:} Low and large frequencies

We decompose $v=v_L+w+v_H$ according to \eqref{3218}. By assumption 
\eqref{337}, estimates \eqref{3219} hold. Then $ii)$ of Lemma~\ref{ref:3.2.3} implies that 
$v_L,v_H$ satisfy the first two inequalities~\eqref{3225}. 

\underline{Second step:} Elliptic estimates for $w$ outside a neighborhood of $\Lambda$. 

We define, for $j\in J(h,C)$ with $C$ large enough
\be\label{3311}
\ba
w_j &=\Tmj \Djh w,\\ \index{Unknowns in Chapter~\ref{chap:3}!$w_j$}
\zj & =\Tmj \mZ \Tj = (Z,2^{j/2}h\px),\quad \zj^k=(Z^{k_1}(2^{j/2}h\px)^{k_2})_{k_1+k_2\le k},
\ea
\ee
so that $w=\sum_{j\in j(h,C)}\Tj w_j$.

Let $\Phi\in C^\infty_0(\Rs)$ be equal to one on a domain $C^{-1}\le |\xi|\le C$ 
for a large constant $C$ and let $\Gamma$ be in $C^\infty_0(\xR)$, with small enough support, equal 
to one close to zero. We define
\be\label{3312}
\ba
\gL(x,\xi)&=\Phi(\xi)\Gamma \bigl(2x\xi+\la \xi\ra^\mez\bigr),\\
\nongL(x,\xi)&=\Phi(\xi)\bigl(1-\Gamma \bigl(2x\xi+\la \xi\ra^\mez\bigr)\bigr).
\ea
\ee
We obtain two symbols of $S(1,F)$ where $F=\{ \xi\,;\, C'^{-1}\le \la\xi\ra\le C'\}$ 
for a large enough $C'$. Moreover, since $2x\xi+\la \xi\ra^{1/2}=0$ is an equation of $\Lambda$, 
we see that on the domain where $\Phi\equiv 1$, $\gL$ (resp.\ $\nongL$) 
cuts-off close to $\Lambda$ (resp.\ outside a neighborhood of $\Lambda$). We shall prove 
the following estimates. 

\begin{prop}\label{ref:3.3.2}
Let $k\ge 1$, $N\in\xN$. We denote by $\kappa$ some fixed small enough positive number (say $\kappa=1/24$). There is a constant $C_k>0$, an element $P_k$ 
of $\Tr_k$ such that for any $v$ satisfying $\E{k-1}(v)\le h^{-1/16}$, one has for any 
$j$ in $j(h,C)$,
\be\label{3313}
\ba
&\blA \Ophj(\nongL)\zj^k w_j\brA_{L^3}\\
&\qquad \le C_k\bigg[ \sqrt{h} 2^{2j/3-j_+(b+\frac{1}{3}(a-b)}
\sum_{k_1+k_2\le k}\prod_{\ell=1}^2 \F{k_\ell}(v)^{\frac{1}{3}}\E{k_\ell}(v)^{\frac{2}{3}}\\
&\qquad\quad +h 2^{j/6-j_{+}b}h^{-4\beta-0}
\sum_{k_1+k_2+k_3\le k}[\F{k_1}(v)\F{k_2}(v)\E{k_3}(v) + \E{k_1}(v)\E{k_2}(v)\E{k_3}(v)]\\
&\qquad\quad+2^{j/4-j_+ a} h_j^{5/6}\F{k+1}(v)+2^{j/6-j_+b}h^{1+\kappa}P_k(v)\\
&\qquad\quad+2^{j/6-j_+(b+\frac{2}{3}(a-b))}h_j^N\F{k}(v)^{2/3}\E{k}(v)^{1/3}\bigg]
\ea
\ee
where $h^{-4\beta-0}$ means a bound in $C_\theta h^{-4\beta-\theta}$ for any $\theta>0$.
\end{prop}

To prove the proposition, we need to estimate the action of $\ophjdeux$ 
on $\zj^k w$ in various spaces. 

\begin{lemm}\label{ref:333}
$i)$ Let $k\ge 1$. There are an element $P_k$ of $\Tr_k$, 
a matrix $A(h_j)$ with uniformly bounded coefficients, a constant $C_k>0$ 
such that, for any $v$ satisfying $\E{k-1}(v)\le h^{-1/16}$ when $h$ stays in some interval $]h',1]$, 
one gets for any $h$ in that interval, any $j$ in $J(h,C)$,
\be\label{3316}
\ba
&\blA \ophjdeux \zj^k w_j -2^{-j/2} \sqrt{h}A(h_j)\zj^k\Tmj \Djh 
\QW\brA_{L^2}\\
&\qquad \le C_k 2^{j/4-j_{+}(b+\mez)}\bigg[ \sum _{k_1+k_2+k_3\le k}
\prod _{\ell=1}^3 \F{k_\ell}(v)^{1/3}
\E{k_\ell}(v)^{2/3}h^{1-0}\\
&\qquad\quad +h^{9/8} P_k(v)+h^{1-0}2^{-j/2} \F{k+1}(v)\bigg].
\ea
\ee

$ii)$ Under the preceding assumptions, we get as well
\be\label{3317}
\ba
&\blA \ophjdeux\zj^k w_j\brA_{L^3}\\
&\qquad \le C_k\bigg[ \sqrt{h} 2^{2j/3-j_+(b+\frac{1}{3}(a-b))}
\sum_{k_1+k_2\le k}\prod_{\ell=1}^2 \F{k_\ell}(v)^{\frac{1}{3}}\E{k_\ell}(v)^{\frac{2}{3}}\\
&\qquad\qquad\quad +h 2^{j/6-j_{+}b}h^{-4\beta-0}
\sum_{k_1+k_2+k_3\le k}[\F{k_1}(v)\F{k_2}(v)\E{k_3}(v)+ \E{k_1}(v)\E{k_2}(v)\E{k_3}(v)]\\
&\qquad\qquad\quad+2^{j/4-j_+ a} h_j^{5/6}\F{k+1}(v)+2^{j/6-j_+b}h^{1+\kappa}P_k(v)\bigg].
\ea
\ee

$iii)$ Under the preceding assumptions
\be\label{3318}
\ba
&\blA \ophjdeux \zj^k w_j -2^{-j/2} \sqrt{h}A(h_j)\zj^k\Tmj \Djh 
\QW\brA_{L^\infty}\\
&\qquad \le C_k h^{1-0}2^{-j_+(b-2)}\sum_{k_1+k_2+k_3\le k} 
\E{k_1}(v)\E{k_2}(v)\E{k_3}(v)\\
&\qquad \quad + C_k h_j2^{-j_+ b}\E{k+1}(v)
\ea
\ee
and
\be\label{3318a}
\blA 2^{-j/2} \sqrt{h}A(h_j)\zj^k\Tmj \Djh \QW\brA_{L^\infty}
\le \sqrt{h}2^{-j_+(b-\mez)}\sum_{k_1+k_2\le k}\E{k_1}(v)\E{k_2}(v).
\ee
\end{lemm}
\begin{proof}
We apply $\Djh$ to equation \eqref{3222a}. Denoting 
$\widetilde{\Delta}_j^h=\widetilde{\varphi}(2^{-j} h\OD)$ 
for a new smooth function $\widetilde{\varphi}$ satisfying 
$\supp \widetilde{\varphi}\subset \supp \va$, we get
\begin{align*}
\opdeux\Djh w&=-\sqrt{h}\Djh \QW\\
&\quad +h\Bigl[ \frac{i}{2}\Djh w-i \widetilde{\Delta}_j^h w-iZ \Djh w-\Djh \CW \Bigr]\\
&\quad -h^{5/4} \Djh R(V).
\end{align*}
Applying $\Tmj$ and using \eqref{3113}, we get
\be\label{3319a}
\ba
\ophjdeux w_j&=-\sqrt{h}2^{-j/2}\Tmj\Djh \QW\\
&\quad +h2^{-j/2}\Bigl[ \frac{i}{2}\widetilde{w}_j-i Zw_j-\Tmj\Djh \CW \Bigr]\\
&\quad -2^{-j/2} h^{5/4}\Tmj  \Djh R(V)
\ea
\ee
where $\widetilde{w}_j=w_j-2\Tmj \widetilde{\Delta}_j^h w$ satisfies the same estimates as $w_j$. 
We commute $\zj^k$ to the equation, using that
\begin{align*}
&\big[ Z,\ophjdeux\bigr]=-\ophjdeux,\\
&\bigl[ 2^{j/2}h\px, \ophjdeux\bigr]=-2ih_j \bigl( 2^{j/2} h\px\bigr).
\end{align*}
We get
\be\label{3320}
\ba
\ophjdeux \bigl( \zj^k w_j\bigr)
&=2^{-j/2}\sqrt{h} A(h_j)\zj^k\Tmj \Djh \QW\\
&\quad +h_j \Bigl[ \widetilde{B}(h_j)\zj^k \widetilde{w}_j
+B(h_j)\zj^{k+1}w_j+C(h_j)\zj^k\Tmj\Djh \CW\Bigr]\\
&\quad -h^{1/4} h_j D(h_j)\zj^k\Tmj \Djh \RV
\ea
\ee
where $A(h_j),B(h_j),\widetilde{B}(h_j),C(h_j),D(h_j)$ are matrices with uniformly bounded coefficients.

Let us control the cubic terms in \eqref{3320}. We write $\CW$ as 
$T_0(W,W,W)$ as in $ii)$ of Lemma~\ref{ref:3.2.1}. We express 
$\zj^k\Tmj \Djh \CW=\Tmj\mZ^k\Djh \CW$ from $\Tmj \Djh \mZ^{k'}T_0(W,W,W)$ 
for $k'\le k$ (changing eventually the definition of the spectral 
cut-off $\Djh$) and decompose each argument $W$ as $\sum_{j_\ell} \Delta_{j_\ell}^hW$. 
Applying estimate \eqref{3215} with $p_{k_1}=p_{k_2}=p_{k_3}=6$ and 
writing $\lA \cdot\rA_{L^6}\le \lA \cdot\rA_{L^2}^{1/3}\lA \cdot \rA_{L^\infty}^{2/3}$, we obtain 
\be\label{3321}
\ba
&\blA \Djh \mZ^{k'}T_0(W,W,W)\brA_{L^2}\\
&\qquad \le C 2^{j/2} \sum_{j_1,j_2,j_3}\sum_{k_1+k_2+k_3\le k'} 
2^{2\max (j_1,j_2,j_3)}\indicator{\max(j_1,j_2,j_3)\ge j-C}\\
&\qquad \qquad\qquad \times \prod_{\ell=1}^3 \blA \mZ^{k_\ell}\Delta_{j_\ell}^h W\brA_{L^2}^{1/3}
\blA \mZ^{k_\ell}\Delta_{j_\ell}^h W\brA_{L^\infty}^{2/3}.
\ea
\ee
Using \eqref{335}, \eqref{336}, we bound the last factor by
$$
2^{-(j_{1+}+j_{2+}+j_{3+})[b+\frac{1}{3}(a-b)]}\prod_{\ell=1}^3 \F{k_\ell}(W)^{1/3}\E{k_\ell}(W)^{2/3}.
$$
Summing in $j_1,j_2,j_3$ in $J(h,C)$, we obtain
$$
\blA \Djh \mZ^{k'}\CW \brA_{L^2}\le C|\log h|^{2} 2^{j/2-j_+(b-2+\frac{1}{3}(a-b))}
\times \sum_{k_1+k_2+k_3\le k'} \prod_\ell \F{k_\ell}^{1/3}\E{k_\ell}^{2/3}.
$$
Remembering that $\blA \Tmj\brA_{\Fl{L^2}{L^2}}=O(2^{j/4})$ we conclude that the 
$L^2$-norm of the cubic term in the right hand side of \eqref{3320} is bounded from above 
by the first term in right hand side of \eqref{3316} 
(since $a-b$ is large enough for $\beta\ll 1$, according to \e{334}).

To estimate the $R_0$-term in \eqref{3320}, we use \eqref{3217a}. We notice first that the right 
hand side of this inequality may be controlled from $\E{k_\ell}(v)$, $\F{k_\ell}(v)$: 
actually
\begin{align*}
\blA \mZ^{k_\ell}\japon^{\alpha+d}V\brA_{L^\infty}
\le C\sum_{k_\ell'\le k_\ell} \bigg[ &\blA \oplow \mZ^{k_\ell'} v\brA_{L^\infty}\\
&+\sum_{j\ge j_0(h,C)}2^{j_+(\alpha+d)}\blA \Djh \mZ^{k_\ell'}v\brA_{L^\infty}\bigg],
\end{align*}
where the constant $C$ depends only on 
$h^{\frac{1}{16}}\blA \langle h\OD\rangle^{\alpha+d}\Zr^{(k-1)_+ v\brA_{L^\infty}}$.

We shall take $d=b-\alpha-0$, so that the bounds \eqref{335} imply that the $j$-series converges and gives a bound 
\be\label{3321a}
\blA \mZ^{k_\ell}\japon^{\alpha+d}V\brA_{L^\infty}\le C \E{k_\ell}(v)|\log h|.
\ee
Similarly, we get
\be\label{3321b}
\blA \mZ^{k_\ell}\japon^{\alpha+d}V\brA_{L^2}\le C \F{k_\ell}(v)|\log h|.
\ee
Plugging this in \eqref{3217a}, we obtain, using again that $\blA \Tmj\brA_{\Fl{L^2}{L^2}}=O(2^{j/4})$,
$$
h^{1/4}h_j\blA \zj^k\Tmj \Djh \RV\brA_{L^2}
\le C h^{3/4-0}h_j^{1/2} 2^{j/2-j_+(b-\alpha-0)}
\sum_{k_1+k_2+k_3\le k}\E{k_1}(v)\E{k_2}(v)\F{k_3}(v).
$$
We notice that since $2^j=O(h^{-2\beta})$, we may bound $2^{-j_+(b-\alpha-0)}$ by 
$2^{-j_+(b+\mez)}h^{-2(\alpha+\mez+0)\beta}$. For $\beta$ small enough, this negative power of $h$ will be compensated consuming a $O(h^{1/8})$-factor, so that we end up with a bound
of the remainder in \eqref{3320} by 
$$
C h^{9/8} 2^{j/4-j_+(b+\mez)}\sum_{k_1+k_2+k_3\le k} \E{k_1}(v)\E{k_2}(v)\F{k_3}(v)
$$
so by the $P_k$-term in the right hand side of \eqref{3316}. 

Finally, the linear terms in the right hand side of \eqref{3320} are bounded by the last contribution in \eqref{3316}, remembering 
that $w_j$ may be expressed from $\Djh w$ by \eqref{3311} 
and that $\blA \Tmj\brA_{\Fl{L^2}{L^2}}=O(2^{j/4})$. This concludes the proof of $i)$ of the lemma.

$ii)$ To prove \eqref{3317}, let us bound the $L^3$-norm of the right hand side of \eqref{3320}. We express first $\zj^k\Tmj\Djh \QW$ 
from $\Tmj \Djh \mZ^{k'}\QW$, $k'\le k$ write $\QW=B_0(W,W)$, decompose $W=\sum_{j_\ell}\Delta_{j_\ell}^h W$ in each factor 
and apply \eqref{3213} with $p=3$, $p_{k_1}=p_{k_2}=6$. We get
\begin{align*}
\blA \Djh \mZ^{k'}\QW\brA_{L^3}
&\le C 2^{j} \sum_{j_1,j_2\in J(h,C)}\sum_{k_1+k_2\le k'} 
2^{\mez\min(j_1,j_2)}\indicator{\max(j_1,j_2)\ge j-C}\\
&\qquad \qquad\times \prod_{\ell=1}^2 \blA \mZ^{k_\ell}\Delta_{j_\ell}^h W\brA_{L^2}^{1/3}
\blA \mZ^{k_\ell}\Delta_{j_\ell}^h W\brA_{L^\infty}^{2/3}.
\end{align*}
Using \eqref{335}, \eqref{336}, we see that this quantity is smaller than
$$
2^j \sum_{k_1+k_2\le k}\bigg( \prod_{\ell=1}^2 \F{k_\ell}(v)^{\frac{1}{3}}\E{k_\ell}(v)^{\frac{2}{3}}\bigg)
2^{-j_+(b+\frac{1}{3}(a-b))}.
$$
We thus get the first term in the right hand side of \eqref{3317}, using that $\blA \Tmj\brA_{\Fl{L^3}{L^3}}=O(2^{j/6})$.

Let us study next the $L^3$-norm of the cubic term in \eqref{3320}. We proceed as in the proof of \eqref{3321}, applying \eqref{3215} with 
$p=3$, $p_{k_1}=p_{k_2}=6$, $p_{k_3}=\infty$. We obtain
\begin{align*}
\blA \Djh \mZ^{k'}\CW\brA_{L^3}
&\le C 2^{j/2} \sum_{j_1,j_2,j_3}\sum_{k_1+k_2+k_3\le k'} 
2^{2\max (j_1,j_2,j_3)}\indicator{\max(j_1,j_2,j_3)\ge j-C}\\
& \qquad\qquad \times 
\bigg(\prod_{\ell=1}^2 \blA \mZ^{k_\ell}\Delta_{j_\ell}^h W\brA_{L^2}^{1/3}
\blA \mZ^{k_\ell}\Delta_{j_\ell}^h W\brA_{L^\infty}^{2/3}\bigg)\blA \mZ^{k_3}\Delta_{j_\ell}^h W\brA_{L^\infty}.
\end{align*}
We bound the general term of this sum by
$$
C2^{2\max (j_1,j_2,j_3)}\indicator{\max(j_1,j_2,j_3)\ge j-C} 2^{-(j_{1+}+j_{2+}+j_{3+})b}\\
\times \Bigl(\F{k_1}(v)\F{k_2}(v)\Bigr)^{1/3}\Bigl(\E{k_1}(v)\E{k_2}(v)\Bigr)^{2/3}\E{k_3}(v).
$$
As $2^j\le Ch^{-2\beta}$, we conclude, using the convexity inequality $a^{1/3}b^{2/3}\leq (a+2b)/3$
$$
\blA \Djh \mZ^{k'}\CW\brA_{L^3} \le Ch^{-4\beta-0}2^{j/2-j_+b}
\sum_{k_1+k_2+k_3\le k} [\F{k_1}(v)\F{k_2}(v) \E{k_3}(v) + \E{k_1}(v)\E{k_2}(v) \E{k_3}(v)].
$$
This gives in \eqref{3317} a contribution to the second term in the right hand side, 
using again that $\blA \Tmj\brA_{\Fl{L^3}{L^3}}$ is $O(2^{j/6})$.

Consider next the remainder. We estimate $\blA \zj^k\Tmj \Djh\RV\brA_{L^3}$ 
from
$$
h_j^{-1/6} \blA \zj^k\Tmj \Djh\RV\brA_{L^2}
$$
using that the expression 
to be bounded is spectrally supported in a ball of radius 
$O(h_j^{-1})$. We apply next estimate \eqref{3217a} together with 
\eqref{3321a}, \eqref{3321b}. We obtain, using that 
$\blA \Tmj\brA_{\Fl{L^3}{L^3}}=O(2^{j/6})$, that the $L^3$-norm 
of the last term in \eqref{3320} is bounded from above by 
$$
Ch^{1/4-0}h_j^{5/6}2^{2j/3-j_+d}\sum_{k_1+k_2+k_3\le k} \E{k_1}(v)\E{k_2}(v) \F{k_3}(v)
$$
with $d=b-\alpha-0$. We get finally as a coefficient 
$2^{j/6-j_+b}h^{13/12-0-2\beta(\alpha+1/12)}$ if we use again that 
$2^j=O(h^{-2\beta})$. If $\beta$ is small enough, we see that the remainder 
in \eqref{3320} contributes to the last term in the right hand side of \eqref{3317}. 

Finally, the contribution of the linear terms in \eqref{3320} is bounded from above by
\begin{align*}
h_j\blA \zj^k \widetilde{w}_j\brA_{L^3}+h_j\blA \zj^{k+1} w_j\brA_{L^3}
&\le Ch_j^{5/6}\Big( \blA \zj^k \widetilde{w}_j\brA_{L^2}+\blA \zj^{k+1} w_j\brA_{L^2}\Big)\\
&\le C2^{j/4-j_+a}h_j^{\frac{5}{6}} \F{k+1}(v)
\end{align*}
where we have used Sobolev injection and the fact that $\zj^k\widetilde{w}_j$, $\zj^{k+1} w_j$ is spectrally supported 
for $h_j\la\xi\ra\sim 1$, and where the gain $2^{j/4}$ comes from 
$\blA \Tmj\brA_{\Fl{L^2}{L^2}}$ when expressing $w_j$ from 
$\Djh w$ in \eqref{3311}. 

$iii)$ Let us prove \eqref{3318} and \eqref{3318a}. 
Applying \eqref{3213} with $p=p_{k_1}=p_{k_2}=\infty$, 
we get for $k'\le k$,
\begin{align*}
\blA \Djh \mZ^{k'}\QW\brA_{L^\infty}
&\le C 2^{j} \sum_{j_1,j_2\in J(h,C)}\sum_{k_1+k_2\le k'} 
2^{\mez\min(j_1,j_2)}\indicator{\max(j_1,j_2)\ge j-C}\\
& \qquad\qquad \times  \blA \mZ^{k_1}\Delta_{j_1}^h W\brA_{L^\infty}
\blA \mZ^{k_2}\Delta_{j_2}^h W\brA_{L^\infty}
\end{align*}
which gives \eqref{3318a}. To get \eqref{3318}, we use again \eqref{3320}. The cubic term 
in the right hand side of this expression is bounded using \eqref{3215} with 
$p=p_{k_1}=p_{k_2}=p_{k_3}=\infty$ and gives the first term in the right hand side of \eqref{3318}. To estimate 
the $L^\infty$-norm of the $R_0$-term in \eqref{3320} we use \eqref{3217b} with 
$d=b-\alpha-0$ and \e{3321a}. The loss $2^{j_+(\alpha+0)}
\le C h^{-2\beta(\alpha+0)}$ may be absorbed by the extra $h^{1/4}$ factor 
in front of the remainder in \eqref{3320}, so that we get again a contribution bounded by 
the first term in the right hand side of \eqref{3318}. Finally, the linear term in \eqref{3320} is controlled by 
$h_j 2^{-j_+b}\E{k+1}(v)$. This concludes the proof.
\end{proof}

\begin{proof}[Proof of Proposition~\ref{ref:3.3.2}] 
We apply corollary \ref{ref:A.2.3} with the weight $m(x,\xi)=\langle x\rangle$. By the definition \eqref{3312} 
of $\nongL$, $2x\xi+\la \xi\ra^{\mezl} \ge c \langle x\rangle$ on the support of $\nongL$. Consequently, 
for any $N$ in $\xN$, we may find symbols $q$ in $S(\langle x\rangle^{-1}),r$ in $S(1)$ such that 
$\nongL=q\# (2x\xi+\la \xi\ra^{\mezl})+h_j^N r$. It follows that for any $p\ge 1$, 
\be\label{3320a}
\blA \Ophj(\nongL)\zj^k w_j\brA_{L^p}\le C \blA \ophjdeux \zj^k w_j\brA_{L^p}
+h_j ^N\blA \zj^k w_j\brA_{L^p}.
\ee
Applying this with $p=3$, we may bound by \eqref{3317} the first term in the right hand side in terms of the \rhs of \eqref{3313}. 
The last contribution is smaller than 
$$
h_j^N\blA \zj^kw_j\brA_{L^3}\le h_j^N\blA \zj^kw_j\brA_{L^2}^{2/3}\blA \zj^kw_j\brA_{L^\infty}^{1/3}
$$
going back to the estimates of $w_j=\Tmj \Djh w$ from 
$\E{k}(v)$, $\F{k}(v)$, we obtain the last term in the \rhs of \eqref{3313}. This concludes the proof of the proposition.
\end{proof}

The $L^3$-estimate we obtained in Proposition~\ref{ref:3.3.2} outside a microlocal 
neighborhood of $\Lambda$ will be useful as 
auxiliary bounds in the third step of our proof of Proposition~\ref{ref:3.3.1}. We also need $L^\infty$-estimates for $w$ cut-off 
outside $\Lambda$. They are given by the following
\begin{prop}\label{ref:3.3.4}
Let $k\ge 1$. There is an element $P_k$ in $\Tr_k^\infty$ such that for any $v$ satisfying $\E{k-1}(v)\le h^{-1/16}$
\be\label{3324}
\blA \zj^k \Ophj(\nongL)w_j\brA_{L^\infty}
\le C\Big[ h^{1/4} P_k(v)+h^{1/2}\F{k+1}(v)\Big]2^{-j_+b}.
\ee
\end{prop}
\begin{proof}
We notice first the commutation relations
\be\label{3324a}
\ba
&\bigl[ tD_t+x\OD , \Ophj(\nongL)\bigr] =\Ophj\bigl(\gamma_{\Lambda,1}^c\bigr),\\
&\bigl[h_j\OD , \Ophj(\nongL)\bigr] =h_j\Ophj\bigl(\gamma_{\Lambda,2}^c\bigr),
\ea
\ee
where $\gamma_{\Lambda,j}^c$ is in $S(1)$ with support contained in $\supp(\nongL)$. This 
shows that, up to a modification of the definition of $\nongL$, it is enough to control 
$\blA \Ophj(\nongL)\zj^kw_j\brA_{L^\infty}$. Let us show 
\be\label{3325}
\blA \ophjdeux \zj^k w_j\brA_{L^\infty}
\le C\Big[ h^{1/4} P_k(v)+h^{1/2}\F{k+1}(v)\Big]2^{-j_+b}
\ee
for some $P_k$ in $\Tr_k^\infty$. This will imply \eqref{3324} using \eqref{3320a} 
with $p=\infty$ and $N$ large enough, since $h_j=O(h^\sigma)$ 
and we may estimate $\blA \zj^k w_j\brA_{L^\infty}\les h_j^{-\mez}\Fr_k(w)2^{-j_+ b}$ by Sobolev. We prove \eqref{3325} estimating 
the $L^\infty$-norm of \eqref{3320}. We bound $\blA Z^k\Djh \QW\brA_{L^\infty}$ using \eqref{3213} with 
$p=p_{k_1}=p_{k_2}=\infty$ and 
$\blA Z^{k_\ell}\Delta_{j_\ell}^h w\brA_{L^\infty}=O(\E{k_\ell}(v)2^{-j_{\ell+}b})$. We obtain a bound 
in $2^{j-j_{+}b}P_k(v)$ for some $P_k$ in $\Tr_k^\infty$, which gives a contribution to the first term in the \rhs 
of \eqref{3325}, writing $2^{j/2}\sqrt{h}=O(h^{1/4})$ as $2^{j/2}\le Ch^{-\beta}$. 
To bound the cubic term in \eqref{3320}, we apply \eqref{3216} with $p=\infty$, $d=b-\alpha-0$, and \e{3321a}, and control 
the loss $2^{j_{+}(\alpha+0)}$ by a small negative power of $h$ using again $2^j\le Ch^{-2\beta}$. 
We obtain that the cubic term in \eqref{3320} is $O(h^{1/4}P_{k}(v))$. The $R_0$ term of \eqref{3320} 
is estimated in the same way, using \eqref{3217b}, \eqref{3321a}. 

Finally, we must bound the linear contributions in \eqref{3320}. Their $L^2$-norms are 
$$
O(h_j 2^{j/4-j_+b}\F{k+1}(v))
$$
according to the definition of $\F{k+1}(v)$ and the expression 
$w_j=\Tmj\Djh w$. Moreover, they are spectrally localized at $h_j|\xi|\sim 1$, so 
that by Sobolev injection, the $L^\infty$-norms are bounded by the $L^2$-norms multiplied by $Ch_j^{-1/2}$. 
This gives a contribution to the last term in \eqref{3325}.
\end{proof}

\underline{Third step:} Estimates on a microlocal neighborhood of $\Lambda$.

We have obtained in Proposition~\ref{ref:335} $L^\infty$-estimates for $\Zr_j^k w_j$ truncated outside a neighborhood of $\La$. We 
want here to prove similar $L^\infty$-estimates for $\Zr_j^k w_j$ truncated close to $\La$. They will be deduced from $L^2$-estimates 
for $\ophjdeux\Ophj(\gL)\Zr_j^k w_j$ that follow from \e{3320} truncated close to $\La$. Let us introduce the following decomposition 
of the function $w=\sumj \Tj w_j$ introduced in \e{3218}, \e{3211}: we define, using notation \e{3312}
\be\label{3326}
w_\La=\sumj \Tj\Ophj(\gL)w_j \index{Unknowns in Chapter~\ref{chap:3}!$w_\La$, $W_\La$}
\ee
and denote $W_\La=(w_\La,\overline{w}_{\La})$. We shall prove

\begin{prop}\label{ref:335}
Let $k\ge 1$. There are $C>0$, an element $P_k$ in $\Tr_k$ such that for any $v$ satisfying 
$\Fr_k(v)=O(\eps h^{-\uq})$, $\Er_{k-1}(v)\le h^{-1/16}$, for any $j$ in $J(h,C)$
\be\label{3327}
\ba
\blA \Ophj(\gL)\zjk w_j\brA_{L^\infty}\le C h^{-0}\bigg[ &\sum_{k_1+k_2+k_3\le k}
\prod _{\ell=1}^{3}\Fr_{k_\ell}(v)^{\frac 1 3}\Er_{k_\ell}(v)^{\frac 2 3}\\
&\quad +h^{\frac 1 8} P_k(v)+\Fr_{k+1}(v)\bigg]2^{-j_+b}.
\ea
\ee
\end{prop}
To prove the proposition, we shall use \e{3320} with $\QW$ replaced by $Q_0(w_\La)$ 
in the right hand side. Let us estimate the error that is done.

\begin{lemm}\label{ref:336}
For any $k\in\xN$, there are $C>0$, an element $P_k$ in $\Tr_k$ such that for any $v$ with $\Er_{k-1}(v)\le h^{-1/16}$
\be\label{3328}
\ba
\blA \Zr^k \Djh \bigl( \QW-Q_0(W_\La)\bigr) \brA_{L^2}&\le 
C \sqrt{h}2^{-j_+b}\sum_{k_1+k_2+k_3\le k}
\prod _{\ell=1}^{3}\Fr_{k_\ell}(v)^{\frac 1 3}\Er_{k_\ell}(v)^{\frac 2 3}\\
&\quad +h^{\tq} 2^{-j_+b}P_k(v).
\ea
\ee
\end{lemm}
\begin{proof}
We have to bound for $j$ in $J(h,C)$,
$$
\blA \Zr^k \Djh B_0(W,W-W_\La)\brA_{L^2}+\blA \Zr^k \Djh B_0(W_\La,W-W_\La)\brA_{L^2}.
$$
Consider for instance the first term. We decompose each argument using 
$$
w=\sumjell \Djellh w=\sumjell \Tjell w_{\jell}.
$$
By using \e{3326}, 
$$
w-w_\La =\sumjell \Tjell \Ophjell(\nongL)w_{\jell}.
$$
We write for $j_1\in J(h,C)$, 
\be\label{3328a}
\ba
\blA \Zr^{k_1}\Delta_{j_1}^h w\brA_{L^6}&\le \blA \Zr^{k_1}\Delta_{j_1}^h w\brd^{\frac 1 3}  \blA \Zr^{k_1}\Delta_{j_1}^h w\bri^{\frac 2 3}\\
&\le 2^{-j_{1+}(b+\frac{1}{3}(a-b))}\Fr_{k_1}(v)^{\frac 1 3}\Er_{k_1}(v)^{\frac 2 3}.
\ea
\ee
Moreover $\Zr^{k_2}\Delta_{j_2}^h (w-w_\La)$ may be written as
$$
\sum_{j_2'\, ;\, |j_2'-j_2|\le N_0}\Zr^{k_2}\Delta_{j_2}^h \Theta_{j_2'}^* \Op_{h_{j_2'}}\big(\nongL\big)w_{j_2'}
$$
for some large enough $N_0$, up to a remainder whose $L^3$ norm is 
$$
O\Big( h^\infty 2^{-j_{2+}(b+\frac{2}{3}(a-b))}\Fr_{k_2}(v)^{\frac 2 3}\Er_{k_2}(v)^{\frac 1 3}\Big).
$$
This follows from the fact that
$$
\Delta_{j_2}^h \Theta_{j_2'}^* \Op_{h_{j_2'}}\big(\nongL\big)w_{j_2'}
=\Theta_{j_2'}^* \Op_{h_{j_2'}}\Bigl(\varphi\big(2^{j_2'-j_2}\xi\big)\Bigr)
\Op_{h_{j_2'}}\big(\nongL\big)w_{j_2'}
$$
and that $\nongL$ is supported 
for $|\xi|\sim 1$. 

If we apply \e{3313}, we conclude that, since $\blA \Theta_{j_2}^*\brA_{\Fl{L^3}{L^3}}=O\bigl(2^{-j_2/6}\big)$, 
\be\label{3328b}
\ba
\blA \Zr^{k_2}\Delta_{j_2}^h (w-w_\La)\brA_{L^3}\le 
C_0\bigg[ \sqrt{h}2^{j_2/2-j_{2+}(b+\frac{1}{3}(a-b))}\sum_{k_1+k_2'\le k_2}\prod_{\ell=1}^2
\Fr_{k_\ell}(v)^{\frac 1 3}\Er_{k_\ell}(v)^{\frac 2 3}\\
\quad +h 2^{-j_{2+}b } h^{-4\beta-0}\sum_{k_1'+k_2'+k_3'\le k_2}
[\Fr_{k_1'}(v)\Fr_{k_2'}(v)\Er_{k_3'}(v) + \Er_{k_1'}(v)\Er_{k_2'}(v)\Er_{k_3'}(v)]\\
\quad +2^{j_2/12-j_{2+}a}h_{j_2}^{5/6} \Fr_{k_2+1}(v)+2^{-j_{2+}b}h^{1+\kappa}
P_{k_2}(v)\\
\quad +2^{-j_{2+}(b+\frac{2}{3}(a-b))}\Big(h_{j_2}^N +h^\infty\Big) 
\Fr_{k_2}(v)^{\frac 2 3}\Er_{k_2}(v)^{\frac 1 3}\bigg].
\ea
\ee
We plug \e{3328a}, \e{3328b} in \e{3213} with $p_{k_1}=6$, 
$p_{k_2}=3$ and we sum for $k_1+k_2\le k$, $j_1,j_2$ in $J(h,C)$. We obtain that, for some $P$ in $\Tr^k$,
\begin{align*}
&\blA \Zr^k \Djh B_0(W,W-W_\La)\brA_{L^2}\\
&\qquad\qquad\le C_1\sqrt{h}2^{j-j_+(b+\frac{1}{3}(a-b)-\mez)}\sum_{k_1+k_2+k_3\le k}\prod_{\ell=1}^3
\Fr_{k_\ell}(v)^{\frac 1 3}\Er_{k_\ell}(v)^{\frac 2 3}\\
&\qquad\qquad \quad +C_1 h^{1-4\beta-0}2^{j-j_+ b}P(v)\\
&\qquad\qquad\quad +C_1 2^{j-j_+(b+\frac{1}{3}(a-b))}h^{\frac 5 6} P(v).
\end{align*}
Since $2^j\le C h^{-2\beta}$, for $\beta$ small enough and $a-b\gg 1$, we get a quantity bounded from above by \e{3328}. This concludes the proof.
\end{proof}

Let us deduce from Lemma~\ref{ref:336} a sharp version of \e{3316}.

\begin{coro}\label{ref:337}
Let $k\ge 1$. There are $C>0$, and an element $P_k$ of $\Tr_k$ such that for any $j$ in $J(h,C)$, any $v$ with $\Er_{k-1}(v)\le h^{-1/16}$
\be\label{3329}
\ba
&\blA \ophjdeux \zjk w_j -2^{-j/2} \sqrt{h} A(h_j)\zjk \Tmj \Djh Q_0(W_\La)\brd\\
&\qquad \qquad \le C h^{\mez-0}h_j^\mez 2^{-j_+b}
\bigg[ \sum_{k_1+k_2+k_3\le k}\prod_{\ell=1}^3\Fr_{k_\ell}(v)^{\frac 1 3}\Er_{k_\ell}(v)^{\frac 2 3}
+h^{\frac 1 8}P_{k}(v)+\Fr_{k+1}(v)\bigg].
\ea
\ee
\end{coro}
\begin{proof}
We start from estimate \e{3316}. If in the left hand side, we replace $\QW$ by $Q_0(W_\La)$, the resulting 
error is bounded from above by the product of \e{3328} and of
$$
\sqrt{h}2^{-j/2}\blA \Tmj \brA_{\Fl{L^2}{L^2}}=O\big(h_j^{1/2}\big).
$$
We obtain 
a contribution to the \rhs of \e{3329}. On the other hand, the \rhs of \e{3316} is bounded from above by the \rhs of \e{3329} if 
we write that $2^{j/4}h^{9/8}\le 2^{j/2}h_j^{1/2} h^{5/8}$. This concludes the proof.
\end{proof}

\begin{proof}[Proof of Proposition~\ref{ref:335}] 
Let us prove that for any $j$ in $J(h,C)$
\be\label{3330}
\ba
&\blA \ophjdeux \Ophj(\gL)\zjk w_j\brd \\
&\qquad \qquad \le C h^{\mez-0}h_j^\mez 2^{-j_+b}
\left[ \sum_{k_1+k_2+k_3\le k}\prod_{\ell=1}^3\Fr_{k_\ell}(v)^{\frac 1 3}\Er_{k_\ell}(v)^{\frac 2 3}
+h^{\frac 1 8}P_{k}(v)+\Fr_{k+1}(v)\right]
\ea
\ee
for some element $P_k$ in $\Tr_k$. 

We notice first that since $w_j=\Tmj \Djh w$, the definition of $\Fr_{k}$ shows that
\be\label{3331}
\blA \zjk w_j\brd \le \Fr_{k}(v)2^{j/4-j_+a}.
\ee
Consequently, by the gain of one power of $h_j$ coming from symbolic calculus, we see that
$$
\blA \bigl[ \ophjdeux , \Ophj(\gL)\bigr] \zjk w_j\brd 
$$
is estimated by the last term in the \rhs of \e{3330}. We are reduced to estimating 
$\blA \Ophj(\gL)\ophjdeux \zjk w_j\brd$ which, according to \e{3329} is smaller than the \rhs 
of \e{3330} modulo the quantity
\be\label{3332}
2^{-j/2} \sqrt{h} \blA \Ophj(\gL)\zjk \Tmj \Djh Q_0(W_\La)\brd.
\ee
Since $w_\La$ is given by \e{3326}, we may write for $k'\le k$
$$
\Zr^{k'}w_\La=\sumj \Tj \big( \Zr_j^{k'}\Ophj(\gL)w_j\big)
$$
and by definition of $\Fr_k(v)$, and the fact that $w_j=\Tmj \Djh w$, we have 
$$
\blA \Zr_j^{k'}\Ophj(\gL)w_j\brd \le C2^{j/4-j_+a}\Fr_{k'}(v).
$$
Since $w_j$ is microlocally supported for $h_j|\xi|\sim 1$, we deduce from that 
$$
\blA \Zr_j^{k'}\Ophj(\gL)w_j\bri \le C h^{-\mez}2^{j/2-j_+a}\Fr_{k'}(v).
$$
By Definition~\ref{ref:3.1.5}, this shows that the family $\big(\Zr_j^{k'}\Ophj(\gL)w_j\big)_j$ is a family 
of elements of $\big( h^{-\mez}\mB_\infty^{1,a}[K]\big)\cap \big( \mB_2^{0,a}[K]\big)$ 
for some compact $K$ of $\cotangent\setminus 0$, contained in a small neighborhood of $\La$, and that 
$$
\blA \Zr_j^{k'}\Ophj(\gL)w_j\brA_{h^{-\mez}\mB_\infty^{1,a}[K]},
$$
which is by definition the best constant in \e{318}, is smaller than $C\Fr_{k'}(v)$. A similar estimate holds 
for $\blA \Zr_j^{k'}\Ophj(\gL)w_j\brA_{\mB_2^{0,a}[K]}$. We shall now prove that 
$Q_0(W_\La)$ is microlocally supported outside a neighborhood of $\La$.

Let us express $\Zr^k Q_0(W_\La)$ as a combination of terms $B_0\big(\Zr^{k_1}W_\La,\Zr^{k_2}W_\La\big)$, $k_1+k_2\le k$, so as 
a combination of expressions deduced from \e{327}.
\be\label{3332a}
\ba
&\Oph \big(a_0(\xi)\big)\Big[ 
\big(\Oph\big(a_1(\xi)\big)\Zr^{k_1}w_\La\big)
\big(\Oph\big(a_2(\xi)\big)\Zr^{k_2}w_\La\big)\Big],\\
&\Oph \big(a_0(\xi)\big)\Big[ 
\big(\Oph\big(a_1(\xi)\big)\Zr^{k_1}w_\La\big)
\big(\Oph\big(a_2(\xi)\big)\Zr^{k_2}\overline{w}_\La\big)\Big],\\
&\Oph \big(a_0(\xi)\big)\Big[ 
\big(\Oph\big(a_1(\xi)\big)\Zr^{k_1}\overline{w}_\La\big)
\big(\Oph\big(a_2(\xi)\big)\Zr^{k_2}\overline{w}_\La\big)\Big],
\ea
\ee
where $a_0,a_1,a_2$ are homogeneous of non negative order. We have just seen that $\Zr^{k'}w_\La$ is in 
$\big( h^{-\mez}\widetilde{\mB}_\infty^{1,a}[K]\big)\cap \big( \widetilde{\mB}_2^{0,a}[K]\big)$ with 
norm in that space bounded from above by $C\Fr_{k'}(v)$. It follows from Proposition~\ref{ref:3110} and the fact that $a$ is large enough relatively to $b$ 
that the first (resp.\ second, resp.\ third) expression \e{3332a} belongs to $h^{-\mez}\widetilde{\mB}_2^{1,b}[K_2]$ 
(resp.\ $h^{-\mez}\widetilde{\mB}_2^{1,b}[K_0]$, resp.\ $h^{-\mez}\widetilde{\mB}_2^{1,a}[K_{-2}]$) where 
$K_2$ (resp.\ $K_0$, resp.\ $K_{-2}$) is a compact subset of $\cotangent$ contained in a small neighborhood of $2\cdot\La$ (resp.\ $0\cdot \La$, 
resp.\ $-2\cdot \La$), and that the norm of these functions in those spaces is $O\big( \Fr_{k_1}(v)\Fr_{k_2}(v)\big)$. Consequently 
$\zjk \Tmj \Djh Q_0(W_\La)$ is microlocally supported far away from $\La$. When we apply a $\Ophj(\gL)$ cut-off as in \e{3332}, we gain a 
$O(h_j^\infty)=O(h^\infty)$ factor. We conclude that \e{3332} is bounded from above $h^N\sum_{k_1+k_2\le k}\Fr_{k_1}(v)\Fr_{k_2}(v)$ 
so that \e{3332} is controlled by the $h^{1/8}$--term in \e{3330}. 

To finish the proof of Proposition~\ref{ref:335}, we are left with showing 

\begin{lemm}\label{ref:338}
Assume that \e{3330} holds. Then estimate \e{3327} holds as well.
\end{lemm}
\begin{proof}
The definition of $\Fr_k(v)$ and the fact that $w_j=\Tmj\Djh v$ implies that
\be\label{3333}
\blA \Ophj(\gL)\zjk w_j\brd \le C2^{\frac{j}{4}-j_+a}\Fr_{k}(v).
\ee
Since $\bigl(2x\xi+\la\xi\ra^\mez\bigr)=(\xi-\dom(x))g(x,\xi)$ for some elliptic symbol $g$, on a neighborhood 
of the support of $\gamma_\La$, we deduce from \e{3330}, \e{3333} and symbolic calculus that
\be\label{3334}
\blA \Ophj(\xi-\dom(x))\Ophj(\gL)\zjk w_j\brd \le h_j^\mez h^{\mez-0}M2^{-j_+b}
\ee
where
$$
M=C\left[ \sum_{k_1+k_2+k_3\le k}\prod_{\ell=1}^3\Fr_{k_\ell}(v)^{\frac 1 3}\Er_{k_\ell}(v)^{\frac 2 3}
+h^{\frac 1 8}P_{k}(v)+\Fr_{k+1}(v)\right]
$$
for some $P_k$ in $\Tr_k$. We may rewrite \e{3333} and \e{3334} as 
\begin{align*}
&\lA e^{-i\omega(x)/h_j}\Ophj(\gL)\zjk w_j\rA_{L^2}
\le C 2^{\frac{j}{4}-j_+a}\Fr_k(v),\\
&\lA \bigl( h_j \OD\bigr) \Bigl(e^{-i\omega(x)/h_j}\Ophj(\gL)\zjk w_j\Bigr)\rA_{L^2}
\le C h_j^\mez h^{\mez-0}M 2^{-j_+b}.
\end{align*}
Using that $\lA f\rA_{L^\infty}=O\Bigl( \lA f\rA_{L^2}^\mez \lA \OD f\rA_{L^2}^\mez\Bigr)$, we get
\begin{align*}
\blA \Ophj(\gL)\zjk w_j\bri &\le C h^{-0}2^{\frac{j}{4}-j_+\frac{a+b}{2}}(\Fr_k(v)M)^\mez\\
&\le Ch^{-0}2^{-j_+b}M.
\end{align*}
This implies \e{3327}.
\end{proof}

{\em Proof of Proposition~\ref{ref:3.3.1}:} We combine estimates \e{3324} and \e{3327}. We obtain 
\be\label{3335}
\ba
\blA \zjk w_j\bri \le C h^{-0}\bigg[ &\sum_{k_1+k_2+k_3\le k}
\prod _{\ell=1}^{3}\Fr_{k_\ell}(v)^{\frac 1 3}\Er_{k_\ell}(v)^{\frac 2 3}\\
&\quad +h^{\frac 1 8} P_k(v)+\Fr_{k+1}(v)\bigg]2^{-j_+b}
\ea
\ee
for any $j$ in $J(h,C)$ and some $P_k$ in $\Tr_k$, assuming an a priori bound $\Er_{k-1}(v)\le h^{-1/16}$. 

We assume that \e{337} holds and that \e{339} has been proved up to order $k-1$. Consequently, by $(ii)$ of Lemma~\ref{ref:3.2.3}, we know that
\begin{align*}
\sup_j \Bigl( 2^{j_+ b} \blA \Djh \ophigh \Zr^k v \brA_{L^\infty}\Big)&\le C_k \eps h^{\frac 7 8},\\
\blA \Oph\bigl(\varphi_0\bigl( h^{-2(1-\sigma)}\xi\bigr)\Zr^k v\brA_{L^\infty}&\le C_k\eps h^{\frac{7}{16}-\sigma},
\end{align*}
where $C_K$ depends only on $A_0,\ldots,A_k$. 

On the other hand, \e{3335}Ê
gives a control of $\blA \Djh \Zr^k v\brA_{L^\infty}$ for 
$j\in J(h,C)$. Going back to the definition \e{335} of $\Er_k(v)$ we obtain
\be\label{3336}
\ba
\Er_k(v)\le C h^{-0}\bigg[ &\sum_{k_1+k_2+k_3\le k}
\prod _{\ell=1}^{3}\Fr_{k_\ell}(v)^{\frac 1 3}\Er_{k_\ell}(v)^{\frac 2 3}\bigg]\\
&\quad +h^{\frac{1}{16}} \big[ P_k(v)+C_k\eps \big]+h^{-0}\Fr_{k+1}(v).
\ea
\ee
Let us deduce from that that \e{339} holds at rank $k$. By the assumption \e{337} and the fact that by \e{330} 
$\delta_k'>\delta_{k+1}$, we may bound $h^{-0}\Fr_{k+1}(v)$ by 
$\eps A_k' h^{-\delta_k'}$ for some $A_k'>0$ depending only on $A_{k+1}$. 
The same is true for $h^{\frac{1}{16}}C_k\eps$, with $A_k'$ depending only on $A_0,\ldots,A_{k+1}$. 
Consider the $h^{\frac{1}{16}}P_k(v)$ contribution. By definition of the class $\Tr_k$ and \e{336b}, this term has modulus 
bounded from above by quantities of the form
\be\label{3337}
h^{\frac{1}{16}}\Er_{k_1}(v)\cdots \Er_{k_{\ell'}}(v)\Fr_{k_{\ell'+1}}(v)\cdots \Fr_{k_{\ell+1}}(v)
\ee
where $\ell'\le \ell \le 4$, $k_1+\cdots+k_\ell\le k$. Assume first that one of the $k_j$, $1\le j\le \ell'$, is equal to $k$, 
so that the other ones equal $0$. We obtain, according to assumption \e{337} a bound in 
$Ch^{\frac{1}{16}-(\ell'-1)\delta_0'-(\ell-\ell'+1)\delta_1}\times \Er_{k}(v)$, 
with a constant $C$ depending only on $\widetilde{A}_0,A_1$. 

By \e{330} and \e{333} this is smaller than $C\Er_k(v)h^{\frac{1}{32}}$ with 
a constant $C$ depending only on $\widetilde{A}_0,A_1$. On the other hand, if all $k_j$, $0\le j\le\ell'$, are strictly 
smaller than $k$, we may apply the induction hypothesis to estimate $\Er_{k_j}(v)$ and \e{337} to control 
$\Fr_{k_{j+1}}(v)$. We obtain for \e{3337} a bound in $C\eps h^{\frac{1}{16}-\delta_k'}$, according to the 
first inequality \e{330}, where the constant depends only on $\widetilde{A}_0,A_0,\ldots, A_k$. 

Let us study now the first term in the \rhs of \e{3336}. When $k_1<k$, $k_2<k$, $k_3<k$, we write 
\be\label{3338}
\ba
h^{-0}\prod _{\ell=1}^{3}\Fr_{k_\ell}(v)^{\frac 1 3}\Er_{k_\ell}(v)^{\frac 2 3}\le \frac{h^{-0}}{3}
\Big(& \Fr_{k_1}(v)\Er_{k_2}(v)\Er_{k_3}(v)\\
&+\Er_{k_1}(v)\Fr_{k_2}(v)\Er_{k_3}(v)+
\Er_{k_1}(v)\Er_{k_2}(v)\Fr_{k_3}(v)\Big).
\ea
\ee
By \e{337}, the fact that \e{339} is assumed to hold for $k_\ell<k$ and the first inequality \e{330}, we get that 
\e{3338} is $O\big(\eps h^{-\delta_k'}\big)$ with a constant depending only on $\widetilde{A}_0$, $A_1,\ldots,A_{k+1}$. 
Finally, we are left with studying the first term in the \rhs of \e{3336} when one the $k_j$ is equal to $k$, i.e.\ 
$$
h^{-0}\Er_{k}(v)^{\frac 2 3}\Fr_{k}(v)^{\frac 1 3}\Fr_{0}(v)^{\frac 2 3}\Er_{0}(v)^{\frac 4 3}
\le \frac{2}{3} \delta \Er_k(v)+\frac{1}{3} \delta^{-2}h^{-0}
\Fr_{k}(v)\Fr_{0}(v)^2\Er_{0}(v)^4
$$
for any $\delta>0$ (where in the right-hand side, $h^{-0}$ denotes $h^{-3\theta}$ if in the left hand side $h^{-0}$ stands for $h^{-\theta}$ with 
$\theta>0$ small). The last term in the above inequality is $O\big(\eps h^{-\delta_k'}\big)$ according to assumption \e{337} and the second inequality 
\e{330}, with a constant depending only on $\widetilde{A}_0,A_1,\ldots,A_{k+1}$. Summing up, we have obtained
$$
\Er_k(v)\le \Big[ \frac{2}{3}\delta +Ch^{\frac{1}{32}}\Big]\Er_{k}(v)+\eps A_k' h^{-\delta_k'}
$$
from which \e{339} at rank $k$ follows if $h$ and $\delta$ are taken small enough.
\end{proof}

\section{Decomposition of the solution in oscillating terms}\label{S:34}

The goal of this subsection is to give a description of 
the component $w$ in the decomposition~\e{3218} of $v$ in terms 
of oscillating contributions. More precisely, we expect $w$ to be a sum 
of a main term, oscillating along the phase $\omega$ 
(i.e.\ a term which is a lagrangian distribution along $\Lambda$), of $O(\sqrt{h})$ 
terms, 
coming from the quadratic part of the nonlinearity, 
that will oscillate along the phases $\pm 2\omega$ (so, which are associated to the 
lagrangians $\pm 2\La$), of $O(h)$ terms, coming from the cubic 
part of the nonlinearity, oscillating along the phases $\pm 3\omega$, $\pm \omega$, 
and a remainder. 
Moreover, we shall need, in preparation 
for next subsection, to get an explicit expression for contributions oscillating 
on $\pm 2\Lambda$. 

We consider a solution $v$ of $\e{326}$ satisfying for $h$ in some interval $]h',h_0]$ 
the a priori estimate \e{339} for 
$k'\le \frac{s}{2}+N_1$ for some fixed $N_1\ll s$. In particular, 
for $k\le \frac{s}{2}+N_1$, 
\be\label{341}
\blA \Djh \Zr^k v\brA_{L^\infty} \le \eps A_k'h^{-\delta_{k}'}2^{-j_+b}
\ee
for $j\in J(h,C)$. In this section, we shall denote by $K$ compact subsets of 
$\cotangent$ contained in a small neighborhood of one of the lagrangians $\ell\cdot\La$, 
$\ell\neq 0$, by $L$ compact subsets of $\cotangent$ and by $F$ closed subsets of $T^*\xR$ whose second projection 
is compact in $\Rs$.

We first obtain a rough decomposition of $v$.

\begin{lemm}\label{ref:341}
One may write $v=v_L+w_\Lambda+w_{\Lambda^c}+v_H$, where 
$v_L,v_H$ are defined in \e{3218} and, for some compact subset $K$ of $\cotangent$, lying in a small enough neighborhood of $\Lambda$ 
and intersecting $\Lambda$, some closed set $F$ as above, $\Zr^k w_\La$, $0\le k\le \frac{s}{2}+N_1$ is an $O(\eps)$ element of 
$h^{-\delta_{k+1}'}L^\infty\tilde{I}_\La^{0,b-2}[K]$ and $\Zr^k w_{\La^c}$, $0\le k\le \frac{s}{2}+N_1$ is an $O(\eps)$ element 
of $h^{\mez-\delta_{k+1}'}\tilde{\mB}_\infty^{0,b-2}[F]+h^{1-\dek}\tilde{\mB}_\infty^{-1,b-2}[F]$. 
Moreover, $\Oph(x\xi)\Zr^k w_{\La^c}$ is in 
$h^{\mez-\dek}\ti{\mB}_\infty^{1,b-2}[F]+h^{1-\dek}\ti{\mB}_\infty^{0,b-2}[F]$. 
\end{lemm}
\begin{proof}
We have written in \e{3218} $v=v_L+w+v_H$ and using notations \e{3311}, we may decompose $w=\sum_{j\in J(h,C)}\Tj w_j$. 
Recall definition \e{3312} of symbol $\gL$ and set \index{Unknowns in Chapter~\ref{chap:3}!$w_{j,\La}$}
$$
w_{j,\La}=\Ophj(\gL)w_j, \quad j\in J(h,C)
$$
so that $\Zr^k w_\La=\sum_{j\in J(h,C)}\Tj\bigl(\Zr_j^k w_{j,\La}\bigr)$. Since, 
when commuting $\Zr_j^k $ to $\Ophj(\gL)$, we get 
expressions of the form 
$\Ophj(\tilde{\gamma}_\La^{k'})\Zr_j^{k'}$ with $k'\le k$ and $\tilde{\gamma}_\La^{k'}$ a new symbol, we deduce from 
estimates \e{339} that $\Zr^k w_{\La}$ belongs to $h^{-\delta_k'}\tilde{\mB}_\infty^{0,b}[K]$ 
for a compact set $K$ satisfying the conditions of the statement if $\gL$ is supported in a small enough neighborhood of $\Lambda$. Moreover, 
$\Zr^k w_\La$ is $O(\eps)$ in the preceding space. 
If we use symbolic calculus, estimates \e{3318}, \e{3318a} and the assumed a priori estimates 
\e{339} together with \e{330}, we get
\be\label{341a}
\blA \ophjdeux \Zr_j^k w_{j,\La}\brA_{L^\infty}
\le C \eps 2^{-j_+ (b-2)}h^{-\delta_{k+1}'}\bigl[ h^\mez+h_j\bigr]
\ee
i.e.\ $\bigl(\Zr_j^k w_{j,\La}\bigr)_j$ belongs to $h^{-\delta_{k+1}'}L^\infty I_\La ^{0,b-2}$ and is 
of size $O(\eps)$ in that space. This gives the statement concerning $w_\La$ of the Lemma.

Set $w_{j,{\La^c}}=\Ophj(\gamma_\La^c)w_j$ \index{Unknowns in Chapter~\ref{chap:3}!$w_{j,\La^c}$}so that 
$w_{\La^c}=\sumj\Tj w_{j,\La^c}$. We use \e{3320a} with $p=\infty$. 
We estimate the first term in the right hand side of this inequality using \e{3318}, \e{3318a}, the bounds 
\e{339} together with inequalities \e{330}. We get
\be\label{a341b}
\blA \Ophj(\gamma_\La^c)\Zr_j^k w_j\brA_{L^\infty}\le C 
\eps 2^{-j_+(b-2)}h^{-\de_{k+1}'}\bigl(h^\mez +h_j\bigr)
+C_N \eps h_j^N h^{-\delta_{k}'}2^{-j_+b},
\ee
where the last term has been estimated from \e{341}. 

If $N$ is large enough, since $h_j\le Ch^\sigma$ we get that $\Zr^k w_{\La^c}$ is in 
$$
h^{\mez-\delta_{k+1}'}\tilde{\mB}_\infty^{0,b-2}[F]+h^{1-\dek}\tilde{\mB}_\infty^{-1,b-2}[F]
$$ 
and is $O(\eps)$ in that space (where $F$ is a closed set as described before the statement of Lemma~\ref{ref:341}). 

To study $\Oph(x\xi)w_{\La^c}$, we write
$$
\Oph(x\xi)w_{\La^c}=\sumj 2^{\jd}\Tj \Op_{h_j}(x\xi)w_{j,\La^c}.
$$
By symbolic calculus, we may write $\Op_{h_j}(x\xi)w_{j,\La^c}$ from 
$$
\Op_{h_j}\big( \ti{\gamma}_{\La^c}\big)w_j,\quad \Op_{h_j}\big(\gamma_{\La^c}\big)\big( \Op_{h_j}(x\xi)w_j\big),
$$
where $\ti{\gamma}_{\La^c}$ is a cut-off with support contained in the one of $\gamma_{\La^c}$. The $L^\infty$-norm of the action of 
$\Zr_j^k$ on the first of these expressions is bounded like \e{a341b}. The second expression may be written from
$$
\Op_{h_j}\big(\gamma_{\La^c}\big)\ophjdeux w_j,\quad \Op_{h_j}\big( \gamma_{\La^c}\big)\Op_{h_j}\big(\la \xi\ra^\mez\big)w_j.
$$
The $L^\infty$-norm of the action of $\Zr_j$ on the last term is bounded using \e{a341b} by the right hand side of this inequality. 
For the first term, we use again \e{3318}, \e{3318a} as in the proof of \e{341a} to get a similar upper bound. This concludes the proof.
\end{proof}

The decomposition $w=w_{\La}+w_{\La^c}$, in terms of a contribution $w_\La$ localized close to $\La$ and another one $w_{\La^c}$ 
supported outside a neighborhood of $\La$ is not precise enough for our purposes. We need to refine it, writing $w_{\La^c}$ as a sum of terms 
oscillating on the lagrangians $\pm 2 \Lambda$, of size of order $\sqrt{h}$, and of a remainder that is $O(h)$. Moreover, we need also to check 
that $w_\La$ is in $h^{-2\de_{k+1}'}L^\infty\tilde{J}_\La^{0,b'}[K]$. This is the goal of next proposition, that will be proved plugging 
the decomposition of Lemma~\ref{ref:341} in the equation \e{3222} satisfied by $w$, written under the form
\be\label{342}
\opdeux w=-\sqrt{h}\QW+h\Big[ \frac{i}{2}w-iZw-\CW\Big] -h^{\frac{5}{4}}R(V).
\ee

\begin{prop}\label{ref:342}Let $b'< b-5$ 
and $N_0<N_1$ such that $(N_1-N_0-1)\sigma\ge 1$. We may write the first decomposition 
of $w$
\be\label{343}
w=w_\La+\sqrt{h}\bigl(w_{2\La}+w_{-2\La}\bigr)+hg \index{Unknowns in Chapter~\ref{chap:3}!$w_{\pm 2\La}$}
\ee
where, for any $k\le \frac{s}{2}+N_0$, $\Zr^k w_{\pm 2\La}$ is a $O(\eps)$ element of 
$h^{-2\delta_{k+1}'}L^\infty \tilde{I}_{\pm 2\La}^{2,b'+\frac{3}{2}}[K_{\pm 2}]$, $\Zr^k w_\La$ is an 
$O(\eps)$ element of $h^{-\delta_{k+1}'}L^\infty \tilde{J}_{\La}^{0,b'}[K]$, 
$\Zr^k g$ is a $O(\eps)$ element 
of $h^{-3\delta_{k+1+N_1-N_0}'}\tilde{\mB}_\infty^{0,b'}[F]$ and 
$\Zr^k\Oph(x\xi)g$ is an $O(\eps)$ element of $h^{-3\delta_{k+1+N_1-N_0}'}\ti{\mB}_\infty^{1,b'}[F]$,  
for 
some compact subsets $K_{\pm 2}$ of $\cotangent\setminus 0$ 
contained in small neighborhoods of $\pm 2\La$, some 
closed subset $F$ of $T^*\R$ whose second projection is compact in $\Rs$. 
Moreover, $w_{\pm 2\La}$ are given by 
\be\label{344}
\ba
w_{2\La}&=-i(1-\chi)(xh^{-\beta})\frac{1+\sqrt{2}}{4}\la \diff\omega(x)\ra w_\La^2,\\
w_{-2\La}&=-i(1-\chi)(xh^{-\beta})\frac{1-\sqrt{2}}{4}\la \diff\omega(x)\ra \overline{w}_\La^2
\ea
\ee
where $\chi\in C^\infty_0(\xR)$, $\chi\equiv 1$ close to zero has small enough support.
\end{prop}

In order to prove the proposition, we shall compute the main contribution to $\QW$ obtained when plugging inside \e{327} 
the decomposition $w=w_\La+w_{\La^c}$ obtained in Lemma~\ref{ref:341}. We make at the same time a similar 
(and more precise) computation when one knows that an expansion of the form \e{343} holds.

\begin{lemm}\label{ref:343}
$i)$ Assume that $w=w_\La+w_{\La^c}$, where for all $k\le \frac{s}{2}+N_1$, $\Zr^k w_\La$ (resp.\ 
$\Zr^k w_{\La^c}$) is an $O(\eps)$ element of $h^{-\delta_{k+1}'}L^\infty \tilde{I}_\La ^{0,b-2}[K]$ 
(resp.\ of $h^{\mez-\delta_{k+1}'}\tilde{\mB}_\infty^{0,b-2}[F]+h^{1-\dek}\tilde{\mB}_\infty^{-1,b-2}[F]$ such that $\Zr^k\Oph(x\xi)w_{\La^c}$ 
is in $h^{\mez-\dek}\ti{\mB}_\infty^{1,b-2}[F]+h^{1-\dek}\ti{\mB}_\infty^{0,b-2}[F]$). 
Denote by $b'$ any number $b'<b-5$. Then, 
there are functions $\widetilde{w}_{\pm 2\La}$ such that for $k\le \frac{s}{2}+N_1$
$$
\Zr^k \widetilde{w}_{\pm 2\Lambda}\text{ is }O(\eps)\text{ in }h^{-2\delta_{k+1}'}L^\infty \tilde{I}_{\pm 2\La} ^{3,b'+\tdm}[K_{\pm 2}]
$$
so that
\be\label{345}
\ba
\sumj \Djh \QW&=\widetilde{w}_{2\La}+\widetilde{w}_{-2\La}+\sqrt{h}\tilde{g}_2,\\
\sumj \Djh \sqrt{h}\CW&=\sqrt{h} \tilde{g_3}
\ea
\ee
where for any $k\le \frac{s}{2}+N_1$,
$$
\Zr^k\ti{g}_2,\Zr^k \Oph(x\xi)\ti{g}_2\in h^{-2\de_{k+1}'}\ti{\mB}_\infty^{1,b'+\mez}[F],\quad 
\Zr^k \ti{g}_3,\Zr^k \Oph(x\xi)\ti{g}_3 \in h^{-3\de_{k+1}'-0}\ti{\mB}_\infty^{1,b'+\mez}[F]
$$
for some new closed subset $F$ 
of $T^*\xR$. 
Moreover one may write
\be\label{346}
\ba
\widetilde{w}_{2\La}&=-i(1-\chi)(x h^{-\beta})\la \diff \omega\ra^{\tdm} \frac{\sqrt{2}}{4}w_\La^2,\\
\widetilde{w}_{-2\La}&=-i(1-\chi)(x h^{-\beta})\la \diff \omega\ra^{\tdm} \frac{\sqrt{2}}{4}\overline{w}_\La^2,
\ea
\ee
for some $\chi\in C^\infty_0(\xR)$, $\chi\equiv 1$ close to zero, with small enough support. 

$ii)$ Assume one is given a decomposition of $w$ of the form \e{343} and denote by 
$b'$ any number $b'<b-8$. 
Then there are elements 
$\widetilde{w}_{\pm 2\La}$ such that for $0\le k\le \frac{s}{2}+N_0$, $\Zr^k\widetilde{w}_{\pm 2\La}$ 
is $O(\eps)$ in 
$h^{-2\delta_{k+1}'}L^\infty \tilde{J}_{\pm 2\La}^{3,b'+3/2}\bigl[K_{\pm 2}\bigr]$, and for 
$\ell \in \{\pm 1,\pm 3\}$ elements $\widetilde{w}_{\ell \La}$, such that, for any $0\le k\le \frac{s}{2}+N_0$, 
$\Zr^k \widetilde{w}_{\ell \La}$ is $O(\eps)$ in $h^{-3\de_{k+1}'}L^\infty \tilde{I}_{\ell\La}^{3,b'+3/2}\bigl[K_{\ell}\bigr]$ so that
\be\label{346a}
\sumj\Djh \big[\QW+\sqrt{h}\CW\big]=\widetilde{w}_{2\La}+\widetilde{w}_{-2\La}+\sqrt{h}\bigl( \widetilde{w}_{3\La}+\widetilde{w}_\La
+\widetilde{w}_{-\La}+\widetilde{w}_{-3\La}\bigr)+h\tilde{g},
\ee
where $\Zr^k\tilde{g}$ is in $h^{-4\de_{k+P}'}\tilde{\mB}^{1,b'+\mez}[F]$ and 
$\Zr^k\Oph(x\xi)\ti{g}$ is in $h^{-4\delta_{k+P}'}\ti{\mB}_\infty^{1,b'}[F]$ 
with $P=N_1-N_0+1$.

Moreover, $\widetilde{w}_\La$ is given in terms of $w_\La$ and of the functions in \e{344} by 
\be\label{347}
\ba
\widetilde{w}_\La&=\frac{i}{2}(1-\chi)(x h^{-\beta})\la \diff \omega\ra^{\tdm}\overline{w}_\La\bigl(w_{2\La}-\overline{w}_{-2\La}\bigr)\\
&\quad+\frac{1}{4}\unchi |\dom(x)|^{\frac{5}{2}}\big|w_\La \big|^2w_\La,
\ea
\ee
for some $\chi\in C^\infty_0(\xR)$, $\chi\equiv 1$ close to zero, with small enough support, and 
$\widetilde{w}_{3\La}$, $\widetilde{w}_{-\La}$, $\widetilde{w}_{-3\La}$ have similar expressions
\be\label{347a}
\ba
\widetilde{w}_{3\La}&=\frac{i}{2}(1-\chi)(x h^{-\beta})\la \diff \omega\ra^{\tdm}w_\La\bigl(\mu_3'w_{2\La}+\mu_3''\overline{w}_{-2\La}\bigr)\\
&\quad+\unchi |\dom(x)|^{\frac{5}{2}}\mu_3''' w_\La^3,\\
\widetilde{w}_{-\La}&=\frac{i}{2}(1-\chi)(x h^{-\beta})\la \diff \omega\ra^{\tdm}w_\La\bigl(\mu_{-1}'w_{-2\La}+\mu_{-1}''\overline{w}_{2\La}\bigr)\\
&\quad+\unchi|\dom(x)|^{\frac{5}{2}}\mu_{-1}''' |w_\La|^2\overline{w}_\La,\\
\widetilde{w}_{-3\La}&=\frac{i}{2}(1-\chi)(x h^{-\beta})\la \diff \omega\ra^{\tdm}\overline{w}_\La\bigl(\mu_{-3}'w_{-2\La}+\mu_{-3}''\overline{w}_{2\La}\bigr)\\
&\quad +\unchi |\dom(x)|^{\frac{5}{2}}\mu_{-3}''' \overline{w}_\La^3,
\ea
\ee
for some real coefficients $\mu_\ell'$, $\mu_\ell''$, $\mu_\ell'''$, $\ell\in \{-3,-1,3\}$.
\end{lemm}

Before starting the proof, we make the following remark that will be used several times below. 
\begin{rema}\label{ref:Remark1}
Let $\kappa$ be a smooth function on $\xR^*$, such that for some real numbers $\ell,\ell'$ and 
for any integer $k$, $\px^k \kappa=O\bigl(|x|^{-\ell-k}\langle x\rangle^{-\ell'}\bigr)$. Let 
$\chi$ be in $C^\infty_0(\xR)$, equal to one close to zero and let $r$ be an element 
of $\ti{\mB}_\infty^{\mu,\gamma}[K]$ for some compact subset of $\cotangent$. 
Then $(1-\chi)(x h^{-\beta})\kappa(x)r$ belongs to $\ti{\mB}_\infty^{\mu+(\ell+\ell'),\gamma+\ell'/2}[K]$.
\end{rema}
\begin{proof}
We decompose $r=\sumj \Tj r_j$ where $(r_j)_j$ is a bounded family in $\mB_\infty^{\mu,\gamma}[K]$. Then 
$$
(1-\chi)(x h^{-\beta})\kappa r=\sumj \Tj \ti{r}_j
$$
with $\ti{r}_j=(1-\chi)(x 2^{-j/2}h^{-\beta})\kappa(2^{-j/2}x)r_j$. 
Since $r_j$ is microlocally supported in $K$ we may, modulo a $O(h_j^\infty)=O(h^\infty)$ remainder, 
replace $r_j$ by $\theta(x)r_j$ for some $\theta\in C^\infty_0(\xR)$, equal to one 
on a large enough compact subset of $\xR^*$. Since 
$$
x\rightarrow (1-\chi)(x 2^{-j/2}h^{-\beta})2^{-\frac{j}{2}(\ell+\ell')+j_+\frac{\ell'}{2}}\kappa\bigl(2^{-\frac{j}{2}}x\bigr)\theta(x)
$$
is in $C^\infty_0(\xR^*)$ and has derivatives uniformly estimated in $j,h$, we see that $(\ti{r}_j)_j$ 
is microlocally supported on $K$ and satisfies uniform bounds in 
$\mB_\infty^{\mu+(\ell+\ell'),\gamma+\frac{\ell'}{2}}[K]$.
\end{proof}

\begin{rema}\label{ref:Remark2}
Let $\chi_1,\chi_2$ be two $C^\infty_0(\xR)$ functions equal to one close to zero and $r$ be in 
$\ti{\mB}_\infty^{\mu,\gamma}[K]$ for some compact set $K$ of $\cotangent$. Then, if 
$\supp \chi_1$ and $\supp\chi_2$ are small enough, $(1-\chi_1)(xh^{-\beta})r$ and 
$(1-\chi_2)(xh^{-\beta})r$ coincide modulo $O(h^\infty)$ 
(so that they are identified).
\end{rema}
\begin{proof}
We write again
$$
\Bigl[(1-\chi_1)(xh^{-\beta})-(1-\chi_2)(xh^{-\beta})\Bigr]r=\sumj \Tj 
\Bigl[(\chi_2-\chi_1)(x 2^{-j/2}h^{-\beta})r_j\Bigr].
$$
As above, modulo $O(h^\infty)$, 
we may insert some cut-off $\theta$ against $r_j$. 
We may then notice that $(\chi_2-\chi_1)(x 2^{-j/2}h^{-\beta})\theta(x)\equiv 0$ if 
$\supp \chi_\ell$ is small enough, as $2^{-j/2}h^{-\beta}>c$ for some $c>0$ since 
$j$ is in $J(h,C)$.
\end{proof}

To prove lemma \ref{ref:343}, it will be necessary to compute explicitly the action of some multilinear operators on functions of the type 
$w=w_\La+w_{\La^c}$.

Let us fix some notation. If $p_1,p_2$ are in $\xZ^*$, $K_{p_1}$, $K_{p_2}$ are compact 
subsets contained in small neighborhoods of 
$p_1\cdot\La$, $p_2\cdot\La$ and if $w_{p_\ell\cdot \La}^\ell$ is an 
element of $L^\infty\ti{I}_{p_\ell\cdot\La}^{\mu_\ell,\gamma_\ell}\bigl[ K_{p_\ell}\bigr]$, 
Proposition~\ref{ref:3110} shows that the product 
$w_{p_1\cdot\La}^1 \cdot w_{p_2\cdot\La}^2$ belongs to 
$L^\infty\ti{I}_{(p_1+p_2)\cdot\La}^{\mu_1+\mu_2,\gamma_1+\gamma_2}\bigl[ K_{p_1+p_2}\bigr]$ 
for some compact subset $K_{p_1+p_2}$ of $\cotangent$ contained in a small 
neighborhood of 
$(p_1+p_2)\cdot\La$, if $K_{p_1}$ and $K_{p_2}$ were contained 
in small enough neighborhoods of $p_1\cdot\La$, $p_2\cdot\La$ respectively. 
In the sequel, to avoid heavy notations, we shall eventually denote by $K_{p_\ell}$ different 
compact subsets of $\cotangent$ contained in a small 
enough neighborhood of $p_\ell\cdot\La$. All of them 
will be constructed from a compact subset $K$ of $\cotangent$ 
contained in a small enough neighborhood of $\La$. 
To simplify some notations, $p\Lambda$ will sometimes stand for $p\cdot \Lambda$. 
We shall also denote by $L$ some compact subset of $\cotangent$ which may vary from line to line.

\begin{lemm}\label{ref:345}
Let $b_\ell\colon \xR^*\rightarrow \xC$, $\ell=1,2,3$ be smooth functions positively homogeneous 
of degree $d_\ell$ and $a_0,a_1\colon \xR^*\rightarrow \xC$ be smooth, 
positively homogeneous of degree $m_0,m_1$. 

Let $p_\ell$ be in $\xZ^*$, $|p_\ell|\le 3$, $\ell=1,2,3$. If 
$p_1+p_2=0$ (resp.\ $p_1+p_2+p_3=0$), assume 
moreover that $a_1$ (resp. $a_0$) 
is an homogeneous polynomial of order $m_1\in \xN^*$ (resp.\ $m_0\in \xN^*$). 
Let $b'$ be a large enough positive number. Let $\chi$ be in $C^\infty_0(\xR)$, $\chi\equiv 1$ 
close to zero, with small enough support. 

$i)$ Assume given for $\ell=1,2,3$ functions $w_{p_\ell\La}^\ell$ such that for some 
$N$ and any $k\le \frac{s}{2}+N$, $\Zr^k w_{p_\ell \La}^\ell$ is in 
$h^{-\delta_{k+1}'}L^\infty\ti{I}_{p_\ell\cdot\La}^{0,b'}\bigl[K_{p_\ell}\bigr]$, for compact 
subsets $K_{p_\ell}$ satisfying the above conditions. 
Denote 
$\mu_2=2(m_1+d_1+d_2)$, $\mu_3=2(m_0+m_1+d_1+d_2+d_3)$. 
Then,
\be\label{3411}
\Oph(a_1)\Bigl[ \bigl( \Oph (b_1)w_{p_1\La}^1\bigr)\bigl( \Oph(b_2)w_{p_2\La}^2\bigr)\Bigr]
\ee
may be written as the sum of 
\be\label{3411a}
(1-\chi)(xh^{-\beta})a_1\bigl((p_1+p_2)\diff \omega\bigr)
b_1(p_1\diff \omega)b_2(p_2\diff \omega)
w_{p_1\La}^1w_{p_2\La}^2,
\ee
which is an element of $h^{-2\de_1'}L^\infty\ti{I}_{(p_1+p_2)\cdot\La}^{\mu_2,2b'}\bigl[ K_{p_1+p_2}\bigr]$ 
such that the action of $\Zr^k$ on it gives an element of 
$h^{-2\de_{k+1}'}L^\infty\ti{I}_{(p_1+p_2)\cdot\La}^{\mu_2,2b'}\bigl[ K_{p_1+p_2}\bigr]$ 
for $k\le \frac{s}{2}+N$, and of a remainder $R$ such that, for those $k's$, $\Zr^k R$ is in 
\be\label{3412a}
h^{\mez-2\de_{k+1}'}\ti{\mB}_\infty^{\mu_2-1,2b'-\mez}[L]
+h^{1-2\de_{k+1}'}\ti{\mB}^{\mu_2-2,2b'-1}[L].
\ee

In the same way, a cubic term 
\be\label{3411b}
\Oph(a_0)\Bigl[\Oph(a_1)\Bigl\{ \bigl( \Oph (b_1)w_{p_1\La}^1\bigr)
\bigl( \Oph(b_2)w_{p_2\La}^2\bigr)\Bigr\}\bigl( \Oph(b_3)w_{p_3\La}^3\bigr)\Bigr]
\ee
may be written as the sum of 
\be\label{3411c}
(1-\chi)(xh^{-\beta})
a_0\big((p_1+p_2+p_3)\dom \big)a_1\bigl((p_1+p_2)\diff \omega\bigr)
\prod_{\ell=1}^3 b_\ell(p_\ell\diff \omega) w_{p_\ell\La}^\ell,
\ee
which is a function such that the action of $\Zr^k$ on it, 
$k\le \frac{s}{2}+N$, gives an element of 
$$
h^{-3\de_{k+1}'}L^\infty\ti{I}_{(p_1+p_2+p_3)\cdot\La}^{\mu_3,3b'}\bigl[ K_{p_1+p_2+p_3}\bigr]
$$ 
and of a remainder $R$ such that $\Zr^kR$ is in 
\be\label{3411d}
\sum_{j=1}^3 h^{\frac{j}{2}-3\de_{k+1}'}\ti{\mB}_\infty^{\mu_3-j,3b'-j/2}[L].
\ee

$ii)$ Assume that we are given a function
\be\label{3412}
w^\ell=w_{p_\ell\cdot\La}^\ell +\sqrt{h}\Bigl[ w_{2p_\ell \cdot\La}^\ell +w_{-2p_\ell\cdot \La}^\ell\Bigr]
\ee
where for $k\le \sd+N$,
\begin{alignat*}{3}
&\Zr^k w_{\pm 2p_\ell\cdot\La}^\ell &&\text{ is in }
&&h^{-2\dek}L^\infty\ti{I}_{\pm 2 p_\ell\cdot\La}^{0,b'}\bigl[ K_{\pm 2p_\ell}\bigr],\\
&\Zr^kw_{p_\ell\cdot\La}^\ell&&\text{ is in }&&h^{-\dek}L^\infty\ti{J}_{p_\ell\cdot\La}^{0,b'}\bigl[K_{p_\ell}\bigr].
\end{alignat*}
Assume also that $p_1±\pm  2p_2\neq 0$, 
$p_2\pm 2p_1\neq 0$. Then \e{3411} may be written as the sum of a quadratic term, given by \e{3411a}, which is 
such that the action of $\Zr^k$ on it gives an element of $h^{-2\dek}L^\infty\ti{J}_{(p_1+p_2)\cdot\La}^{\mu_2,2b'}\bigl[K_{p_1+p_2}\bigr]$, 
of a cubic term, which may be written as the product of $\sqrt{h}$ and
\be\label{3412b}
\ba
(1-\chi)(x h^{-\beta})\Big[ &a_1\bigl( (p_1+2p_2)\diff \omega\bigr) 
b_1(p_1\diff \omega) b_2(2p_2\diff \omega) w_{p_1\La}^1 w_{2p_2\La}^2\\
&+a_1\bigl( (p_1-2p_2)\diff \omega\bigr) 
b_1(p_1\diff \omega)b_2(-2p_2\diff \omega) w_{p_1\La}^1 w_{-2p_2\La}^2\\
&+a_1\bigl( (2p_1+p_2)\diff \omega\bigr) 
b_1(2p_1\diff \omega)b_2(p_2\diff \omega) w_{2p_1\La}^1 w_{p_2\La}^2\\
&+a_1\bigl( (-2p_1+p_2)\diff \omega\bigr) 
b_1(-2p_1\diff \omega)b_2(p_2\diff \omega) w_{-2p_1\La}^1 w_{p_2\La}^2\Big]
\ea
\ee
and of a remainder term $R$. Moreover the action of $\Zr^k$ on 
\e{3412b}, $0\le k\le \sd+N$, gives an element of 
\begin{align*}
&\sum_{+,-}h^{-3\dek}L^\infty\ti{I}_{(p_1\pm 2p_2)\cdot\La}^{\mu_2,2b'}\bigl[ K_{p_1\pm 2 p_2}\bigr]\\
&\qquad +\sum_{+,-}h^{-3\dek}L^\infty\ti{I}_{(p_2\pm 2p_1)\cdot\La}^{\mu_2,2b'}\bigl[ K_{p_2\pm 2 p_1}\bigr]
\end{align*}
and the action of $\Zr^k$ on $R$ gives an element of 
\be\label{3412bb}
h^{1-4\delta_{k+1}'}\ti{\mB}_\infty^{\mu_2-1,2b'-\mez}[L].
\ee
\end{lemm}
\begin{proof}
$i)$ We use Proposition~\ref{ref:3.1.9}. By \e{3116} applied with a symbol $a\equiv b_\ell(\xi)$, $\ell=1,2$, we may write
\be\label{3412c}
\Oph (b_\ell)w_{p_\ell\La}^\ell=\unchi b_\ell(p_\ell \diff\omega)w_{\pl\La}^\ell+\tilde{r}^\ell
\ee
where the action of $\Zr^k$ on $\tilde{r}^\ell$ (resp.\ on the left hand side, resp.\ on the first term in the right hand side) 
of \e{3412c} is in $h^{\mez-\dek} \ti{\mB}_\infty^{2d_\ell-1,b'-\mez}[L]$ 
(resp.\ in $h^{-\dek} L^\infty\ti{I}_{p_\ell \La}^{2d_\ell,b'}[K_{p_\ell}]$). 

By Proposition~\ref{ref:3110}, 
\be\label{3412d}
Z^k\Bigl[ \bigl(\Oph(b_1)w_{p_1\La}^1\bigr)\bigl(\Oph(b_2)w_{p_2\La}^2\bigr)-
\unchi^2 b_1(p_1\diff\omega)b_2(p_2\diff\omega)w_{p_1\La}^1w_{p_2\La}^2\Bigr]
\ee
is in 
$$
h^{\mez-2\dek} \ti{\mB}_\infty^{2(d_1+d_2)-1,2b'-\mez}[L]
+h^{1-2\dek} \ti{\mB}_\infty^{2(d_1+d_2)-2,2b'-1}[L] 
$$
and the second term in \e{3412d} belongs to 
$$
h^{-2\dek} L^\infty\ti{I}_{(p_1+p_2)\cdot \La}^{2(d_1+d_2),2b'}\bigl[K_{p_1+p_2}\bigr].
$$
We make act $\Oph(a_1)$ on the bracket in \e{3412d}. By Proposition~\ref{ref:3.1.9}, this gives 
a remainder $R$ satisfying the conclusions of the statement. Moreover, the action of 
$\Oph(a_1)$ on
$$
\unchi ^2 b_1(p_1\diff\omega)b_2(p_2\diff \omega)w_{p_1\La}^1w_{p_2\La}^2
$$ 
may be written as \e{3411a} modulo similar remainders. Notice that the second remark after the statement of 
Lemma~\ref{ref:343} allows one to replace any power $(1-\chi)^p\bigl(x h^{-\beta}\bigr)$ by 
$\unchi$ if $\supp \chi$ is small enough.

One studies the cubic expressions \e{3411b} in the same way. 

$ii)$ We start from the stronger assumption \e{3412}. By \e{3116} 
and the lines following that formula 
we may write $\Oph(b_\ell)w^\ell$ as 
\be\label{3413}
\ba
&\unchi b_\ell(\pl\dom)w_{\pl\La}^\ell\\
&\quad +\sqrt{h}\unchi \Bigl[ b_\ell (2\pl\dom)w_{2\pl\La}^\ell
+b_\ell(-2\pl\dom)w_{-2\pl\La}^\ell\Bigr]\\
&\quad +\ti{r}^\ell
\ea
\ee
where the action of $\Zr^k$ on $\ti{r}^\ell$ is in $h^{1-2\de_{k+1}'} \ti{\mB}_\infty^{2d_\ell-1,b'-\mez}[L]$. 

Moreover 
$\Zr^k \bigl[ \unchi b_\ell(\pl\dom)w_{\pl\La}^\ell\bigr]$ (resp.\ 
$\Zr^k \bigl[ \unchi b_\ell(\pm 2\pl\dom)w_{\pm 2\pl\La}^\ell\bigr]$) belongs to 
$$
h^{-\dek}L^\infty\ti{J}_{\pl\cdot\La}^{2d_\ell,b'}\bigl[ K_{p_\ell}\bigr] 
\quad(\text{resp.\ }h^{-2\dek}L^\infty\ti{I}_{\pm 2\pl\cdot\La}^{2d_\ell,b'}\bigl[ K_{\pm 2p_\ell}\bigr]). 
$$
Applying Proposition~\ref{ref:3110}, we obtain that 
$\bigl(\Oph(b_1)w^1\bigr)\bigl(\Oph(b_2)w^2\bigr)$ may be written as the sum of quadratic terms 
\be\label{3414}  
\unchi ^2 b_1(p_1\dom)b_2(p_2\dom)\bigl(w_{p_1\La}^1\bigr)\bigl(w_{p_2\La}^2\bigr),
\ee
of cubic terms
\be\label{3415}
\ba
\sqrt{h}\unchi ^2\Bigl[ &b_1(p_1\dom)b_2(2p_2\dom)\bigl(w_{p_1\La}^1\bigr)\bigl(w_{2p_2\La}^2\bigr)\\
&+b_1(p_1\dom)b_2(-2p_2\dom)\bigl(w_{p_1\La}^1\bigr)\bigl(w_{-2p_2\La}^2\bigr)\\
&+b_1(2p_1\dom)b_2(p_2\dom)\bigl(w_{2p_1\La}^1\bigr)\bigl(w_{p_2\La}^2\bigr)\\
&+b_1(-2p_1\dom)b_2(p_2\dom)\bigl(w_{-2p_1\La}^1\bigr)\bigl(w_{p_2\La}^2\bigr)\Bigr]
\ea
\ee
and of a remainder $\ti{R}^\ell$ such that $\Zr^k \ti{R}^\ell$ is in 
$h^{1-4\de_{k+1}'} \ti{\mB}_\infty^{2(d_1+d_2)-1,2b'-\mez}[L]$. 

To study \e{3411}, we 
make act $\Oph(a_0)$ on \e{3414}, \e{3415} and on the remainder. We know from Proposition~\ref{ref:3110} that 
$\unchi^2 w_{p_1\La}^1 w_{p_2\La}^2$ is in 
$h^{-2\dek}L^\infty \ti{J}_{(p_1+p_2)\cdot\La}^{0,2b'}\bigl[ K_{p_1+p_2}\bigr]$ 
and 
that $\unchi^2 w_{p_1\La}^1w_{\pm 2p_2\La}^2$ (resp.\ $\unchi ^2 w_{\pm 2p_1\La}^1w_{p_2\La}^2$) belongs to 
$$
h^{-3\dek}L^\infty\ti{I}_{(p_1\pm 2p_2)\cdot\La}^{0,2b'}\bigl[ K_{p_1 \pm 2p_2}\bigr]
\quad (\text{resp.\ }h^{-3\dek}L^\infty\ti{I}_{(\pm 2p_1+p_2)\cdot\La}^{0,2b'}\bigl[ K_{\pm 2p_1 +p_2}\bigr]).
$$ 
To study the action of $\Oph(a_0)$ on \e{3414}, \e{3415}, we may use \e{3116a}, noticing that, since $\dom$ 
is homogeneous of degree $-2$,
\begin{align*}
&a_1(\xi)b_1(p_1\dom(x))b_2(p_2\dom(x)),\\
&a_1(\xi)b_1(p_1\dom(x))b_2(\pm 2p_2\dom(x)),\\
&a_1(\xi)b_1(\pm 2p_1\dom(x))b_2(p_2\dom(x))
\end{align*}
satisfy the assumptions 
\e{3115} with $(\ell,\ell',d,d')$ replaced by $(-2(d_1+d_2),0,m_1,0)$. We conclude that 
\e{3411} is given by the sum of \e{3411a}, \e{3412b} and remainders $R$ such that the action of 
$\Zr^k$ on $R$ gives elements of $h^{1-4\de_{k+1}'} \ti{\mB}_\infty^{\mu_2-1,2b'-\mez}[L]$. 
Moreover, \e{3411a} and \e{3412b} belong to the spaces indicated in the statement of the lemma. This concludes the proof.
\end{proof}

\begin{proof}[Proof of Lemma~\ref{ref:343}] $i)$ Let us prove the first equality \e{345}. Recall that we denote 
$W=(w,\bar{w})$. In the same way, set $W_\La=(w_\La,\bar{w}_\La)$, $W_{\La^c}=(w_{\La^c},\bar{w}_{\La^c})$. 
If $B_0$ denotes the polar form of $Q_0$, we have
$$
Q_0(W)=Q_0\big(W_\La\big)+2B_0\big(W_\La,W_{\La^c}\big)+Q_0\big(W_{\La^c}\big).
$$
For $j$ in $J(h,C)$, we set
$$
\ti{g}_{2,j}=h^{-\mez}\Tmj \Djh\big[ 2B_0\big(W_\La,W_{\La^c}\big)+Q_0\big(W_{\La^c}\big)\big].
$$
Let us show that $\ti{g}_2=\sumj \Tj \ti{g}_{2,j}$ satisfies the conclusions of the lemma. 
By assumption $\Zr^k w_\La$ is in $h^{-\dek}\ti{\mB}_\infty^{0,b-2}[K]$ and 
$\Zr^k w_{\La^c}$ is in $h^{\mez-\dek}\ti{\mB}_\infty^{0,b-2}[F]+h^{1-\dek}\ti{\mB}_\infty^{-1,b-2}[F]$. We plug these informations inside \e{3213}. We get
\begin{align*}
\blA \Djh \Zr^k B_0\big(W_\La,W_{\La^c}\big)\brA_{L^\infty}\le 
C 2^j\sum_{\substack{\max (j_1,j_2)\ge j-C\\ j_1,j_2\in J(h,C)}}2^{\mez \min(j_1,j_2)}\Big[ &h^{\mez-2\dek}+h^{1-2\dek}2^{-\frac{j_2}{2}}\Big]\\
&\times 2^{-(j_{1+}+j_{2+})(b-2)}.
\end{align*}
Summing using the fact that the number of negative $j_\ell$'s in $J(h,C)$ 
is $O(|\log(h)|)$, we obtain a bound in $2^{j-j_+(b-2)}h^{\mez-2\dek}$ which is the bound 
characterizing elements of $h^{\mez-2\dek}\mB_\infty^{2,b-2}[F]$ 
(where $F$ is a closed set of the form $C^{-1}\le |\xi|\le C$). In the same way
\begin{align*}
\blA \Djh \Zr^k B_0\big(W_{\La^c},W_{\La^c}\big)\brA_{L^\infty}\le 
C 2^j\sum_{\substack{\max (j_1,j_2)\ge j-C\\ j_1,j_2\in J(h,C)}}2^{\mez \min(j_1,j_2)}&\Big[ h^{\mez-\dek}+h^{1-\dek}2^{-\frac{j_1}{2}}\Big]\\
&\times \Big[h^{\mez-\dek}+h^{1-\dek}2^{-\frac{j_2}{2}}\Big]\\
&\times 2^{-(j_{1+}+j_{2+})(b-2)}.
\end{align*}
Summing we get a bound in 
$C\Big[ 2^j h^{1-2\dek}+2^{\jd}h^{2-0-2\dek}\Big] 2^{-j_+(b-2)}$. 
This characterizes an element of $h^{1-2\dek}\mB_\infty^{2,b-2}[F]+h^{2-0-2\dek}\mB_\infty^{1,b-2}[F]$. 
Summarizing, we get finally that $\ti{g}_2$ is in $h^{-2\dek}\ti{\mB}_\infty^{1,b-\frac{5}{2}}[F]$ which is the wanted conclusion since we assume $b'<b-3$. 

To estimate $\Oph(x\xi)\ti{g}_2$, we have to perform similar estimates replacing 
$B_0\big(W_\La,W_{\La^c}\big)$ (resp.\ $B_0\big(W_{\La^c},W_{\La^c}\big)$) by 
$(xhD_x)B_0\big(W_\La,W_{\La^c}\big)$ (resp.\ $(xhD_x)B_0\big(W_{\La^c},W_{\La^c}\big)$). If $S(\xi)$ is a positively 
homogeneous function of order $\lambda>0$, smooth outside zero, 
$$
[\Oph(x\xi),\Oph(S(\xi))]=i\lambda h \Oph(S).
$$
Consequently, the expression \e{327} of $Q_0$ and Leibniz rule show that 
$\Oph(x\xi)Q_0(V)$ may be expressed from $B_0\big(\Oph(x\xi)V,V\big)$ and from 
$h\ti{B}_0(V,V)$, where $\ti{B}_0$ is a bilinear form satisfying the same estimates \e{3213} as $B_0$ 
(Actually, $\ti{B}_0$ is either a multiple of the polar form of the quadratic form in the first line of the right hand side of \e{327}, or a multiple of the polar forms of the sum of 
the second and third lines). The last property stated in Lemma~\ref{ref:341} implies that
$$
\Zr^k \Oph(x\xi)w_{\La^c}\in h^{\mez-\dek}\ti{\mB}_\infty^{0,b-\frac{5}{2}}[F]+h^{1-\dek}\ti{\mB}_\infty^{-1,b-\frac{5}{2}}[F].
$$
Moreover, still because of this lemma, $\Zr^k w_\La$ is in $h^{-\dek}\ti{\mB}_\infty^{0,b-2}[K]$ for a compact subset 
$K$ of $\cotangent$. It follows from \e{3113} and the fact that $x\xi$ restricted to such a compact set is in the class of symbols $S(1)$, that 
$\Zr^k\Oph(x\xi)w_\La$ is in $h^{-\dek}\ti{\mB}_\infty^{1,b-2}[K]\subset h^{-\dek}\ti{\mB}_\infty^{0,b-\frac{5}{2}}[F]$. This shows that to estimate 
$\Oph(x\xi)B_0(W_\La,W_{\La^c})$, $\Oph(x\xi)B_0(W_{\La^c},W_{\La^c})$, it suffices to use the bounds obtained above 
for $B_0(W_\La,W_{\La^c})$, $B_0(W_{\La^c},W_{\La^c})$ replacing $b$ by $b-\mez$. We conclude that 
$\Oph(x\xi)\ti{g}_2$ is in $h^{-2\dek}\ti{\mB}_\infty^{1,b-3}[F]\subset h^{-2\dek}\ti{\mB}_\infty^{1,b'+\mez}[F]$.

We compute next $Q_0(W_\La)$ from \e{327}. Let us examine 
first the contributions that are bilinear in $(w_\La,\bar{w}_\La)$ i.e.\
\be\label{3417}
\ba
&-\frac{i}{4}\Oph(|\xi|^{\mez}) \Bigl[ \bigl( \Oph(\xi|\xi|^{-\mez}) w_\La\bigr)
 \bigl( \Oph(\xi|\xi|^{-\mez}) \bar{w}_\La\bigr)\Bigr]\\
&-\frac{i}{4}\Oph(|\xi|^{\mez})
\Bigl[\bigl(\Oph(|\xi|^\mez) w_\La\bigr)\bigl(\Oph(|\xi|^\mez) \bar{w}_\La\bigr)\Bigr]\\
&+\frac{i}{4} \Oph(|\xi|)\Bigl[ w_\La\Oph(|\xi|^\mez)\bar{w}_\La
-\bar{w}_\La\Oph(|\xi|^\mez)w_\La\Bigr]\\
&-\frac{i}{4}\Oph(\xi)\Bigl[ w_\La\Oph(\xi|\xi|^{-\mez})\bar{w}_\La-\bar{w}_\La\Oph(\xi|\xi|^{-\mez})w_\La\Bigr].
\ea
\ee
We use that \e{3411} may be computed from \e{3411a}, up to 
a remainder given by \e{3412a} with $\mu_2=3$, $b'=b-2$, 
that contributes to $\sqrt{h}\ti{g}_2$ in \e{345} (since $b'<b-\frac{5}{2}$ and $b$ is large enough). 
Notice that the main contribution, computed from 
\e{3411a} vanishes. For the terms inside the first two brackets in \e{3417}, this follows 
from a two by two cancellation between the two contributions in each bracket. For the last term in \e{3417}, 
we remark that the symbol $\xi$ of the outside operator 
$\Oph(\xi)$ is an homogeneous polynomial, which allows us to make use of expansion \e{3411a} with $a_1\equiv\xi$, 
$p_1+p_2=0$, and implies as well the vanishing of that term. 

We are left with studying the quadratic terms in $w_\La$ and the quadratic terms in $\bar{w}_\La$ in \e{327}. We may apply to 
both of them $i)$ of Lemma~\ref{ref:345} with $(p_1,p_2)=(1,1)$ or $(p_1,p_2)=(-1,-1)$. We get the contribution to 
$Q_0(W_\La)$ given by the sum of the two expressions \e{346}. 

To study $C_0(W)$, we use that the assumptions imply that $\Zr^k w$ is in 
$h^{-\dek}\ti{\mB}_\infty^{0,b-2}[F]$ (This follows from the fact that $h\ti{\mB}_\infty^{-1,b-2}[F]\subset 
h^\sigma \ti{\mB}_\infty^{0,b-2}[F]$, as a consequence of the inequality $h_j=O(h^\sigma)$). To bound 
$\blA \Djh \Zr^k C_0(W)\brA_{L^\infty}$, we apply \e{3216} with $p=\infty$, $d=b-\alpha-2-0$, $V_1=V_2=V_3=W$. 
Our assumptions on $w$ and $d$ imply that
$$
\blA \Zr^k \japon^{\alpha+d}w\brA_{L^\infty}=O\big( h^{-\dek-0}\big)
$$
as is seen from the expansion $w=\sumj \Tj w_j$ and the bounds on the $w_j$'s. It follows from \e{3216} that 
\be\label{a3417b}
\blA \Djh \Zr^k C_0(W)\brA_{L^\infty}=O\Big( 2^{\jd-j_+(b-\alpha-2-0)}h^{-3\dek-0}\Big).
\ee
The conclusion $\Zr^k g_3\in h^{-3\dek-0}\ti{\mB}_\infty^{1,b'+\mez}[F]$ follows if we assume $b'<b-\frac{9}{2}$ 
(since $\alpha$ is any number strictly larger than $2$). 

To obtain that $\Zr^k\Oph(x\xi)\ti{g_3}$ is in $h^{-3\dek-0}\ti{\mB}_\infty^{1,b'+\mez}[F]$ we make act 
$\Oph(x\xi)$ on $C_0(W)$ and we argue as in the study of quadratic terms, distributing $xhD_x$ on the different factors using Leibniz rule. 
We have seen that $\Zr^k\Oph(x\xi)W$ is $h^{-\dek}\ti{\mB}_\infty^{0,b-\frac{5}{2}}[F]$. It follows as above that we get for 
$\blA \Djh \Zr^k \Oph(x\xi)C_0(W)\brA_{L^\infty}$ the same estimate as \e{a3417b}, with $b$ replaced by $b-1/2$. This gives the wanted bound as 
$b'<b-5$. This concludes the proof of $i)$ of the lemma.

$ii)$ Let us show first that we may replace in the quadratic (resp.\ cubic) part of the left hand side of \e{346a} $w$ by 
$w_p=w_\La+\sqrt{h}\big(w_{2\La}+w_{-2\La}\big)$ (resp.\ by $w_\La$) up to a contribution to the 
$h\ti{g}$ term in the right hand side. By assumption, $\Zr^k w_p$ is in $h^{-\dek}\ti{\mB}_\infty^{0,b'}[F]$ for some 
$b'<b-\frac{9}{2}$. Set $W_p=(w_p,\bar{w}_p)$, $G=(g,\bar{g})$ and let us show that the contributions 
of $\Zr^k B_0(W_p,G)$ and $h\Zr^k B_0(G,G)$ are in $h^{-4\delta_{k+P}'}\ti{\mB}_\infty^{2,b'}[F]$. 
We use \e{3213} and the assumption that $\Zr^k g$ is in $h^{-3\delta_{k+P}'}\ti{\mB}_\infty^{0,b'}[F]$ to bound using \e{3213}
\begin{align*}
\blA \Djh \Zr^k B_0\big(W_p,G\big)\brA_{L^\infty}\le 
C 2^j\sum_{\substack{\max (j_1,j_2)\ge j-C\\ j_1,j_2\in J(h,C)}}2^{\mez \min(j_1,j_2)} h^{-\dek-3\delta_{k+P}'}2^{-(j_{1+}+j_{2+})b'}.
\end{align*}
We get an estimate in $O\big( h^{-4\delta_{k+P}'}2^{j-j_+ b'}\big)$ which shows the wanted conclusion. One argues in the same way for 
$h B_0(G,G)$. 

Considering the cubic term, we write $w=w_\La+\sqrt{h}g'$, where 
$g'=(w_{2\La}+w_{-2\La})+\sqrt{h}g$ satisfies $\Zr^k g'\in h^{-2\dek}\ti{\mB}_\infty^{0,b'}[F]$ and where $\Zr^k w_\La$ is in $h^{-\dek}\ti{\mB}_\infty^{0,b'}[F]$. 
If we set $G'=(g',\bar{g}')$, we have to study $\blA \Djh \Zr^k \big[ C_0(W_\La+\sqrt{h}G')-C_0(W_\La)\big]\brA_{L^\infty}$ i.e.\ 
$\sqrt{h} \blA \Djh \Zr^k T_0(W_\La,W_\La,G')\brA_{L^\infty}$, $h\blA \Djh \Zr^k T_0(W_\La,G',G')\brA_{L^\infty}$ and 
$h^{\tdm}\blA \Djh \Zr^k T_0(G',G',G')\brA_{L^\infty}$. If $d=b'-\alpha-0$, we bound
$$
\blA \Zr^k\japon^{\alpha+d}W_\La\brA_{L^\infty}=O\big(h^{-0-\dek}\big),\quad 
\blA \Zr^k\japon^{\alpha+d} G'\brA_{L^\infty}=O\big(h^{-2\dek-0}\big).
$$
Plugging these estimates in \e{3216}, we get for the quantities under study a bound in terms of 
$2^{\jd-j_+(b'-\alpha-0)}h^{-4\dek+\mez}$. 
Let us study as well
$$
\OPHX B_0(W_\La,G),\quad  
\OPHX B_0(G,G),\quad 
\OPHX \big[ C_0\big(W_\La+\sqrt{h}G'\big)-C_0\big(W_\La\big)\big].
$$ 
As in the proof of $i)$, we may express these quantities from 
\begin{align*}
&B_0\big(\OPHX W_p,G\big),\quad 
B_0\big( W_p,\OPHX G\big), \\
&\sqrt{h}T_0\big(\OPHX W_\La,W_\La,G'\big),\quad 
\sqrt{h}T_0\big(W_\La,W_\La,\OPHX G'\big), \\
&hT_0\big(W_\La,\OPHX G',G'\big), \quad 
hT_0\big(W_\La,\OPHX G',G'\big), \quad h^{\tdm}T_0\big(\OPHX G',G',G'\big),
\end{align*}
and from quadratic and cubic 
quantities of the form of those already estimated. By assumption, the $g$-term in \e{343} satisfies 
$\Zr^k\OPHX g\in h^{-3\delta_{k+P}'}\timb{1,b'}[F]\subset 
h^{-3\delta_{k+P}'}\timb{0,b'-\mez}[F]$. Moreover, by definition of that quantity, 
$\Zr^k\OPHX w_p$ is in $h^{-\dek}\timb{0,b'}[F]$. The above estimate of $B_0$, with 
$b'$ replaced by $b'-\mez$, shows that
$$
\blA \Djh \Zr^kB_0\big(\OPHX W_p,G\big)\brA_{L^\infty}+
\blA \Djh \Zr^kB_0\big( W_p,\OPHX G\big)\brA_{L^\infty}=O\big( h^{-4\delta_{k+P}'}2^{j-j_+(b'-\mez)}\big).
$$
We obtain a $h^{-4\delta'_{k+P}}\timb{2,b'-\mez}[F]\subset h^{-4\delta'_{k+P}}\timb{1,b'-\tdm}[F]$ contribution to 
$\Zr^k\OPHX \ti{g}$ in the action of $\OPHX$ on \e{346a}. The definition of $g'$ implies that 
$\Zr^k\OPHX g'$ is in $h^{-2\dek}\timb{0,b'-\mez}[F]$. Bounding, with $d=b'-\mez-\alpha-0$
\begin{align*}
&\blA \Zr^k \japon^{\alpha+d}\OPHX W_\La\brA_{L^\infty}=O\big(h^{-0-\dek}\big),\\
&\blA \Zr^k \japon^{\alpha+d}\OPHX G'\brA_{L^\infty}=O\big(h^{-0-2\dek}\big),
\end{align*}
we get again from \e{3115} that
$$
\blA \Djh \Zr^k\OPHX \big[ C_0\big(W_\La+\sqrt{h}G'\big)-C_0\big(W_\La\big)\big]\brA_{L^\infty}
=O\big( 2^{\jd -j_+(b'-\mez-\alpha-0)}h^{-4\dek+\mez}\big).
$$
If we replace above $b'$ by $b'+3$ (since $\alpha=2+0$), which corresponds to decreasing by $3$ units the assumption made on 
$b'-b$ in Proposition~\ref{ref:342} (i.e.\ imposing $b'<b-8$), we obtain finally that the contributions of 
$Q_0(W)-Q_0(W_p)$ and $C_0(W)-C_0(W_\La)$ to \e{346a} 
may be incorporated into the $h\ti{g}$ term of the right hand side. We are reduced to the study of 
$$
Q_0\big(W_\La +\sqrt{h}(W_{2\La}+W_{-2\La})\big), \quad C_0(W_\La).
$$
To treat the first expression, we use $ii)$ of Lemma~\ref{ref:345}, which allows us to compute expressions \e{327} using \e{3212}. The remainders satisfy 
bounds of the form \e{3412bb} with $\mu_2=3$, so may be incorporated to the $h\ti{g}$ term in \e{346a}. We have 
already seen in the proof of $i)$ that the $O(1)$ term in \e{346a} is given by \e{346}.

The $O(\sqrt{h})$ term is computed from \e{3412b} applied to the different contributions to $Q_0$ 
given by \e{327}. 
We need to compute explicitly only the 
$\La$-oscillating term i.e.\ the contributions to \e{3412b} corresponding to $p_1\pm 2p_2=1$ and 
$p_2\pm 2p_1=1$ ($p_1,p_2\in \{-1,1\}$). From the expression \e{327} of $Q_0$ and 
\e{3412b}, we get a contribution 
\be\label{3418}
\unchi \frac{i}{2} \la \dom\ra^\tdm \overline{w}_\La\bigl(w_{2\La}-\overline{w}_{-2\La}\bigr).
\ee
In the same way, using \e{3411c}, we compute the $\La$-oscillating 
cubic term coming from the expression \e{328} of $C_0(W_\La)$. 
We obtain a contribution $\frac{1}{4}\unchi |\dom(x)|^{\frac{5}{2}}|w_\La |^2 w_\La$. 
Summing up, we get
\begin{align*}
Q_0(w,\bar{w})+\sqrt{h}C_0(w,\bar{w})&=\widetilde{w}_{2\La}+\widetilde{w}_{-2\La}\\
&\quad +\sqrt{h}\bigl( \widetilde{w}_{3\La}+\widetilde{w}_{\La}+\widetilde{w}_{-\La}+\widetilde{w}_{-3\La}\bigr)\\
&\quad +h\ti{g}_1
\end{align*}
where according to \e{3411d} and \e{3412bb}, 
$$
\Zr^k \ti{g}_1\in 
h^{-4\de_{k+1}'} \ti{\mB}_\infty^{2,b'+1}[F]\subset h^{-4\delta_{k+1}'}\ti{\mB}_\infty^{1,b'+\mez}[F]
$$ 
(if $b'$ is large enough), 
and where 
$\widetilde{w}_\La$ is given by \e{347}. The contributions $\widetilde{w}_{3\La}$, $\widetilde{w}_{-\La}$, 
$\widetilde{w}_{-3\La}$ have expressions \e{347a}, for which we do not need to compute 
explicitly the coefficients $\mu_\ell'$, $\mu_\ell''$. This concludes the proof of the lemma.
\end{proof}

The next step of the proof of Proposition~\ref{ref:342} will be to deduce from equation~\e{3211}, and from the description provided by 
Lemma~\ref{ref:343} of the \rhs of this equation, an expansion of $\Oph(\nongL )w$, exploiting that $2x\xi+\la\xi\ra^\mez$ is an elliptic symbol on the support 
of $\nongL$. In a first step, we establish some a priori bounds for the components $w_j$ of $w$ cut-off outside a neighborhood of $\La$. 

\begin{lemm}\label{ref:346}
Assume that \e{339} holds for $k\le \sd +N_1$ for some integer $N_1$ 
satisfying $(N_1-N_0)\sigma\ge 1$. Then for any 
symbol $a$ in $S(1)$, microlocally supported outside a neighborhood of $\La$, the following estimate holds for $k\le \sd+N_0$, 
and any $j$ in $J(h,C)$,
\be\label{3419}
\blA \Ophj (a) \Zr_j^k w_j\brA_{L^\infty}
\le C_k \eps h^{\mez-\de_{k+N_1-N_0}'}
\bigl(2^{j/2}+h^{1/2}\bigr)2^{-j_+ b'}
\ee
for any $b'<b-\alpha<b-2$.
\end{lemm}
\begin{proof}
Let us construct for any $1\le \ell\le N_1-N_0$, any $k\le \sd+N_1-\ell$, a family of symbols 
in $S(1)$, $(b_{\ell'}^\ell)_{0\le \ell'\le \ell}$, vanishing close to $\La$ and a sequence $(r_j^{\ell,k})_{j\in J(h,C)}$ with 
\be\label{3420}
\blA r_j^{\ell,k}\brA_{L^\infty}\le C\eps h^{\mez-\de_{k+\ell}'}\bigl( 2^{j/2}+h^{1/2}\bigr)2^{-j_+b'}
\ee
such that
\be\label{3421}
\Zr_j^k\Ophj(a)w_j=h_j^\ell \left[ \sum_{\ell'=0}^\ell 
\Zr_j^{k+\ell'}\Ophj\bigl(b_{\ell'}^\ell\bigr)
w_j^{(\ell,\ell')}\right]+r_j^{\ell,k}
\ee
where $w_j^{(\ell,\ell')}$ denotes a function defined like $w_j=\Tmj\Djh w$ but with 
$\Djh$ replaced by another cut-off of the same type. 
Then \e{3421} with $\ell=N_1-N_0$ implies \e{3419} since 
$h_j^{N_1-N_0}\le h^{\sigma(N_1-N_0)}\le h$ and since $w_j^{(\ell,\ell')}$ satisfies the same 
$L^\infty$ estimates \e{339} as $w_j$.

We remark that to prove \e{3421}, we just need to treat the case $\ell=1$ and iterate the formula. 
Finally, to obtain \e{3421} with $\ell=1$, we use that, by the symbolic calculus of appendix, and since $a$ vanishes 
close to $\La$, we may find a symbol $q$ in $S\bigl(\langle x\rangle^{-1}\bigr)\subset S(1)$, vanishing close to $\La$, 
a symbol $\rho$ in $S(1)$ such that
$$
\Ophj (a)=\Ophj(q)\ophjdeux+h_j^N \Ophj(\rho)
$$
for an arbitrary integer $N$. If $N$ is large enough, the fact that \e{339} holds implies that $\Zr_j^k\bigl(h_j^N \Ophj(\rho)w_j\bigr)$ 
satisfies estimate \e{3420} with $\ell=1$ for all $j\in J(h,C)$. We are thus reduced to showing that
\be\label{3422}
\Zr_j^k\Ophj(q)\ophjdeux w_j
\ee
may be written as the \rhs of \e{3421} with $\ell=1$. We use \e{3319a}. On the one hand, we get a contribution to 
\e{3422} of the form
$$
\Zr_j^k\Bigl[ h_j\Ophj(q)\Bigl( \frac{i}{2}\widetilde{w}_j-iZw_j\Bigr)\Bigr]
$$
that forms part of the sum in \e{3421} with $\ell=1$. On the other hand, the nonlinear terms in \e{3319a} bring an expression 
$$
\Zr_j^k \Ophj(q)\Bigl[ -\sqrt{h}2^{-\frac{j}{2}}\Tmj \Djh \QW
-h 2^{-\jd}\Tmj \Djh \CW-2^{-\jd}h^{\frac{5}{4}}\Tmj \Djh R(V)\Bigr].
$$
Let us check that these terms satisfy estimates \e{3420} with $\ell=1$. For the quadratic terms, this follows from \e{3213} 
(with $p=p_{k_1}=p_{k_2}=\infty$) and from the assumption $\delta_{k_1}'+\de_{k_2}'\le \dek$ 
if $k_1+k_2\le k$, that follows from \e{330}. For the cubic term (resp.\ the remainder) we use \e{3216} (resp.\ \e{3217b}) 
and the estimates of $\blA \Zr^k \japon^{\alpha+d}V\brA_{L^\infty}$ deduced from \e{339} 
with 
$d=b-\alpha-0$, $b'<b-\alpha-0$. This concludes the proof.
\end{proof}

We shall use the preceding lemma to give an asymptotic expansion of 
$\Ophj(\nongL)w_j$, assuming that we know a priori that $Q_0(w,\bar{w})$ admits the expansion given by equality \e{345} in 
Lemma~\ref{ref:343}, or that $Q_0(w,\bar{w})+\sqrt{h}C_0(w,\bar{w})$ obeys the equality \e{346a} of the same lemma.

\begin{lemm}\label{ref:347}
$i)$ Assume $b'<b-5$ and that $\QW$ satisfies \e{345}. Then there are functions 
$w_{\pm 2\La}=\sumj \Tj w_{\pm 2\La,j}$ 
such that for any $k\le \sd+N_0$, $\Zr^k w_{\pm 2\La}$ is an $O(\eps)$ element of 
$h^{-2\dek}L^\infty\ti{I}_{\pm 2\La}^{2,b'+\tdm}\bigl[ K_{\pm 2}\bigr]$ and a function 
$g=\sumj \Tj g_j$ such that $\Zr^k g$ is an $O(\eps)$ element of 
$h^{-3\de_{k+1+N_1-N_0}'}\ti{\mB}_\infty^{0,b'}[F]$ for some closed subset $F$ of 
$T^*\xR$ whose second projection is compact in $\xR^*$ 
and $\Zr^k\OPHX g$ is an $O(\eps)$ element of $h^{-3\delta_{k+1+N_1-N_0}'}\timb{1,b'}[F]$, so that
\be\label{3423}
\ba
w_{2\La}&=-i\unchi \frac{1+\sqrt{2}}{4}\la \dom(x)\ra w_\La^2\\
w_{-2\La}&=-i\unchi \frac{1-\sqrt{2}}{4}\la \dom(x)\ra \bar{w}_\La^2
\ea
\ee
and for any $j$ in $J(h,C)$
\be\label{3424}
\Ophj (\nongL)w_j=\sqrt{h}\bigl(w_{2\La,j}+w_{-2\La,j}\bigr)+h g_j.
\ee

$ii)$ Assume that $b'<b-8$ and $Q_0(w,\bar w)+\sqrt{h}C_0(w,\bar w)$ obeys \e{346a}. 
Then there are functions $w_{\pm 2\La}$ such that for any 
$k\le \sd +N_0$, 
$\Zr^k w_{\pm 2\La}$ is an $O(\eps)$ element of 
$h^{-2\dek}L^\infty\ti{J}_{\pm 2\La}^{2,b'+\tdm}\bigl[ K_{\pm 2}\bigr]$, there are 
functions $w_{\ell\La}=\sumj \Tj w_{\ell\La,j}$ for $\ell=-3,-1,3$, such that 
for $k\le \sd+N_0$, $\Zr^k w_{\ell\La}$ is an $O(\eps)$ element in 
$h^{-3\dek}L^\infty\ti{I}_{\ell\La}^{2,b'+\tdm}\bigl[ K_{\ell}\bigr]$, 
a function $g=\sumj \Tj g_j$ such that $\Zr^k g$ is $O(\eps)$ in 
$h^{-4\de_{k+1+N_1-N_0}'}\ti{\mB}_\infty^{0,b'}[L]$ and 
$\Zr^k\OPHX g$ is an $O(\eps)$ element of $h^{-4\delta_{k+1+N_1-N_0}'}\timb{0,b'-\mez}[F]$ so that
\be\label{3425}
\ba
\Ophj(\nongL)w_j&=\sqrt{h}\bigl(w_{2\La,j}+w_{-2\La,j}\bigr)\\
&\quad +h \bigl(w_{3\La,j}+w_{-\La,j}+w_{-3\La,j}\bigr)\\
&\quad +h^{1+\sigma}g_j.
\ea
\ee
Moreover, $w_{\pm 2\La}$ is still given by \e{3423} and 
\be\label{3426}
\ba
w_{3\La}&=\unchi \lambda_3 \la\dom(x)\ra^2 w_\La^3\\
w_{-\La}&=\unchi \lambda_{-1} \la\dom(x)\ra^2 \la w_\La\ra^2\bar{w}_\La\\
w_{-3\La}&=\unchi \lambda_{-3} \la\dom(x)\ra^2 \bar{w}_\La^3
\ea
\ee
for some real constants $\lambda_{3}$, $\lambda_{-1}$, $\lambda_{-3}$. 
\end{lemm}
\begin{proof}
$i)$ By Corollary \ref{ref:A.2.3} of the appendix, we may find symbols $a$ in $S(\langle x\rangle^{-1})$, $c$ in $S(1)$, supported in a domain 
$C_0^{-1}\le \la \xi\ra\le C_0$ and outside a neighborhood of $\La$ such that 
$\nongL=a\# (2x\xi+\la \xi\ra^\mez)+h_j^N c$. Moreover, we may write 
\be\label{3426a}
a=(2x\xi+\la \xi\ra^\mez)^{-1}\nongL+h_j a_1
\ee
for some symbol $a_1$ in $S(\langle x\rangle^{-1})$. We get
\be\label{3427}
\Ophj(\nongL)w_j=\Ophj(a)\ophjdeux w_j +h_j^N\Ophj(c)w_j.
\ee
Since $h_j\le h^\sigma$, taking $N$ large enough, we see that assumption \e{339} implies that 
the last term in \e{3427} may be written as $h^\tdm g_j$ with 
$\lA \Zr_j^k g_j\rA_{L^\infty}=O(h^{-\de_k'} 2^{-j_+ b'})$ for $k\le \sd +N_0$. 
By construction, $g_j$ is microlocally supported in a closed set of the form 
$C_0^{-1}\le |\xi|\le C_0$.

Moreover, since by Lemma~\ref{ref:341} $\Zr^kw_\La$ is in $h^{-\dek}\timb{0,b-2}[K]$, 
we get that $\Zr^k \OPHX w_\La$ belongs to $h^{-\dek}\timb{1,b-2}[K]$. 
Since $\Zr^k \OPHX w_{\La^c}$ is by the same lemma in 
$h^{\mez-\dek}\timb{1,b-2}[F]+h^{1-\dek}\timb{0,b-2}[F]$, we get that 
$\Zr^k\OPHX w$ belongs to $h^{-\dek}\timb{1,b-2}[F]+h^{1-\dek}\timb{0,b-2}[F]$. 
Since $h2^{-\jd}=h_j=O(h^\sigma)=O(1)$, this implies for 
$\blA \Zr_j^k \Op_{h_j}(x\xi)w_j\brA_{L^\infty}$ a bound in 
$$
2^{-\jd}h^{-\dek}\Big[ 2^{\jd-j_+(b-2)}+h2^{-j_+(b-2)}\Big]\le C h^{-\dek-j_+(b-2)}
$$
This implies that 
$\blA \Zr_j^k\Op_{h_j}(x\xi)g_j\brA_{L^\infty}$ is $O\big(h^{-\dek}2^{-j_+b'}\big)$ so that 
$\sumj \Tj g_j$ brings a contribution to the $g$ function in the statement of the lemma.

We use expression \e{3319a} to study the first term in the \rhs of \e{3427}. The contribution 
of 
\be\label{3427a}
\Zr_j^k\Bigl[ h_j\Ophj(a)\Bigl(\frac{i}{2}\widetilde{w}_j-iZw_j\Bigr)\Bigr]
\ee
has according to \e{3419} a bound of the form 
\be\label{3427b}
C_k\eps h^{\mez-\de_{k+1+N_1-N_0}'}\bigl( h+h_j h^\mez\bigr)2^{-j_+b'}
\le C_k \eps
h^{1+\sigma-\de_{k+1+N_1-N_0}'}2^{-j_+b'}
\ee
using $h_j=O(h^\sigma)$. This will give a contribution to $g_j$ in \e{3424} since the action of 
$\Op_{h_j}(x\xi)$ on \e{3427a} admits similar bounds as $a$ is in $S(\langle x\rangle^{-1})$. 

Let us examine the contribution of
\be\label{3428}
\ba
\Zr_j^k \Bigl[ &-\sqrt{h}2^{-j/2}\Ophj (a)\Tmj \Djh \QW
-h2^{-j/2}\Ophj(a)\Tmj \Djh C_{0}(W)\\
&-2^{-\jd}h^{\frac{5}{4}}\Op_{h_j}(a)\Tmj\Djh R(V)
\Bigr]
\ea
\ee
to \e{3427}. 
We use expressions \e{345}. The contributions of $\ti{g}_2$, $\ti{g}_3$ to \e{3428} induce in \e{3424} an expression contributing to 
$hg_j$. Actually, they give terms whose $L^\infty$-norm is $O\big(h^{-2\dek+1}2^{-j_+b'}\big)$. 
Moreover, the action of $\Op_{h_j}(x\xi)$ on these terms admit similar bounds, again because \e{3428} 
contains a $\Op_{h_j}(a)$ operator in factor, with $a$ in $S(\langle x\rangle^{-1})$. In the same way, the 
$L^\infty$-norm of the last term in \e{3428} (and of the action of $\Op_{h_j}(x\xi)$ on it) may be estimated using 
\e{3217b} with $d = b'-\alpha-0$ and the fact that $\blA \Zr^k\japon^{\alpha+d}V\brA_{L^\infty}$ is bounded using \e{339}
(The loss in $2^{j_+(\alpha+0)} = O(h^{-2\beta(\alpha+0)})$ coming from the right hand side of \e{3217b} is absorbed by part of
the $h^{1/4}$-extra factor in the last term in \e{3428}). This brings another contribution to 
$hg_j$. 
Consequently the only contribution to \e{3428} that we are left with is 
\be\label{3428a}
-\Zr_j^k \Bigl[ 2^{-j/2}\sqrt{h}\Ophj(a)\Tmj \Djh \bigl( \widetilde{w}_{2\La}
+\widetilde{w}_{-2\La}\bigr)\Bigr].
\ee
We shall study this expression in part $ii)$ of the proof below. 

$ii)$ We assume that $Q_0+\sqrt{h}C_0$ obeys \e{346a} 
and write again \e{3427}, expressing the \rhs from 
\e{3319a}. The contribution 
\e{3427a} brings, according to \e{3427b}, 
part of the term $h^{1+\sigma}g$ in \e{3425}. The same holds for the remainder term in \e{3319a}. 
We are thus reduced to the study of the quadratic and cubic terms in \e{3428}. By \e{346a}, we have an expression for $\QW+\sqrt{h}\CW$. 
The term $\ti{g}$ in that expansion will bring 
part of the $g_j$ term in \e{3425}. Consequently, we are reduced to studying
\be\label{3428b}
\ba
&\Zr_j^k \Bigl[ \sqrt{h} 2^{-j/2}\Ophj(a)\Tmj \Djh \bigl( \widetilde{w}_{2\La}
+\widetilde{w}_{-2\La}\bigr)\Bigr],\\
&\Zr_j^k \Bigl[ h 2^{-j/2}\Ophj(a)\Tmj \Djh \bigl( \widetilde{w}_{3\La}
+\widetilde{w}_{\La}+\widetilde{w}_{-\La}+\widetilde{w}_{-3\La}\bigr)\Bigr].
\ea
\ee
We notice first that $\widetilde{w}_\La$ is microlocally supported close to $\La$ while $\Ophj(a)$ cut-offs outside a neighborhood 
of that set. Consequently, the $\widetilde{w}_\La$ term in the second formula \e{3428b} gives rise to a remainder. 
For $\ell\in \{-3,-2,-1,2,3\}$ and $j$ in $J(h,C)$, set 
\be\label{3429}
w_{\ell\La,j}^{(1)}=-2^{-j/2}\Ophj(a)\Tmj \Djh \widetilde{w}_{\ell\La}.
\ee
Expressions \e{3428}, \e{3428a}, \e{3428b} show that \e{3425} holds with $w_{\ell\La,j}$ replaced by $w_{\ell\La,j}^{(1)}$. 
We have to show that, up to a modification of $g_j$ in \e{3425}, $w_{\ell\La,j}^{(1)}$ may be replaced by a function 
$w_{\ell\La,j}$ such that $\sumj\Tj w_{\ell\La,j}=w_{\ell\La}$ satisfies the conclusions of the lemma. 

We write $\widetilde{w}_{\ell\La}=\sum_{j'\in J(h,C)}\Theta_{j'}^* \widetilde{w}_{\ell\La,j'}$ and set
$$
\widetilde{w}_{\ell\La,j}^{(1)}\defn \Tmj\Djh \widetilde{w}_{\ell\La}=\sum_{j'\in J(h,C)}\Theta_{-j+j'}^* \Op_{h_{j'}}\bigl(\varphi\bigl(2^{-j+j'}\xi\bigr)\bigr)
\widetilde{w}_{\ell\La,j'}.
$$
Since $\widetilde{w}_{\ell\La,j'}=\Op_{h_{j'}}\bigl(\widetilde{\varphi}(\xi)\bigr)\widetilde{w}_{\ell\La,j'}$ for some $\widetilde{\varphi}$ in 
$C^\infty_0(\xR^*)$, we may limit the sum above to those $j'$ satisfying $|j-j'|\le M$ for some 
$M$. This shows that $\bigl(\Zr_j^k\widetilde{w}_{\ell\La,j}^{(1)}\bigr)_j$ is a bounded family in 
$$
h^{-2\dek}L^\infty J_{\pm 2\La}^{3,b'+\tdm}\bigl[ K_{\pm 2}\bigr]\subset 
h^{-2\dek}L^\infty J_{\pm 2\La}^{2,b'+1}\big[K_{\pm 2}\big]
$$
when $\ell=\pm 2$ and in 
$h^{-3\dek}L^\infty I_{\ell\La}^{3,b'+\tdm}\bigl[ K_{\ell}\bigr]$ if $\ell\in \{-3,-1,3\}$, according to the assumptions made on 
$\widetilde{w}_{\ell\La}$. In the expression \e{3429} of
$$
w_{\ell\La,j}^{(1)}=-2^{-j/2}\Ophj(a)\widetilde{w}_{\ell\La,j}^{(1)},
$$
we insert the decomposition \e{3426a} of $a$. Since $\nongL$ may be assumed to be equal to one close to $\ell\La$, $\ell\neq 1$, if the support of $\gamma_\La$ is close enough to $\La$, we may write
$$
a\arrowvert_{\ell\La}=\bigl(|\ell|^\mez-\ell)^{-1}\la \dom (x)\ra^{-\mez}+h_j a_1\arrowvert_{\ell\La}
$$
for $\ell\in \{-3,-2,-1,2,3\}$. Consequently, \e{3429} may be written as the sum of 
\be\label{3431}
w_{\lL,j}=-2^{-\jd}\big(|\ell|^\mez-\ell\big)^{-1}\la \dom (x)\ra^{-\mez}\Tmj \Djh \tw_{\lL}
\ee
and of
\be\label{3432}
-2^{-\jd}\Ophj\big( c^\ell +h_jd^\ell\big)\tw_{\lL,j}^{(1)}
\ee
where $c^\ell,d^\ell$ are symbols, with $c^\ell$ vanishing on $\ell\cdot\La$. 

Let us show first that \e{3432} multiplied by $\sqrt{h}$ 
when $\ell=\pm 2$ and by $h$ when $\ell$ belongs to $\{-3,-1,3\}$ provides 
a contribution to $h^{1+\sigma}g_j$ in \e{3425}. 
Since $\big(\zjk \tw_{\pm 2\La,j}^{(1)}\big)_j$ is a $O(\eps)$ 
family in $h^{-2\dek}L^\infty J_{\pm 2\La}^{2,b'+1}\big[K_{\pm 2}\big]$ 
and $c^{\pm 2}$ vanishes on $\pm 2\La$, we see that the $L^\infty$-norm 
of the action of $\zjk$ on \e{3432} with $\ell=\pm 2$ is bounded from above by 
$$
C\eps 2^{-\jd}h^{-2\dek}2^{j-j_+(b'+1)}h_j\le 
C\eps h^{1-2\dek }2^{-j_+(b'+1)}.
$$
Consequently, when $\ell=\pm 2$, if we make act $\zjk$ on \e{3432} multiplied by $\sqrt{h}$, we obtain an element of 
$h^{\tdm-2\dek }\mB_\infty^{0,b'}[L]$ i.e.\ a contribution to $h^{1+\sigma}g_j$ in \e{3425}. In the same way, when 
$\ell\in \{-3,-1,3\}$, using that $\big(\zjk \tw_{\lL,j}^{(1)}\big)_j$ is in 
$h^{-3\dek}L^\infty I_{\lL}^{2,b'+1}\big[K_{\ell}\big]$, we may estimate the $L^\infty$-norm of \e{3432} 
on which acts $\zjk$ by 
$$
C\eps 2^{-\jd}h^{-3\dek}2^{j-j_+(b'+1)}\big[ h^\mez+h_j\big]\le 
C\eps h^{\mez-3\dek }2^{-j_+(b'+\mez)}.
$$ 
Again, after multiplication by $h$, this gives a contribution to $h^{1+\sigma}g_j$ in \e{3425}. 
Notice that the fact that $\OPHX g=\sumj 2^{\jd}\Tj \Op_{h_j}(x\xi)g_j$ satisfies the same estimates as $g$, 
with $b'$ replaced by $b'-\mez$, follows from the above bounds since $\widetilde{w}_{\ell\La,j}^{(1)}$ is microlocally supported in a compact set 
of $\cotangent$.

We have thus shown that $\Ophj(\nongL)w_j$ is given by the \rhs of \e{3425}, with $w_{\lL,j}$ given by \e{3431}. In particular 
since $\la \dom\ra$ is positively homogeneous of degree $-2$, we get
$$
w_{\lL}=\sumj \Tj w_{\lL,j}=-\big(|\ell|^\mez-\ell)^{-1}\la \dom(x)\ra^{-\mez}\tw_{\lL}.
$$
Combining this with \e{346}, \e{347a}, we obtain \e{3423} and \e{3426}. Moreover, expressions \e{3431} and the properties of $\tw_{\lL}$ 
obtained in $ii)$ if Lemma~\ref{ref:343} show that $\Zr^k w_{\pm 2\La}$ is in the space $h^{-2\dek}L^\infty \tilde{J}_{\pm 2\La}^{2,b'+\tdm}\big[K_{\pm 2}\big]$ 
and that $w_{\lL}$, $\ell\in \{-3,-1,3\}$ belongs 
to $L^\infty \tilde{I}_{\ell\La}^{2,b'+\tdm}\big[K_{\ell}\big]$, and are $O(\eps)$ in these spaces.
\end{proof}

\begin{proof}[Proof of Proposition~\ref{ref:342}] 
Let $b'$ satisfying the assumption of the proposition. 

By Lemma~\ref{ref:341}, $\Zr^k w_\La$ is an $O(\eps)$ element 
of $h^{-\dek}L^\infty\ti{I}_\La^{0,b-2}[K]$ and $\Zr^k w_{\La^c}$ is an $O(\eps)$ 
element of $h^{\mez-\dek}\ti{\mB}_\infty^{0,b-2}[F]+h^{1-\dek}\ti{\mB}_\infty^{-1,b-2}[F]$, 
for $0\le k\le \sd+N_1$. We may therefore apply $i)$ of Lemma~\ref{ref:343} which shows that \e{345} holds. 
This allows us to use $i)$ of Lemma~\ref{ref:347}. In that way, we obtain functions $w_{\pm 2\La}$, in the spaces indicated in the 
statement of that lemma, such that \e{344} holds. Writing
$$
w=w_\La+\sumj \Tj \Ophj(\nongL)w_j
$$
and using \e{3424}, we obtain equality \e{343}. 

We still have to check that $\Zr^k w_\La$ is an $O(\eps)$ element in $h^{-2\dek}L^\infty\ti{J}_\La^{0,b'}[K]$ 
since Lemma~\ref{ref:341} was only ensuring that this function is 
$O(\eps)$ in the space $h^{-\dek}L^\infty\ti{I}_\La^{0,b'}[K]$. To do so, we must show that
\be\label{3433}
\blA \zjk \ophjdeux w_{\La,j}\brA_{L^\infty}\le C\eps h^{-2\dek}h_j 2^{-j_+b'}.
\ee
We notice that, by symbolic calculus and assumption \e{339}
$$
\blA \zjk \big[ \ophjdeux,\Ophj(\gL)\big]w_j\brA_{L^\infty}
$$
satisfies the wanted bound, since the commutators between the vector fields and $\Ophj(e)$, for a symbol $e$, are of the form 
$\Ophj(\ti{e})$ for another symbol $\ti{e}$. We may therefore study
$$
\blA \zjk \Ophj(\gL)\ophjdeux w_j\brA_{L^\infty}.
$$
Using the commutation relation 
$$
\big[ tD_t+xD_x,\ophjdeux\big]=i\ophjdeux
$$
we see that the above quantity may be estimated from
$$
\blA \Ophj\big(\widetilde{\gamma}_\La\big)\ophjdeux \Zr^{k'}_jw_j\brA_{L^\infty}
$$
for $k'\le k$ and $\widetilde{\gamma}_\La$ a symbol with $\supp \widetilde{\gamma}_\La\subset \supp \gL$. 
We use now \e{3318}, which provides the wanted bound of type \e{3433}, up to a similar estimate for
$$
2^{-\jd}\sqrt{h}\blA \Ophj\big( \widetilde{\gamma}_\La\big)\Zr_j^{k'}\Tmj \Djh \QW\brA_{L^\infty}.
$$
To study this quantity, we need to exploit the structure of $\QW$ given by \e{345}. The remainder 
in the first equation \e{345} gives a contribution bounded by the \rhs of \e{3433}. On the other hand,
$$
\blA \Ophj\big(\widetilde{\gamma}_\La\big)\Zr_j^{k'}\Tmj \Djh \big( \tw_{2\La}+\tw_{-2\La}\big)\brA_{L^\infty}
$$
is $O\big(\eps h^\infty \big)$ 
since $\widetilde{\gamma}_\La$ cuts-off on a neighborhood of $\La$ while 
$\tw_{\pm 2\La}$ are supported close to $\pm 2 \La$, so outside such a neighborhood. 
This concludes the proof 
of \e{3433}, whence the proposition.
\end{proof}

Let us deduce from expansion \e{343} of $w$ a second refined decomposition of $w$ in oscillating factors.

\begin{coro}\label{ref:348}
Under the assumptions of Proposition~\ref{ref:342} with moreover $b'<b-8$, 
we may write
\be\label{3435}
w=w_\La+\sqrt{h}\big(w_{2\La}+w_{-2\La}\big)+h\big( w_{3\La}+w_{-\La}+w_{-3\La}\big)+h^{1+\sigma}g
\ee
where $w_{\pm 2\La}$, $w_{\pm 3\La}$, $w_{-\La}$ satisfy the conclusions of 
$ii)$ of Lemma~\ref{ref:347} and are given in terms of 
$w_\La$ by \e{3423}, \e{3426}, and where for 
$k\le \sd +N_0$, $\Zr^k g$ is $O(\eps)$ in $h^{-4\delta_{k+1+N_1-N_0}'}\ti{\mB}_\infty^{0,b'}[F]$ 
and $\Zr^k\OPHX g$ is $O(\eps)$ in $h^{-4\delta_{k+1+N_1-N_0}'}\timb{0,b'-\mez}[F]$. 
\end{coro}
\begin{proof}
By Proposition~\ref{ref:342}, $w$ may be written as \e{343}. Consequently, the assumptions of 
$ii)$ of Lemma~\ref{ref:343} hold and this lemma implies a decomposition \e{346a} for
$$
\sumj\Djh\big[\QW+\sqrt{h}\CW\big].
$$ 
This shows that the assumptions of $ii)$ of Lemma~\ref{ref:347} hold. According to this lemma, 
$\Ophj\big(\nongL\big)w_j$ is given by \e{3425}. We define $w_{\lL}=\sumj \Tj w_{\lL,j}$ and get 
\e{3435}, remembering that we defined $w_\La=\sumj \Tj \Ophj(\gL)w_j$ if $w=\sumj \Tj w_j$. The expansion 
in terms of $w_\La,\overline{w}_\La$ follow from \e{3423}, \e{3426}.
\end{proof}

We have seen in Proposition~\ref{ref:342}, that $\Zr^kw_\La$ is an $O(\eps)$ element in 
$h^{-\dek}L^\infty \ti{J}_\La^{0,b'}[K]$. We need a more precise description of this quantity. 

\begin{prop}\label{ref:349}
Let $w_\La=\sumj \Tj w_{\La,j}$ be the function introduced in Proposition~\ref{ref:342}. There are elements 
$f=\sumj \Tj f_j$ where $\Zr^k f$ is $O(\eps)$ in $h^{-3\delta_{k+2}'}L^\infty \ti{I}_\La^{0,b'}[K]$ for 
$k\le \sd +N_0$ and $r=\sumj \Tj r_j$, with $\Zr^k r$ of size $O(\eps)$ in 
$h^{-4\delta_{k+N_1-N_0+1}'}\ti{\mB}_\infty^{0,b'}[F]$ such that
\be\label{3436}
\ophjdeux w_{\La,j}=h_j \big[ f_j+h^\uq r_j\big].
\ee
\end{prop}
\begin{proof}
We use the definition of $w_{j,\La}=\Ophj(\gL)w_j$ and
\e{3319a} to write
\be\label{3436a}
\ba
\ophjdeux w_{j,\La}&=\big[ \ophjdeux , \Ophj(\gL)\big] w_j\\
&\quad -\sqrt{h} 2^{-\jd}\Ophj(\gL)\Tmj \Djh \QW\\
&\quad +h2^{-\jd}\Ophj(\gL)\Big[ \frac{i}{2}\tw_j-iZ w_j-\Tmj \Djh \CW\Big]\\
&\quad -2^{-\jd}h^{\frac 5 4}\Ophj(\gL)\Tmj \Djh R(V).
\ea
\ee
The commutator term may be written $h_j\Ophj(e)w_j$ for some symbol 
$e$ in $S(1)$, with support contained in $\supp \gL$ (up to 
a $O\big(h_j^\infty\big)=O(h^\infty)$ remainder). Consequently $\Ophj(e)w_j$ will 
satisfy the same type of properties as $w_{\La,j}$ i.e.\ by Lemma~\ref{ref:341}, 
$\big(\zjk \Ophj(e)w_j\big)_j$ will be a $O(\eps)$ family in $h^{-\dek}L^\infty I_\La^{0,b'}[K]$ so that the first term 
in the \rhs of \e{3436a} contributes to $h_j f_j$ in \e{3436}. 

To study the quadratic and cubic terms in \e{3436a}, we use expression \e{346a} for $\QW+\sqrt{h}\CW$. The 
remainder $\ti{g}$ in \e{346a} will bring a contribution to $h_j h^\mez r_j$ in \e{3436}. The contribution
$$
\sqrt{h}2^{-\jd}\Ophj(\gL)\Tmj \Djh \Big( \tw_{2\La}+\tw_{-2\La}+\sqrt{h}\big( \tw_{3\La}
+\tw_{-\La}+\tw_{-3\La}\big)\Big)
$$
and its $\Zr$-derivatives are $O(\eps h^\infty)$ since $\gL$ cuts-off close to $\La$, while the terms on which it acts are 
supported close to $\ell\cdot\La$, $|\ell|\le 3$, $\ell\neq 1$. Consequently, the only remaining term coming from \e{346a} is
$$
-h2^{-\jd}\Ophj(\gL)\Tmj \Djh \tw_\La.
$$
Since
$$
\Zr^k\tw_\La=\sum_{j'\in J(h,C)}\Zr^k \Theta_{j'}^* \tw_{\La,j'} \quad\text{ is }O(\eps) \text{ in } 
h^{-3\dek}L^\infty \ti{I}_\La^{0,b'}[K]
$$
we get a contribution to $h_j f_j$ in \e{3436}. 

The remainder term $2^{-\jd}h^{\frac 5 4} \Ophj(\gL)\Tmj \Djh R(V)$ will contribute to the last term 
in \e{3426} 
as it has been seen in the estimate of the last term in \e{3428}. 

Finally, we are left with studying 
\be\label{3437}
h2^{-\jd}\Ophj(\gL)\Big[ \frac{i}{2}\tw_j-iZ w_j\Big]
\ee
We use $iii)$ of Lemma~\ref{ref:333} to bound the action of $\zjk \ophjdeux$ on \e{3437}. We obtain an $L^\infty$ bound 
of the form
\begin{align*}
Ch_j 2^{-j_+ b'}\bigg[ &h^{\mez}\sum_{k_1+k_2\le k+1}\Er_{k_1}(v)\Er_{k_2}(v)
+h^{\mez}\sum_{k_1+k_2+k_3\le k+1}\Er_{k_1}(v)\Er_{k_2}(v)\Er_{k_3}(v)+h_j \Er_{k+2}(v)\bigg].
\end{align*}
Using \e{330} we bound this by $C h_j \big(h^\mez+h_j\big)\Er_{k+2}(v)$ so, according 
to \e{339}, by
$$
Ch_j h^{-\delta_{k+2}'}\big( h^\mez +h_j\big)2^{-j_+b'}.
$$
We obtain in that way a contribution to 
$h_j f_j$ in \e{3436}.
\end{proof}

In the following section, we shall need estimates not only for $w$, but also for
$$
w^{(\ell)}=\Oph\big(\langle\xi\rangle^\ell \big)w,\quad 
w_\La^{(\ell)}=\Oph\big(\langle\xi\rangle^\ell \big)w_\La,
$$
where $\ell$ is an integer $0\le \ell \le \sd+N_0+b$. Let us deduce from 
Corollary \ref{ref:348} an expansion for $w^{(\ell)}$ in terms of $w_\La^{(\ell)}$, $\ell=0,\ldots,\sd+N_0+b'$.

\begin{coro}\label{ref:3410}
Under the assumptions of Proposition~$\ref{ref:342}$, for $\ell=0,\ldots,\sd+N_0+b'$ 
we may write 
\be\label{3441}
\wl=\wl_\La+\sqrt{h}\big(\wl_{2\La}+\wl_{-2\La}\big)+h\big( \wl_{3\La}+\wl_{-\La}+\wl_{-3\La}\big)+h^{1+\sigma}g^{(\ell)}
\ee
where for any $k\le \min\big( \sd +N_0,\sd+N_0+b'-\ell\big)$, $\Zr^k g^{(\ell)}$ is in 
$h^{-4\delta_{k+1+N_1-N_0}'}\ti{\mB}_\infty^{0,b'-\ell}[F]$, $\Zr^k\OPHX g^{(\ell)}$ is in $h^{-4\delta_{k+1+N_1-N_0}'}\timb{0,b'-\ell-\mez}[F]$ 
and of size $O(\eps)$ in that space and 
\be\label{3442}
\ba
\wl_{2\La}&= -i\unchi \langle 2\dom \rangle^{\ell}\langle \dom \rangle^{-2\ell}\la \dom \ra \frac{1+\sqrt{2}}{4}\big( \wl_\La \big)^2\\
\wl_{-2\La}&= -i\unchi \langle 2\dom \rangle^{\ell}\langle \dom \rangle^{-2\ell}\la \dom \ra \frac{1-\sqrt{2}}{4}\big( \owl_\La \big)^2\\
\wl_{3\La}&=\phantom{-i} \unchi \langle 3\dom \rangle^{\ell}\langle \dom \rangle^{-3\ell}\la \dom \ra^2 \lambda_3^\ell\big( \wl_\La \big)^3\\
\wl_{-\La}&=\phantom{-i} \unchi \langle \dom \rangle^{-2\ell}\la \dom \ra^2 \lambda_{-1}^\ell \big|\wl_\La\big|^2 \owl_\La\\
\wl_{-3\La}&= \phantom{-i}\unchi \langle 3\dom \rangle^{\ell}\langle \dom \rangle^{-3\ell}\la \dom \ra ^2 \lambda_{-3}^\ell\big( \owl_\La \big)^3
\ea
\ee
where $\lambda_{3}^\ell,\lambda_{-1}^\ell,\lambda_{-3}^\ell$ are real constants, $\chi \in C^\infty_0(\xR)$, $\chi\equiv 1$ close to zero, 
with small enough support. 

Finally, in the decomposition
$$
\wl_\La=\sumj \Tj \wl_{\La,j}
$$
of $\wl_\La$ deduced from the one of $w_\La$, we may write
\be\label{3443}
\ophjdeux \wl_{\La,j}=h_j\big( f^{(\ell)}_j+h^{\uq}r^{(\ell)}_j\big)
\ee
where
\begin{alignat*}{3}
&\big(\zjk f^{(\ell)}_j\big)_j &&\text{ is a $O(\eps)$ family in } &&h^{-3\delta_{k+2}'}L^\infty I_\La^{0,b'-\ell}[K]\\
&\big(\zjk r^{(\ell)}_j\big)_j &&\text{ is a $O(\eps)$ family in } &&h^{-4\delta_{k+1+N_1-N_0}'}\ti{\mB}_\infty^{0,b'-\ell}[L].
\end{alignat*}
\end{coro}
\begin{proof}
By definition, $w_\La=\Oph\big(\langle \xi\rangle^{-\ell}\big)\wl_\La$ and $\Zr^k \wl_\La$ belongs to 
$h^{-\dek}L^\infty \ti{J}_\La^{0,b'-\ell}[K]$. By Proposition~\ref{ref:3.1.9}, we may write 
$$
w_\La=\unchi \langle \dom \rangle^{-\ell}\wl_\La+h r_1
$$
where $\Zr^k r_1$ is in $h^{-\delta_{k+1}'}\ti{\mB}_\infty^{-1,b'}[K]$ and the action of $\Zr^k$ on 
$\unchi \langle \dom \rangle^{-\ell}\wl_\La$ is in $h^{-\dek}L^\infty \ti{J}_\La^{0,b'}[K]$. We apply Proposition~\ref{ref:3110} 
to compute powers of $w_\La$
\be\label{3444}
\ba
\big(w_{\La}\big)^2&= \unchi^2 \langle \dom \rangle^{-2\ell} \big(\wl_\La\big)^2+hr_2\\
\big(\overline{w}_{\La}\big)^2&= \unchi^2 \langle \dom \rangle^{-2\ell} \big(\owl_\La\big)^2+hr_{-2}\\
\big(w_{\La}\big)^3&= \unchi^3 \langle \dom \rangle^{-3\ell} \big(\wl_\La\big)^3+hr_3\\
\big\vert w_\La\big\vert^2 w_{\La}^2&= \unchi^3 \langle \dom \rangle^{-3\ell} \big\vert \wl_\La\big\vert^2 \wl_\La+hr_1\\
\big\vert w_\La\big\vert^2 \overline{w}_{\La}^2&= \unchi^3 \langle \dom \rangle^{-3\ell} \big\vert \wl_\La\big\vert^2 \owl_\La+hr_{-1}\\
\big(\overline{w}_{\La}\big)^3&= \unchi^3 \langle \dom \rangle^{-3\ell} \big(\owl_\La\big)^3+hr_{-3}
\ea
\ee
where the action of $\Zr^k$ on $r_q$ gives an element of $h^{-2\dek}\ti{\mB}_\infty^{-2,2b'-\mez}\big[K_q\big]$ 
if $q = \pm 2$ and of $h^{-3\dek}\ti{\mB}_\infty^{-3,3b'-1}\big[ K_q\big]$ if $q=3,1,-1,-3$.

On the other hand, consider the contributions $w_{q\La}$, $|q|\le 3$, 
$q\neq 1$, to the expansion \e{3435} and define $w^{(\ell),1}_{q\La}=\Oph\big(\langle \xi\rangle^{\ell}\big)w_{q\La}$ 
so that \e{3435} may be written
\be\label{3445}
\wl=\wl_\La+\sqrt{h}\big(w^{(\ell),1}_{2\La}+w^{(\ell),1}_{-2\La}\big)
+h\big( w^{(\ell),1}_{3\La}+w^{(\ell),1}_{-\La}+w^{(\ell),1}_{-3\La}\big)+h^{1+\sigma}g^{(\ell)}
\ee
where $g^{(\ell)}$ satisfies the bounds of the remainder in \e{3441}. We apply again Proposition~\ref{ref:3.1.9} to get an expansion 
of $w^{(\ell),1}_{q\La}$. Since by $ii)$ of Lemma~\ref{ref:347}, 
\begin{alignat*}{3}
&\Zr^k w_{\pm 2\La} &&\text{ is in } &&h^{-2\dek}L^\infty \ti{J}_\La^{2,b'+\tdm}\big[ K_{\pm 2}\big],\\ 
&\Zr^k w_{q\La} &&\text{ is in } &&h^{-3\dek}L^\infty \ti{I}_\La^{2,b'+\tdm}\big[ K_{q}\big], \text{ for }q=-3,-1,3,
\end{alignat*}
we obtain that 
\be\label{3446}
\ba
w^{(\ell),1}_{\pm 2\La}&=\unchi \langle 2\dom \rangle^\ell w_{\pm 2\La}+h r^{(\ell),1}_{\pm 2}\\
w^{(\ell),1}_{q\La}&=\unchi \langle q\dom \rangle^\ell w_{q\La}+h^\mez r^{(\ell),1}_{q},\quad q=-3,-1,3,
\ea
\ee
where $\Zr^k r^{(\ell),1}_{\pm 2}$ is in $h^{-2\dek}\ti{\mB}_\infty^{1,b'+\tdm-\ell}\big[ K_2\big]$ and for $q=-3,-1,3$, 
$$
\Zr^k r^{(\ell),1}_{q}\text{ is in }h^{-3\dek}\ti{\mB}_\infty^{1,b'+1-\ell}\big[ K_q\big].
$$ 
We deduce from \e{3445} that 
\be\label{3447}
\ba
\wl&=\wl_\La+\sqrt{h}\langle 2\dom \rangle^\ell \unchi \big(w_{2\La}+w_{-2\La}\big)\\
&\quad + h\unchi \big(\langle 3\dom \rangle^\ell w_{3\La} +\langle \dom \rangle^\ell w_{-\La}
+\langle 3\dom \rangle^\ell w_{-3\La}\big)\\
&\quad +h^{1+\sigma}g^{(\ell)}
\ea
\ee
with a new remainder $g^{(\ell)}$ as in \e{3441}. We use next \e{3423}, \e{3426} to express $w_{q\La}$ from $w_\La$, $\overline{w}_\La$ 
and \e{3444} to compute the resulting quantities from $\wl_\La$, $\owl_\La$. We get expressions 
\e{3442}, with $(1-\chi)$ replaced eventually by some of its powers. As already seen, these powers may be replaced by $(1-\chi)$, 
up to $O(h^\infty)$ remainders. 

The remainders coming from the ones in \e{3444} may be expressed as the product of $h^\tdm$ (resp.\ $h^2$) with 
$\unchi \la \dom \ra\langle 2 \dom \rangle^\ell r_{\pm 2}$ 
(resp.\ $\unchi \la \dom \ra^2 \langle q \dom \rangle^\ell r_{\ell}$, $\ell=-3,-1,3$). By Proposition~\ref{ref:3.1.9}, and since 
$\la \dom \ra\langle 2 \dom \rangle^\ell$ (resp.\ $\la \dom \ra^2 \langle q \dom \rangle^\ell$) satisfies \e{3115} with $(\ell,\ell',d,d')$ replaced by 
$(-2\ell-2,2\ell,0,0)$ (resp.\ $(-2\ell-4,2\ell,0,0)$), we obtain that the action of $\Zr^k$ on 
these functions gives elements belonging to \[h^{-2\dek}\ti{\mB}_\infty^{0,2b'-\frac{1}{2}-\ell}[K_\ell]\subset h^{-2\dek}\ti{\mB}_\infty^{0,b'-\ell}[K_\ell]\] (resp.\ 
$h^{-3\dek}\ti{\mB}_\infty^{1,3b'-1-\ell}[K_\ell]\subset h^{-3\dek}\ti{\mB}_\infty^{0,b'-\ell}[K_\ell]$) for $\ell\in \{-3,\ldots,3\}$ 
so that we obtain again a contribution 
to $g^\ell$. (Notice that the action of $\OPHX$ on these remainders give elements of the same spaces with $b$ replaced by 
$b-1/2$, since they are microlocally supported in a compact subset of $\cotangent$). This concludes the proof of \e{3441}. 

To prove \e{3443}, we first write, according to the definition of $\wl_\La$ and \e{3113}, that $\wl_{\La,j}=\Ophj\big(\bigl\langle 2^j \xi\bigr\rangle^\ell\big)w_{\La,j}$. 
Making act $\Ophj\big(\bigl\langle 2^j \xi\bigr\rangle^\ell\big)$ on \e{3436}, we get for the left hand side of \e{3443} an expression given by its right hand side, modulo a term
$$
\Bigl[ \ophjdeux , \Ophj\big(\bigl\langle 2^j \xi\bigr\rangle^\ell\big)\Bigr]w_{\La,j}. 
$$
Since $\big(\Zr^k w_{\La,j}\big)_j$ is a bounded family in $h^{-3\dek}L^\infty J_\La^{0,b'}[K]$, we see using 
symbolic calculus, that this expression contributes to the $h_j f_j^{(\ell)}$ term in \e{3443}.
\end{proof}

\section{\texorpdfstring{Ordinary differential equation for $w_\La$}{Ordinary differential equation}}\label{S:35}

We consider a solution $v$ of \e{326}, satisfying for $h$ in some interval $]h',1]$ the 
a priori estimate \e{337} for $k'\le k+1$, with $k\le s-a-1$. By Proposition~\ref{ref:3.3.1}, we know that $v$ 
satisfies then \e{339}, and by Corollary \ref{ref:348}, that $v=v_L+w+v_H$, where $w$ has an expansion \e{3435}. Our goal here is to deduce from that and from 
the equation satisfied by $w$, a uniform estimate for $\blA \opjl w(t,\cdot)\bri$, and estimates 
for $\blA \Zr^k \opjl w(t,\cdot)\bri$ which are not uniform, but which are better than \e{339} (i.e.\ that involve exponents closer to zero than the 
$\delta'_{k'}$). 

For $0\le \ell \le \sd+N_0$ and $0\le k\le \sd+N_0-\ell$ we define
\be\label{350}
\Wkl=\bigl( Z^{k'}\wl\big)_{0\le k'\le k}
\ee
where
$$
\wl=\opjl w,
$$
as in the preceding section. The estimates we are looking for will follow from an ordinary 
differential equation satisfied by $\Wkl$. 

\begin{prop}\label{ref:351}
Under the preceding assumptions, the function $\wl$ satisfies the equation
\be\label{351}
\ba
D_t \wl &=\mez \unchi \la \dom(x)\ra^\mez \wl\\
&\quad -i \frac{\sqrt{h}}{8}\unchi \la \dom\ra^{\tdm}\angledom^{-2\ell}\deuxangledom^{\ell}
\Big[ (1+\sqrt{2})\big(\wl\big)^2-3(1-\sqrt{2})\big(\owl\big)^2\Big]\\
&\quad +h\Big[ \Phi_3^{(\ell)}\big(\wl\big)^3+\Phi_1^{(\ell)}| \wl |^2 \wl +\Phi_{-1}^{(\ell)}|\wl|^2\owl
+\Phi_{-3}^{(\ell)}\big(\owl\big)^3\Big]\\
&\quad +h^{1+\kappa}\rl(t,x)
\ea
\ee
where $\chi$ is in $C^\infty_0(\xR)$, equal to one close to zero, with small enough support, where $\kappa$ is a small 
positive number, where 
$\Phi_j^{(\ell)}$, 
$-3\le j\le 3$ are given by
\be\label{353}
\ba
\Phi_1^{(\ell)}(x)&=\unchi |\dom|^{\frac 5 2}\angledom^{-2\ell}\left[ 
\frac{\deuxangledom^{2\ell}}{\angledom^{2\ell}}\frac{3(3-2\sqrt{2})}{16}+\mez\right]\\
\Phi_j^{(\ell)}(x)&=\unchi |\dom|^{\frac 5 2}\Gamma_j^{(\ell)}(\dom)\quad \ell\neq 1
\ea
\ee
for some real valued symbols of order $-2\ell$, $\Gamma_j^{(\ell)}$, 
and where $\blA (hD_x)^p \Zr^k \rl (t,x)\bri$ is $O(\eps)$ for any integers $k,p,\ell$ with 
$k\le \sd+N_0-\ell$, $0\le p\le b'-1$. Moreover, $\Wkl$ defined by \e{350} satisfies a system of the form
\be\label{353a}
\ba
D_t\Wkl &=\mez \unchi |\dom(x)|^\mez \Wkl \\
&\quad +\sh Q^{k,(\ell)}\big[ x,h;\Wkl,\oWkl\big]\\
&\quad +h\Cr^{\kl}\big[ x,h;\Wkl,\oWkl\big]
+h^{1+\kappa}R^\kl(t,x)
\ea
\ee
$\bullet$ where 
$\blA (hD_x)^pZ^{k'}R^\kl (t,\cdot)\bri=O(\eps)$ for $0\le p\le b''\le b'-2/\beta$ and $k'\le \sd +N_0-\ell-k$;

$\bullet$ where 
$Q^\kl$ is a vector valued quadratic map in $(\Wkl,\oWkl)$ whose components are linear combination of functions of the form
\be\label{353b}
\ba
&\theta\big(xh^{-\beta}\big)\Phi(x)\big(Z^{k_1}\wl\big)\big(Z^{k_2}\wl\big)\\
&\theta\big(xh^{-\beta}\big)\Phi(x)\big(Z^{k_1}\wl\big)\big(Z^{k_2}\owl\big)\\
&\theta\big(xh^{-\beta}\big)\Phi(x)\big(Z^{k_1}\owl\big)\big(Z^{k_2}\owl\big)
\ea
\ee
for $k_1+k_2\le k$, with smooth functions $\theta$ bounded as well as their derivatives, $\theta\equiv 0$ close to zero and 
$\Phi$ satisfying $\big\vert Z^k \Phi(x)\big\vert\le C|x|^{2\ell-3}\langle x\rangle^{-2\ell}$ for any $k$;

$\bullet$ where $\Cr^\kl$ is a vector valued cubic map, whose components are linear combination of quantities
\be\label{353c}
\ba
&\theta\big(xh^{-\beta}\big)\Phi(x)\big(Z^{k_1}\wl\big)\big(Z^{k_2}\wl\big)\big(Z^{k_3}\wl\big)\\
&\theta\big(xh^{-\beta}\big)\Phi(x)\big(Z^{k_1}\wl\big)\big(Z^{k_2}\wl\big)\big(Z^{k_3}\owl\big)\\
&\theta\big(xh^{-\beta}\big)\Phi(x)\big(Z^{k_1}\wl\big)\big(Z^{k_2}\owl\big)\big(Z^{k_3}\owl\big)\\
&\theta\big(xh^{-\beta}\big)\Phi(x)\big(Z^{k_1}\owl\big)\big(Z^{k_2}\owl\big)\big(Z^{k_3}\owl\big)
\ea
\ee
for $k_1+k_2+k_3\le k$, $\big\vert Z^k \Phi(x)\big\vert\le C|x|^{4\ell-5}\langle x\rangle^{-4\ell}$ for any $k$.
\end{prop}

We shall prove first \e{351}, deducing it from \e{3222} on which we make act $\opjl$. Let us study first the action of this operator 
on the nonlinearity. 

As in the preceding section, we shall call $K$ or $K_\ell$ compact subsets of $\cotangent$ 
contained in a small neighborhood of $\ell\La$ for 
$\ell\in \{\pm 3,\pm 2,\pm 1\}$, by $L$ compact subsets of $\cotangent$ and by $F$ closed subsets of $T^*\xR$ 
whose second projection is compact in $\Rs$.

\begin{lemm}\label{ref:351a}
Under the assumptions of the proposition, we may write for $\ell\le \sd+N_0$, 
\be\label{354}
\ba
&\sumj\Djh \Big[\opjl \big[ \sh \QW+h\CW\big]\Big]
\\
&\qquad \qquad=-i\sh \unchi |\dom|^\tdm \frac{\sqrt{2}}{4}\deuxangledom^\ell \angledom^{-2\ell} \Big[ \big(\wl_\La\big)^2+\big(\owl_\La\big)^2\Big]\\
&\qquad \qquad \quad +\frac{h}{2}\unchi |\dom|^{\frac{5}{2}}\angledom^{-3\ell}
\troisangledom^\ell\Big[\lambda_3^{(\ell)}\big(\wl_\La\bigr)^3+\lambda_{-3}^{(\ell)}\big(\owl_\La\bigr)^3\Big]\\
&\qquad \qquad \quad +\frac{h}{2}\unchi |\dom|^{\frac{5}{2}}\angledom^{-2\ell} \Big[\big\vert\wlla\big\vert^2\wlla+\lambda_{-1}^{(\ell)}\big\vert\wlla\big\vert^2\owlla\Big]\\
&\qquad \qquad\quad +h^\tdm \rl
\ea
\ee
where for $k\le \sd+N_0-\ell$ $Z^k\rl $ belongs to $h^{-3\delta_{k+1+\ell}'}\ti{\mB}_\infty^{0,b'}[F]$ and 
$\lambda_{\pm 3}^{(\ell)}$, $\lambda_{-1}^{(\ell)}$ are real constants.
\end{lemm}
\begin{proof}
We apply $\opjl$ to \e{346a} and write the resulting right hand side as in \e{354}. By $ii)$ if Lemma~\ref{ref:343}, we know 
that $\Zr^{k+\ell}\ti{w}_{\pm 2 \La}$ is $O(\eps)$ in $h^{-2\delta_{k+\ell+1}'}L^\infty \ti{J}_{\pm 2\La}^{3,b'+\tdm}\big[K_{\pm 2}\big]$, 
so $Z^k\ti{w}_{\pm 2\La}$ is $O(\eps)$ in 
$h^{-2\delta_{k+\ell+1}'}L^\infty \ti{J}_{\pm 2\La}^{3,b'+\tdm+\ell}\big[K_{\pm 2}\big]$. 

In the same way $Z^k\ti{w}_{q\La}$ is $O(\eps)$ in $h^{-3\delta_{k+\ell+1}'}L^\infty \ti{I}_{q\La}^{3,b'+\tdm+\ell}\big[K_q\big]$, for 
$q\in \{\pm 1,\pm 3\}$. Consequently, Proposition~\ref{ref:3.1.9} shows that
\begin{align*}
\opjl \ti{w}_{\pm 2 \La}&=\unchi \deuxangledom^\ell \ti{w}_{\pm 2 \La}+h\ti{r}^{(\ell)}_{\pm 2},\\
\opjl \ti{w}_{q \La}&=\unchi \langle q\dom \rangle^{\ell} \ti{w}_{q \La}+h^\mez\ti{r}^{(\ell)}_{q},
\end{align*}
where $Z^k \ti{r}^{\el}_{\pm 2}$ (resp.\ $Z^k \ti{r}^{\el}_q$) is $O(\eps)$ in 
$h^{-2\delta'_{k+\ell+1}}\ti{\mB}_\infty^{2,b'+\tdm}\big[K_{\pm 2}\big]$ (resp.\ $h^{-3\delta_{k+1+\ell}}\ti{\mB}_\infty^{2,b'+1}\big[K_{q}\big]$). We 
combine this with the expressions \e{346}, \e{347}, \e{347a} of $\ti{w}_{q\La}$ in terms of $w_\La$, $w_{\pm 2\La}$, and with 
the formulas \e{3423} expressing $w_{\pm 2\La}$ in terms of $w_\La$. If moreover we compute the powers of $w_\La$, $\overline{w}_\La$ 
from $\wlla$, $\owlla$ using \e{3444}, we get 
\begin{align*}
\opjl \ti{w}_{2\La}&=-i\unchi  |\dom|^\tdm \frac{\sqrt{2}}{4}\deuxangledom^\ell\angledom^{-2\ell}\big(\wlla\big)^2+h \rl_{2},\\
\opjl \ti{w}_{-2\La}&=-i\unchi  |\dom|^\tdm \frac{\sqrt{2}}{4}\deuxangledom^\ell\angledom^{-2\ell} \big(\owlla\big)^2+h \rl_{-2},
\end{align*}
with remainders $\rl_{\pm 2}$ satisfying again, because of Proposition \ref{ref:3.1.9}, that 
$Z^k\rl_{\pm 2}$ is $O(\eps)$ in $h^{-2\delta_{k+1+\ell}'}\ti{\mB}_\infty^{0,b'} \big[ K_{\pm 2}\big]$. 
In the same way
\begin{align*}
\opjl \ti{w}_{q\La}&=\unchi  |\dom|^{\frac 5 2} \lambda_q^{(\ell)}\langle q\dom\rangle^\ell\angledom^{-3\ell} P_q \big(\wlla,\owlla\big)+h^{\mez} \rl_{q},
\end{align*}
where $P_q \big(\wlla,\owlla\big)$ is equal to $\big(\wlla\big)^3$ (resp.\ $\big\vert \wlla\big\vert^2\wlla$, 
resp.\ $\big\vert \wlla\big\vert^2\owlla$, resp.\ $\big(\owlla\big)^3$) if $q=3$ (resp.\ $q=1$, resp.\ $q=-1$, resp.\ $q=-3$), 
where $\lambda_q^{(\ell)}$ are real constants with $\lambda_1^{(\ell)}=\mez$, and where 
$Z^k\rl_q$ is $O(\eps)$ in 
$h^{-3\delta_{k+1+\ell}'}\ti{\mB}_\infty^{0,b'}\big[K_q\big]$. (We used again remark~\ref{ref:Remark2} to replace different powers 
of $\unchi$ by $1$.)

This concludes the proof of the lemma.
\end{proof}

Let us study next the action of $\opjl$ on the linear term $\opx w$ of \e{3222}, writing $w=\Oph\big(\lxi^{-\ell}\big) \wl$. 

\begin{lemm}\label{ref:352}
One may write, for $\ell\le \sd+N_0$
\be\label{357}
\ba
&\opjl \opx \Oph\big(\langle\xi\rangle^{-\ell}\big) \wl+i\ell h\Oph\Big(\frac{\xi^2}{\langle\xi\rangle^2}\Big)\wl\\
&\qquad \qquad =\unchi\bigg[ \Big(\mez |\dom|^\mez+\frac{i}{2} h\Big)\wl_\La \\
&\qquad \qquad \phantom{=\unchi\Big[ }-i\frac{\sh}{4} |\dom|^\tdm \deuxangledom^{\ell}
\angledom^{-2\ell}\Big[\big(\wlla\big)^2-\big(\owlla\big)^2\Big]\\
&\qquad \qquad \phantom{=\unchi\Big[ }+h|\dom |^{\frac{5}{2}}\angledom^{-3\ell}\Big[ \troisangledom^\ell \Big(\mu_3^\el \big(\wlla\big)^3
+\mu_{-3}^\el \big(\owlla\big)^3\Big)\\
&\qquad \qquad \phantom{=\unchi\Big[ +h|\dom |^{\frac{5}{2}}\Big[ }+\angledom^\ell \mu_{-1}^\el \big\vert \wlla\big\vert^2\owlla\Big]\bigg]\\
&\qquad\qquad\quad +h^{1+\sigma}r^\el
\ea
\ee
for some real constants $\mu_3^\el$, $\mu_{-1}^\el$, $\mu_{-3}^\el$ and where 
for $k\le \sd+N_0-\ell$, $Z^kr^\el$ belongs to $h^{-4\delta_{k+\ell+1+N_1-N_0}'}\ti{\mB}_\infty^{0,b'-\mez}[F]$. 
\end{lemm}

The proof of the above lemma will use

\begin{lemm}\label{ref:352a}
We may write for $\ell\le \sd+N_0$
\be\label{357a}
\opx \wlla =\mez\unchi \left[ |\dom|^\mez \wlla+ih \wlla\right] +h^{1+\sigma}r^\el_1
\ee
where for $k\le \sd+N_0-\ell$, $Z^k r^\el_1$ is in $h^{-4\delta_{k+2+\ell}'}\ti{\mB}_\infty^{0,b'}[L]$.
\end{lemm}
\begin{proof}
We write remembering \e{3113}
\be\label{357b}
\opx \wlla =\sumj 2^{\jd}\Tj \ophjx \wl_{\La,j}.
\ee
Let us show that we may write
\be\label{3510}
\ba
\ophjx &=\mez |\dom(x)|^\mez +\Ophj(e_1)\Big( \ophjdeux\Big)^2\\
&\quad +ih_j \Ophj(e_2)\ophjdeux\\
&\quad -ih_j \Ophj(e_1)\Ophj\big(\la\xi\ra^\mez\big)+h_j^2\Ophj(e_3)
\ea
\ee
where $e_j$ are symbols in $S(1,K)$, for some large enough compact subset $K$ of $\cotangent$, satisfying
\be\label{3511}
e_1\arrowvert_\La=-\mez |\dom(x)|^{-\mez}.
\ee
Denote $a(x,\xi)=x\xi+|\xi|^\mez$ and take
$$
e_1(x,\xi)=\frac{a(x,\xi)-a(x,\dom)}{\big(2x\xi+\la\xi\ra^\mez\big)^2}\cdot
$$
A direct computation shows that the numerator vanishes at second order on $\La$, so that the quotient is smooth, 
and that its restriction to $\La$ is given by \e{3511}. If we set $e(x,\xi)=2x\xi+\la\xi\ra^\mez$, we obtain by symbolic calculus
$$
e\# e=e^2-ih_j \partial_\xi e \px e+h_j^2 \tilde{e}
$$
for some symbol $\tilde{e}$, so that by an immediate computation
\be\label{3512}
e^2=e\# e+2ih_j e-ih_j\la\xi\ra^\mez-h_j^2\tilde{e}.
\ee
On the other hand, by symbolic calculus $e_1 e^2=e_1\# e^2+h_j e_2'\# e+h_j^2\tilde{e}'$ for some symbols 
$e_2',\tilde{e}'$ in $S(1,K)$, so that taking \e{3512} into account
$$
e_1 e^2=e_1\# e\# e-ih_j   e_1\# \la\xi\ra^\mez+ih_je_2\# e+h_j^2e_3
$$
for new symbols $e_2,e_3$. Since $e_1e^2=a(x,\xi)-a(x,\dom)=x\xi+\la \xi\ra^\mez-\mez \la \dom(x)\ra^\mez$, we obtain \e{3510} by quantification. 

Let us use \e{3510} to show that \e{357a}Ê
holds. Actually, the contribution of the first term in the \rhs of \e{3510} to \e{357a} gives the $\la\dom\ra^\mez$ term 
in the \rhs of \e{357a} (Again, we may insert a cut-off $\unchi$ as $w_{\La,j}$ is microlocally supported on a compact subset of 
$\cotangent$ and $j$ stays in $J(h,C)$, if we accept some $O(h^\infty)$ remainder). 
The contribution of the last but one term in \e{3510} to \e{357b} may be written
$$
-ih\sumj \Tj \Ophj(e_1)\Ophj\big(\la\xi\ra^\mez\big)\wl_{\La,j}=-ih\Oph(e_1)\Oph\big(\la\xi\ra^\mez\big)\wlla.
$$
By Proposition~\ref{ref:3.1.9} and \e{3511}, this is equal to 
$$
\frac{i}{2}h\unchi \wlla +h^2\rl_{1,1}
$$
where $Z^k \rl_{1,1}$ belongs to $h^{-\delta_{k+1+\ell}'}\ti{\mB}_\infty^{-1,b'}[K]$. Actually noticing that 
$e_1(x,\xi)=|x|e_1\big(\frac{x}{|x|}, |x|^2\xi)$, and that 
$(\px^\alpha\partial_\eta^\beta e_1)(\pm 1,\eta)=O\big(|\eta|^{-1-|\beta|}\big)$, $|\eta|\rightarrow 0$ and $|\eta|\rightarrow +\infty$, one checks 
that $e_1(x,\xi)$ satisfies \e{3115} with $(\ell,\ell',d,d')=(-1,0,-1,0)$ so that $e_1(x,\xi)|\xi|^\mez$ obeys these estimates for 
$(\ell,\ell',d,d')=(-1,0,-1/2,0)$. Since $Z^k\wlla$ is in $h^{-\delta_{k+1+\ell}'}L^\infty\ti{J}_\La^{0,b'}[K]$, the above 
statement holds. Since for $j\in J(h,C)$, $2^j\ge h^{2(1-\sigma)}$, the remainder may be rewritten as the product of $h^{1+\sigma}$ with 
an element whose $\Zr^k$-derivatives are in $h^{-\delta_{k+1+\ell}'}\ti{\mB}_\infty^{0,b'}[K]$ i.e. contributes to $h^{1+\sigma}\rl_1$ in \e{357a}. 

We are reduced to showing that the contributions of the second, third and last terms in the \rhs of \e{3510} provide remainders. This is evident for 
the last term as $2^{\jd}h_j^2=h h_j=O(h^{1+\sigma})$. Using \e{3443}, we may write the sum of the two remaining terms
\be\label{3514}
h_j \Ophj(e_1)\ophjdeux \Big( f_j^\el+h^{\uq} \rl_j\Big)
+ih_j^2\Ophj(e_2)\Big(f_j^\el+h^{\uq} \rl_j\Big).
\ee
Since $\big(Z^k f_j^\el\big)_j$ is $O(\eps)$ in $h^{-3\delta_{k+2+\ell}'}L^\infty I_\La^{0,b'}[K]$ and 
$\big(Z^k \rl_j\big)_j$ is $O(\eps)$ in $h^{-4\delta_{k+1+\ell}'}\ti{\mB}_\infty^{0,b'}[L]$, the last term 
as well as the $\rl_j$ contribution to the first one induce in \e{357b} a 
contribution that may be included in the $h^{1+\sigma}\rl_1$ remainder term of 
\e{357a}. On the other hand, the fact that $\big(Z^k f_j^\el\big)_j$ is $O(\eps)$ in 
$h^{-3\delta_{k+2+\ell}'}L^\infty I_\La^{0,b'}[K]$ implies that 
$Z^k \ophjdeux f_j^\el$ belongs to a $\eps$-neighborhood 
of zero in $h^{-3\delta_{k+2+\ell}'}\big(h^\mez+h_j\big)\ti{\mB}_\infty^{0,b'}[L]$. Consequently, the first term in \e{3514} induces also in \e{357b} 
a contribution forming part to the $h^{1+\sigma}\rl_1$ term in \e{357a}. This concludes the proof of the lemma.
\end{proof}

\begin{proof}[Proof of Lemma~$\ref{ref:352}$] We notice first that
\be\label{3514a}
\opjl \opx \Oph\big(\lxi^{-\ell}\big) +i\ell h\Oph\Big(\frac{\xi^2}{\langle\xi\rangle^2}\Big)=\opx.
\ee
We make act this operator on the expression of $\wl$ from $\wlla$ given in \e{3441}. The action of 
$\opx$ on $\wlla$ has been computed in Lemma~\ref{ref:352a}. Let us study 
$$
\opx\left(\sh \big(\wl_{2\La}+\wl_{-2\La}\big)\right).
$$
One may express $\wl_{\pm 2\La}$ from 
$\wlla$, $\owlla$ by \e{3442}. Since, according to Proposition~\ref{ref:342}, $Z^k \wlla$ is in 
$h^{-\delta_{k+1+\ell}'}L^\infty\ti{J}_\La^{0,b'}[K]$, it follows from Proposition \ref{ref:3110} that 
$Z^k\big(\wlla\big)^2$ (resp.\ $Z^k\big(\owlla\big)^2)$ belongs to 
$h^{-2\delta_{k+1+\ell}'}L^\infty \ti{J}_{2\La}^{0,2b'}[K_2]$ (resp.\ $h^{-2\delta_{k+1+\ell}'}L^\infty \ti{J}_{-2\La}^{0,2b'}[K_{-2}]$). 

We apply next Proposition~\ref{ref:3.1.9}, with $a$ replaced by $(x\xi+|\xi|^\mez)\deuxangledom^{\ell} \angledom^{-2\ell}|\dom|$. 
Since, because of the fact that $\wl_{\pm 2\La}$ is microlocally supported close to $\pm 2\La$, we may assume that 
$x\xi+|\xi|^\mez$ is cut-off close to this manifold, we see that the above symbol satisfies \e{3115} with 
$(\ell,\ell',d,d')=(-2+2\ell,-2\ell,1/2,0)$ or $(-1+2\ell,-2\ell,1,0)$. It follows from \e{3116} that
\be\label{3514b}
\ba
&\sh \opx \big(\wl_{2\La}+\wl_{-2\La}\big)\\
&\qquad\qquad=-\frac{i}{4}\sh \unchi \deuxangledom^{\ell}\angledom^{-2\ell}|\dom|^\tdm \Big[ \big(\wlla\big)^2-\big(\owlla\big)^2\Big]\\
&\qquad \qquad \quad +h^\tdm \rl
\ea
\ee
where $Z^k\rl$ is in $h^{-2\delta_{k+1+\ell}'}\ti{\mB}_\infty^{2,2b'}[L]\subset h^{-2\delta_{k+1+\ell}'}\ti{\mB}_\infty^{0,b'}[L]$. 

In the same way
\be\label{3514c}
\ba
&h \opx \big(\wl_{3\La}+\wl_{-\La}+\wl_{-3\La}\big)\\
&\qquad=h |\dom|^{\frac 5 2} \angledom^{-3\ell}\Big[ \troisangledom^\ell \Big(\mu_3^\el \big(\wlla\big)^3+\mu_{-3}^\el\big(\owlla\big)^3\Big) +\angledom^\ell\mu_{-1}^\el \big\vert\wlla\big\vert\owlla\Big]\\
&\qquad \quad+h^2\rl
\ea
\ee
with $Z^k\rl$ in $h^{-3\delta_{k+1+\ell}'}\ti{\mB}_\infty^{4,3b'}[L]\subset h^{-3\delta_{k+1+\ell}'}\ti{\mB}_\infty^{0,b'}[L]$ and 
some real constants $\mu^\el_{\pm 3}$, $\mu^\el_{-1}$. 

Finally, since the action of $\opx$ on the remainder $g^\el$ of \e{3441} gives a function $\rl$ such that $Z^k\rl$ 
is in $h^{-4\delta_{k+N_1-N_0+1+\ell}'}\ti{\mB}_\infty^{0,b'-\mez}[F]$ we conclude, 
summing \e{357a}, \e{3514b}, \e{3514c} that \e{3514a} is given by formula \e{357}.
\end{proof}

We may now prove Proposition~\ref{ref:351}.

\begin{proof}[Proof of Proposition~\ref{ref:351}]
Let us compute
$$
D_t\wl =D_t \opjl w =i\ell h \Oph\Big(\frac{\xi^2}{\langle\xi\rangle^2}\Big)\wl +\opjl D_t w.
$$
According to \e{3222} this is the sum of $-\frac{i}{2}h\wl$, of \e{354}, of \e{357} and of a remainder $h^{\frac 5 4}R^\el(V)$ where 
$\Zr^k R^\el (V)$ is in $\mR_\infty^b$. (Notice that by definition of $w$, we may always insert on the left hand side of \e{3222} 
a cut-off $\sumj \Djh$ for some large enough $C$, so that the sum of quadratic and cubic contributions is really given by
\e{354}). 
Remembering the expression \e{3441} of $\wl$ in terms of $w_\La^{(\ell)}$, we may write 
$$
-\frac{i}{2}h\wl=-\frac{i}{2}h\unchi \wlla +h^\tdm \rl
$$
with $Z^k \rl$ in $h^{-4\delta_{k+N_1-N_0+1+\ell}'}\ti{\mB}_\infty^{0,b'}[F]$ (We used again that the microlocal support properties of $\wlla$ 
allow to multiply it by some cut-off $\unchi$ up to a $O(h^\infty)$-remainder). We obtain
\be\label{3515}
\ba
D_t \wl&=\mez \unchi |\dom|^\mez \wlla \\
&\quad -i\frac{\sh}{4}\unchi |\dom|^\tdm \deuxangledom^\ell \angledom^{-2\ell}
\Big[ (1+\sqrt{2})\bigl(\wlla\big)^2+(\sqrt{2}-1)\big(\owlla\big)^2\Big]\\
&\quad +\frac{h}{2}\unchi |\dom|^{\frac 5 2}\angledom^{-3\ell}\Big[ \troisangledom^\ell\Big( \mu_3^\el \bigl(\wlla\big)^3
+\mu_{-3}^\el \bigl(\owlla\big)^3\Big)\\
&\qquad\qquad\quad +\angledom^\ell \big\vert \wlla\big\vert^2\wlla +\angledom^\ell \mu_{-1}^\el \big\vert \wlla\big\vert\owlla\Big]\\
&\quad +h^{1+\sigma}R(V)
\ea
\ee
where $Z^k R(V)$ is $O(\eps)$ in $h^{-4\delta_{k+N_1-N_0+1+\ell}'}\mR_\infty^{b'-\mez}$, and where $\mu_q^\el$ are some new real constants. 

We express next $\wlla$ from $\wl$ inverting relation \e{3441} i.e.\ writing, taking \e{3442} into account, 
\be\label{3516}
\ba
\wlla&=\wl +i\frac{\sh}{4}\unchi \deuxangledom^\ell \angledom^{-2\ell}|\dom|\Big[ (1+\sqrt{2})\big(\wl\big)^2+(1-\sqrt{2})\big(\owl\big)^2\Big]\\
&\quad +h|\dom|^2\unchi \Big[ \Gamma_3^\el (\dom)\big(\wl)^3+\Gamma_1^\el (\dom)\big\vert \wl\big\vert^2\wl\\
&\quad\phantom{+h|\dom|^2\unchi \Big[}\quad +\Gamma_{-1}^\el (\dom)\big\vert \wl\big\vert^2\owl+\Gamma_{-3}^\el (\dom)\big(\owl)^3\Big]\\
&\quad +h^{1+\sigma}g^\el,
\ea
\ee
where $Z^k g^\el$ is in $h^{-4\delta_{k+N_1-N_0+1+\ell}'}\mR_\infty^{b'}$, and where $\Gamma_q^\el(\zeta)$ is a symbol 
of order $-2\ell$, with
$$
\Gamma_1^\el(\zeta)=\langle 2\zeta\rangle^{2\ell}\langle\zeta\rangle^{-4\ell}\left( \frac{3-2\sqrt{2}}{8}\right).
$$
We plug this expansion in \e{3515} to get \e{351}. The remainder satisfies the conditions of the statement of the proposition if we assume that 
$4\delta_{k+N_1-N_0+1+\ell}'<\frac{\sigma}{2}$, so that we may take $\kappa=\sigma/2$. 

We prove now \e{353a} by induction from \e{351}.

To deduce \e{353a} at order $k$ from the similar equality at order $k-1$, we notice first that the action of $Z$ on the quadratic 
(resp.\ cubic, resp.\ remainder) terms of \e{353a} at order $k-1$ gives contributions to $Q^\kl$ (resp.\ $\Cr^\kl$, resp.\ $\mR^\kl$). Moreover,
\begin{align*}
&\Big[ Z, D_t-\mez\unchi |\dom(x)|^\mez\Big]\\
&\qquad \qquad =-\Big(D_t-\mez \unchi |\dom(x)|^\mez\Big)\\
&\qquad \qquad \quad +\frac{1+\beta}{2}\big(xh^{-\beta}\big)\chi'\big(xh^{-\beta}\big)|\dom(x)|^\mez.
\end{align*}
The product of the last term with $W^{k-1,\el}$ may be computed from expressions of the form 
$h^{-\beta}\Gamma\big(xh^{-\beta}\big)Z^{k'}w^\el$ where 
$k'\le k-1$ and $\Gamma$ is in $C^\infty_0(\xR^*)$. We just have to check that 
such terms contribute to the remainder in \e{353a}. Because of the expression \e{3441} of $\wl$ in terms of $\wlla$, 
we see that we need to check that for $p\le b''$
\be\label{3516a}
\lA (hD_x)^p\left[ \Gamma\xb Z^{k'}\wlla\right]\rA_{L^\infty} =O\big(h^{1+\kappa+\beta}\big).
\ee
(The contribution coming from the remainder in \e{3441} satisfies the wanted bound as we assume after \e{333} that 
$\beta<\sigma/2$.) We remember that $\wlla =\sumj \Tj \wl_{\La,j}$, where $\wl_{\La,j}$ is 
microlocally supported for $x$ in a compact subset of $\xR^*$, so that
$$
\Gamma\xb Z^{k'}\wlla =\sumj \Gamma\xb Z^{k'}\left( \Gamma_1 \big(2^{\jd}x\big)\Tj \wl_{\La,j}\right)
$$
for some $\Gamma_1$ in $C^\infty_0(\xR)$. This shows that the sum is limited to those $j$ for which $2^j\sim h^{-2\beta}$. 

Since $Z^{k'}\wlla$ is in $h^{-\delta_{k'+\ell+1}'}L^\infty \ti{J}_\La^{0,b'}[K]$ according to Proposition~\ref{ref:342},
$$
\blA \Tj Z^{k'}\wllaj\bri =O\big(h^{-\delta_{k'+\ell+1}'}2^{-j_+b'}\big).
$$
Using that $2^j\sim h^{-2\beta}$, we bound $2^{-j_+b'}\le 2^{-j_+ b''}h^{\delta_{k'+\ell+1}'+\beta+\kappa+2}$ 
since the assumption on $b''$ relatively to $b'$ implies that 
$2\beta(b'-b'')>\delta_{k+\ell+1}'+\kappa+\beta+2$, as $\delta_{k+\ell+1}'$, $\kappa$, $\beta$ are small enough. Consequently, 
we get \e{3516a} for all integers $p\le b''$. 

This concludes the proof of the proposition.
\end{proof}

\begin{prop}\label{ref:353}
Let $T_0$ be a large enough positive number, $\kappa$ a small positive constant, $C_0>0$. 

Let $\ell$ be an integer, with $\ell\le \sd+N_0$. Assume given a function $(t,x)\rightarrow \rl(t,x)$ from a domain 
$[T_0,T[\times \xR$ to $\xC$, satisfying 
for $p\le b''$,
$$
\sup_x \la (hD_x)^p\rl (t,x)\ra\le C_0\eps
$$
for any $t\in [T_0,T[$. Assume given a solution $\wl\colon [T_0,T[\times\xR\rightarrow \xC$ of equation \e{351}, such that 
$\la \wl(T_0,x)\ra\le C_0\eps$ for any $x$. 

Then there are $\eps_0>0$, $C_1>0$, depending only on $C_0$, such that for any $\eps\in ]0,\eps_0[$,
\be\label{3517}
\sup_{[T_0,T[}\blA \wl(t,\cdot)\bri\le C_1\eps.
\ee
Moreover, if we assume that $r^{(\ell)}$ is defined and satisfies the above assumption on $[T_0,+\infty[\times\xR$, 
then $w^{(\ell)}$ is defined on $[T_0,+\infty[\times\xR$ and 
there are a continuous bounded function $\alpha\colon \xR\rightarrow \xC$, vanishing like $|x|^{2b''}$ when $x$ goes to zero, 
$(t,x)\rightarrow \rho(t,x)$ a bounded function on $[T_0,+\infty[\times \xR$ with values in $\xC$ and $\kappa>0$ such that
\be\label{3518}
w(t,x)=\eps \alpha (x)\exp\left[ \frac{i}{4|x|}\int_{T_0}^t (1-\chi)\big(\tau^\beta x\big)\, d\tau
+\frac{i}{64}\eps^2\frac{\la \alpha(x)\ra^2}{\la x\ra^{5}}\log t\right]+\eps t^{-\kappa}\rho(t,x)
\ee
where $\chi\in C^\infty_0(\xR)$, $\chi\equiv 1$ close to zero.
\end{prop}
\begin{rema*}
We may write \e{3518} on a more explicit fashion. Assume that $b''>\kappa/(2\beta)$. The contribution of the first term in expansion 
\e{3518} localized for $|x|<Ct^{-\kappa/(2b'')}$ may be incorporated to the remainder, because of the vanishing 
of $\alpha$ at order $2b''$ at $x=0$. On the other hand, if $|x|>Ct^{-\kappa/(2b'')}$, our assumption on $b''$ implies that 
$|x|^{-1/\beta}<C^{-1/\beta}t$, so that, if $C$ is large enough, 
$$
\int_{T_0}^t \chi(\tau^\beta x)\, d\tau=\int_{T_0}^{+\infty}\chi(\tau^\beta x)\, d\tau.
$$
If we define 
$$
\underline{\alpha}(x)=\alpha(x)\exp\left[ -\frac{i}{4|x|}T_0- \frac{i}{4|x|}\int_{T_0}^{+\infty}\chi\big(\tau^\beta x\big)\, d\tau\right]
$$
we obtain
\be\label{3518a}
w(t,x)=\eps \underline{\alpha}(x)\exp\left[ \frac{it}{4|x|}+
\frac{i}{64}\eps^2\frac{\la \underline{\alpha}(x)\ra^2}{\la x\ra^{5}}\log t\right]+\eps t^{-\kappa}\underline{\rho}(t,x)
\ee
for a new bounded remainder $\underline{\rho}$.
\end{rema*}
\begin{proof}
We shall establish the proposition performing a normal form transform on equation \e{351}. 
Denote by $\Gr$ the space of continuous bounded 
functions on $[T_0,T[\times \xR$. Let $\chi_0$ be in $C^\infty_0(\xR)$, $\chi_0\equiv 1$ close to zero, 
$\supp \chi_0\subset \{x\,;\,\chi(x)\equiv 1\}$ and set 
\be\label{3519}
\ba
f^\el&=\wl\\
&\quad+i\frac{\sh}{4}\unchiz |\dom| \deuxangledom^\ell \angledom^{-2\ell} 
\Big[ (1+\sqrt{2})\big(\wl\big)^2+(1-\sqrt{2})\big(\owl\big)^2\Big]\\
&\quad +h \unchiz ^2 |\dom |^2\Big[ M_3^\el (\dom)\big(\wl\big)^3+M_{-3}^\el (\dom)\big(\owl\big)^3\\
&\quad \phantom{+h \unchiz |\dom |^2\Big[ }\quad 
+M_{-1}^\el (\dom)\big\vert \wl\big\vert^2\owl\Big]
\ea
\ee
where $M_p^\el(\zeta)$ are symbols of order $-2\ell$ in $\zeta$ to be chosen, $p=-3,-1,3$. 

We consider the polynomial map 
$\Phi\colon \begin{pmatrix} \wl \\ \owl \end{pmatrix} \rightarrow \begin{pmatrix} f^\el \\ \overline{f}^\el\end{pmatrix}$ 
defined on $\Gr$. For $h=t^{-1}$ small enough (i.e.\ $t\ge T_0$ large enough), this is a local diffeomorphism 
at zero in $\Gr$. The inverse $\Phi^{-1}$ sends 
$\begin{pmatrix} f^\el \\ \overline{f}^\el\end{pmatrix}$ to $\begin{pmatrix} \wl \\ \owl \end{pmatrix}$, where $\wl$ may be expressed explicitly as 
\be\label{3520}
\ba
\wl&=f^\el\\
&\quad -i\frac{\sh}{4}\unchiz |\dom| \deuxangledom^\ell \angledom^{-2\ell} 
\Big[ (1+\sqrt{2})\big(f^\el\big)^2+(1-\sqrt{2})\big(\overline{f}^\el\big)^2\Big]\\
&\quad +h  |\dom |^2 (1-\chi_0)^2\xb \bigg[ \widetilde{M}_3^\ell (\dom)\big(f^\el\big)^3\\
&\quad \phantom{+h  |\dom |^2\Big[}\quad + \frac{3-2\sqrt{2}}{8}\deuxangledom^{2\ell}\angledom^{-4\ell}
\big\vert f^\el\big\vert^2 f^\el\\
&\quad \phantom{+h  |\dom |^2\Big[}\quad +\widetilde{M}_{-1}^\ell (\dom)\big\vert f^\el\big\vert^2 \overline{f}^\el\\
&\quad \phantom{+h  |\dom |^2\Big[}\quad +\widetilde{M}_{-3}^\ell (\dom)\big(\overline{f}^\el\big)^3\bigg]\\
&\quad +h^{1+\kappa}R_4\big(x,h;f^\el,\overline{f}^\el\big)
\ea
\ee
where $\kappa$ is some positive constant, where $R_4\big(x,h;f^\el,\overline{f}^\el\big)$ is some analytic function of 
$\big(f^\el,\overline{f}^\el\big)$, vanishing at order four at zero, with bounds uniform in $(x,h)$, 
and 
$\widetilde{M}_p^\el (\zeta)=-M_p^\el (\zeta)+\Gamma_p^\ell(\zeta)$, 
$p=-3,-1,1$ for symbols $\Gamma_p^\ell$ of order $-2\ell$,  
independent of $M_p^\ell$. 

We compute $D_t f^\el$ from \e{3519}, expressing in the \rhs $D_t \wl$, $D_t\owl$ using \e{351}. We get
\be\label{3521}
\ba
D_t f^\el &=\mez \unchi |\dom(x)|^\mez \wl\\
&\quad +i\frac{\sh}{8}\unchi |\dom(x)|^\tdm \deuxangledom^\ell \angledom^{-2\ell}
\Big[ (1+\sqrt{2})\big(\wl\big)^2+(1-\sqrt{2})\big(\owl\big)^2\Big]\\
&\quad +h\unchi |\dom(x)|^{\frac 5 2}\bigg[\mez \angledom^{-2\ell}\big\vert \wl\big\vert^2\wl\\
&\quad \qquad +\Big( \tdm M_3^\el (\dom)+\widetilde{\Gamma}_3^\el (\dom)\Big)\big(\wl\big)^3
+\Big( -\mez M_{-1}^\el (\dom)+\widetilde{\Gamma}_{-1}^\el (\dom)\Big)\big\vert \wl\big\vert^2\owl\\
&\quad \qquad +\Big( -\tdm M_{-3}^\el (\dom)+\widetilde{\Gamma}_{-3}^\el (\dom)\Big)\big(\owl\big)^3\bigg]\\
&\quad +h^{1+\kappa}\Big( \rl(t,x)+R_2\big(x,h;\wl,\owl\big)\Big)
\ea
\ee
where $\widetilde{\Gamma}_p^\el(\zeta)$ are symbols of order $-2\ell$ in $\zeta$, that depend only on the coefficients of 
$\big(\wl\big)^2$, $\big(\wl\big)^3$, \ldots in the \rhs of \e{351}, where $r^\el$ is the remainder in \e{351} and 
where $R_2\big(x,h;\wl,\owl\big)$ is some polynomial in $\big(\wl,\owl\big)$, vanishing  at order $2$ at zero, with uniform bounds 
in $(x,h)$. We express $\wl$ in the \rhs of \e{3521} using formula \e{3520}. The quadratic terms in the definition \e{3519} 
of $f^\el$ have been chosen in such a way that the quadratic contributions in the \rhs of the resulting expression for 
$D_t f^\el$ vanish. We get
\be\label{3522}
\ba
D_t f^\el &=\mez \unchi |\dom(x)|^\mez f^\el\\
&\quad +h\unchi |\dom(x)|^{\frac 5 2}\bigg[ \mez \angledom^{-2\ell}\big\vert f^\el\big\vert^2f^\el\\
&\quad\qquad +\Big(M_3^\el (\dom)-\widetilde{\Gamma'}_3^\el (\dom)\Big)\big(f^\el\big)^3\\
&\quad\qquad +\Big( -M_{-1}^\el (\dom)-\widetilde{\Gamma'}_{-1}^\el (\dom)\Big)\big\vert f^\el\big\vert^2\overline{f}^\el\\
&\quad \qquad +\Big(-2 M_{-3}^\el (\dom)-\widetilde{\Gamma'}_{-3}^\el (\dom)\Big)\big(\overline{f}^\el\big)^3\bigg]\\
&\quad +h^{1+\kappa}\Big( \rl(t,x)+R_2\big(x,h;f^\el,\overline{f}^\el\big)\Big)
\ea
\ee
where $\widetilde{\Gamma'}_p^\el (\zeta)$ are new symbols of order $-2\ell$ that do not depend on $M_p^\el$, and where 
$R_2$ is a new analytic function of $(f^\el,\overline{f}^\el)$ vanishing at order $2$ at zero, with uniform bounds in~$(x,h)$. 

We choose now the free symbols $M_p^\el$, $p=3,-1,-3$ introduced in the definition \e{3519} of $f^\el$ so that the coefficients of 
$\big(f^\el\big)^3$, $\big\vert f^\el\big\vert^2\overline{f}^\ell$ and $\big(\overline{f}^\el\big)^3$ vanish. In that way, we are reduced to 
\be\label{3523}
\ba
D_t f^\el &=\mez \unchi |\dom(x)|^\mez \bigg[ 1+\frac{|\dom|^2}{t}\angledom^{-2\ell}\big\vert f^\el\big\vert^2\bigg]f^\el\\
&\quad +t^{-1-\kappa}\rl(t,x)+t^{-1-\kappa}R_2\big(x,h;f^\el,\overline{f}^\el\big)
\ea
\ee
where $\blA (hD_x)^pZ^k \rl (t,x)\bri$ is $O(\eps)$ for any $p\le b''$, $k+\ell\le \sd+N_0$. It follows from 
\e{3523} that $\la \partial_t\big\vert f^\el\big\vert^2\ra\le \big(C_0\eps +C_0'\big\vert f^\el\big\vert^2\big)t^{-1-\kappa}$ as long as 
$\big\vert f^\el\big\vert$ stays smaller than $1$. Since at time $t=T_0$, 
$\big\vert f^\el\big\vert=O(\eps)$, we obtain that $\big\vert f^\el(t,x)\big\vert\le C_1'\eps$ for some constant $C_1'>0$, 
as long as the solution exists. Using expression \e{3520} for $\wl$ in terms of $f^\el$, we get \e{3517}. 
If $r_\ell$ is defined for $t\in [T_0,+\infty[$, we get that $f^{(\ell)}$ and thus $w^{(\ell)}$ is defined on 
$[T_0,+\infty[\times \xR$.

Let us prove the asymptotic expansion for $w$. If $\ell\le \sd+N_0$, we define
$$
\Theta_\ell(t,x)=\mez |\dom(x)|^\mez\int_{T_0}^t(1-\chi)\big(\tau^\beta x\big) 
\left[ 1+\frac{|\dom(x)|^2}{\tau}\angledom^{-2\ell}\bla f^\el(\tau,x)\bra^2\right]\, d\tau.
$$
Then \e{3523} and the uniform a priori bound just obtained for $f^\el$ show that
$$
\frac{d}{dt}\left[ f^\el(t,x)\exp\Big[- i\Theta_\ell(t,x)\Big]\right]=O\big(\eps t^{-1-\kappa}\big)
$$
uniformly for $x\in \xR$. It follows that the uniform limit when $t$ goes to $+\infty$ of 
$$
f^\el(t,x)\exp\big[ -i\Theta_\ell(t,x)\big]
$$
exists and defines a continuous function $\eps\alpha_\ell(x)$ on $\xR$, which is $O(\eps)$ in 
$L^\infty(\xR)$. Moreover
\be\label{3524}
\blA f^\el (t,x)-\eps\alpha_\ell(x)\exp\big(i\Theta_\ell(t,x)\big)\bri
=O\big(\eps t^{-\kappa}\big),\quad t\rightarrow +\infty.
\ee
We write
\be\label{3525}
\ba
\Theta_\ell(t,x)&=\mez |\dom(x)|^\mez \int_{T_0}^t (1-\chi)\big(\tau^\beta x\big) \, d\tau\\
&\quad +\frac{\eps^2}{2}(1-\chi)\big(t^\beta x\big)|\dom(x)|^{\frac 5 2}\langle \dom(x)\rangle^{-2\ell}\la \alpha_\ell(x)\ra^2\log t\\
&\quad -\frac{\eps^2}{2}(1-\chi)\big(T_0^\beta x\big)|\dom(x)|^{\frac 5 2}\langle \dom(x)\rangle^{-2\ell}\la \alpha_\ell(x)\ra^2\log T_0\\
&\quad +\frac{\eps^2}{2}\int_{T_0}^t \beta \tau^\beta x \chi'\big(\tau^\beta x\big)\la \dom(x)\ra^{\frac 5 2} \langle \dom(x)\rangle^{-2\ell}\la \alpha_\ell(x)\ra^2
\frac{\log \tau}{\tau}\, d\tau\\
&\quad +\frac{1}{2}\int_{T_0}^t (1-\chi)\big(\tau^\beta x\big)\la \dom(x)\ra^{\frac 5 2} \langle \dom(x)\rangle^{-2\ell}\Big(
\bla f^\el\bra^2-\la \eps \alpha_\ell(x)\ra^2\Big)\, \frac{d\tau}{\tau}.
\ea
\ee
We notice that $\la \dom(x)\ra^{\frac 5 2}\angledom^{-2\ell}$ is $O(\langle x\rangle^{-5})$ if $\ell\geq 5/4$ and $\tau^\beta x \chi'(\tau^\beta x)\la \dom(x)\ra^{\frac 5 2}\langle \dom(x)\rangle^{-2\ell}$ is $O\big(\tau^{-\beta}\langle x\rangle^{-6}\big)$ if 
$\ell\ge 3/2$. 

These bounds and the estimate $\blA \bla f^\el\bra^2-\eps^2\la \alpha_\ell\ra^2\bri=O\big(\eps^2 t^{-\kappa}\big)$ that follows from 
\e{3524} imply that the last three terms in \e{3525} may be written as $\eps^2\Gamma_\ell(x)+\eps^2R_\ell(t,x)$ for some continuous function 
$\Gamma_\ell(x)$, which is $O\big(\langle x\rangle^{-5}\big)$ and some remainder $R(t,x)$ satisfying 
$\la R(t,x)\ra=O\big( t^{-\kappa}\lx^{-5}\big)$ (assuming $0<\kappa<\beta$). Modifying 
the definition of that remainder, we get finally
\begin{align*}
\Theta_\ell(t,x)&=\mez \la \dom(x)\ra^\mez\int_{T_0}^t (1-\chi)(\tau^\beta x)\, d\tau\\
&\quad +\frac{\eps^2}{2}\la \dom(x)\ra^{\frac 5 2}\angledom^{-2\ell}\la \alpha_\ell(x)\ra^2\log t\\
&\quad +\eps^2\Gamma(x)+\eps^2 R(t,x)
\end{align*}
when $\ell>3/2$. It follows from this and from \e{3524} that 
\be\label{3526}
\ba
f^\el(t,x)&=\eps \widetilde{\alpha}_\ell (x)\exp\left[ \frac{i}{4|x|}\int_{T_0}^t (1-\chi)\big(\tau^\beta x\big)\, d\tau
+\frac{i}{64}\eps^2\frac{\bla \widetilde{\alpha}_\ell(x)\bra^2}{\la x\ra^{5}}\angledom^{-2\ell} \log t\right]\\
&\quad+\eps t^{-\kappa}\rho(t,x)
\ea
\ee
where $\widetilde{\alpha}_\ell(x)=e^{i\eps^2\Gamma(x)}\alpha_\ell(x)$ and where 
$\rho(t,x)$ is uniformly bounded. If we express $\wl$ from $f^\el$ using \e{3520}, we conclude that the same expansion 
\e{3526} holds for $\wl$ (with a different remainder). Let us compute $w(t,\cdot)=\Oph\big(\lxi^{-\ell}\big) \wl$. 
The action of $\Oph\big(\lxi^{-\ell}\big)$ on the remainder gives a term of the same type, if $\ell$ is large enough. On the other hand, 
by the expression \e{3441}, \e{3442} of $\wl$ from $\wlla$ (and the converse expression), we get that $\Oph\big(\lxi^{-\ell}\big)w=
\Oph\big(\lxi^{-\ell}\big)w_\La$ up to a remainder bounded in $L^\infty(dx)$ by $C\eps t^{-\kappa}$. As $w_\La$ is in 
$h^{-\delta_1'}L^\infty\ti{J}_\La^{0,b'}[K]$, Proposition~\ref{ref:3.1.9} applies and shows that 
$\Oph\big(\lxi^{-\ell}\big)w_\La$ may be written as $\angledom^{-\ell}w_\La$ modulo a remainder in 
$h^{1-\delta'_1} \ti{\mB}_\infty^{-1,b'}[K]\subset h^{\sigma-\delta'_1}\ti{\mB}_\infty^{0,b'}[K]$, which is 
$O(\eps t^{-\kappa})$ in $L^\infty$ for small enough $\kappa>0$ since by \e{333} $\delta'_1<\sigma/8$. Using again the expression of $w_\La$ from $w$ 
deduced from \e{3441}, we deduce that 
$$
\lA w(t,\cdot)-\angledom^{-\ell}\wl(t,\cdot)\rA_{L^\infty}=O(\eps t^{-\kappa}).
$$
If we define $\alpha(x)=\angledom^{-\ell}\widetilde{\alpha}_\ell(x)$ with $\ell$ equal to $b''$, we obtain a function 
continuous and bounded on $\xR$, vanishing like $|x|^{2b''}$ when $x$ goes to zero and such that 
$w(t,x)$ is given by the asymptotic expansion \e{3518}. This concludes the proof.
\end{proof}

We prove now a statement concerning the $Z$-derivatives of $\wl$. Let 
$(A_k'')_{k\ge 1}$ be a sequence of positive numbers satisfying $A_k''\ge A_{k_1}''+A_{k_2}''+A_{k_3}''$ 
if $k_1+k_2+k_3=k$, $k_j<k$, $j=1,2,3$ and $A_k''$ large enough relatively to the constant $C_1$ in \e{3517}.

\begin{prop}\label{ref:354}
There is a constant $C_2>0$ such that, if we set $\widetilde{\delta}_k'=A_k''\eps^2$, for any 
$k,\ell$ with $k+\ell\le \sd +N_0-2$, the solution $\wl$ of \e{351} satisfies 
\be\label{3527}
\blA Z^k \wl(t,\cdot)\bri\le C_2\eps t^{\widetilde{\delta}_k'}.
\ee
\end{prop}
\begin{rema*}
The gain in \e{3527}, in comparison with \e{339}, 
is that the exponents $\widetilde{\delta}_k'$ depend only on 
the size $\eps$ of the Cauchy data and not on the exponents 
$\delta_k$ that are used in the $L^2$-estimates. In particular, 
taking $\eps$ small enough, we may arrange so that $\widetilde{\delta}_k'\ll \delta_k$. 
\end{rema*}
\begin{proof}
We apply a normal forms method to remove the 
quadratic terms in \e{353a}. For $(k,\ell)$ satisfying 
$k+\ell\le \sd+N_0$, we define a new quadratic 
map $\widetilde{Q}^{k,\el}\Bigl[x,h;W^\kl,\overline{W}^\kl\Bigr]$ 
in the following way: The components of this map 
are defined taking the same linear combinations as those used to define the 
components of $Q^\kl$ from \e{353b} of the quantities
\be\label{3528}
\ba
&\frac{2\theta \xb \Phi(x)}{\unchi \la \dom\ra^{\mez}}\big( Z^{k_1}\wl\big)\big(Z^{k_2}\wl\big)\\
&\frac{-2\theta \xb \Phi(x)}{\unchi \la \dom\ra^{\mez}}\big( Z^{k_1}\wl\big)\big(Z^{k_2}\owl\big)\\
&\frac{-2\theta \xb \Phi(x)}{3\unchi \la \dom\ra^{\mez}}\big( Z^{k_1}\owl\big)\big(Z^{k_2}\owl\big).
\ea
\ee
If we make act $D_t-\mez \unchi \la \dom(x)\ra^\mez$ on each line of \e{3528}, we see that we obtain, using 
\e{353a}, the corresponding line of \e{353b} and the following contributions

$\bullet$ Quantities of the form $\sh \widetilde{\Cr}^\kl \Bigl[x,h;W^\kl,\overline{W}^\kl\Bigr]$ for cubic forms $\widetilde{\Cr}^\kl$ which have the same structure 
\e{353c} as $\Cr^\kl$. 

$\bullet$ Quantities given by the product of $h$ and of homogeneous expressions of order $4$ in 
$$
\big(Z^{k_j}\wl,Z^{k_j}\owl\big),\quad  k_1+\cdots+k_4\le k
$$
with coefficients depending on $x$ which are 
$O(h^{-7\beta})$. If we use that $Z^{k_j}\wl$ satisfies the a priori estimates \e{339}, we see that these contributions may be written as 
$h^{\mez+\kappa}R^\kl$ for some $\kappa>0$ and a bounded function $R^\kl$. 

$\bullet$ Contributions coming from the remainder in \e{353a} or from the action of $D_t$ on the cut-offs in \e{3528}, that may be written also as 
$h^{\mez+\kappa}R^\kl$. 

Consequently, if we set for $k+\ell\le \sd+N_0$,
$$
\widetilde{W}^\kl=W^\kl-\sh\widetilde{Q}^{k,\el}\Bigl[x,h;W^\kl,\overline{W}^\kl\Bigr],
$$
we obtain that $\widetilde{W}^\kl$ satisfies bounds of the form \e{339} and solves the equation
\be\label{3529}
\Bigl( D_t-\mez\unchi \la \dom \ra^\mez\Bigr)\widetilde{W}^\kl=h \Cr^\kl \Bigl[x,h;W^\kl,\overline{W}^\kl\Bigr]
+h^{1+\kappa}\mR^\kl
\ee
where $\Cr^\kl$ is a new cubic map given in terms of monomials of the form \e{353c} and $\mR^\kl$ a uniformly 
bounded remainder. Notice that, up to a modification of $\mR^\kl$, we may replace $W^\kl$ by $\widetilde{W}^\kl$ 
in the argument of $\Cr^\kl$.

Assume by induction that for given $k,\ell$ with $k+\ell\le \sd+N_0$, 
$\ell\ge 2$, \e{3527} has been established with 
$k$ replaced by $k-1$. Then $W^{k-1,\el}$ and $\widetilde{W}^{k-1,\el}$ are under control, and 
we need to obtain \e{3527} for the last component $Z^k\wl$ of $W^\kl$, or equivalently, for the last component $\widetilde{W}_k^\kl$
of $\widetilde{W}^\kl$. We sort the different contributions to 
\be\label{3530}
\Cr^\kl\Bigl[x,h;\widetilde{W}^\kl,\overline{\widetilde{W}}^\kl\Bigr].
\ee
On the one hand, we get terms given by expressions of the form 
\be\label{3531}
\ba
&\theta \xb \Phi(x)\widetilde{W}^\kl_{k_1}\widetilde{W}^\kl_{k_2}\widetilde{W}^\kl_{k_3}\\
&\theta \xb \Phi(x)\widetilde{W}^\kl_{k_1}\widetilde{W}^\kl_{k_2}\overline{\widetilde{W}}^\kl_{k_3}\\
&\theta \xb \Phi(x)\widetilde{W}^\kl_{k_1}\overline{\widetilde{W}}^\kl_{k_2}\overline{\widetilde{W}}^\kl_{k_3}\\
&\theta \xb \Phi(x)\overline{\widetilde{W}}^\kl_{k_1}\overline{\widetilde{W}}^\kl_{k_2}\overline{\widetilde{W}}^\kl_{k_3}
\ea
\ee
where $\theta,\Phi$ satisfy the same conditions as in \e{353c}, so are bounded since $\ell\ge 2$, and where two among $k_1,k_2,k_3$ are zero and the other one 
is equal to $k$. We call $F$ the sum of contributions of that type, so that
$$
\la F(t,x)\ra\le C\blA \widetilde{W}_0^\kl(t,\cdot)\bri ^2 \bla \widetilde{W}^\kl_k(t,x)\bra.
$$
Proposition \ref{ref:353} gives a uniform estimate for $\blA \wl(t,\cdot)\bri$, so also for 
$$
\widetilde{W}_0^\kl=\wl-\sh \widetilde{Q}_0^\kl \Bigl[x,h;W^\kl,\overline{W}^\kl\Bigr].
$$
We conclude that for some constant $B>0$, depending only on the constant $C_1$ in \e{3517},
\be\label{3532}
\la F(t,x)\ra\le B\eps^2 \bla \widetilde{W}_k^\kl(t,x)\bra.
\ee
On the other hand, \e{3530} is also made of terms of the form \e{3531} with
$$
k_1+k_2+k_3\le k,\quad  k_1,k_2,k_3<k.
$$
The assumption of induction, together with the inequality between the constants $A_k$ made in the statement 
of the proposition, imply that the contribution $G$ of these terms satisfies
\be\label{3533}
\la G(t,x)\ra\le C\eps^3 t^{\widetilde{\delta}_k'}.
\ee
We deduce from the equation for the last component $\widetilde{W}^\kl_k$ of $\widetilde{W}^\kl$ given by \e{3529}
\begin{align*}
\bla \widetilde{W}^\kl_k(t,x)\bra^2&\le \bla \widetilde{W}^\kl_k(T_0,x)\bra^2+\int_{T_0}^t |F(\tau,x)|\bla \widetilde{W}^\kl_k(\tau,x)\bra\, \frac{d\tau}{\tau}\\
&\quad +\int_{T_0}^t \bla G(\tau,x)\bra \bla \widetilde{W}_k^\kl (\tau,x)\bra\, \frac{d\tau}{\tau}\\
&\quad +\int_{T_0}^t \bla R_k^\kl(\tau,x)\bra\bla \widetilde{W}_k^\kl(\tau,x)\bra\, \frac{d\tau}{\tau^{1+\kappa}}.
\end{align*}
Using \e{3532}, \e{3533}, and the fact that at $t=T_0$, $\widetilde{W}^\kl_k(T_0,\cdot)$ is $O(\eps)$ we deduce that 
\be\label{3534}
\ba
\bla \widetilde{W}^\kl_k(t,x)\bra&\le C\eps+B\eps^2 \int_{T_0}^t \bla \widetilde{W}^\kl_k(\tau,x)\bra\, \frac{d\tau}{\tau}\\
&\quad +C\eps^2 \int_{T_0}^t \tau^{\widetilde{\delta}_k'-1}\, d\tau+C\eps \int_{T_0}^t \frac{d\tau}{\tau^{1+\kappa}}.
\ea
\ee
If we use Gronwall inequality, and assume that the constant $A_k$ in the definition $\widetilde{\delta}_k'=A_k\eps^2$ of 
$\widetilde{\delta}_k'$ is large enough relatively to $B$, we deduce from \e{3534} that 
$$
\bla \widetilde{W}^\kl_k(t,x)\bra\le C'\eps t^{\widetilde{\delta}_k'}
$$
when $k+\ell\le \sd+N_0$, $\ell\ge 2$. By definition of $\widetilde{W}_k^\kl$, the same inequality holds for 
$Z^k\wl$. Since $w^{(\ell-2)}=\Oph\big(\lxi^{-2}\big)\wl$, we conclude that $\blA Z^k\wl (t,\cdot)\bri$ is 
$O(\eps t^{\widetilde{\delta}_k'})$ when $k+\ell\le \sd+N_0-2$. This concludes the proof of the proposition. 
\end{proof}

To finish this section, we deduce from the results established so far the proof of Theorem~\ref{ref:132}. This will conclude the demonstration of our main 
theorem. 

\begin{proof}[Proof of Theorem~\ref{ref:132}]
We notice first that it is enough to prove the following apparently weaker statement: Assume that for some constants $B_2>0$,
$\tilde{A}_0>0$, any $t\in [T_0,T[$, any $\epsilon\in ]0,1]$, any $k\leq s_1$
\begin{equation}
  \label{3534a}
  \begin{split}
    M_s^{(k_1)}(t)\leq B_2\epsilon t^{\delta_k},\ N_\rho^{(s_0)}(t)\leq\sqrt{\epsilon}<1,\\
\norm{\eta(t,\cdot)}_{C^\gamma} + \norm{\abs{D_x}^{1/2}\psi(t,\cdot)}_{C^{\gamma-\frac{1}{2}}} \leq \tilde{A}_0t^{-\frac{1}{2}+\delta'_0}.
  \end{split}
\end{equation}
Then, \e{138} holds.

Actually, if the preceding implication is proved with $\rho\geq\gamma$, and if we assume only \e{137}, then \e{3534a} holds
true on some interval $[T_0,T']$, $T'>T_0$ taking $\tilde{A}_0$ large enough in function of $T_0$ (because the last condition
in \e{3534a} follows then from the second one, taking $T'$ close enough to $T_0$). We conclude that \e{138} holds on
$[T_0,T']$, and taking $\epsilon<\epsilon_0'$ small enough so that $\epsilon B_\infty <\tilde{A}_0$, we see that, by
continuity, \e{3534a} holds on some interval $[T_0,T'']$ with $T''>T'$. By bootstrap, we conclude that \e{138} will then be
true on the whole interval $[T_0,T]$.

Consequently, we have reduced ourselves to the proof of the fact that \e{3534a} implies \e{138}.

Recall that we have fixed in \e{334} large enough numbers $a,b$. We introduced also at the beginning of Section~\ref{S:33} integers 
$N_0,N_1$ and we assumed in Proposition~\ref{ref:342} that $(N_1-N_0-1)\sigma\ge 1$. 
Let us fix $\gamma\in ]\max (7/2,b),+\infty[\setminus 
\mez \xN$, and assume that $N_0$ is taken large enough so that 
$N_0\ge 2\gamma+\frac{13}{2}$. We define 
\be\label{3535}
s_1=\sd+N_1+1,\quad s_0=\sd+N_0-3-[\gamma]
\ee
where $s$ is an even integer taken large enough so that the following conditions hold 
\be\label{3536}
s\ge s_1\ge s_0\ge \mez (s+2\gamma)
\ee
and that moreover
\be\label{3537}
s_1\le s-a-\mez.
\ee
We set $\rho=s_0+\gamma$. 
It follows from equation (5.2.157) \todo{Ref Compagnon. ATTENTION : la ref est \`a corriger dans l'article compagnon et ici} of the companion paper~\cite{AlDel}  
that if 
$\mathcal{C}_{s_0}N_\rho^{(s_0)}=C\bigl( N_\rho^{(s_0)}\bigr)N_\rho^{(s_0)}$ 
is small enough, we have for any $k\le s_1$
$$
\blA Z^k \eta(t)\brA_{H^{s-k}}+\blA \Dxmez Z^k \psi(t)\brA_{H^{s-k-\mez}}\le B_2\eps t^{\delta_k}
$$
for a new value of the constant $B_2$. 
The smallness condition above 
is satisfied for $\eps<\eps_0'\ll 1$ using 
the second estimate \e{3534a}. Since we have set at the beginning of Section~\ref{S:32}
$$
u(t,x)=\Dxmez\psi+i\eta \text{ and }u(t,x)=\frac{1}{\sqrt{t}}v\left(t,\frac{x}{t}\right)
$$
it follows, denoting by the same notation $Z$ the vector field in $(t,x)$ and in $(t,\frac{x}{t})$-coordinates, that 
$$
\blA (hD_x)^\ell Z^{k'}v(t,\cdot)\brA_{L^2}\le B_2\eps t^{\delta_k}
$$
for $k'\le s_1$, $\ell\le a$ since $s_1+a\le s-\mez$. This, together with the definition \e{336} of $\Fr_k$, 
shows that the second condition in \e{337} holds with $k=s_1-1$. The first condition \e{337} holds because of the second estimate \e{3534a} and the fact that 
$\rho\ge \gamma>b$. Consequently, Proposition~\ref{ref:3.3.1} implies that \e{339} holds for any 
$k\le s_1-1=\sd+N_1$, with constants $A_{k'}'$ depending only on $B_2$ in \e{3534a}. The assumption \e{341} is thus satisfied, and since 
we assumed $(N_1-N_0-1)\sigma\ge 1$, we may apply Proposition~\ref{ref:342} and Corollary~\ref{ref:348} which provides development 
\e{3435}. This development is the assumption that allows one to apply the results of Section~\ref{S:35}: in particular inequality 
\e{3527} will hold, with a constant $C_2$ depending only on the constant $B_2$ of \e{3534a} (and of universal quantities). 
If $B_\infty$ is taken large enough relatively to $B_2$ and if $B_\infty'$ is larger than the constant $A_{s_0+[\gamma]+1}''$ introduced in 
Proposition~\ref{ref:354}, we deduce from \e{3527}
$$
\blA (h\px)^\ell Z^k w(t,\cdot)\bri \le \mez B_\infty \eps t^{B_\infty'\eps^2}
$$
for $k+\ell\le s_0+[\gamma]+1$ (since $\sd+N_0-2= s_0+[\gamma]+1$ by our choice \e{3535} of $s_0$). Coming back to the expression 
of $u=\Dxmez\psi+i\eta$ from $t^{-\mez} v$, and using that by definition $\rho=s_0+\gamma$ this will give the bound
\be\label{3537a}
N_\rho^{(s_0)}(t)\le \mez B_\infty \eps t^{-\mez+B_\infty'\eps^2}
\ee
if we prove that in the decomposition $v=v_L+w+v_H$, the contributions $v_L$ and $v_H$ satisfy also a bound of the form
\be\label{3538}
\blA Z^k v_H(t,\cdot)\brA_{\eC{\rho-k}}+\blA Z^k v_L(t,\cdot)\brA_{\eC{\rho-k}}\le \frac{1}{4} B_\infty\eps t^{B_\infty'\eps^2} 
\ee
if $k\le s_0$. Since our assumption \e{3534a} implies that \e{337} holds (with constants $A_{k'}$ depending only on $B_2$), for 
$k'\le s_1$, we deduce from \e{336} and the definition \e{3218} of $v_L$ that 
$$
\blA Z^k v_L(t,\cdot)\brA_{L^2}\le \eps A_k h^{-\delta_k}, \quad k\le s_1.
$$
Since $v_L$ is spectrally supported for $h|\xi|=O\big(h^{2(1-\sigma)}\big)$, we deduce from that by Sobolev injection that
\be\label{3539}
\blA Z^k v_L(t,\cdot)\bri =O\big(\eps h^{-\delta_k+\mez-\frac{\sigma}{2}}\big)
\ee
with constants depending only on $B_2$, which gives for $v_L$ a better estimate than the one \e{3538} we are looking for 
(since $v_L$ is spectrally supported for small frequencies, estimating $L^\infty$ or $\eC{\rho-k}$ norms is equivalent). 

Consider next the $v_H$-contribution. As \e{339} holds for $k=s_1-1$ with constants depending only on $B_2$, we may write for any 
$j\ge j_0(h,C)$, any $k,\ell$ with $k+\ell\le s_1-1$,
$$
\blA \Djh (hD)^\ell Z^kv_H(t,\cdot)\bri \le C\eps 2^{-j_+b}h^{-\delta_k'}.
$$
This holds in particular for $k+\ell \le s_0+\gamma+1$ as $s_0+\gamma+1\le s_1-1$ 
by \e{3535}. Since $v_H$ is spectrally supported for $|h\xi|\ge ch^{-\beta}$, we conclude that
\be\label{3540}
\blA Z^k v_H(t,\cdot)\brA_{\eC{\rho-k}}\le C \eps h^{b\beta-\delta_k'}\le C\eps h^{2-\delta_k'}
\ee
with a constant $C$ depending only on $B_2$, as we assumed in \e{334} that $b\beta>2$. This largely 
implies estimate \e{3538} for $v_H$, and so concludes the proof of \e{3537a}. 

We thus have obtained the first inequality \e{138}. We are left with showing the second estimate. This follows from \e{3517} that holds 
for $\ell\le \sd+N_0$, so for $\ell\le s_0+\gamma+1$. This concludes the proof of Theorem~\ref{ref:132}. 
\end{proof}

{\sc{Final remark on the proof of Theorem~\ref{ref:122}:}} In Section~\ref{S:13}, we did not justify the asymptotic expansion \e{127} of 
$u(t,x)=\frac{1}{\sqrt{t}}v\big(t,\frac{x}{t}\big)$. This follows from \e{3518a}, since we have seen in the proof above that in the decomposition $v=v_L+w+v_H$, 
$v_L$ and $v_H$ are $O(\eps t^{-\kappa})$ for some $\kappa>0$ (see \e{3539} and \e{3540}).

%% file: appendix2.tex


\renewcommand{\thesection}{A}
\renewcommand{\theequation}{\thesection.\arabic{equation}}

\section{Appendix: Semi-classical pseudo-differential operators}\label{S:A.2}

We recall here some definitions and results concerning semi-classical pseudo-differential operators 
in one dimension. We refer to the books of Dimassi-Sj\"ostrand~\cite{DiSj}  Martinez~\cite{Ma} and Zworski~\cite{Zw}. 

Let $h$ be a parameter in $]0,1]$. An order function $m$ is a 
function $m\colon (x,\xi)\mapsto m(x,\xi)$ from $T^*\xR$ (identified with $\xR\times \xR$) 
to $\xR_{+}$, smooth, such that there are constants $N_{0}\in\xN$, $C_{0}>0$ with
$$
m(x,\xi)\le C_0(1+|x-y|+|\xi-\eta|)^{N_0}m(y,\eta)
$$
for any $(x,\xi)$, $(y,\eta)$ in $T^*\xR$. 

\begin{defi}\label{ref:A.2.1}
Let $m$ be an order function on $T^*\xR$. One denotes\index{Symbols!$S(m)$} by $S(m)$ the set of functions $a\colon T^*\xR\times ]0,1]\rightarrow \xC$, 
$(x,\xi,h)\mapsto a(x,\xi,h)$ such that for any $(\alpha,\beta)$ in $\xN\times \xN$, there is $C_{\alpha\beta}>0$, 
and for any $(x,\xi)\in T^*\xR$, any $h$ in $]0,1]$
$$
\la \px^\alpha\partial_\xi^\beta a(x,\xi,h)\ra\le C_{\alpha\beta}m(x,\xi).
$$
\end{defi}

If $(u_h)_h$ is a family indexed by $h\in]0,1]$ of elements of $\mathcal{S}'(\xR)$, and $a\in S(m)$, we define a family of elements of $\mathcal{S}'(\xR)$ by\index{Pseudo-differential operators!$\Oph(a)$, semi-classical operators}
\be\label{A21}
\Oph(a)u_h=\frac{1}{2\pi}\int_\xR e^{ix\xi}a(x,h\xi,h)\widehat{u_h}(\xi)\, d\xi.
\ee
If $m\equiv 1$, $\Oph(a)$ is a bounded family indexed by $h\in ]0,1]$ of bounded operators in $L^2(\xR)$. If moreover 
$\xi\mapsto a(x,\xi,h)$ is supported in a compact subset independent of $(x,h)$, the kernel of 
$\Oph(a)$, is
$$
K_h(x,y)=\frac{1}{h} k_h\left(x,\frac{x-y}{h}\right)
$$
where $k_h(x,z)=(\mathcal{F}_\xi^{-1} a)(x,z,h)$ is a smooth function satisfying 
estimates $\la\px^\alpha\partial_z^\beta k_h(x,z)\ra\le C_{\alpha\beta N}(1+|z|)^{-N}$ for any 
$\alpha,\beta,N$ so that $\Oph(a)$ is uniformly bounded on any $L^p$-space, $p\in [1,\infty]$. 

Let us recall the main result of symbolic calculus (Theorem~$7.9$,  
Proposition~$7.7$, formulas 
(7.16) and (7.3) in \cite{DiSj}).

\begin{theo}\label{ref:A.2.2}
Let $m_1,m_2$ be two order functions, $a_j$ an 
element of $S(m_j)$, $j=1,2$. There is an element $a_1\# a_2$ of $S(m_1m_2)$ such that 
$\Oph(a_1\# a_2)=\Oph(a_1)\Oph(a_2)$. Moreover, one has the expansion
\be\label{A22}
a_1\# a_2 -\sum_{j=0}^N \frac{1}{j!} \left(\frac{h}{i}\right)^j 
(\partial_\xi^j a_1)(\partial_x^j a_2)\in h^{N+1} S(m_1m_2).
\ee
Let $m$ be an order function, $a$ an element 
of $S(m)$. There is $b$ in $S(m)$ such that 
$\Oph(a)^*=\Oph(b)$. Moreover, 
$b=\bar{a}+hb_1$ with $b_1$ in $S(m)$. 
\end{theo}
\begin{coro}\label{ref:A.2.3}
Let $m$ be an order function such that $m^{-1}$ is also an order function. Let $a$ be in $S(1)$, $e$ be in $S(m)$ and 
assume that $e\ge cm$ for some $c>0$ on a neighborhood of the support of $a$. 
Then for any $N\in\xN$, there are $q\in S(m^{-1})$, $r\in S(1)$ such that $a=e\# q+h^Nr$ (resp.\ $a=q\# e+h^N r$). 
Moreover, we may write $q=q_0+h q_1$ where $q_0,q_1$ are in $S(m^{-1})$ and $q_0=\frac{a}{e}$.
\end{coro}
\begin{proof}
We define $q_0=\frac{a}{e}$, which is an element of $S(m^{-1})$ by assumption. 
Then Theorem~\ref{ref:A.2.2} shows that $a-e\# q_0$ (resp.\ $a-q_0\# e$) may be written 
$ha_1+h^N r_0$ with $a_1$ in $S(1)$, $\supp a_1 \subset \supp a$. We iterate the construction to 
get the result.
\end{proof}

%% file: article.bbl
\begin{thebibliography}{10}

\bibitem{ABZ3}
T.~Alazard, N.~Burq, and C.~Zuily.
\newblock On the {C}auchy problem for gravity water waves.
\newblock arXiv:1212.0626.

\bibitem{ABZ1}
T.~Alazard, N.~Burq, and C.~Zuily.
\newblock On the water-wave equations with surface tension.
\newblock {\em Duke Math. J.}, 158(3):413--499, 2011.

\bibitem{ABZ2}
T.~Alazard, N.~Burq, and C.~Zuily.
\newblock Strichartz estimates for water waves.
\newblock {\em Ann. Sci. \'Ec. Norm. Sup\'er. (4)}, 44(5):855--903, 2011.

\bibitem{Bertinoro}
T.~Alazard, N.~Burq, and C.~Zuily.
\newblock The water-wave equations: from {Z}akharov to {E}uler.
\newblock In {\em Studies in Phase Space Analysis with Applications to PDEs},
  pages 1--20. Springer, 2013.

\bibitem{AlDel}
T.~Alazard and J.-M. Delort.
\newblock Sobolev estimates for two dimensional gravity water waves.
\newblock Preprint, 2013.

\bibitem{AM}
T.~Alazard and G.~M{\'e}tivier.
\newblock Paralinearization of the {D}irichlet to {N}eumann operator, and
  regularity of three-dimensional water waves.
\newblock {\em Comm. Partial Differential Equations}, 34(10-12):1632--1704,
  2009.

\bibitem{Alipara}
S.~Alinhac.
\newblock Paracomposition et op\'erateurs paradiff\'erentiels.
\newblock {\em Comm. Partial Differential Equations}, 11(1):87--121, 1986.

\bibitem{ASL}
B.~Alvarez-Samaniego and D.~Lannes.
\newblock Large time existence for 3{D} water-waves and asymptotics.
\newblock {\em Invent. Math.}, 171(3):485--541, 2008.

\bibitem{Beale}
J.~T. Beale.
\newblock Large-time regularity of viscous surface waves.
\newblock {\em Arch. Rational Mech. Anal.}, 84(4):307--352, 1983/84.

\bibitem{BO}
T.~B. Benjamin and P.~J. Olver.
\newblock Hamiltonian structure, symmetries and conservation laws for water
  waves.
\newblock {\em J. Fluid Mech.}, 125:137--185, 1982.

\bibitem{BG}
K.~Beyer and M.~G{\"u}nther.
\newblock On the {C}auchy problem for a capillary drop. {I}. {I}rrotational
  motion.
\newblock {\em Math. Methods Appl. Sci.}, 21(12):1149--1183, 1998.

\bibitem{Bourdaud}
G.~Bourdaud.
\newblock Realizations of homogeneous {S}obolev spaces.
\newblock {\em Complex Var. Elliptic Equ.}, 56(10-11):857--874, 2011.

\bibitem{BGSW}
B.~Buffoni, M.~D. Groves, S.-M. Sun, and E.~Wahl{\'e}n.
\newblock Existence and conditional energetic stability of three-dimensional
  fully localised solitary gravity-capillary water waves.
\newblock {\em J. Differential Equations}, 254(3):1006--1096, 2013.

\bibitem{CCFGGS-arxiv}
A.~Castro, D.~C{\'o}rdoba, C.~Fefferman, F.~Gancedo, and J.~G{\'o}mez-Serrano.
\newblock Finite time singularities for the free boundary incompressible
  {E}uler equations.
\newblock arXiv:1112.2170.

\bibitem{CCFGGS-PNAS}
A.~Castro, D.~C{\'o}rdoba, C.~L. Fefferman, F.~Gancedo, and
  J.~G{\'o}mez-Serrano.
\newblock Splash singularity for water waves.
\newblock {\em Proc. Natl. Acad. Sci. USA}, 109(3):733--738, 2012.

\bibitem{CCFGLF}
A.~Castro, D.~C{\'o}rdoba, C.~L. Fefferman, F.~Gancedo, and
  M.~L{\'o}pez-Fern{\'a}ndez.
\newblock Turning waves and breakdown for incompressible flows.
\newblock {\em Proc. Natl. Acad. Sci. USA}, 108(12):4754--4759, 2011.

\bibitem{Cauchy}
A.~L. Cauchy.
\newblock Th\'eorie de la propagation des ondes \`a la surface d'un fluide
  pesant d'une profondeur ind\'efinie.
\newblock p.5-318. M{\'e}moires pr{\'e}sent{\'e}s par divers savants {\`a}
  l'Acad{\'e}mie royale des sciences de l'Institut de France et imprim{\'e}s
  par son ordre. Sciences math{\'e}matiques et physiques. Tome I, imprim{\'e}
  par autorisation du Roi {\`a} l'Imprimerie royale; 1827. Disponible sur le
  site http://mathdoc.emath.fr/.

\bibitem{CMSW}
R.~M. Chen, J.~L. Marzuola, D.~Spirn, and J.~D. Wright.
\newblock On the regularity of the flow map for the gravity-capillary
  equations.
\newblock {\em J. Funct. Anal.}, 264(3):752--782, 2013.

\bibitem{CHS}
H.~Christianson, V.~M. Hur, and G.~Staffilani.
\newblock Strichartz estimates for the water-wave problem with surface tension.
\newblock {\em Comm. Partial Differential Equations}, 35(12):2195--2252, 2010.

\bibitem{CCG}
A.~C{\'o}rdoba, D.~C{\'o}rdoba, and F.~Gancedo.
\newblock Interface evolution: water waves in 2-{D}.
\newblock {\em Adv. Math.}, 223(1):120--173, 2010.

\bibitem{CS2}
D.~Coutand and S.~Shkoller.
\newblock On the finite-time splash and splat singularities for the 3-d
  free-surface {E}uler equations.
\newblock arXiv:1201.4919.

\bibitem{CS}
D.~Coutand and S.~Shkoller.
\newblock Well-posedness of the free-surface incompressible {E}uler equations
  with or without surface tension.
\newblock {\em J. Amer. Math. Soc.}, 20(3):829--930 (electronic), 2007.

\bibitem{Craig1985}
W.~Craig.
\newblock An existence theory for water waves and the {B}oussinesq and
  {K}orteweg-de {V}ries scaling limits.
\newblock {\em Comm. Partial Differential Equations}, 10(8):787--1003, 1985.

\bibitem{Craig96}
W.~Craig.
\newblock Birkhoff normal forms for water waves.
\newblock In {\em Mathematical problems in the theory of water waves ({L}uminy,
  1995)}, volume 200 of {\em Contemp. Math.}, pages 57--74. Amer. Math. Soc.,
  Providence, RI, 1996.

\bibitem{Craig02}
W.~Craig.
\newblock Non-existence of solitary water waves in three dimensions.
\newblock {\em R. Soc. Lond. Philos. Trans. Ser. A Math. Phys. Eng. Sci.},
  360(1799):2127--2135, 2002.
\newblock Recent developments in the mathematical theory of water waves
  (Oberwolfach, 2001).

\bibitem{CSS}
W.~Craig, U.~Schanz, and C.~Sulem.
\newblock The modulational regime of three-dimensional water waves and the
  {D}avey-{S}tewartson system.
\newblock {\em Ann. Inst. H. Poincar\'e Anal. Non Lin\'eaire}, 14(5):615--667,
  1997.

\bibitem{CrSu}
W.~Craig and C.~Sulem.
\newblock Numerical simulation of gravity waves.
\newblock {\em J. Comput. Phys.}, 108(1):73--83, 1993.

\bibitem{CrSuSu}
W.~Craig, C.~Sulem, and P.-L. Sulem.
\newblock Nonlinear modulation of gravity waves: a rigorous approach.
\newblock {\em Nonlinearity}, 5(2):497--522, 1992.

\bibitem{CW}
W.~Craig and C.~E. Wayne.
\newblock Mathematical aspects of surface waves on water.
\newblock {\em Uspekhi Mat. Nauk}, 62(3(375)):95--116, 2007.

\bibitem{Craik}
A.~D.~D. Craik.
\newblock {\em Wave interactions and fluid flows}.
\newblock Cambridge Monographs on Mechanics and Applied Mathematics. Cambridge
  University Press, Cambridge, 1988.

\bibitem{Darrigol}
O.~Darrigol.
\newblock {\em Worlds of flow}.
\newblock Oxford University Press, New York, 2005.
\newblock A history of hydrodynamics from the Bernoullis to Prandtl.

\bibitem{Delort}
J.-M. Delort.
\newblock Existence globale et comportement asymptotique pour l'\'equation de
  {K}lein-{G}ordon quasi lin\'eaire \`a donn\'ees petites en dimension 1.
\newblock {\em Ann. Sci. \'Ecole Norm. Sup. (4)}, 34(1), 2001.
\newblock Erratum : Ann. Sci. \'Ecole Norm. Sup. (4) 39 (2006), no. 2,
  335--345.

\bibitem{DiSj}
M.~Dimassi and J.~Sj{\"o}strand.
\newblock {\em Spectral asymptotics in the semi-classical limit}, volume 268 of
  {\em London Mathematical Society Lecture Note Series}.
\newblock Cambridge University Press, Cambridge, 1999.

\bibitem{GMS2}
P.~Germain, N.~Masmoudi, and J.~Shatah.
\newblock Global existence for capillary water waves.
\newblock arXiv:1210.1601.

\bibitem{GMS}
P.~Germain, N.~Masmoudi, and J.~Shatah.
\newblock Global solutions for the gravity water waves equation in dimension 3.
\newblock {\em Ann. of Math. (2)}, 175(2):691--754, 2012.

\bibitem{GT}
Y.~Guo and I.~Tice.
\newblock Decay of viscous surface waves without surface tension.
\newblock arXiv:1011.5179.

\bibitem{HayashiNaumkin}
N.~Hayashi and P.~I. Naumkin.
\newblock Asymptotics of small solutions to nonlinear {S}chr\"odinger equations
  with cubic nonlinearities.
\newblock {\em Int. J. Pure Appl. Math.}, 3(3):255--273, 2002.

\bibitem{Hormander}
L.~H{\"o}rmander.
\newblock {\em Lectures on nonlinear hyperbolic differential equations},
  volume~26 of {\em Math\'ematiques \& Applications (Berlin) [Mathematics \&
  Applications]}.
\newblock Springer-Verlag, Berlin, 1997.

\bibitem{Hur1}
V.~M. Hur.
\newblock No solitary waves exist on 2d deep water.
\newblock arXiv:1209.1926.

\bibitem{IonescuPusateri}
A.~Ionescu and F.~Pusateri.
\newblock Global solutions for the gravity water waves system in 2d.
\newblock arXiv:1303.5357.

\bibitem{IonescuPusateri0}
A.~Ionescu and F.~Pusateri.
\newblock Nonlinear fractional {S}chr\"odinger equations in one dimension.
\newblock arXiv:1209.4943.

\bibitem{IP-SW1}
G.~Iooss and P.~Plotnikov.
\newblock Multimodal standing gravity waves: a completely resonant system.
\newblock {\em J. Math. Fluid Mech.}, 7(suppl. 1):S110--S126, 2005.

\bibitem{IP}
G.~Iooss and P.~I. Plotnikov.
\newblock Small divisor problem in the theory of three-dimensional water
  gravity waves.
\newblock {\em Mem. Amer. Math. Soc.}, 200(940):viii+128, 2009.

\bibitem{Kl2}
S.~Klainerman.
\newblock Global existence of small amplitude solutions to nonlinear
  {K}lein-{G}ordon equations in four space-time dimensions.
\newblock {\em Comm. Pure Appl. Math.}, 38(5):631--641, 1985.

\bibitem{Kl}
S.~Klainerman.
\newblock Uniform decay estimates and the {L}orentz invariance of the classical
  wave equation.
\newblock {\em Comm. Pure Appl. Math.}, 38(3):321--332, 1985.

\bibitem{LannesBourbaki}
D.~Lannes.
\newblock Space time resonances [after {G}ermain, {M}asmoudi, {S}hatah].
\newblock S{\'e}minaire BOURBAKI 64{\`e}me ann{\'e}e, 2011-2012, no 1053.

\bibitem{LannesLivre}
D.~Lannes.
\newblock Water waves: mathematical analysis and asymptotics.
\newblock To appear.

\bibitem{LannesJAMS}
D.~Lannes.
\newblock Well-posedness of the water-waves equations.
\newblock {\em J. Amer. Math. Soc.}, 18(3):605--654 (electronic), 2005.

\bibitem{LannesKelvin}
D.~Lannes.
\newblock A {S}tability {C}riterion for {T}wo-{F}luid {I}nterfaces and
  {A}pplications.
\newblock {\em Arch. Ration. Mech. Anal.}, 208(2):481--567, 2013.

\bibitem{LindbladAnnals}
H.~Lindblad.
\newblock Well-posedness for the motion of an incompressible liquid with free
  surface boundary.
\newblock {\em Ann. of Math. (2)}, 162(1):109--194, 2005.

\bibitem{Ma}
A.~Martinez.
\newblock {\em An introduction to semiclassical and microlocal analysis}.
\newblock Universitext. Springer-Verlag, New York, 2002.

\bibitem{MR}
N.~Masmoudi and F.~Rousset.
\newblock Uniform regularity and vanishing viscosity limit for the free surface
  {N}avier-{S}tokes equations.
\newblock arXiv:1202.0657.

\bibitem{MRT}
M.~Ming, F.~Rousset, and N.~Tzvetkov.
\newblock Multi-solitons and related solutions for the water-waves system.
\newblock arXiv:1304.5263.

\bibitem{Nalimov}
V.~I. Nalimov.
\newblock The {C}auchy-{P}oisson problem.
\newblock {\em Dinamika Splo\v sn. Sredy}, (Vyp. 18 Dinamika Zidkost. so
  Svobod. Granicami):104--210, 254, 1974.

\bibitem{SW}
G.~Schneider and C.~E. Wayne.
\newblock The rigorous approximation of long-wavelength capillary-gravity
  waves.
\newblock {\em Arch. Ration. Mech. Anal.}, 162(3):247--285, 2002.

\bibitem{Shatah}
J.~Shatah.
\newblock Normal forms and quadratic nonlinear {K}lein-{G}ordon equations.
\newblock {\em Comm. Pure Appl. Math.}, 38(5):685--696, 1985.

\bibitem{SZ}
J.~Shatah and C.~Zeng.
\newblock Geometry and a priori estimates for free boundary problems of the
  {E}uler equation.
\newblock {\em Comm. Pure Appl. Math.}, 61(5):698--744, 2008.

\bibitem{SZ2}
J.~Shatah and C.~Zeng.
\newblock A priori estimates for fluid interface problems.
\newblock {\em Comm. Pure Appl. Math.}, 61(6):848--876, 2008.

\bibitem{Shinbrot}
M.~Shinbrot.
\newblock The initial value problem for surface waves under gravity. {I}. {T}he
  simplest case.
\newblock {\em Indiana Univ. Math. J.}, 25(3):281--300, 1976.

\bibitem{SuSu}
C.~Sulem and P.-L. Sulem.
\newblock {\em The nonlinear {S}chr\"odinger equation}, volume 139 of {\em
  Applied Mathematical Sciences}.
\newblock Springer-Verlag, New York, 1999.
\newblock Self-focusing and wave collapse.

\bibitem{Sun}
S.-M. Sun.
\newblock Some analytical properties of capillary-gravity waves in two-fluid
  flows of infinite depth.
\newblock {\em Proc. Roy. Soc. London Ser. A}, 453(1961):1153--1175, 1997.

\bibitem{TW}
N.~Totz and S.~Wu.
\newblock A rigorous justification of the modulation approximation to the 2{D}
  full water wave problem.
\newblock {\em Comm. Math. Phys.}, 310(3):817--883, 2012.

\bibitem{WuInvent}
S.~Wu.
\newblock Well-posedness in {S}obolev spaces of the full water wave problem in
  2-{D}.
\newblock {\em Invent. Math.}, 130(1):39--72, 1997.

\bibitem{WuJAMS}
S.~Wu.
\newblock Well-posedness in {S}obolev spaces of the full water wave problem in
  3-{D}.
\newblock {\em J. Amer. Math. Soc.}, 12(2):445--495, 1999.

\bibitem{Wu09}
S.~Wu.
\newblock Almost global wellposedness of the 2-{D} full water wave problem.
\newblock {\em Invent. Math.}, 177(1):45--135, 2009.

\bibitem{Wu10}
S.~Wu.
\newblock Global wellposedness of the 3-{D} full water wave problem.
\newblock {\em Invent. Math.}, 184(1):125--220, 2011.

\bibitem{Yosihara}
H.~Yosihara.
\newblock Gravity waves on the free surface of an incompressible perfect fluid
  of finite depth.
\newblock {\em Publ. Res. Inst. Math. Sci.}, 18(1):49--96, 1982.

\bibitem{Zakharov1968}
V.~E. Zakharov.
\newblock Stability of periodic waves of finite amplitude on the surface of a
  deep fluid.
\newblock {\em Journal of Applied Mechanics and Technical Physics},
  9(2):190--194, 1968.

\bibitem{Zw}
M.~Zworski.
\newblock {\em Semiclassical analysis}, volume 138 of {\em Graduate Studies in
  Mathematics}.
\newblock American Mathematical Society, Providence, RI, 2012.

\end{thebibliography}
